\documentclass[preprint,11pt]{elsarticle}

\usepackage{amsfonts, amsmath, amscd}
\usepackage[psamsfonts]{amssymb}

\usepackage{amssymb}

\usepackage{pb-diagram}

\usepackage[all,cmtip]{xy}

\usepackage[usenames]{color}

\newtheorem{theorem}{Theorem}[section]
\newtheorem{lemma}[theorem]{Lemma}
\newtheorem{corollary}[theorem]{Corollary}

\newtheorem{remark}[theorem]{Remark}
\newtheorem{proposition}[theorem]{Proposition}
\newtheorem{definition}[theorem]{Definition}
\newtheorem{example}[theorem]{Example}

\newtheorem{notation}[theorem]{Notation}
\newtheorem{problem}[theorem]{Problem}

\newproof{proof}{Proof}

\numberwithin{equation}{section}
\numberwithin{theorem}{section}

\newcommand{\e}{\varepsilon}
\newcommand{\w}{\omega}

\newcommand{\NN}{\mathbb{N}}

\newcommand{\IR}{\mathbb{R}}

\newcommand{\ff}{\mathbb{F}}

\newcommand{\IF}{\mathbb{F}}

\newcommand{\xxx}{\mathbf{x}}

\newcommand{\aaa}{\mathbf{a}}

\newcommand{\id}{\mathrm{id}}

\newcommand{\GG}{\mathfrak{G}}

\newcommand{\TTT}{\mathcal{T}}

\newcommand{\BB}{\mathcal{B}}

\newcommand{\KK}{\mathcal{K}}
\newcommand{\Nn}{\mathcal{N}}
\newcommand{\AAA}{\mathcal{A}}

\newcommand{\LL}{\mathcal{L}}

\newcommand{\A}{\mathcal{A}}

\newcommand{\PPP}{\mathcal{P}}

\newcommand{\supp}{\mathrm{supp}}

\newcommand{\cl}{\mathrm{cl}}
\newcommand{\Ra}{\Rightarrow}

\newcommand{\LRa}{\Leftrightarrow}
\newcommand{\cacx}{\overline{\mathrm{acx}}}

\newcommand{\Bo}{\mathsf{Bo}}

\newcommand{\Int}{\mathsf{Int}}

\newcommand{\Id}{\mathsf{id}}

\newcommand{\ind}{\mathsf{ind}}

\newcommand{\spn}{\mathsf{span}}
\newcommand{\cspn}{\overline{\mathsf{span}}}

\newcommand{\CC}{C_k}

\newcommand{\SI}{\underrightarrow{\mbox{ s-$\mathsf{ind}$}}_n\,}
\newcommand{\SM}{{\setminus}}

\begin{document}

\begin{frontmatter}

\title{Pe{\l}czy\'{n}ski's type sets and Pe{\l}czy\'{n}ski's geometrical properties \\ of locally convex spaces}

\author{Saak~Gabriyelyan}
\ead{saak@bgu.ac.il}
\address{Department of Mathematics, Ben-Gurion University of the Negev, Beer-Sheva, P.O. 653, Israel}

\begin{abstract}
For $1\leq p\leq q\leq\infty$ and a locally convex space $E$, we introduce and study the $(V^\ast)$ subsets of  order $(p,q)$ of $E$ and the $(V)$ subsets of order $(p,q)$ of the topological dual $E'$ of $E$. Using these sets we define and study the (sequential) Pe{\l}czy\'{n}ski's property $V^\ast$ of order $(p,q)$,   the (sequential) Pe{\l}czy\'{n}ski's property $V$ of order $(p,q)$, and  the Pe{\l}czy\'{n}ski's property $(u)$ of order $p$ in the class of all locally convex spaces. To this end,  we also introduce and study several new completeness type properties, weak barrelledness conditions, Schur type properties, the Gantmacher property for locally convex spaces, and $(q,p)$-summing operators between locally convex spaces. Applications to some classical function spaces are given.
\end{abstract}


\end{frontmatter}

\newpage

\tableofcontents\newpage

\section{Introduction}


It follows from a result of Orlicz \cite{Orlicz} that every weakly compact operator between Banach spaces sends every weakly unconditionally Cauchy (wuC) sequence into an unconditionally convergent (u.c.) series. Recall that a series $\sum_{n\in\w} x_n$ in a Banach space $X$ is called {\em wuC} (resp., {\em u.c.}) if for every real null (resp., bounded) sequence $\{t_n\}_{n\in\w}$, the series  $\sum_{n\in\w} t_n x_n$ is convergent. Later in \cite{Pelcz-60}, Pe{\l}czy\'{n}ski proved that every operator from a Banach space $C(K)$ into a Banach space no subspace of which is isomorphic to $c_0$ is weakly compact. A complete explanation of these results was given by Pe{\l}czy\'{n}ski in his fundamental article \cite{Pelcz-62}. Following \cite{Pelcz-62}, an operator $T:X\to Y$ between Banach spaces $X$ and $Y$ is called {\em unconditionally converging} (u.c.) if it sends wuC series in $X$ into u.c. series in $Y$. Then, for a Banach space $X$, the following conditions are equivalent: (a) for every Banach space $Y$, every u.c. operator $T:X\to Y$ is weakly compact, (b) each subset $B$ of the Banach dual $X'$ satisfying the condition
\[
\lim_{n\to\infty} \sup\{ |\langle \chi, x_n\rangle|: \chi\in B\}=0
\]
for every wuC series $\sum_n x_n$ in $X$, is sequentially compact in the weak topology of $X'$. This amazing result motivates to introduce the following classes of Banach spaces.
\begin{definition}[\cite{Pelcz-62}] \label{def:V-V*} {\em
A Banach space $X$ is said to have
\begin{enumerate}
\item[$\bullet$] the {\em property $V$} if each subset $B$ of the Banach dual $X'$ satisfying the condition
\[
\lim_{n\to\infty} \sup\{ |\langle \chi, x_n\rangle|: \chi\in B\}=0
\]
for every wuC series $\sum_n x_n$ in $X$, is a weakly sequentially compact subset of $X'$;
\item[$\bullet$] the {\em property $V^\ast$} if each subset $A$ of $X$ satisfying the condition
\begin{equation} \label{equ:V*-1}
\lim_{n\to\infty} \sup\{ |\langle \chi_n, x\rangle|: x\in A\}=0
\end{equation}
for every wuC series $\sum_n \chi_n$ in $X'$, is weakly sequentially compact.\qed
\end{enumerate} }
\end{definition}

The following remarkable assertion was proved by  Pe{\l}czy\'{n}ski.

\begin{theorem}[\cite{Pelcz-62}] \label{t:Pel-C(K)}
For every compact space $K$, the Banach space $C(K)$ has the property $V$.
\end{theorem}
As a corollary he proved that every abstract $L$-space has the property $(V^\ast)$. It should be mentioned that a Banach space is reflexive if and only if it has both properties $(V)$ and $(V^\ast)$, see Proposition 7 in \cite{Pelcz-62}.

Since numerous important results were obtained by many authors. A characterization of Banach spaces with the property $V^\ast$ was obtained by Emmanuele \cite{Emmanuele-V}. Randrianantoanina \cite{Randrian} extended Theorem \ref{t:Pel-C(K)} by showing that if $K$ is a compact space and $E$ is a separable Banach space, then the Banach space $C(K,E)$ has the property $V$ if and only if $E$ has the property $V$. Pfitzner proved in \cite{Pfitzner} that every $C^\ast$-algebra has the property $V$. In \cite{GS-N},  Godefroy and Saab showed that any Banach space which is an $M$-ideal in its bidual has the property $V$. Saab and Saab proved in \cite{Saab-Saab} that a Banach lattice $X$ has the property $V^\ast$ if and only if $X$ contains no subspaces isomorphic to $c_0$. Answering a problem of Godefroy and Saab \cite{GS}, Castillo and Gonz\'{a}lez showed in \cite{CasGon} that the property $V$ is not a three-space property. It turns out that a lot of important concrete Banach spaces have the property $V$, see for example Bourgain \cite{Bourgain-83}, Delbaen \cite{Delbaen} and Kislyakov \cite{Kislyakov}. Johnson and Zippin showed in \cite{JZ} that all real preduals have the property $V$. Chu and Mellon proved in \cite{Chu-M} that all $JB^\ast$-triples have the property $V$.

In \cite{Bombal}, Bombal introduced the weak version of the property $V^\ast$: A Banach space $X$ has the {\em weak $V^\ast$} if each subset $A$ of $X$ satisfying the condition (\ref{equ:V*-1})  is weakly sequentially precompact. A characterization of Banach spaces with the weak $V^\ast$ was given by Saab and Saab in \cite[p.~530]{Saab-Saab}. An operator characterization of $V$-sets and sufficient and necessary conditions to have the property $V$ were obtained recently by Cilia and Emmanuele \cite{Cilia-Em}.

An important tool in the study of subspaces of Banach spaces and Banach spaces with unconditional bases is  Pe{\l}czy\'{n}ski's property $(u)$.
\begin{definition}[\cite{Pelcz-58}] \label{def:u-Banach} {\em
A Banach space $X$ is said to have the {\em property $(u)$} if, for every  weakly Cauchy sequence $\{y_n\}_{n\in\w}$ in $X$, there exists a sequence $\{x_n\}_{n\in\w}$ in $X$ such that
\begin{enumerate}
\item[{\rm(i)}] the series $\sum_{n\in\w} x_n$ is weakly unconditionally convergent, i.e. $\sum_{n\in\w} |\langle \chi,x_n\rangle|<\infty$  for every $\chi\in X'$,
\item[{\rm(ii)}] the sequence $\{y_n-\sum_{j\leq n} x_j\}_{n\in\w}$ is weakly null. \qed
\end{enumerate} }
\end{definition}
In \cite{Pelcz-58}  Pe{\l}czy\'{n}ski proved that every closed subspace of a Banach space with an unconditional basis has the property $(u)$. For Banach lattices a similar result was proved by Tzafriri \cite{Tzafriri}: any order continuous Banach lattice $E$ has the property $(u)$. It follows from a result of  Karlin \cite{Karlin} that the Banach space $C[0,1]$ cannot be embedded in a space with unconditional basis. In \cite{God-Li}, Godefroy and Li proved that Banach spaces which are $M$-ideals in their bidual have the property $(u)$. Castillo and Gonz\'{a}lez showed in \cite{CasGon} that the property $(u)$ is not a three-space property. There is a natural connection between the property $(u)$ and the property $V$ found by Pe{\l}czy\'{n}ski \cite{Pelcz-62}: If a  Banach space $X$ has the property $(u)$ and has no isomorphic copy of $\ell_1$, then $X$ has the property $V$. Cembranos, Kalton, Saab and  Saab proved in \cite{CKSS} that for every compact space $K$ and each Banach space $X$ which does not contain an isomorphic copy of $\ell_1$ and with the property $(u)$, the space $C(K,X)$ has the property $V$. In \cite{Rosen-94}, Rosenthal proved that a Banach space $X$ has the property $(u)$ and has no isomorphic copy of $\ell_1$ if and only if each closed subspace of $X$ has the property $V$.

For other important results related to  Pe{\l}czy\'{n}ski's properties  $V$,  $V^\ast$ and  $(u)$, we refer the reader to the articles \cite{Bombal,CKSS,ChengHe,Cilia-Em,Emmanuele-V,Emmanuele-DP,GS-N,Pfitzner-93,Randrian,Randrian-V*,Saab-Saab-86} and the classical books \cite{Al-Kal,CasGon-Book,DJT,HWW,LT-1,LT-2}.

Let $p\in[1,\infty]$, $E$ be a locally convex space (lcs for short),  and let $E'$ be the topological dual of $E$. Recall (see Section 19.4 in \cite{Jar} and \cite{CasSim}) that a sequence $\{x_n\}_{n\in\w}$ in $E$ is called {\em weakly $p$-summable} or a {\em weak $\ell_p$-sequence} if  for every $\chi\in E'$ it follows that $(\langle\chi,x_n\rangle)\in \ell_p$ if $p\in[1,\infty)$ and $(\langle\chi,x_n\rangle)\in c_0$ if $p=\infty$.
The family of all weakly $p$summable sequences in $E$ is denoted by $\ell_p^w(E)$ or $c_0^w(E)$ if $p=\infty$.

Unifying the notion of u.c. operator and the notion of completely continuous operators (i.e., they transform weakly null sequences into norm null), Castillo and S\'{a}nchez selected in \cite{CS} the class of $p$-convergent operators. An operator $T:X\to Y$ between Banach spaces is called {\em $p$-convergent} if it transforms weakly $p$-summable sequences into norm null sequences. Using this notion they introduced and study Banach spaces with the Dunford--Pettis property of order $p$ ($DPP_p$ for short)  for every $p\in[1,\infty]$. A Banach space $X$ is said to have the $DPP_p$ if every weakly compact operator from $X$ into a Banach space $Y$ is $p$-convergent.

The influential article  of Castillo and S\'{a}nchez \cite{CS} inspired an intensive study of $p$-versions of numerous geometrical properties of Banach spaces,  in particular, Pe{\l}czy\'{n}ski's properties  $V$ and  $V^\ast$ of order $p$ \cite{CCDL,LCCD}, the Gelfand--Phillips property of order $p$ \cite{Ghenciu-pGP}, the Schur property of order $p$ \cite{DM,FZ-pL} and others. Let us recall the definitions of $V^\ast$ sets of order $p$ and $V$ sets of order $p$ which  were defined and studied by Chen, Ch\'{a}vez-Dom\'{\i}nguez, and Li in \cite{LCCD} and \cite{CCDL}, respectively.

\begin{definition} \label{def:small-bounded-p} {\em
Let $p\in[1,\infty]$, $X$ be a Banach space, and let $A$ and $B$ be bounded subsets of $X$ and the Banach dual $X'$ of $X$, respectively. Then:
\begin{enumerate}
\item[{\rm(i)}] $A$ is called a {\em $p$-$(V^\ast)$ set} if
\[
\lim_{n\to\infty} \sup_{a\in A} |\langle\chi_n,a\rangle|=0\;\;
\mbox{ for every $(\chi_n)\in \ell_p^w (X')$ \big(or $(\chi_n)\in c_0^w(X')$ if $p=\infty$\big);}
\]
\item[{\rm(ii)}] $B$ is called a {\em $p$-$(V)$ set} if
\[
\lim_{n\to\infty} \sup_{\chi\in B} |\langle\chi,x_n\rangle|=0\;\;
\mbox{ for every $(x_n)\in \ell_p^w (X)$ \big(or $(x_n)\in c_0^w(X)$ if $p=\infty$\big).\qed}
\]
\end{enumerate}}
\end{definition}
The notion of $p$-$(V)$ sets  was applied in \cite{CCDL} to characterize  $DPP_p$. Being motivated by  Pe{\l}czy\'{n}ski's properties  $V$ and  $V^\ast$, one can naturally define (to unify both notions, for the property  $V$ of order $p$ we use an equivalent definition given in  Theorem 2.4 of \cite{LCCD})
\begin{definition}[\cite{LCCD}] \label{def:Vp-property} {\em
Let $p\in[1,\infty]$. A Banach space $X$ is said to have
\begin{enumerate}
\item[{\rm(i)}] {\em Pe{\l}czy\'{n}ski's property  $V$ of order $p$} if every $p$-$(V)$ subset of $X'$ is relatively weakly compact;
\item[{\rm(ii)}] {\em Pe{\l}czy\'{n}ski's property  $V^\ast$ of order $p$} if every $p$-$(V^\ast)$ subset of $X$ is relatively weakly compact.
\end{enumerate}}
\end{definition}
Chen, Ch\'{a}vez-Dom\'{\i}nguez, and Li proved in \cite{LCCD} that for every $1<p<\infty$, the James $p$-space $J_p$ has Pe{\l}czy\'{n}ski's property  $V^\ast$ of order $p$ and the Pe{\l}czy\'{n}ski's property  $V$ of order $p^\ast$, where $p^\ast$ is the conjugate number of $p$. However, since $J_p$ does not have the property $V$ it follows that Pe{\l}czy\'{n}ski's properties $V$ of order $p$ are indeed depend on $p$. They also defined a quantitative version of the property $V^\ast$. For $p\in[1,\infty)$, the $p$-version of the weak $V^\ast$ property was introduced and thoroughly  studied by Ghenciu \cite{Ghenciu-pGP}. It is worth mentioning that a quantitative version of the property $V$ was defined and studied by Kruli\v{s}ov\'{a} \cite{Krulisova,Krulisova-C}.

The aforementioned  results motivate to introduce and study analogous type of sets and properties in the general theory of locally convex spaces. This is the main goal of the article.

Now we describe the content of the article. In Section \ref{sec:Prel} we fix basic notions and prove some necessary results used in the article.  Especially this is important for compact type sets as, for example,  (pre)compact sets or sequentially (pre)compact sets. Indeed, it is well known that a Banach space $X$ is weakly angelic, and hence (relatively) weakly compact sets are exactly (relatively) weakly sequentially compact sets. By this reason, in Banach space theory, weakly sequentially compact sets are considered usually as weakly compact. However, if $E$ is a locally convex space which is not weakly angelic, these notions are completely different. Moreover, even countable precompact sets can be not sequentially precompact, see Lemma \ref{l:seq-precom-precom}. Therefore it seems  reasonable to distinguish ``compact'' and  ``sequentially compact'' constructions as it is done by Pe{\l}czy\'{n}ski in Definition \ref{def:V-V*}.
We note also  that some of results of Section \ref{sec:Prel} are of independent interest, see for example Propositions \ref{p:bounded-covering} and \ref{p:large-charac}.

The classical notions of being a quasi-complete, sequentially complete or locally complete space will play a considerable role in the article. It should be mentioned that the study of completeness type properties of locally convex spaces  is an important area of the theory of locally convex spaces, we refer the reader to the classical books \cite{Jar}, \cite{PB}  and \cite{Wilansky} and references therein.
It turns out that  to obtain some of our main results some new completeness type properties (which lie between the quasi-completeness and the sequential completeness) appear naturally. By this reason we separate a new section, Section \ref{sec:completeness}, in which we recall classical and define  new completeness type properties. Applications to the space $C_p(X)$ of continuous functions on a Tychonoff space $X$ endowed with the topology of pointwise convergence are given. It is well known that $C_p(X)$ is quasi-complete if and only if $X$ is discrete,  see Theorem 3.6.6 of \cite{Jar}. In Theorem \ref{t:Cp-vNc}  we generalize this result by showing that the space  $C_p(X)$ is von Neumann complete if and only if $X$ is discrete.
The Buchwalter--Schmets theorem states that $C_p(X)$ is sequentially complete if and only if $X$ is a $P$-space. Recall that a Tychonoff space is called a {\em $P$-space} if any countable intersection of open sets is open. Theorem 1.1 of \cite{FKS-P} generalizes this result: the space $C_p(X)$ is locally complete if and only if $X$ is a $P$-space.
In Theorem \ref{t:Cp-lc} below we  strengthen both these results and give an independent and simpler  proof of the second one (using a  functional characterization of $P$-spaces given in Proposition \ref{p:P-space}): {\em$C_p(X)$  is  separably quasi-complete if and only if  $X$ is a $P$-space.}

In Section \ref{sec:conv-sum} we recall the definitions of weakly $p$-summable, weakly $p$-convergent and weakly $p$-Cauchy sequences in locally convex spaces. The families $\ell_p^w(E)$ and $c_0^w(E)$  of all weakly $p$-summable sequences in $E$ endowed with natural locally convex vector topologies are locally convex spaces denoted by $\ell_p[E]$ and $c_0[E]$.  We recall in detail these topologies and extend Proposition 19.4.2 of \cite{Jar} (which states that if $E$ is complete, then the spaces $\ell_p[E]$ and $c_0[E]$ are complete as well), see Proposition \ref{p:L^w-E-complete}.
If $X$ is a Banach space, the well-known result of Grothendieck \cite{Grot-56} (see also Proposition 2.2 in \cite{DJT}) states that there is the canonical isomorphism between the families $\LL(\ell_{p^\ast}, X)$ of all operators from $\ell_{p^\ast}$ into $X$ and  $\ell_p^w(X)$ defined by $T\mapsto \{T(e_n)\}_{n\in\w}$ (if $p=1$ or $p=\infty$, then between $\LL(c_{0},X)$ and $\ell_1^w(X)$ and, respectively, between $\LL(\ell_{1}, X)$ and $\ell_\infty(X)$). In general, an analogous result does not hold because an lcs $E$ can be not complete. Therefore to get a similar correspondence, instead of $\ell_p$ and $c_0$,  we consider (dense) linear spans $\ell_p^0$ and $c_0^0$ of the canonical basis, see Proposition \ref{p:Lp-E-operator}.
Similarly, for a Banach space $X$ and $p\in[1,\infty]$, the map $\LL(X,\ell_p)\to \ell_p^w(X^\ast), T\mapsto \{T^\ast(e^\ast_n)\}_{n\in\w}$, is an isomorphism playing an essential role in the study of  Pe{\l}czy\'{n}ski's properties of order $p$. However, in the general case of locally convex spaces an analogous result  is satisfied if an lcs $E$ is barrelled, see Proposition \ref{p:p-sum-operator}. A necessary condition on an lcs $E$, under which  the map $T\mapsto \{T^\ast(e^\ast_n)\}_{n\in\w}$ is an isomorphism, is the property of being a $p$-(quasi)barrelled space, see Proposition \ref{p:operator-Lp}. An lcs $E$ is called {\em $p$-(quasi)barrelled} if every weakly $p$-summable sequence in $(E,\sigma(E',E))$ (resp., in $(E,\beta(E',E))$) is equicontinuous.

Taking into account that  $p$-(quasi)barrelled spaces play a considerable role in our article we study this class of spaces in Section \ref{sec:p-barrelled} being motivated also by the fact that weak barrelledness concepts are the cornerstone in the study of general locally convex spaces, see for example the classical books \cite{Jar,Kothe,PB}. In Example \ref{exa:c0-1-barrel} we show that there exist metrizable spaces which are not $p$-barrelled, and Example \ref{exa:c0-quasi-1quasi} shows that there are $c_0$-quasibarrelled spaces which are not $p$-quasibarrelled. Therefore the $p$-(quasi)barrelledness is indeed a new weak barrelledness  notion, and hence to find conditions on an lcs $E$ under which this new concept coincides with the classical notion of $c_0$-(quasi)barrelled spaces is an interesting problem. We show in Corollary \ref{c:loc-complete-p-quasi} that if $E$ is a Mackey space, then $E$ is $c_0$-barrelled if and only if it is $p$-barrelled for some (every) $p\in[1,\infty]$. In Proposition \ref{p:p-bar-product} we show that the class of $p$-(quasi)barrelled spaces is stable under taking quotients, products, countable locally convex direct sums and completions (of large subspaces).

A property of being a Schur space is important in the study of Pe{\l}czy\'{n}ski's properties. Locally convex spaces with the Schur property are thoroughly studied in \cite{Gabr-free-resp}. For $p\in[1,\infty]$, Banach spaces with the $p$-Schur property was defined by Deghani and Moshtaghioun \cite{DM} and Fourie and Zeekoei \cite{FZ-pL}:  a Banach space $E$ has the {\em $p$-Schur property} (or $E\in C_p$, see \cite{CS}) if every weakly $p$-summable sequence is a norm null-sequence. In Section \ref{sec:p-Schur}, we introduce and briefly study locally convex spaces with the $p$-Schur property. Considering the spaces $\ell_p$ and $\ell_p^0$ we show that the property of being a $p$-Schur space is different for distinct $p$, see Proposition \ref{p:Lp-Schur}.

In Section \ref{sec:V*-sets}, for $p,q\in[1,\infty]$, we introduce and study $(p,q)$-$(V^\ast)$ and $(p,q)$-$(EV^\ast)$ sets in a locally convex space $E$ (where the letter `$E$' means `equicontinuous'). If  $E$ is a Banach space, $(p,\infty)$-$(V^\ast)$ sets are exactly $p$-$(V^\ast)$ defined in Definition \ref{def:small-bounded-p}.  To consider the double index  $(p,q)$ is motivated not only but its natural generalization of one index $p$-$(V^\ast)$ sets, but also by the fact that if $1<p<\infty$, then $(p,\infty)$-$(V^\ast)$ sets of a Banach space $X$ are exactly coarse $p$-limited sets introduced and studied by Galindo and Miranda \cite{GalMir} (this fact immediately follows from Theorem \ref{t:p-V*-set}, for more details see \cite{Gab-GP}). Therefore considering double index we unify the study of some important classes of bounded sets. Moreover, the $(q,p)$-summing operators between Banach spaces give another strong motivation to consider double index $(p,q)$, see Sections \ref{sec:oper-qp} and \ref{sec:p-summing}.
On the other hand, the reason to introduce the `equicontinuous' version of $(V^\ast)$-type sets is motivated by the non-uniqueness of extension of these sets from the case of Banach spaces to the general case of locally convex spaces.  Indeed, if $X$ is a Banach space (or, more generally, $X$ is a $p$-quasibarrelled space), then every sequence $(\chi_n)\in \ell_p^w (X')$ is automatically equicontinuous. Therefore generalizing the notion of $(V^\ast)$ sets from the Banach case to the general case of locally convex spaces  we can assume that the sequence $(\chi_n)$ satisfies one of the following conditions: (1) it is simply weakly null in the strong dual $E'_\beta$, or (2) it is weakly null in $E'_\beta$ and additionally equicontinuous. For example, to define limited sets in a  locally convex space $E$ the first condition (for weak$^\ast$ null sequences) was used in our article  \cite{BG-GP-lcs} and the second one was used by Lindstr\"{o}m and  Schlumprecht \cite{Lin-Schl-lim}. The classes of $(p,q)$-$(V^\ast)$ sets and $(p,q)$-$(EV^\ast)$ sets are saturated, closed under taking images and well-behaved under sums and products, see  Lemma \ref{l:V*-set-1} and Proposition \ref{p:product-sum-V*-set}. For every $p\in[1,\infty]$, each precompact subset $A$ of $E$ is a $p$-$(EV^\ast)$ set and if $E$ is $p$-quasibarrelled, then every precompact subset $A$ of $E$ is a $p$-$(V^\ast)$ set, see Proposition \ref{p:precompact-p-V*}.
In Theorem \ref{t:Bo=Vp} we show that for every $p\in[1,\infty]$,  every bounded subset of an lcs $E$ is a $(p,\infty)$-$(V^\ast)$ set if and only if $E'_\beta$ has the $p$-Schur property. In Theorem \ref{t:V*-set-precompact} we give an operator characterization of locally convex spaces whose $(p,q)$-$(V^\ast)$ set are precompact.

Let $p,q\in[1,\infty]$. In Section \ref{sec:V-sets} we introduce and study $(p,q)$-$(V)$ and $(p,q)$-$(EV)$ sets in the topological dual $E'$ of an lcs $E$. We show in Lemma \ref{l:V-set-1} and Proposition \ref{p:product-sum-V-set} that the classes of $(p,q)$-$(V)$ and $(p,q)$-$(EV)$ sets are saturated,  closed under taking images and well-behaved under sums and products. Also we introduce and study weak$^\ast$ $(p,q)$-$(V)$ sets and  weak$^\ast$ $(p,q)$-$(EV)$ sets in $E'$, see Lemma \ref{l:*V-set-1}. In Example \ref{exa:weak*-non-Vp-set} we show that these classes are different in general.

Let $E$ be a locally convex space. Any family of subsets in $E'$, in particular the family of all $(p,q)$-$(V)$ sets, naturally defines weak barrelledness conditions. For example, $E$ is said to be {\em $V_{(p,q)}$-$($quasi$)$barrelled} if every (resp., strongly bounded) set $B$ in $E'$ such that $B$ is a  $(p,q)$-$(V)$ set is equicontinuous. In Corollary \ref{c:Cp-Vp-barrelled} we show that for every Tychonoff space $X$, the space $C_p(X)$  is $V_{(p,q)}$-barrelled. Consequently, if $X$ contains an infinite functionally bounded subset, then the space $C_p(X)$ is $V_{(p,q)}$-barrelled but not $c_0$-barrelled. The class of locally convex spaces with such weak barrelledness conditions is closed under taking quotients, products, countable locally convex direct sums and completions, see Proposition \ref{p:Vp-bar-property}.

In Section \ref{sec:property-Vp} we introduce $V$ type properties of locally convex spaces. For $p,q\in[1,\infty]$, an lcs $E$ is said to have a {\em property $V_{(p,q)}$} (a {\em property $sV_{(p,q)}$} or a {\em weak property $sV_{(p,q)}$})  if every $(p,q)$-$(V)$ set in $E'_\beta$ is relatively weakly compact (resp., relatively weakly sequentially compact or weakly sequentially precompact). Equicontinuous versions of these properties are also defined. Taking into account that Banach spaces with the weak topology are angelic, it is clear that the classical notion of the property $V$ and and the property $V_p$ coincide with the property $sV_{(1,\infty)}$ and the property $sV_{(p,\infty)}$, respectively. Permanent properties and relationships between $V$ type properties are studied. In particular, we show that the class of spaces with these properties are closed under taking products, (countable) direct sums and completions,  see Lemma \ref{l:property-Vp-Vq} and Propositions \ref{p:product-sum-V} and \ref{p:Vp-dense}. If $H$ is a closed subspace of $E$ with one of the  $V$ type properties such that the quotient map $q:E\to E/H$ is almost bounded-covering, then the quotient space $E/H$ has the same property, see Proposition \ref{p:prop-V-quotient}. This result generalizes Corollary 1 of \cite{Pelcz-62} which states that a quotient space of a Banach space with the property $V$ has the property $V$, as well. In Theorem \ref{t:Cp-V-p} we show that for every Tychonoff space $X$, the space $C_p(X)$ has all $V$ type properties. At the end of this section we remark that in the realm of Banach spaces, the property $sV_{(p,\infty)}$ (resp., $wsV_{(p,\infty)}$) coincides with the (resp., weak) reciprocal Dunford--Pettis property of order $p$ introduced and studied by Ghenciu \cite{Ghenciu-pGP}.

In Section \ref{sec:V*-property} we define  $V^\ast$ type properties as follows: if  $p,q\in[1,\infty]$, an lcs $E$ is said to have a {\em property $V^\ast_{(p,q)}$} (a {\em property $sV^\ast_{(p,q)}$} or a {\em weak property $sV_{(p,q)}^\ast$})  if every $(p,q)$-$(V^\ast)$ set in $E$ is relatively weakly compact  (resp., relatively weakly sequentially compact or weakly sequentially precompact). Analogously we define equicontinuous versions of these properties. Therefore, if $E$ is a Banach space and $q=\infty$ we obtain the property $V^\ast_p$ and its weak version mentioned above (see \cite{Bombal} and \cite{Ghenciu-pGP}). In Lemma \ref{l:property-Vp*-Vq*}, Propositions \ref{p:product-sum-V*} and \ref{p:subspace-Vp*} we study categorical properties and relationships between different types of  $V^\ast$ type properties. In Proposition 7 of \cite{Pelcz-62}  Pe{\l}czy\'{n}ski proved that a Banach space $X$ is reflexive if and only if it has both properties $V$ and $V^\ast$. In Proposition \ref{p:semirefl-Vp*} we obtain partial extensions of this result to any locally convex space.
In Corollary \ref{c:Cp-V*-prop} we prove that if  $p,q\in[1,\infty]$ and $X$ is a Tychonoff space, then $C_p(X)$ has the property  $V^\ast_{(p,q)}$ if and only if  $X$ is discrete. We remark that for Banach spaces, the property $sV^\ast_p$ is nothing else than the reciprocal Dunford-Pettis$^\ast$ property of order $p$ introduced by Bator, Lewis and Ochoa \cite{BLO}.

In Section \ref{sec:property-u} we introduce and study locally convex spaces $E$ with the property $p$-$(u)$, where $p\in[1,\infty)$. If $1<p<\infty$, the class of locally convex spaces with the property $p$-$(u)$ seems to be new even for Banach spaces.
It is well known that a Banach space $X$ has the property $(u)$ if and only if every $x^{\ast\ast}$ in the weak$^\ast$ sequential closure of $X$ in its bidual is also the weak$^\ast$ limit of a wuC series in $X$, see for example \cite[p.~32]{LT-2}. In Proposition \ref{p:p-u-B1} we generalize this result to the property $p$-$(u)$ in the class of all locally convex spaces. We show that the property  $p$-$(u)$ depends only on the duality $(E,E')$ and the class of locally convex spaces with the property $p$-$(u)$ is stable under taking direct sums and Tychonoff products, see Propositions \ref{p:product-sum-u} and \ref{p:p-u-comp}.

One of the most important classes of operators between Banach spaces is the class of completely continuous operators. An operator $T$ from a Banach space $X$ into a Banach space $Y$ is called {\em completely continuous } if it sends weakly convergent sequences into norm-convergent. This class of operators was generalized by Castillo and Sanchez \cite{CS} as follows: if $p\in[1,\infty]$,  $T$ is called {\em $p$-convergent} if  it sends weakly $p$-summable sequences in $X$ into null-sequences in $Y$. The class of $p$-convergent operators plays a crucial role in characterization of Banach spaces with the Dunford--Pettis property of order $p$ defined in  \cite{CS} and the Gelfand--Phillips property of order $p$ defined by Ghenciu \cite{Ghenciu-pGP}. It is natural to define an operator $T:E\to L$ between locally convex spaces $E$ and $L$ to be {\em $p$-convergent} if  $T$ sends weakly $p$-summable sequences in $E$ into null-sequences in $L$. In Theorem \ref{t:p-convergent-1} of Section \ref{sec:p-convergent},  we characterize $p$-convergent operators $T:E\to L$ for every $p\in[1,\infty)$. This theorem generalizes and extends the corresponding assertion for Banach spaces, see \cite[p.~45]{CS} and Proposition 13 of \cite{Ghenciu-pGP}. Applying  Theorem \ref{t:p-convergent-1} to the identity operator $\Id:E\to E$ we obtain an operator characterization of locally convex spaces with the $p$-Schur property for the case $p\in[1,\infty)$, see Corollary \ref{c:p-conver-p-Schur}.

The results obtained in Section \ref{sec:p-convergent} are essentially used in Section \ref{sec:small-bound-p-conv} in which we
show natural relationships between images of bounded sets and $p$-convergence. Let $p\in[1,\infty]$, and let $T:E\to L$ be an operator between  locally convex spaces  $E$ and $L$. In Theorem \ref{t:bounded-to-p-V*} we prove that $T$ sends bounded sets of $E$ into $(p,\infty)$-$(V^\ast)$ subsets of $L$ if and only if the adjoint operator $T^\ast:L'_\beta \to E'_\beta$ is $p$-convergent. As a consequence we obtain that every bounded subset of $E$ is a $(p,\infty)$-$(V^\ast)$ set if and only if $E'_\beta$ has the $p$-Schur property, see Corollary \ref{c:Bo=V*}. In Theorem \ref{t:p-convergent-2} we show that if $E$ is quasibarrelled, then $T$ is $p$-convergent if and only if  $T^\ast(B)$ is a $(p,\infty)$-$(V)$ set in $E'$ for every bounded subset $B$ of $L'_\beta$. Consequently, a quasibarrelled space $E$ has the $p$-Schur property if and only if each bounded subset $B$ of $E'_\beta$ is a  $(p,\infty)$-$(V)$ set.
Generalizing and  extending Theorem 15 of Ghenciu \cite{Ghenciu-pGP} we prove in Theorem \ref{t:p-V*-precompact-V*} that the space $E$ has the $wsV^\ast_{(p,\infty)}$ property (resp.,  the $sV^\ast_{(p,\infty)}$ property) if and only if for every normed space $Y$, if $T\in\LL(Y,E)$ is such that the adjoint operator $T^\ast:E'_\beta\to Y'_\beta$ is $p$-convergent, then $T$ is weakly sequentially precompact (resp., weakly sequentially compact). If additionally $E$ is locally complete, then $E$ has the $wsV^\ast_{(p,\infty)}$ property (resp.,  the $sV^\ast_{(p,\infty)}$ property) if and only if whenever $T\in\LL(\ell_1,E)$ is such that the adjoint operator $T^\ast:E'_\beta\to \ell_\infty$ is $p$-convergent, then $T$ is weakly sequentially precompact (resp., weakly sequentially compact).
Generalizing some results from  \cite{Ghenciu-pGP} we give an operator characterization of $(p,\infty)$-$(V^\ast)$ subsets of $E$ in the case when $1<p<\infty$ (Theorem \ref{t:p-V*-set}), and an operator characterization of $(p,\infty)$-$(V)$ subsets of $E'$ in the case when $1\leq p<\infty$ (Theorem \ref{t:p-V-set}).

In Section \ref{sec:oper-sVp} we generalize some known results related to (weakly) compact operators between Banach spaces, see Proposition \ref{p:compact-type-operator} and Theorem \ref{t:operator-compact-type}. Being motivated by the classical Gantmacher theorem (see Theorem \ref{t:Gantmacher} below) we define an lcs $E$ to have the {\em Gantmacher property}  if for every Banach space $L$, an operator $T:E\to L$ is weakly compact if and only if so is its adjoint $T^\ast: L'_\beta\to E'_\beta$. Generalizing an operator characterization of the property $V$ for Banach spaces obtained by Pe{\l}czy\'{n}ski in Proposition 1 of  \cite{Pelcz-62},  we show in Theorem \ref{t:sVp-charac} that an $\ell_\infty$-$V_p$-barrelled space $E$ with the Gantmacher property has the property $sV_p$ if and only if for every Banach space $L$, each $p$-convergent operator $T:E\to L$ is weakly compact. Using this theorem and a strong result of Pe{\l}czy\'{n}ski \cite{Pelcz-60}, we provide a short proof of Theorem \ref{t:Pel-C(K)}, see Theorem \ref{t:sVp-C(K)}.

In Section \ref{sec:oper-qp} we generalize the notion of $p$-convergent operators as follows. Let $1\leq p\leq q\leq\infty$. An operator $T$ from an lcs $E$ to a Banach space $L$  is {\em $(q,p)$-convergent} if it sends weakly $p$-summable sequences in $E$ to strongly $q$-summable sequences in $L$.
In Theorem  \ref{t:qp-convergent-pq-V} we show that $T$ is $(q,p)$-convergent if and only if $T^\ast(B_{L'})$ is a $(p,q)$-$(V)$ subsets of $E'$. If $T$ is an operator from a normed space to $E$, then $T(B_L)$ is a $(p,q)$-$(V^\ast)$ subsets of $E$ if and only if $T^\ast:E'_\beta\to L'_\beta$ is  $(q,p)$-convergent, see Theorem \ref{t:qp-convergent-pq-V-1}. In Theorems \ref{t:qp-convergent-V-compact} and \ref{t:qp*-convergent-compact} we characterize locally convex spaces for which every $(p,q)$-$(V)$ subset of $E'$ is relatively weakly sequentially ($p$-)compact or relatively sequentially compact. An operator characterization of $(p,q)$-$(V^\ast)$ subsets of $E$ is given in Theorem \ref{t:V*-p-summing}. These results essentially generalize some of known results for Banach spaces, see for example Ghenciu \cite{Ghenciu-pGP,Ghenciu-23}.

Let $1\leq p\leq q<\infty$. An operator characterization of  $(p,q)$-$(V^\ast)$ subsets of an lcs $E$ given in Theorem \ref{t:V*-p-summing} motivates the problem to find conditions under which an operator $T$ from $E$ to a Banach space $L$ is $(q,p)$-convergent. In the case when $E$ is also a Banach space such a conditions is known: $T$ is $(q,p)$-convergent if and only if it is $(q,p)$-summing, see \cite[p.~197]{DJT}. In Section \ref{sec:p-summing} we introduce the notion of $(q,p)$-summing  operators between locally convex space (see Definition \ref{def:U-qp-summing}) and show in Proposition \ref{p:qp-summing-norm} that every $(q,p)$-summing operator is $(q,p)$-convergent and the converse is true if the space $\ell_p[E]$ is barrelled. We study the vector space of all $(q,p)$-summing  operators and show that numerous basic results concerning $(q,p)$-summing operator between Banach spaces (as, for example, Ideal Property, Inclusion Property, Injectivity, Pietsch Domination Theorem and  Pietsch Factorization Theorem) remain true in more general cases. As an application of the obtained results we show in Theorem \ref{t:Banach-Vpp} that
every Banach space $E$ has the property $V^\ast_{(p,p)}$.


It should be mentioned that the obtained results will play a crucial role in the study of Dunford--Pettis type  properties and Gelfand--Phillips type properties of locally convex spaces, see \cite{Gab-DP} and \cite{Gab-GP}.



\section{Preliminary results} \label{sec:Prel}


We start with some necessary definitions and notations used in the article. Set $\NN:=\{ 1,2,\dots\}$ and $\w:=\{ 0,1,2,\dots\}$.
All topological spaces are assumed to be Tychonoff (= completely regular and $T_1$). The closure of a subset $A$ of a topological space $X$ is denoted by $\overline{A}$, $\overline{A}^X$ or $\cl_X(A)$.
Recall that a function $f:X\to Y$ between topological spaces $X$ and $Y$ is called {\em sequentially continuous} if for any convergent sequence $\{x_n\}_{n\in\w}\subseteq X$, the sequence $\{f(x_n)\}_{n\in\w}$ converges in $Y$ and $\lim_{n}f(x_n)=f(\lim_{n}x_n)$.

All topological vector spaces are over the field $\ff$ of real $\IR$ or complex $\mathbb{C}$ numbers.  The closed unit ball of the field $\ff$ is denoted by $\mathbb{D}$. We denote by $C(X)$ the vector space of all continuous $\ff$-valued functions on $X$. The space $C(X)$ endowed with the pointwise topology or the compact-open topology is denoted by $C_p(X)$ and $\CC(X)$, respectively.

Although the spaces of continuous functions are the most important classes of function spaces on a space $X$, there are other classes of (discontinuous) functions which are of significant importance and widely studied in General Topology and Analysis; for example, the classes of Baire type functions introduced and studied by Baire \cite{Baire}.
Let $X$ be a Tychonoff  spaces. For $\alpha=0$, we put $B_0(X):=C_p(X)$. For every nonzero countable ordinal $\alpha$, let $B_\alpha(X)$ be the family of all functions $f:X\to \IF$ that are pointwise limits of sequences $\{f_n\}_{n\in\w}\subseteq \bigcup_{\beta<\alpha}B_\beta(X)$ and $B(X):=\bigcup_{\alpha<\w_1}B_\alpha(X)$. All the spaces $B_\alpha(X)$ and $B(X)$ are endowed with the topology of pointwise convergence, inherited from the Tychonoff product $\IF^X$.

Recall that a subset $A$ of a topological space $X$ is called
\begin{enumerate}
\item[$\bullet$] {\em relatively compact} if its closure ${\bar A}$ is compact;
\item[$\bullet$] ({\em relatively})  {\em countably compact} if each countably infinite subset in $A$ has a cluster point in $A$ (resp., of $X$);
\item[$\bullet$] ({\em relatively}) {\em sequentially compact} if each sequence in $A$ has a subsequence converging to a point of $A$ (resp., of $X$);
\item[$\bullet$] {\em functionally bounded in $X$} if every $f\in C(X)$ is bounded on $A$.
\end{enumerate}

The following lemma is well known and can be easily proved using the diagonal method (see also the proof of Lemma \ref{l:seq-precom-precom} below).
\begin{lemma} \label{l:pr-rsc}
Let $X:=\prod_{i\in\w} X_n$ be the Tychonoff product of a sequence $\{X_n\}_{n\in\w}$ of topological spaces, $A$ be a subset of $X$, and for every $n\in\w$, let $A_n$ be the projection of $A$ onto $X_n$. Then $A$ is a $($relatively$)$ sequentially compact subset of $X:=\prod_{i\in\w} X_n$ if and only if $A_n$ is a $($relatively$)$ sequentially compact subset of $X_n$ for all $n\in\w$.
\end{lemma}


Recall that a Tychonoff space $X$  is called {\em Fr\'{e}chet--Urysohn} if for any cluster point $a\in X$ of a subset $A\subseteq X$ there is a sequence $\{ a_n\}_{n\in\w}\subseteq A$ which converges to $a$. A Tychonoff space $X$ is called an {\em angelic space} if (1) every relatively countably compact subset of $X$ is relatively compact, and (2) any compact subspace of $X$ is Fr\'{e}chet--Urysohn. Note that any subspace of an angelic space is angelic, and a subset $A$ of an angelic space $X$ is compact if and only if it is countably compact if and only if $A$ is sequentially compact, see Lemma 0.3 of \cite{Pryce}. Note also that if $\tau$ and $\nu$ are regular topologies on a set $X$ such that $\tau\leq\nu$ and the space $(X,\tau)$ is  angelic, then the space $(X,\nu)$ is also angelic, see \cite{Pryce}. By the famous Eberlein--\v{S}mulyan theorem every Banach space with the weak topology is angelic. The Grothendieck extension of Eberlein--\v{S}mulyan's theorem states that $C_p(K)$ is angelic for every compact space $K$.

Let $E$ be a topological vector space (tvs for short). The span of a subset $A$ of $E$ and its closure are denoted by $\spn(A)$ and $\cspn(A)$, respectively. We denote by $\Nn_0(E)$ (resp., $\Nn_{0}^c(E)$) the family of all (resp., closed absolutely convex) neighborhoods of zero of $E$. The family of all bounded subsets of $E$ is denoted by $\Bo(E)$. The topological dual space of $E$  is denoted by $E'$. The value of $\chi\in E'$ on $x\in E$ is denoted by $\langle\chi,x\rangle$ or $\chi(x)$. If $E'$ separates the points of $E$, $E$ is called {\em separated}. A sequence $\{x_n\}_{n\in\w}$ in $E$ is said to be {\em Cauchy} if for every $U\in\Nn_0(E)$ there is $N\in\w$ such that $x_n-x_m\in U$ for all $n,m\geq N$. It is easy to see that a sequence $\{x_n\}_{n\in\w}$ in  $E$ is Cauchy if and only if $x_{n_k}-x_{n_{k+1}}\to 0$ for every (strictly) increasing sequence $(n_k)$ in $\w$. If $E$ is a normed space, $B_E$ denotes the closed unit ball of $E$.

Let $E$ be a topological vector space.  We denote by $E_w$ and $E_\beta$ the space $E$ endowed with the weak topology $\sigma(E,E')$ or with the strong topology $\beta(E,E')$, respectively. The topological dual space $E'$ of $E$ endowed with weak$^\ast$ topology $\sigma(E',E)$ or with the strong topology $\beta(E',E)$ is denoted by $E'_{w^\ast}$ or $E'_\beta$, respectively. The closure of a subset $A$ in the weak topology is denoted by $\overline{A}^{\,w}$ or $\overline{A}^{\,\sigma(E,E')}$, and $\overline{B}^{\,w^\ast}$ (or $\overline{B}^{\,\sigma(E',E)}$) denotes the closure of $B\subseteq E'$ in the weak$^\ast$ topology. The {\em polar} of a subset $A$ of $E$ is denoted by
\[
A^\circ :=\{ \chi\in E': \|\chi\|_A \leq 1\}, \quad\mbox{ where }\quad\|\chi\|_A=\sup\big\{|\chi(x)|: x\in A\cup\{0\}\big\}.
\]
A subset $B$ of $E'$ is {\em equicontinuous} if $B\subseteq U^\circ$ for some $U\in \Nn_0(E)$.

Let $H$ be a subspace of a separated tvs $E$. Then the {\em annihilator $H^\perp$ of $H$} is the subspace $E'$ defined by
\[
H^\perp:=\{ \chi\in E': \langle\chi, x\rangle=0 \mbox{ for every } x\in H\}. 
\]

Two vector topologies $\tau$ and $\TTT$ on a vector space $L$ are called {\em compatible} if $(L,\tau)'=(L,\TTT)'$ algebraically. If $(E,\tau)$ is a locally convex space, then there is a finest locally convex vector topology $\mu(E,E')$  compatible with $\tau$. The topology $\mu(E,E')$ is called the {\em Mackey topology}, and if $\tau=\mu(E,E')$, the space $E$ is called a {\em Mackey space}. Set $E_\mu:=\big(E,\mu(E,E')\big)$.

A subset $A$ of a separated topological vector space $E$ is called
\begin{enumerate}
\item[$\bullet$] {\em precompact} if for every $U\in\Nn_0(E)$ there is a finite set $F\subseteq E$ such that $A\subseteq F+U$;
\item[$\bullet$] {\em sequentially precompact} if every sequence in $A$ has a Cauchy subsequence.
\end{enumerate}
It is easy to see that any precompact subset of $E$ is bounded. It is well known that a subset of a locally convex space $E$ is bounded if and only if it is weakly precompact. Using standard terminology we shall say that $A$ is weakly  compact (resp., weakly sequentially compact, relatively weakly compact, relatively weakly sequentially compact, or weakly sequentially precompact) if $A$ is a compact  (resp., sequentially compact, relatively compact, relatively sequentially compact, or  sequentially precompact) subset of $E_w$.

\begin{lemma} \label{l:seq-precom-precom}
Any sequentially precompact subset $A$ of a topological vector space $E$ is precompact hence bounded, but the converse is not true in general.
\end{lemma}

\begin{proof}
Suppose for a contradiction that $A$ is not precompact. Then there is $U\in\Nn_0(E)$ such that for every finite set $F\subseteq E$, it follows $A\not\subseteq F+U$. Take an arbitrary $a_0\in A\SM U$. By induction on $n\geq 1$, choose $a_{n+1}\in A$ such that $a_{n+1} \not\in \{a_0,\dots,a_n\}+U$. Then, by construction, $a_n-a_m\not\in U$ for every $n>m$. Therefore the sequence $\{a_n\}_{n\in\w}$ does not contain a Cauchy subsequence and hence $A$ is not sequentially precompact, a contradiction.

To show that the converse is not true in general we consider the following example. Let $E$ be the Banach space $C(3^\w)$ of all real-valued continuous functions on the product $3^\w$ endowed with the weak topology, where $3:=\{-1,0,1\}$ is endowed  with the discrete topology.  For every $n\in\w$, let $f_n$  be the projection of the compact space $3^\w$ onto the $n$th coordinate. Then the sequence $S=\{f_n\}_{n\in\w}$ is equivalent to the standard unit basis of $\ell_1$.
Indeed,  for every $n\in\w$ and each scalars $c_0,\dots,c_n\in \IR$, set
\[
x:=\big( \mathrm{sign}(c_0),\dots,\mathrm{sign}(c_n),0,0,\dots\big)\in 3^\w.
\]
Then $\big|\sum_{i\leq n} c_i f_i(x)\big| = \sum_{i\leq n} |c_i|$, and hence $S$ is an $\ell_1$-sequence in $C(3^w)$. Therefore, by the Rosenthal $\ell_1$-theorem, the sequence $S$ has no weakly Cauchy subsequences and hence it is not  sequentially precompact in $E$. On the other hand, the set $S$,  being a bounded subset of $C(3^\w)$, is precompact in the weak topology.\qed
\end{proof}

Let $E$ and $L$ be two topological vector spaces. The map $T$ is {\em weak-weak } or {\em weakly continuous} if it is continuous as a map from $E_w$ to $L_w$. If $T:E\to L$ is a weakly continuous linear map, we denote by $T^\ast:L'\to E'$ the adjoint linear map defined by $\langle T^\ast(\chi),x\rangle=\langle \chi, T(x)\rangle$ for every $\chi\in L'$ and each $x\in E$.  The family of all operators (= continuous linear maps) from $E$ to $L$ is denoted by $\LL(E,L)$.
Observe that if $B\subseteq L'$ is equicontinuous and $T\in\LL(E,L)$, then $T^\ast(B)\subseteq E'$ is also equicontinuous. 
An operator $T:E\to L$ is called {\em bounded-covering} if for every bounded subset $A$ of $L$ there is a bounded subset $B$ of $E$ such that $A\subseteq T(B)$. Evidently, bounded-covering operators are  surjective. It is clear that if $H$ is a closed subspace of a normed space $E$, then the quotient map $q:E\to E/H$ is bounded-covering.

Let $E$ and $L$ be two locally convex spaces, and let $T:E\to L$ be a surjective operator. We shall say that $T$ is {\em almost bounded-covering} if for every bounded subset $A$ of $L$ there is a bounded subset $B$ of $E$ such that $A\subseteq \overline{T(B)}$. The meaning of almost bounded-covering maps is explained by the following  assertion which is of independent interest. 

\begin{proposition} \label{p:bounded-covering}
Let $E$ and $L$ be locally convex spaces, and let $T\in\LL(E,L)$ be surjective. Then the adjoint map $T^\ast: L'_\beta \to E'_\beta$ is an embedding  if and only if  $T$ is almost bounded-covering.
\end{proposition}

\begin{proof}
First we recall that $T^\ast$ is an injective strongly continuous operator, see Theorem  8.11.3 of \cite{NaB}. Denote by $H$ the image of $T^\ast$ in $E'_\beta$ with the induced topology. 

Now, assume that $T^\ast$ is an embedding. So, for every bounded subset $A$ of $L$, the image $T^\ast\big(A^\circ\big)$ is a neighborhood of zero in $H$. This means that there is a closed absolutely convex  bounded subset $B$ of $E$ such that $B^\circ\cap H \subseteq T^\ast\big(A^\circ\big)$. We show that $A\subseteq \overline{T(B)}$. Suppose for a contradiction that there is  $z\in E$ such that $T(z)\in A\SM \overline{T(B)}$ (we use here the surjectivity of $T$). Since $\overline{T(B)}$ is bounded closed and absolutely convex, the strong separation theorem guaranties that there is $\eta\in L'$ such that $\eta\in \overline{T(B)}^{\,\circ}$ but $\langle \eta, T(z)\rangle>1$. Then $T^\ast(\eta)\in B^\circ \cap H$ but $\eta\not\in A^\circ$ because $T(z)\in A$. Since $T^\ast$ is injective it follows that $T^\ast(\eta)\not\in T^\ast\big(A^\circ\big)$ and hence $B^\circ\cap H \nsubseteq T^\ast\big(A^\circ\big)$, a contradiction. Thus $A\subseteq \overline{T(B)}$ and hence $T$ is  almost bounded-covering.


Conversely, assume that $T:E\to L$ is almost bounded-covering. Since $T^\ast:L'_\beta\to E'_\beta$ is injective and  continuous it remains to show that $T^\ast$ is open. Let $A^\circ$ be a standard neighborhood of zero in $L'_\beta$, where $A$ is a bounded subset of $L$. We show that $T^\ast(A^\circ)$ is a neighborhood of zero in $H \subseteq E'_\beta$. To this end, choose a bounded subset $B$ of $E$ such that $A\subseteq \overline{T(B)}$. Then every $\chi\in B^\circ \cap H$ has the form  $\chi=T^\ast(\eta)$ for some $\eta\in L'$ such that
\[
|\langle\eta, T(b)\rangle|=|\langle\chi, b\rangle|\leq 1 \quad \mbox{ for every $b\in B$}.
\]
Since $A\subseteq \overline{T(B)}$ it follows that $\eta\in A^\circ$, and hence $\chi\in T^\ast(A^\circ)$. Thus $T^\ast$ is open.\qed
\end{proof}

Recall that an lcs $E$ is called {\em semi-reflexive} if the canonical map $J_E:E\to E''=(E'_\beta)'_\beta$ defined by  $\langle J_E(x),\chi\rangle:=\langle\chi,x\rangle$ ($\chi\in E'$) is an isomorphism; if in addition $J_E$ is a topological isomorphism the space $E$ is called {\em reflexive}.

%
Let $H$ be  a dense subspace  of a locally convex space $E$, and let $\id:H\to E$ be the identity embedding. Then the adjoint map $\id^\ast: E'_\beta \to H'_\beta$ is continuous. Since $E'=H'$ algebraically, this means that the strong topology of $E'_\beta$ is finer than the strong topology of $H'_\beta$. It is natural to ask when both topologies coincide, in other words, when $E'_\beta=H'_\beta$? This is done in the following assertion.
We recall that a dense subspace $H$ of a locally convex space $E$ is called {\em large} if every bounded subset of $E$ is contained in the closure of a bounded subset of $H$ (for more information on large subspaces of $E$, see Chapter 8.3 of \cite{PB}).

\begin{proposition} \label{p:large-charac}
Let $H$ be  a dense subspace  of a locally convex space $E$. Then  $E'_\beta=H'_\beta$ if and only if $H$ is large.
\end{proposition}

\begin{proof}
Assume that $E'_\beta=H'_\beta$, and let $A$ be a closed absolutely convex bounded subset of $E$. Since $E'_\beta=H'_\beta$, there is a closed absolutely convex bounded subset $B$ of $H$ such that $B^\circ \subseteq A^\circ$. We show that $A\subseteq \overline{B}^{\,E}$. First we note that  $A=A^{\circ\circ} \subseteq B^{\circ\circ}$, where the second polar of $B$ is taken in the space $E$. Observe that $B$ is absolutely convex in $E$ as well, and hence $\overline{B}^{\,E}$ is also absolutely convex. Therefore, by the Bipolar Theorem, $B^{\circ\circ}=\big( \overline{B}^{\,E}\big)^{\circ\circ}=\overline{B}^{\,E}$. Thus $A\subseteq \overline{B}^{\,E}$, and hence $H$ is large in $E$.

Conversely, assume that $H$ is large in $E$, and let $A^\circ$ be a standard neighborhood of zero in $E'_\beta$, where $A\subseteq E$ is bounded. Choose a bounded subset $B$ of $H$ such that $A\subseteq \overline{B}^{\,E}$. Then $B^\circ \subseteq A^\circ$ and hence the topology of $H'_\beta$ is finer than the topology of $E'_\beta$. Since the topology of $E'_\beta$ is always finer than the topology of $H'_\beta$  we obtain $E'_\beta=H'_\beta$.\qed
\end{proof}

We denote by $\bigoplus_{i\in I} E_i$ and $\prod_{i\in I} E_i$  the locally convex direct sum and the topological product of a non-empty family $\{E_i\}_{i\in I}$ of locally convex spaces, respectively. If $\xxx=(x_i)\in \bigoplus_{i\in I} E_i$, then the set $\supp(\xxx):=\{i\in I: x_i\not= 0\}$ is called the {\em support} of $\xxx$. The {\em support}  of a subset $A$ of $\bigoplus_{i\in I} E_i$ is the set $\supp(A):=\bigcup_{a\in A} \supp(a)$. We shall also consider elements $\xxx=(x_i) \in \prod_{i\in I} E_i$ as functions on $I$ and write $\xxx(i):=x_i$.
The following assertion is proved in \cite[p.~287]{Kothe}.
\begin{proposition} \label{p:product-sum-strong}
Let $\{E_i\}_{i\in I}$  be a non-empty family of locally convex spaces. Then:
\[ 
\Big(\bigoplus_{i\in I} E_i\Big)'_\beta = \prod_{i\in I} \big(E_i\big)'_\beta \;\; \mbox{ and } \;\;
\Big(\prod_{i\in I} E_i\Big)'_\beta = \bigoplus_{i\in I} \big(E_i\big)'_\beta.
\] 
\end{proposition}

The next easy result will be used in what follows.
\begin{lemma} \label{l:product-normed}
Let $E=\prod_{i\in I} E_i$ be the direct product of a non-empty family $\{E_i\}_{i\in I}$ of locally convex spaces, and let $L$ be a normed space. Then for every $T\in\LL(E,L)$, there is a finite subset $F$ of $I$ such that $\{0\}^F\times \prod_{i\in I\SM F} E_i$ is contained in the kernel of $T$.
\end{lemma}

\begin{proof}
Take a finite subset $F$ of $I$ such that $T\big(\{0\}^F\times \prod_{i\in I\SM F} E_i\big)$ is contained in the unit ball $B_L$. Since $B_L$ contains no non-trivial vector subspaces it follows that  $T\big(\{0\}^F\times \prod_{i\in I\SM F} E_i\big)=\{0\}$, as desired. \qed
\end{proof}

We shall use also the next lemma which complements Lemma \ref{l:pr-rsc}.

\begin{lemma} \label{l:pr-rsc-2}
Let $E:=\prod_{i\in\w} E_n$ be the product of  a sequence $\{E_n\}_{n\in\w}$ of locally convex spaces, $A$ be a subset of $E$, and let $A_n$ be the projection of $A$ onto $E_n$  for every $n\in\w$. Then $A$ is a $($weakly$)$ sequentially precompact subset of $E$ if and only if $A_n$ is a $($weakly$)$ sequentially precompact subset of $E_n$ for all $n\in\w$.
\end{lemma}

\begin{proof}
By Theorem 8.8.5 of \cite{Jar}, we have $E_w=\prod_{i\in\w} (E_n)_w$. Therefore it suffices to prove the lemma only for the sequentially precompact case. The necessity is clear. To prove the sufficiency, let $\{\aaa_k\}_{k\in\w}$ be a sequence in $A$. Since $A_0$ is sequentially precompact in $E_0$, there is a sequence $J_0\subseteq \w$ such that the sequence $\{\aaa_k(0)\}_{k\in J_0}$ is Cauchy in $E_0$. Since $A_1$ is sequentially precompact in $E_1$, there is a sequence $J_1\subseteq J_0$ such that $\{\aaa_k(1)\}_{k\in J_1}$ is Cauchy in $E_1$.  By induction we can find a sequence $J_{i+1} \subseteq J_i$ such that $\{\aaa_k(i+1)\}_{k\in J_{i+1}}$ is Cauchy in $E_{i+1}$. For every $s\in\w$, choose $k_s\in J_s$ such that the sequence $(k_s)\subseteq \w$ is strictly increasing. It is easy to see that the subsequence $\{\aaa_{k_s}\}_{s\in\w}$ of $\{\aaa_k\}_{k\in\w}$ is Cauchy in the product $E$. Thus $A$ is sequentially precompact.\qed
\end{proof}

We denote by $\ind_{n\in \w} E_n$  the inductive limit of a (reduced) inductive sequence $\big\{(E_n,\tau_n)\big\}_{n\in \w}$  of locally convex spaces. If in addition $\tau_m{\restriction}_{E_n} =\tau_n$ for all $n,m\in\w$ with $n\leq m$, the inductive limit $\ind_{n\in \w} E_n$ is called {\em strict} and is denoted by $\SI E_n$. It is well known that $E:=\SI E_n$ is regular, i.e., every bounded set in $E$ is contained in some $E_n$. For more details we refer the reader to Section 4.5 of \cite{Jar}. In the partial case when all spaces $E_n$ are Fr\'{e}chet, the strict inductive limit is called a {\em strict $(LF)$-space}. One of the most important examples of strict $(LF)$-spaces is  the space $\mathcal{D}(\Omega)$ of test functions over an open subset $\Omega$ of $\IR^n$. The strong dual $\mathcal{D}'(\Omega)$ of $\mathcal{D}(\Omega)$ is the space of distributions.
\begin{lemma} \label{l:angelic-strict-LF}
Let $E$ be a locally convex space whose closed bounded sets are weakly angelic {\rm(}for example, $E$ is a strict $(LF)$-space{\rm)}. Then $E$ is weakly angelic.
\end{lemma}

\begin{proof}
Let $A$ be a relatively weakly countably compact subset of $E$. By assumption, $\overline{A}^{\,w}$ is an angelic space. Therefore  $\overline{A}^{\,w}$  is compact and Fr\'{e}chet--Urysohn. In particular, $A$ is relatively  weakly compact. Thus $E$ is a weakly angelic space.

Assume that $E=\SI E_n$ is a strict $(LF)$-space. Then any bounded set of $E$ is sit in some $E_n$ and $E_n$ is  (weakly) closed in $E$. Therefore it suffices to show that any Fr\'{e}chet space is weakly angelic.

It is well known that any Fr\'{e}chet space $E$ embeds into the product $\prod_{n\in\w} X_n$ of a sequence $\{X_n\}_{n\in\w}$ of Banach spaces, see Proposition 8.5.4 of \cite{Jar}, where $X_n=C(K_n)$ for some compact space $K_n$. By Proposition 4 of \cite{Gov}, $X_n$ is even a strictly angelic space. Thus, by Theorem 1 of \cite{Gov}, $\prod_{n\in\w} X_n$ and hence also $E$ are strictly angelic in the weak topology.\qed
\end{proof}

It is well known that any two compatible locally convex vector topologies on a vector space have the same bounded sets. Let us recall that a locally convex space $(E,\tau)$ has
\begin{enumerate}
\item[$\bullet$] the {\em Schur property} if $E$ and $E_w$ have the same convergent sequences;
\item[$\bullet$] the {\em Glicksberg  property} if $E$ and $E_w$ have the same compact sets.
\end{enumerate}
If an lcs $E$ has the  Glicksberg property, then it has trivially the Schur property. The converse is true for strict $(LF)$-spaces (in particular, for Banach spaces), but not in general, see Corollary 2.13 and Proposition 3.5 of \cite{Gabr-free-resp}.

Below we recall the basic classical weak barrelledness conditions, for more details see the books \cite{Jar}, \cite{Kothe}, \cite{NaB}, or \cite{PB}.
\begin{definition} \label{def:weak-barrel}{\em
A locally convex space $E$ is called
\begin{enumerate}
\item[$\bullet$] ({\em quasi}){\em barrelled} if every $\sigma(E',E)$-bounded (resp., $\beta(E',E)$-bounded) subset of $E'$ is equicontinuous;
\item[$\bullet$] {\em $\aleph_0$-}({\em quasi}){\em barrelled} if every $\sigma(E',E)$-bounded (resp., $\beta(E',E)$-bounded) subset of $E'$ which is the countable union of equicontinuous subsets is also equicontinuous;
\item[$\bullet$] {\em $\ell_\infty$-}({\em quasi}){\em barrelled} if every $\sigma(E',E)$-bounded (resp., $\beta(E',E)$-bounded) sequence is equicontinuous;
\item[$\bullet$] {\em $c_0$-}({\em quasi}){\em barrelled} if every $\sigma(E',E)$-null (resp., $\beta(E',E)$-null) sequence is equicontinuous.\qed
\end{enumerate}}
\end{definition}

Below we recall several important classes of locally convex spaces which play an essential role in the article, in particular, to construct (counter)examples.

Let $p\in[1,\infty]$. Then the conjugate number $p^\ast$ of $p$ is defined to be the unique element of $ [1,\infty]$ which satisfies $\tfrac{1}{p}+\tfrac{1}{p^\ast}=1$. For $p\in[1,\infty)$, the space $\ell_{p^\ast}$  is the dual space of $\ell_p$. We denote by $\{e_n\}_{n\in\w}$ the canonical basis of $\ell_p$, if $1\leq p<\infty$, or the canonical basis of $c_0$, if $p=\infty$. The canonical basis of $\ell_{p^\ast}$ is denoted by $\{e_n^\ast\}_{n\in\w}$. In what follows we always identify $\ell_{1^\ast}$ with $c_0$. Denote by  $\ell_p^0$ and $c_0^0$ the linear span of $\{e_n\}_{n\in\w}$  in  $\ell_p$ or $c_0$ endowed with the induced norm topology, respectively. We shall use also the following well known description of relatively compact subsets of $\ell_p$ and $c_0$,  see \cite[p.~6]{Diestel}.
\begin{proposition} \label{p:compact-ell-p}
{\rm(i)} A bounded subset $A$ of $\ell_p$, $p\in[1,\infty)$, is relatively compact if and only if
\[
\lim_{m\to\infty} \sup\Big\{ \sum_{m\leq n} |x_n|^p : x=(x_n)\in A\Big\} =0.
\]
{\rm(ii)} A bounded subset $A$ of $c_0$ is relatively compact if and only if
\[
\lim_{n\to\infty} \sup\{ |x_n|: x=(x_n)\in A\} =0.
\]
\end{proposition}

We denote by $E=(c_0)_p \subseteq \IF^\w$  the Banach space $c_0$ endowed with the pointwise topology inherited  from the direct product $\IF^\w$.
The direct locally convex sum $\IF^{(\w)}$ of a countably-infinite family of the field $\IF$ is denoted by $\varphi$. It is well known that any bounded subset of $\varphi$ is finite-dimensional, see \cite{Jar}. We denote by $\{e_n^\ast\}_{n\in\w}$ the standard coordinate basis of $\varphi$. Observe that $(\IF^\w)'_\beta =\varphi$.

One of the most important classes of locally convex spaces is the class of free locally convex spaces introduced by Markov in \cite{Mar}. The {\em  free locally convex space}  $L(X)$ over a Tychonoff space $X$ is a pair consisting of a locally convex space $L(X)$ and  a continuous map $i: X\to L(X)$  such that every  continuous map $f$ from $X$ to a locally convex space  $E$ gives rise to a unique continuous linear operator $\Psi_E(f): L(X) \to E$  with $f=\Psi_E(f) \circ i$. The free locally convex space $L(X)$ always exists and is essentially unique, and $X$ is the Hamel basis of $L(X)$. So, each nonzero $\chi\in L(X)$ has a unique decomposition $\chi =a_1 i(x_1) +\cdots +a_n i(x_n)$, where all $a_k$ are nonzero and $x_k$ are distinct. The set $\supp(\chi):=\{x_1,\dots,x_n\}$ is called the {\em support} of $\chi$. In what follows we shall identify $i(x)$ with $x$ and consider $i(x)$ as the Dirac measure $\delta_x$ at the point $x\in X$. We also recall that $C_p(X)'=L(X)$ and $L(X)'=C(X)$. It is worth mentioning that $L(X)$ has the Glicksberg property for every Tychonoff space $X$, and if $X$ is non-discrete, then $L(X)$ is not a Mackey space, see \cite{Gab-Respected} and \cite{Gabr-L(X)-Mackey}, respectively.


\section{Completeness type properties in locally convex spaces} \label{sec:completeness}


Let $E$ be a topological vector space. We shall say that a subfamily  $\AAA(E)$ of $\Bo(E)$ is a {\em bornology on $E$} if it satisfies the following conditions:
\begin{enumerate}
\item[(a)] if $A_1,A_2\in \AAA(E)$, then there is $A\in\AAA(E)$ such that $(A_1+ A_2)\cup A_1\cup A_2\subseteq A$;
\item[(b)] if $A\in\AAA(E)$ and $\lambda\in\IF$, then $\lambda A\in\AAA(E)$;
\item[(c)] $\AAA(E)$ contains all finite subsets of $E$ and is closed under taking subsets.
\end{enumerate}
Any bornology $\AAA(E)$ on $E$  defines the {\em $\AAA$-topology } denoted by $\TTT_\AAA$ on the dual space $E'$ of uniform convergence on the members of $\AAA$. We write $E'_\AAA :=(E', \TTT_\AAA)$. Recall that the {\em saturated hull $\AAA^{sat}$ of $\AAA$} is the collection of all subsets of bipolars of elements of $\AAA$, and the bornology $\AAA$ is {\em saturated} if $\AAA=\AAA^{sat}$.

Let us recall some of the most important families of compact-type subsets of a separated topological vector space $E$ which form a bornology on $E$. These families complete the family $\mathcal{PC}(E)$  of precompact sets and the family $\mathcal{SPC}(E)$ of sequentially precompact sets considered in Section \ref{sec:Prel}. 
A subset $A$ of $E$ is called
\begin{enumerate}
\item[$\bullet$] {\em weakly $($sequentially$)$ compact} if $A$ is (sequentially) compact in $E_w$;
\item[$\bullet$] {\em relatively weakly compact} if its weak closure $\overline{A}^{\,\sigma(E,E')}$ is compact in $E_w$;
\item[$\bullet$] {\em relatively weakly sequentially compact} if each sequence in $A$ has a subsequence weakly converging to a point of $E$;
\item[$\bullet$] {\em weakly sequentially precompact} if each sequence in $A$ has a weakly Cauchy subsequence.
\end{enumerate}
\begin{notation} \label{n:family-A} {\em
For a separated tvs $E$, we  denote by $\AAA(E)$ one of the families $\mathcal{RC}(E)$, $\mathcal{RSC}(E)$, $\mathcal{PC}(E)$, $\mathcal{SPC}(E)$, $\mathcal{RWC}(E)$, $\mathcal{RWSC}(E)$, $\mathcal{WSPC}(E)$, or $\Bo(E)$ of all relatively compact, relatively sequentially compact,  precompact, sequentially precompact,  relatively weakly compact, relatively weakly sequentially compact,  weakly sequentially precompact or bounded subsets of $E$, respectively.\qed}
\end{notation}
It is easy to see that all the families in Notation \ref{n:family-A} are bornologies.

Below we recall some of the basic classes of compact-type operators. 

\begin{definition} \label{def:operators} {\em
Let $E$ and $L$ be locally convex spaces. An operator $T\in \LL(E,L)$ is called {\em compact} (resp., {\em sequentially compact, precompact, sequentially precompact, weakly compact, weakly sequentially compact, weakly sequentially precompact, bounded}) if there is $U\in \Nn_0(E)$ such that $T(U)\in \mathcal{RC}(L)$ (resp., $\mathcal{RSC}(L)$, $\mathcal{PC}(L)$, $\mathcal{SPC}(L)$, $\mathcal{RWC}(L)$, $\mathcal{RWSC}(L)$, $\mathcal{WSPC}(L)$, or $\Bo(L)$). The family of all such operators is denoted by $\mathcal{LC}(E,L)$ (resp., $\mathcal{LSC}(E,L)$, $\mathcal{LPC}(E,L)$, $\mathcal{LSPC}(E,L)$, $\mathcal{LWC}(E,L)$, $\mathcal{LWSC}(E,L)$, $\mathcal{LWSPC}(E,L)$, or $\mathcal{LB}(E,L)$). \qed}   
\end{definition}

It is natural to generalize and unify Definition \ref{def:operators} as follows.

\begin{definition} \label{def:operators-g} {\em
Let $E$ and $L$ be locally convex spaces, and let $\AAA(L)$ be a bornology on $L$. An operator $T:E\to L$ {\em belongs to } $\mathcal{LA}(E,L)$ if there is $U\in \Nn_0(E)$ such that $T(U)\in\AAA(L)$.\qed}
\end{definition}

To obtain some characterizations of compact type operators it is essential to find additional conditions on $E$ under which the family $\AAA(E)$ is saturated.
This can be naturally done using completeness type properties of $E$. 
Let us recall that a locally convex space  $E$
\begin{enumerate}
\item[$\bullet$] is {\em complete} if each Cauchy net converges;
\item[$\bullet$] is {\em quasi-complete} if each closed bounded subset of $E$ is complete;
\item[$\bullet$] is {\em von Neumann complete} if every precompact subset of $E$ is relatively compact; 
\item[$\bullet$] is {\em sequentially complete} if each Cauchy sequence in $E$ converges;
\item[$\bullet$] is {\em locally complete} if the closed absolutely convex hull of a null sequence in $E$ is compact;
\item[$\bullet$] has the {\em convex compactness property} (ccp) if the closed absolutely convex hull of each compact subset $K$ of $E$ is compact;
\item[$\bullet$] has the {\em Krein property} if the closed absolutely convex hull of each weakly compact subset $K$ of $E$ is weakly compact (i.e., $E_w$ has the ccp).
\end{enumerate}
The Krein property was defined and characterized in \cite{Gabr-free-resp} being motivated by the Krein theorem \cite[\S~24.5(4)]{Kothe} which states that if $K$ is a weakly compact subset of an lcs $E$, then $\overline{\mathrm{acx}}(K)$ is weakly compact if and only if $\overline{\mathrm{acx}}(K)$ is $\mu(E,E')$-complete. Therefore $E$ is a Krein space if the space $E_\mu:=(E,\mu(E,E'))$ is quasi-complete. 
In the next lemma the clause (i) is well known  
 and the clauses (ii) and (iii) follow from the corresponding definitions.
\begin{lemma} \label{l:AAA-convex}
Let $E$ be a locally convex space, and let $\AAA(E)$  be a bornology on $E$. Then $\AAA(E)$ is saturated if one of the following condition holds:
\begin{enumerate}
\item[{\rm(i)}] $\AAA(E)= \mathcal{PC}(E)$ or $\AAA(E)= \Bo(E)$;
\item[{\rm(ii)}] $E$ has ccp and $\AAA(E)= \mathcal{RC}(E)$;
\item[{\rm(iii)}] $E$ has the Krein property and $\AAA(E)= \mathcal{RWC}(E)$.
\end{enumerate}
\end{lemma}

\begin{lemma} \label{l:seq-p-comp}
A subset $A$ of a sequentially complete separated tvs $E$ is sequentially precompact if and only if it is relatively sequentially compact in $E$.
\end{lemma}

\begin{proof}
The sufficiency is clear. To prove the necessity, let $A$ be a sequentially precompact subset of $E$. Take an arbitrary sequence $\{a_n\big\}_{n\in\w}$ in $A$, and choose  a Cauchy subsequence $\big\{a_{n_k}\big\}_{k\in\w}$ in  $\big\{a_{n}\big\}_{n\in\w}$. Since $E$ is sequentially complete, there is $x\in E$ such that $a_{n_k}\to x$.  Thus $A$  is relatively sequentially compact in $E$.\qed
\end{proof}

Let $A$ be a {\em disk} (= a bounded and absolutely convex subset) of a locally convex space $E$. We denote by $E_A$ the linear span of $A$. The gauge $\rho_A$ of $A$ defines the norm topology on $E_A$. We shall use repeatedly the following result, see Proposition 3.2.2 and Proposition 5.1.6 of \cite{PB}.

\begin{proposition} \label{p:bounded-norm}
Let $A$ be a disk in a locally convex space $E$. Then $(E_A,\rho_A)$ is a normed space and its norm topology is finer than the topology induced from $E$. Moreover, if $E$ is locally complete and $A$ is closed, then $(E_A,\rho_A)$ is a Banach space.
\end{proposition}

We shall use the following extension result. Recall that a subspace $Y$ of a topological space $X$ is {\em sequentially dense} if for every $x\in X$ there is a sequence $\{y_n\}_{n\in\w}$ in $Y$ which converges to $x$.

\begin{proposition} \label{p:extens-bounded}
Let $H$ be a sequentially dense subspace of a locally convex space $E$, $L$ be a sequentially complete locally convex space, and let $T:H\to L$ be a bounded operator. Then $T$ extends  to a bounded operator $\overline{T}$ from $E$ to $L$. Consequently, any bounded operator from a metrizable locally convex space $E$ to a sequentially complete locally convex space $L$ can be extended to the completion $\overline{E}$ of $E$.
\end{proposition}

\begin{proof}
For every $x\in E$, choose a sequence $S=\{h_n\}_{n\in\w}$ in $H$ which converges to $x$. Then $S$ is Cauchy in $H$, and hence the sequence $T(S)$ is Cauchy in $L$. Since $L$ is sequentially complete, $T(S)$ converges to some element $l$. It is easy to see that $l$ does not depend on a sequence converging to $x$, therefore we can define $\overline{T}(x):=l$. Evidently, $\overline{T}$ is a linear map. It remains to prove that $\overline{T}$ is bounded. To this end, fix an absolutely convex, closed neighborhood $U$ of zero in $E$ such that $T(U\cap H)$ is a bounded subset of $L$. Set $W:=\tfrac{1}{2} U$. We show that $\overline{T}(W)$ is a bounded subset of $L$. Let $x\in W$. Choose a sequence $\{h_n\}_{n\in\w}$ in $H$ which converges to $x$. Without loss of generality we can assume that $\{h_n\}_{n\in\w}\subseteq U\cap H$. Therefore $\overline{T}(x)$ belongs to the bounded subset $\overline{T(U\cap H)}$ of $L$. Thus $\overline{T}(W)$ is bounded in $L$, as desired.\qed
\end{proof}

One of the basic theorems in Analysis is the Ascoli theorem which states that {\em if $X$ is a $k$-space, then every compact subset of $\CC(X)$ is evenly continuous}, see Theorem 3.4.20 in \cite{Eng}. The Ascoli theorem motivates us in \cite{BG} to introduce and study the class of Ascoli spaces. A space $X$ is called {\em Ascoli} if every compact subset of $\CC(X)$ is evenly continuous. One can easily show that a compact subset of $\CC(X)$ is evenly continuous if and only if it is equicontinuous. Recall that a subset $H$ of $C(X)$ is {\em equicontinuous} if for every $x\in X$ and each $\e>0$ there is an open neighborhood $U$ of $x$ such that $|f(x')-f(x)|<\e$ for all $x'\in U$ and $f\in H$.  Being motivated by the classical notion of $c_0$-barrelled locally convex spaces we defined in \cite{Gabr-free-lcs} a space $X$ to be {\em sequentially Ascoli} if every convergent sequence in $\CC(X)$ is equicontinuous. Clearly, every Ascoli space is sequentially Ascoli, but the converse is not true in general (every non-discrete $P$-space is sequentially Ascoli but not Ascoli, see Proposition 2.9 in \cite{Gabr-free-lcs}). In \cite{G-Ck-lc}, we proved that the space $\CC(X)$ is locally complete if and only if $X$ is a sequentially Ascoli space. The next assertion complements this result.

\begin{proposition} \label{p:Ascoli-ccp}
If $X$ is an Ascoli space, then $\CC(X)$ has ccp.
\end{proposition}

\begin{proof}
Let $\KK$ be a compact subset of $\CC(X)$. Since $X$ is Ascoli, $\KK$ is equicontinuous. Therefore, by Lemma 3.4.17 of \cite{Eng}, $\cacx(\KK)$ is equcontinuous as well. As $\cacx(\KK)$ is also pointwise bounded and closed, the necessity of the Ascoli theorem \cite[3.4.20]{Eng} implies that $\cacx(\KK)$ is a compact subset of $\CC(X)$. Thus $\CC(X)$ has ccp.\qed
\end{proof}

We shall consider the following completeness type properties.

\begin{definition} \label{def:sep-quasi-comp} {\em
A topological vector space $E$ is said to be
\begin{enumerate}
\item[$\bullet$] {\em separably quasi-complete} if every closed separable bounded subset of $E$ is complete;
\item[$\bullet$] {\em separably von Neumann complete} if every closed separable precompact subset of $E$ is compact.
\end{enumerate}}
\end{definition}
Note that in the realm of topological vector spaces, separably von Neumann complete spaces coincide with topological vector spaces with the $\mathbf{cp}$-property introduced in \cite{Gab-Respected}.

\begin{proposition} \label{p:sep-quasi-comp}
Let $E$ be a topological vector space. Then:
\begin{enumerate}
\item[{\rm(i)}] if $E$ is quasi-complete, then it is separably quasi-complete and von Neumann complete;
\item[{\rm(ii)}] if $E$ is either separably quasi-complete or von Neumann complete, then it is separably  von Neumann complete;
\item[{\rm(iii)}] if $E$ is separably  von Neumann complete, then it is sequentially complete.
\end{enumerate}
\end{proposition}

\begin{proof}
The assertions (i) and (ii) are evident (recall that any precompact set is bounded). To prove (iii) we note that every Cauchy sequence $S=\{x_n\}_{n\in\w}$ in $E$ is a precompact separable set and hence its closure is compact. Then $S$ has a cluster point $z\in E$. Since $S$ is Cauchy it follows that $S$ converges to the point $z$. Thus $E$ is sequentially complete.\qed
\end{proof}

\begin{example} \label{exa:seq-comp-non-sep}
The space $E=(\ell_1)_w$ is sequentially complete but not separably von Neumann complete.
\end{example}

\begin{proof}
The space $E$ is sequentially complete by Corollary 10.5.4 of \cite{Jar} (or by (i) of Proposition \ref{p:l-inf-not-c0-bar} below). 
Observe that $B_{\ell_1}$ is closed, separable and precompact (being bounded) in the weak topology. However, by the Schur property, $B_{\ell_1}$ is not compact in $E$. Thus $E$ is not   separably von Neumann complete.\qed
\end{proof}

\begin{example} \label{exa:sep-quasi-non-vN}
Let $\lambda$ be an uncountable cardinal, and let
\[
E:=\big\{ (x_i)\in \IR^\lambda: |\supp(x_i)|\leq \aleph_0\big\}.
\]
Then $E$ is a separably quasi-complete but not von Neumann complete space.
\end{example}

\begin{proof}
To show that $E$ is separably quasi-complete, let $A$ be a separable bounded subset of $E$. Take a dense countable subset $A_0$ of $A$. Then $A_0$ has countable support $I\subseteq \lambda$. Therefore the closure $\overline{A}=\overline{A_0}$ in $E$  has the same countable support $I$, and hence $\overline{A}$ is a closed, bounded subset of the complete metrizable subspace $\IR^I$ of $E$. Whence $\overline{A}$ is compact in $\IR^I$ and hence in $E$. Thus $E$ is separably quasi-complete.

To show that $E$ is not  von Neumann complete, consider the following set
\[
B:= \big\{ (x_i)\in E: x_i=1 \mbox{ for every } i\in \supp(x_i)\big\}.
\]
Since the closure $\cl_{\IR^\lambda}(B)$ of $B$ in the complete space $\IR^\lambda$ is compact and contains $(1_i)\not\in E$, it follows that $B$ is precompact but its closure in $E$ is not compact. Thus $E$ is not  von Neumann complete.\qed
\end{proof}

The aforementioned results show the following relationships between the introduced completeness type properties (where ``vN'' means ``von Neumann'')
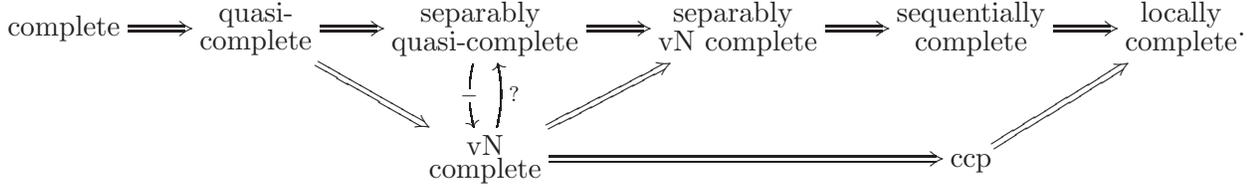
\[
\xymatrix{
\mbox{complete}  \ar@{=>}[r]& {\substack{\mbox{quasi-} \\ \mbox{complete}}}  \ar@{=>}[r]\ar@{=>}[rd] & {\substack{\mbox{separably } \\ \mbox{quasi-complete}}}  \ar@{=>}[r] \ar@/_/[d]|-{-}  & {\substack{\mbox{separably } \\ \mbox{vN complete}}}  \ar@{=>}[r]  &  {\substack{\mbox{sequentially} \\ \mbox{complete}}}  \ar@{=>}[r] & {\substack{\mbox{locally} \\ \mbox{complete}}}. \\
&& {\substack{\mbox{vN} \\ \mbox{complete}}}  \ar@{=>}[ru] \ar@/_/[u]_{?}\ar@{=>}[rr] && \mbox{ccp}  \ar@{=>}[ru]& }
\]

Let $X$ be a Tychonoff space.  It is known that $C_p(X)$ is quasi-complete if and only if $X$ is discrete,  see Theorem 3.6.6 of \cite{Jar}. Below we generalize this result.
\begin{theorem} \label{t:Cp-vNc}
Let $X$ be a Tychonoff space. Then $C_p(X)$ is von Neumann complete if and only if $X$ is discrete.
\end{theorem}

\begin{proof}
If $X$ is discrete, then $C_p(X)=\IF^X$ is even complete. Assume that  $C_p(X)$ is von Neumann complete. It is easy to see that the set $B:=\{f\in C(X): \|f\|_X\leq 1\}$ is a closed precompact subset of $C_p(X)$. Therefore $B$ is compact. Since $B$ is dense in the compact space $\mathbb{D}^X$ it follows that $B=\mathbb{D}^X$. So every bounded function on $X$ is continuous. Thus $X$ is discrete.\qed
\end{proof}

Recall that a Tychonoff space is called a {\em $P$-space} if any countable intersection of open sets is open. The Buchwalter--Schmets theorem states that $C_p(X)$ is sequentially complete if and only if $X$ is a $P$-space, and, by  Theorem 1.1 of \cite{FKS-P},  the space $C_p(X)$ is locally complete if and only if $X$ is a $P$-space.
In Theorem \ref{t:Cp-lc} below we  strengthen  these results and give an independent and simpler  proof of the implication (iv)$\Ra$(i) using the following functional characterization of (non) $P$-spaces. Recall that the {\em support} of a continuous function $f\in C(X)$ is the set $\supp(f):=\cl_X\{x\in X: f(x)\not=0\}$.

\begin{proposition} \label{p:P-space}
For a Tychonoff space $X$, the following assertions are equivalent:
\begin{enumerate}
\item[{\rm(i)}] $X$ is not a $P$-space;
\item[{\rm(ii)}] there is $f\in C(X)$ such that the set $\{x\in X: f(x)>0\}$ is not closed;
\item[{\rm(iii)}] there are a point $z\in X$, a sequence $\{ g_i\}_{i\in\w}$ of continuous functions from $X$ to $[0,2]$, and a sequence $\{ U_i\}_{i\in\w}$ of open subsets of $X$ such that
\begin{enumerate}
\item[{\rm (a)}] $\supp(g_i) \subseteq U_i $ for every $i\in\w$;
\item[{\rm (b)}] $U_i\cap U_j=\emptyset $ for all distinct $i,j\in\w$;
\item[{\rm (c)}] $z\not\in U_i$ for every $i\in\kappa$ and $z\in \cl\big(\bigcup_{i\in\w} \{ x\in X: g_i(x)\geq 1\}\big)$.
\end{enumerate}
\end{enumerate}
\end{proposition}

\begin{proof}
(i)$\Ra$(ii) Suppose for a contradiction that the set $C_f:=\{x\in X: f(x)>0\}$ is closed for every $f\in C(X)$. To get a contradiction we show that $X$ is a $P$-space. To this end, let  $\{ U_n\}_{n\in\w}$  be a decreasing sequence of open sets such that $\bigcap_{n\in\w} U_n$ is not empty. Assuming that $\bigcap_{n\in\w} U_n$ is not open we could find a point $z\in \bigcap_{n\in\w} U_n\SM \Int\big(\bigcap_{n\in\w} U_n\big)$. For every $n\in\w$, choose a continuous function $h_n:X\to [0,1]$ such that $h_n(X\SM U_n)=\{0\}$ and $h_n(z)=1$. Then the function
\[
h:= \sum_{n\in\w} \tfrac{1}{2^{n+1}}\, h_n : X \to [0,1]
\]
is continuous and such that $h(z)=1$. Set $f:=1-h$. Then $f:X\to [0,1]$ is continuous and $f(z)=0$. By assumption the set $C_f$ is closed and hence the set $X\SM C_f$ is open and contains $z$. Since
\[
X\SM C_f =\{x\in X:h(x)=1\} \subseteq \bigcap_{n\in\w}\{x\in X:h_n(x)=1\} \subseteq \bigcap_{n\in\w} U_n,
\]
we obtain that $z\in \Int\big(\bigcap_{n\in\w} U_n\big)$. But this contradicts the choice of $z$.

(ii)$\Ra$(iii) is proved in Case 1 of Proposition 2.4 of \cite{Gabr-L(X)-Mackey}.

(iii)$\Ra$(i) For every $i\in\w$, set $V_i:=\{ x\in X: g_i(x)< 1\}$. Since $\supp(g_i) \subseteq U_i $ and $z\not\in U_i$, it follows that $g_i(z)=0$ and hence $V_i$ is an open neighborhood of $z$. Applying (c) we obtain that $\bigcap_{i\in\w} V_i$ is not a neighborhood of $z$, and hence $\bigcap_{i\in\w} V_i$ is not open. Thus $X$ is not a $P$-space.\qed
\end{proof}

\begin{theorem} \label{t:Cp-lc}
Let $X$ be a Tychonoff space. Then the following assertions are equivalent:
\begin{enumerate}
\item[{\rm(i)}] $X$ is a $P$-space;
\item[{\rm(ii)}] $C_p(X)$  is  separably quasi-complete;
\item[{\rm(iii)}] $C_p(X)$ is sequentially complete;
\item[{\rm(iv)}] $C_p(X)$  is locally complete.
\end{enumerate}
\end{theorem}

\begin{proof}
(i)$\Ra$(ii) Assume that $X$ is a $P$-space. Fix a separable closed bounded subset $B$ of $C_p(X)$, and let $S=\{f_n\}_{n\in\w}$ be a dense subset of $B$. To show that $B$ is complete, it suffices to show that $B$ is closed in the completion $\IF^X$ of $C_p(X)$.  To this end, it is sufficient to show that any $g\in \cl_{\IF^X}(B)$ is continuous (indeed, if $g\in C_p(X)$, then the closeness of $B$ in $C_p(X)$ implies that $g\in B$). Fix an arbitrary point $z\in X$. Since $X$ is a $P$-space, there is a neighborhood $U$ of $z$ such that all functions $f_n$ are constant on $U$. Since $S$ is dense in $B$ we have $\cl_{\IF^X}(S)=\cl_{\IF^X}(B)$, and hence the function $g$ is also a constant function on $U$. Thus $g$ is continuous at $z$, as desired.
\smallskip

(ii)$\Ra$(iii) follows from Proposition \ref{p:sep-quasi-comp}, and the implication (iii)$\Ra$(iv) is well known (see Corollary 5.1.8 \cite{PB}).
\smallskip

(iv)$\Ra$(i) Suppose for a contradiction that $X$ is not a $P$-space. Then, by Proposition \ref{p:P-space}, there are a point $z\in X$, a sequence $\{ g_i\}_{i\in\w}$ of continuous functions from $X$ to $[0,2]$, and a sequence $\{ U_i\}_{i\in\w}$ of open subsets of $X$ such that
\begin{enumerate}
\item[{\rm (a)}] $\supp(g_i) \subseteq U_i $ for every $i\in\w$;
\item[{\rm (b)}] $U_i\cap U_j=\emptyset $ for all distinct $i,j\in\w$;
\item[{\rm (c)}] $z\not\in U_i$ for every $i\in\kappa$ and $z\in \cl\big(\bigcup_{i\in\w} \{ x\in X: g_i(x)\geq 1\}\big)$.
\end{enumerate}
For every $n\in\w$, set $f_n:= 2^{n+1} g_n$. Then (a) and (b) imply that $\{f_n\}_{n\in\w}$ is a null sequence in $C_p(X)$. For every $n\in\w$, set
\[
F_n (x) := \sum_{i\leq n} \tfrac{1}{2^{i+1}} \, f_i(x) =  \sum_{i\leq n} g_i(x) \in\cacx\big(\{f_n\}_{n\in\w}\big).
\]
Since $C_p(X)$  is locally complete, it follows that $\cacx\big(\{f_n\}_{n\in\w}\big)$ is  compact and hence the sequence $\{F_n\}_{n\in\w}$ has a cluster point $F\in C_p(X)$. By (c), $F_n(z)=0$ for every $n\in\w$, and hence $F(z)=0$. Choose a neighborhood $U$ of $z$ such that $F(x)<\tfrac{1}{4}$ for every $x\in U$. By (c), there is an $m\in\w$ and a point $y\in U$ such that $g_m(y)\geq 1$. Therefore, for every $n\geq m$, we have $F_n(y)\geq g_m(y)\geq 1$ and hence
\[
F(y) \geq \inf\{ F_n(y): n\geq m\} \geq 1.
\]
However, since $y\in U$ we have $F(y)<\tfrac{1}{4}$. This is a desired contradiction. \qed
\end{proof}

We finish this section with the following natural  problems.

\begin{problem}
Characterize Tychonoff spaces $X$ for which the space $C_p(X)$ has ccp.
\end{problem}

\begin{problem}
Characterize Tychonoff spaces $X$ for which the space $\CC(X)$ is separably quasi-complete, separably von Neumann complete or has ccp.
\end{problem}


\section{Convergent and summable sequences in topological vector spaces} \label{sec:conv-sum}


Below we define (weakly) unconditionally Cauchy and unconditionally convergent series in a locally convex space naturally generalizing the corresponding notions in Banach spaces.
\begin{definition} \label{def:tvs-p-wuC}{\em
A series $\sum_{n\in\w} x_n$ in a separated topological vector spaces $E$ is called

$\bullet$ {\em unconditionally Cauchy} ({\em uC}) if  for every neighborhood $U$ of zero there is an $N\in\w$ such that
$ \sum_{n\in F} x_n\in U$ for every finite set $F\subseteq [N,\infty)$;

$\bullet$ {\em weakly unconditionally Cauchy} ({\em wuC}) if  $\sum_{n\in\w}|\langle\chi, x_n\rangle|$ converges for every $\chi\in E'$;

$\bullet$ {\em unconditionally convergent} ({\em u.c.}) if $\sum\lambda_nx_n$ converges for all $(\lambda_n)\in \ell_\infty$.\qed}
\end{definition}

Now we extend the definitions of weakly $p$-summable,  weakly $p$-convergent  and weakly $p$-Cauchy sequences from the case of Banach spaces to the general case of separated  topological vector spaces.
\begin{definition}\label{def:tvs-weakly-p-sum}{\em
Let $p\in[1,\infty]$. A sequence  $\{x_n\}_{n\in\w}$ in a separated  tvs $E$ is called
\begin{enumerate}
\item[$\bullet$]  {\em weakly $p$-summable} if  for every $\chi\in E'$, it follows
\[
\mbox{$(\langle\chi, x_n\rangle)_{n\in\w} \in\ell_p$ if $p<\infty$, and  $(\langle\chi, x_n\rangle)_{n\in\w} \in c_0$ if $p=\infty$;}
\]
\item[$\bullet$] {\em weakly $p$-convergent to $x\in E$} if  $\{x_n-x\}_{n\in\w}$ is weakly $p$-summable;
\item[$\bullet$] {\em weakly $p$-Cauchy} if for each pair of strictly increasing sequences $(k_n),(j_n)\subseteq \w$, the sequence  $(x_{k_n}-x_{j_n})_{n\in\w}$ is weakly $p$-summable.\qed
\end{enumerate}}
\end{definition}






If $1\leq p<q\leq\infty$ and $\{x_n\}_{n\in\w}$ is a sequence in a separated  tvs $E$, Definitions \ref{def:tvs-p-wuC} and \ref{def:tvs-weakly-p-sum} immediately imply the following relationships between the introduced notions
\begin{equation} \label{equ:summable}
\xymatrix{
{\substack{\mbox{ weakly $1$-summable =} \\   \mbox{ the series } \sum x_n \mbox{ is wuC}}} \ar@{=>}[r] & {\substack{\mbox{weakly}\\ \mbox{$p$-summable}}} \ar@{=>}[r] & {\substack{\mbox{weakly}\\ \mbox{$q$-summable}}} \ar@{=>}[r] & {\substack{\mbox{weakly $\infty$-summable} \\ \mbox{= weakly null}}} }
\end{equation}


\begin{notation}{\em
We shall denote by $\ell^w_p(E)$ the family of all weakly $p$-summable sequences in $E$. If $p=\infty$, for the simplicity of notations and formulations we shall usually identify $\ell^w_\infty(E)$ with the space $c_0^w(E)$ of all weakly null-sequences in $E$. Elements of $\ell^w_p(E)$ will be written as $(x_n)$.  \qed}
\end{notation}

\begin{example} \label{exa:lp-in-lr}
Let $r,p\in[1,\infty]$, and let $\{e_n\}_{n\in\w}$ be the standard unit basis in $\ell_r$. Then $(e_n)\in\ell_p^w(\ell_r)$ if and only if $p\geq r^\ast$.
\end{example}
\begin{proof}
Assume that $p\geq r^\ast$. If $\chi=(a_n)\in\ell_{r^\ast}$, then $(\langle\chi, e_n\rangle)_{n\in\w}=\chi\in \ell_p$. Therefore $(e_n)\in\ell_p^w(\ell_r)$. Assume that $p< r^\ast$. Take $\chi=(a_n)\in  \ell_{r^\ast}\SM \ell_p$. Then $(\langle\chi, e_n\rangle)_{n\in\w}=\chi\not\in \ell_p$ and hence $(e_n)\not\in\ell_p^w(\ell_r)$.\qed
\end{proof}

We shall use the next standard lemma.
\begin{lemma} \label{l:p-Cauchy-conv}
Let $p\in[1,\infty]$, and let $E$ be a locally convex space. If a weakly $p$-Cauchy sequence $\{x_k\}_{k\in\w}$ has a subsequence $\{x_{k_n}\}_{n\in\w}$ which weakly $p$-converges to some $x\in E$, then also $\{x_k\}_{k\in\w}$ weakly $p$-converges to $x$.
\end{lemma}

\begin{proof}
We prove the lemma only for the case $p<\infty$. Suppose for a contradiction that
\[
\sum_{k\in\w} |\langle\chi,x_k-x\rangle|^p =\infty \;\; \mbox{ for some } \; \chi\in E'.
\]
For every $n\in\w$, let $j_n:=n$. Since $\{x_k\}_{k\in\w}$ is weakly $p$-Cauchy we obtain
\[
\infty =\Big(\sum_{n\in\w} |\langle\chi,x_{j_n}-x\rangle|^p\Big)^p \leq \Big(\sum_{n\in\w} |\langle\chi,x_{j_n}-x_{k_n}\rangle|^p\Big)^p + \Big(\sum_{n\in\w} |\langle\chi,x_{k_n}-x\rangle|^p\Big)^p <\infty,
\]
a contradiction. \qed
\end{proof}

In the next lemma we summarize some simple properties of the set $\ell^w_p(E)$.
\begin{lemma} \label{l:prop-p-sum}
Let $p,q\in[1,\infty]$, and let $(E,\tau)$ be a locally convex space. Then:
\begin{enumerate}
\item[{\rm(i)}] $\ell^w_p(E)$ is a vector space and every $(x_n)\in \ell^w_p(E)$ is weakly null in $E$;
\item[{\rm(ii)}] if $p<q$, then $\ell^w_p(E)\subsetneq \ell^w_q(E)$;
\item[{\rm(iii)}] if $T:E\to L$ is an operator into an lcs $L$ and $(x_n)_{n\in\w}\in \ell_p^w(E)$ $($resp., $(x_n)_{n\in\w}\in \ell_p^w(E_\beta)$$)$, then $\big(T(x_n)\big)_{n\in\w}\in \ell_p^w(L)$ $($resp., $(x_n)_{n\in\w}\in \ell_p^w(L_\beta)$$)$; in particular, $\ell_p^w(E_\beta)\subseteq \ell_p^w(E)$;
\item[{\rm(iv)}] if $\TTT$ is a locally convex vector topology on $E$ compatible with the topology $\tau$ of $E$, then $\ell^w_p(E)=\ell^w_p(E,\TTT)$; consequently, also the notions of being a weakly $p$-convergent sequence or a weakly $p$-Cauchy sequence depend only on the duality $(E,E')$;
\item[{\rm(v)}] $\ell_p^w(E'_{\beta}) \subseteq \ell_p^w(E'_{w^\ast})$ and the equality holds for semi-reflexive spaces;
\item[{\rm(vi)}] if $H$ is a subspace of $E$, then $\ell^w_p(H)\subseteq \ell^w_p(E)$; more precisely, $\ell^w_p(H)=\ell^w_p(E)\cap H^\w$.
\end{enumerate}
\end{lemma}


\begin{proof}
(i), (iv) and (v) are evident, and (vi) follows from the Hahn--Banach extension theorem.
\smallskip

(ii) The inclusion $\ell^w_p(E)\subseteq \ell^w_q(E)$ is clear.  To show that $\ell^w_p(E)\not= \ell^w_q(E)$, fix $p<t<q$ and a nonzero element $x\in E$. For every $n\in\w$,  define $x_n:=\tfrac{1}{(n+1)^{1/t}} \cdot x$. Then $(x_n)_{n\in\w}\in \ell_q^w(E)\SM \ell_p^w(E)$.
\smallskip

(iii) Let $\eta\in L'$. Then $\chi:=\eta\circ T\in E'$ and $(\langle\eta, T(x_n)\rangle)_{n\in\w}=(\langle\chi, x_n\rangle)_{n\in\w}$ belongs to $\ell_p$ (or to  $c_0$, if $p=\infty$). Thus $\big(T(x_n)\big)_{n\in\w}\in \ell_p^w(L)$. The case when $(x_n)_{n\in\w}\in \ell_p^w(E_\beta)$ follows from the proved one because, by Theorem 8.11.3 of \cite{NaB}, the map $T:E_\beta \to L_\beta$ is continuous.
%
Since the topology $\tau$ of $E_\beta$ is finer than the topology of $E$, we have $E'\subseteq (E_\beta)'$. Thus $\ell_p^w(E_\beta)\subseteq \ell_p^w(E)$.\qed
\end{proof}
Note that the inclusion $\ell_p^w(E'_{\beta}) \subseteq \ell_p^w(E'_{w^\ast})$ can be strict, see Example \ref{exa:c0-1-barrel}.

For the case $p=1$, the proof of the next lemma is taken from \S 16.5 of \cite{Jar}.
\begin{lemma} \label{l:topology-L^w}
Let $E$ be a locally convex space, and let  $U\in\Nn_0^c(E)$. Then for every $1\leq p<\infty$, the function
\[
\rho_U: \ell_p^w(E)\to\IR, \quad (x_n)\mapsto \sup_{\chi\in U^\circ} \big\|\big( \langle\chi,x_n\rangle \big)_n\big\|_{\ell_p},
\]
is a seminorm on $\ell_p^w(E)$, and if $p=\infty$, then the function
\[
\rho_U: c_0^w(E)\to\IR, \quad (x_n)\mapsto \sup_{\chi\in U^\circ} \big\|\big( \langle\chi,x_n\rangle \big)_n\big\|_{c_0},
\]
is a seminorm on $c_0^w(E)$.
\end{lemma}

\begin{proof}
It suffices to show that the function $\rho_U$ is well-defined, then the seminorm conditions follow from the corresponding properties of the norm of $\ell_p$ or $c_0$. Fix $(x_i)\in\ell_p^w(E)$ (or $\in c_0^w(E)$ if $p=\infty$). We distinguish between three cases.
\smallskip

{\em Case 1. Let $p=1$.}  Set
$
A_1:=\Big\{ \sum_{i=0}^n \alpha_i x_i: n\in\w \mbox{ and } (\alpha_i)\in B_{\ell_{\infty}}\Big\}.
$
 Then for each $\chi\in E'$ and $\sum_{i=0}^n \alpha_i x_i\in A_1$, we have
\[
\Big|\langle \chi,\sum_{i=0}^n \alpha_i x_i\rangle\Big| \leq \sum_{i=0}^n  |\alpha_i|\cdot |\langle\chi,x_i\rangle| \leq \big\|\big( \langle \chi,x_i\rangle\big)\big\|_{\ell_1}
\]
which means  that $A_1$ is bounded. Choose $\lambda> 0 $ such that $A_1\subseteq \lambda U$. Given $\chi\in U^\circ$, for every $i\in\w$ choose $\alpha_i\in\IF$ such that $\alpha_i \langle \chi,x_i\rangle=|\langle\chi,x_i\rangle|$. Then $(\alpha_i)\in B_{\ell_{\infty}}$ and
$
\sum_{i=0}^n \big|\langle \chi,x_i\rangle\big| =\Big|\langle \chi,\sum_{i=0}^n \alpha_i x_i\rangle\Big|\leq \lambda. 
$
Therefore $\sup_{\chi\in U^\circ} \big\|\big( \langle\chi,x_n\rangle \big)_n\big\|_{\ell_p}$ is finite, and hence $\rho_U$ is well-defined.
\smallskip

{\em Case 2. Let $1<p<\infty$.} Set
$
A_p:=\Big\{ \sum_{i=0}^n \alpha_i x_i: n\in\w \mbox{ and } (\alpha_i)\in B_{\ell_{p^\ast}}\Big\} \subseteq E.
$
We claim that $A_p$ is bounded. Indeed, let $\chi\in E'$ and $\sum_{i=0}^n \alpha_i x_i\in A_p$. Then the H\"{o}lder inequality implies
\[
\Big|\langle \chi,\sum_{i=0}^n \alpha_i x_i\rangle\Big| \leq \sum_{i=0}^n  |\alpha_i|\cdot |\langle\chi,x_i\rangle| \leq \Big( \sum_{i=0}^n  |\alpha_i|^{p^\ast}\Big)^{\tfrac{1}{p^\ast}} \cdot \Big( \sum_{i=0}^n  |\langle\chi,x_i\rangle|^p\Big)^{\tfrac{1}{p}}\leq \big\|\big( \langle\chi,x_i\rangle \big)_i\big\|_{\ell_p}.
\]
As $\chi$ was arbitrary, it follows that $A_p$ is bounded.
Therefore there is $\lambda>0$ such that $A_p\subseteq \lambda U$.

To show that $\rho_U$ is well-defined we prove that
\begin{equation} \label{equ:topology-L^w-1}
\big\|\big( \langle\chi,x_i\rangle \big)_i\big\|_{\ell_p} \leq \lambda \;\; \mbox{ for every $\chi\in U^\circ$}.
\end{equation}
To this end, fix an arbitrary $\chi\in U^\circ$ such that $\big\|\big( \langle\chi,x_i\rangle \big)_i\big\|_{\ell_p} >0$.
Set $C(\chi):= \big\|\big( \langle\chi,x_i\rangle \big)_i\big\|_{\ell_p}^{p/p^\ast}>0$. Denote by $I:=\{i\in\w: \langle\chi,x_i\rangle \not=0\}$ the ``support'' of $\chi$ on $(x_i)$. For every $i\in\w$, define $\alpha_i\in \IF$ by
\[
\alpha_i:=0 \;\mbox{ if }\; i\not\in I, \;\;\mbox{ and } \;\;\alpha_i:=\tfrac{\overline{\langle\chi,x_i\rangle}}{|\langle\chi,x_i\rangle|}\cdot \tfrac{|\langle\chi,x_i\rangle|^{p-1}}{C(\chi)} \; \mbox{ if }\; i\in I.
\]
Then the equality $(p-1)p^\ast=p$ implies
\[
\sum_{i\in\w} |\alpha_i|^{p^\ast} =\sum_{i\in I} \tfrac{|\langle\chi,x_i\rangle|^{(p-1)p^\ast}}{C(\chi)^{p^\ast}} =\tfrac{1}{\big\|\big( \langle\chi,x_i\rangle \big)_i\big\|_{\ell_p}^{p}} \cdot \sum_{i\in I} |\langle\chi,x_i\rangle|^p = 1,
\]
which means that $(\alpha_i)\in B_{\ell_{p^\ast}}$. Therefore, for every $n\in\w$, the inclusion $A_p\subseteq \lambda U$ implies
\[
\begin{aligned}
\tfrac{1}{C(\chi)} \sum_{i\leq n}|\langle\chi,x_i\rangle|^p & =\sum_{i\leq n,\; n\in I} \tfrac{\overline{\langle\chi,x_i\rangle}}{|\langle\chi,x_i\rangle|}\cdot \tfrac{|\langle\chi,x_i\rangle|^{p-1}}{C(\chi)} \cdot \langle\chi,x_i\rangle\\
& = \Big| \sum_{i\leq n} \langle\chi,\alpha_i x_i\rangle\Big|=\Big|  \big\langle\chi,\sum_{i\leq n}\alpha_i x_i \big\rangle\Big|\leq \lambda,
\end{aligned}
\]
and hence
\[
\sum_{i\leq n}|\langle\chi,x_i\rangle|^p \leq\lambda\cdot C(\chi)=\lambda \cdot \left(\sum_{i\in\w}|\langle\chi,x_i\rangle|^p\right)^{1/p^\ast}
\Longleftrightarrow \left(\sum_{i\in\w}|\langle\chi,x_i\rangle|^p\right)^{1-\tfrac{1}{p^\ast}}\leq \lambda.
\]
Thus (\ref{equ:topology-L^w-1}) is proved.
\smallskip

{\em Case 3. Let  $p=\infty$.}  Since $S=\{x_i\}_{i\in\w}$ is a bounded subset of $E$, there is $\lambda> 0 $ such that $S\subseteq \lambda U$. Then for every $i\in\w$ and each $\chi\in U^\circ$, we have $|\langle \chi,x_i\rangle|=\lambda\cdot |\langle \chi,\tfrac{1}{\lambda} \cdot x_i\rangle|\leq \lambda$. Thus $\rho_U$ is well-defined.\qed
\end{proof}

As an immediate corollary of Lemma \ref{l:topology-L^w} we obtain the next result.
\begin{proposition} \label{p:topology-L^w}
Let $p\in[1,\infty]$, and let $E$ be a locally convex space. Then the family $\{\rho_U: U\in\Nn_0^c(E)\}$ defines a locally convex vector topology $\TTT_p$ on the vector space $\ell_p^w(E)$ {\rm(}or on $c_0^w(E)$ if $p=\infty${\rm)}.
\end{proposition}

\begin{notation} \label{n:l^w[E]} {\em
Following notations of \S~19.4 of \cite{Jar}, we set $\ell_p[E]:=\big(\ell_p^w(E),\TTT_p\big)$ if $1\leq p<\infty$ and $c_0[E]:=\big(c_0^w(E),\TTT_\infty\big)$ if $p=\infty$.\qed}
\end{notation}

The next easy assertion shows that $E$ can be considered as a direct summand of the spaces $\ell_p[E]$ and $c_0[E]$.
\begin{proposition} \label{p:L^w-E-summand}
For a locally convex space $E$ and $p\in[1,\infty]$, the map $I:E\to \ell_p[E]$ {\rm(}or $I:E\to c_0[E]$ if $p=\infty${\rm)} defined by  $x\mapsto (x,0,\dots)$ is an embedding onto a direct summand.
\end{proposition}

\begin{proof}
We consider only the case $1\leq p<\infty$, the case $p=\infty$ is considered analogously. Consider the subspace ${\tilde E}:=\{(x_i)\in\ell_p[E]: (x_i)=(x_0,0,\dots)\}$. Since the projection $P:\ell_p[E]\to \ell_p[E]$ onto the first coordinate  defined by $P(x_i):=(x_0,0\dots)$ is the converse to the identity embedding of ${\tilde E}$, we see that ${\tilde E}$ is a direct summand of $\ell_p[E]$. It is clear that $I$ is a linear isomorphism of $E$ onto ${\tilde E}$. If $U\in\Nn_0^c(E)$, then $I(U)={\tilde E}\cap \{(x_i)\in\ell_p[E]: \rho_U(x)\leq 1\}$. Therefore $I$ is a homeomorphism. \qed
\end{proof}

Proposition 19.4.2 of \cite{Jar} states that if $E$ is complete, then the spaces $\ell_p[E]$ and $c_0[E]$ are complete as well. Below we extend this assertion.

\begin{proposition} \label{p:L^w-E-complete}
Let $p\in[1,\infty]$, and let $E$ be a locally convex space. Then $E$ is complete {\rm(}resp., quasi-complete, separably quasi-complete, or sequentially complete{\rm)} if and only if so is  $\ell_p[E]$ {\rm(}or $c_0[E]$ if $p=\infty${\rm)}.
\end{proposition}

\begin{proof}
The condition is sufficient because, by Proposition \ref{p:L^w-E-summand}, $E$ is a direct summand of $\ell_p[E]$ (or of  $c_0[E]$ if $p=\infty$). To prove the necessity, assume that $E$ is complete (resp., quasi-complete, separably quasi-complete, or sequentially complete). We consider only the case $p<\infty$ since the case $p=\infty$ can be proved analogously.
For every $n\in\w$, let $P_n$ be the projection of $\ell_p[E]$ onto the $n$th coordinate. As in the proof of Proposition \ref{p:L^w-E-summand}, we obtain that $P_n(\ell_p[E])$ is topologically isomorphic to $E$.

Fix a Cauchy net $\big\{(x_{n,\alpha})_{n\in\w}\big\}_{\alpha\in\AAA}$ in $\ell_p[E]$ (resp., in a closed bounded subset $A$  of $\ell_p[E]$, a closed bounded and separable subset $A$ of $\ell_p[E]$, or a Cauchy sequence for $\AAA=\w$). For every $n\in\w$, the projection $\big\{x_{n,\alpha}\big\}_{\alpha\in\AAA}$ of the net $\big\{(x_{n,\alpha})_{n\in\w}\big\}_{\alpha\in\AAA}$ onto the $n$th coordinate is Cauchy in $P_n(\ell_p[E])$ (resp., in the closure $B$ of the bounded subset $P_n(A)$ of $P_n(\ell_p[E])$, in the closure $B$ of the bounded and separable subset  $P_n(A)$ of $P_n(\ell_p[E])$, or a Cauchy sequence in $P_n(\ell_p[E])$ for $\AAA=\w$).  Therefore, $\big\{x_{n,\alpha}\big\}_{\alpha\in\AAA}$ has a limit point $z_n$. Set $z:=(z_n)$.

We claim that $z\in \ell_p^w(E)$. Indeed, fix an arbitrary $\chi\in E'$. Choose $V\in\Nn_0^c(E)$ such that $\chi\in V^\circ$. Then
\[
\sum_{n\in\w} |\langle\chi,x_n\rangle|^p \leq 1 \;\;\mbox{ if } \;\; \rho_V\big((x_n)_{n\in\w}\big)\leq 1,
\]
which means that the map $\ell_p[E]\to\ell_p$, $(x_n)\mapsto \big(\langle\chi,x_n\rangle\big)_n$, is a well-defined operator. Therefore $\big\{\big(\langle\chi,x_{n,\alpha}\rangle\big)_{n\in\w}\big\}_{\alpha\in\AAA}$ is a Cauchy net (or a Cauchy sequence if $\AAA=\w$) in $\ell_p$ and hence converges to some $(a_n)\in\ell_p$. It is clear that $a_n=\langle\chi,z_{n}\rangle$ for every $n\in\w$. Thus $z\in \ell_p^w(E)$.

Taking into account that $A$ is closed, it remains to show that $(x_{n,\alpha})_{n\in\w}\to z$. To this end, fix $U\in\Nn_0^c(E)$ and $\delta>0$. Choose $\gamma\in\AAA$ such that
\[
\sum_{n\in\w} |\langle\chi, x_{n,\alpha}-x_{n,\beta}\rangle|^p \leq \delta^p \;\; \mbox{ for all } \; \chi\in U^\circ \; \mbox{ and } \; \alpha,\beta\geq \gamma.
\]
Passing to the limit by $\beta$ in the above inequality, we obtain $\sum_{n\in\w} |\langle\chi, x_{n,\alpha}-z_{n}\rangle|^p \leq \delta^p$ which means that $\rho_U\big( (x_{n,\alpha})_{n\in\w}, (z_{n})_{n\in\w}\big)\leq \delta$ for every $\alpha\geq \gamma$. Thus $z$ is the limit of $\big\{(x_{n,\alpha})_{n\in\w}\big\}_{\alpha\in\AAA}$ in $\ell_p[E]$.\qed
\end{proof}

Recall that an lcs $E$ has a {\em $\GG$-base} or an {\em $\w^\w$-base} if it has a base $\{ U_\alpha : \alpha\in\w^\w\}$ of neighborhoods at zero such that $U_\beta \subseteq U_\alpha$ whenever $\alpha\leq\beta$ for all $\alpha,\beta\in\w^\w$, where $\alpha=(\alpha(n))_{n\in\w}\leq \beta=(\beta(n))_{n\in\w}$ if $\alpha(n)\leq\beta(n)$ for all $n\in\w$. Each metrizable or strict $(LF)$ space has a $\GG$-base.
\begin{corollary} \label{c:L[E]-Banach}
Let $p\in[1,\infty]$, and let $E$ be a locally convex space. Then:
\begin{enumerate}
\item[{\rm(i)}] $E$ is a Banach {\rm(}resp., normed, metrizable or Fr\'{e}chet{\rm)} space if and only if so is $\ell_p[E]$ {\rm(}or $c_0[E]$ if $p=\infty${\rm)};
\item[{\rm(ii)}] $E$ has a $\GG$-base if and only if $\ell_p[E]$ {\rm(}or $c_0[E]$ if $p=\infty${\rm)} has a $\GG$-base.
\end{enumerate}
\end{corollary}

\begin{proof}
Observe that if $U,V\in\Nn_0^c(E)$ are such that $U\subseteq V$, then $\rho_U\geq\rho_V$. Now the corollary immediately follows from Propositions \ref{p:topology-L^w}, \ref{p:L^w-E-summand} and \ref{p:L^w-E-complete}.\qed
\end{proof}

We shall use repeatedly the next observation.
\begin{lemma} \label{l:lp-finite-dim}
Let $p\in[1,\infty]$, and let $E$ be a locally convex space.  A sequence $\{\chi_n\}_{n\in\w}$ in $E'$ is weakly $p$-summable in $E'_\beta$ {\rm(}or in $E'_{w^\ast}${\rm)} and finite-dimensional if and only if there are linearly independent elements $\eta_1,\dots,\eta_s\in E'$ and sequences $(a_{1,n}),\dots,(a_{s,n})\in \ell_p$ {\rm(}or $\in c_0$ if $p=\infty${\rm)} such that
\[
\chi_n=a_{1,n} \eta_1 +\cdots +a_{s,n}\eta_s \;\; \mbox{ for every $n\in\w$}.
\]
\end{lemma}

\begin{proof}
The sufficiency is clear. To prove the necessity, let $\eta_1,\dots,\eta_s\in E'$ be a basis of $\spn\big(\{\chi_n\}_{n\in\w}\big)$. Then for every $n\in\w$, there is a unique representation $\chi_n=a_{1,n} \eta_1 +\cdots +a_{s,n}\eta_s$. Let $\xi_i\in E''$ (resp., $\in E$) be such that $\langle\xi_i,\eta_i\rangle=1$ and $\langle\xi_i,\eta_j\rangle=0$ for every $j\not= i$. Since $\{\chi_n\}_{n\in\w}$ is weakly $p$-summable, we have $\big(\langle\xi_i,\chi_n\rangle\big)_n =(a_{i,n})_n\in \ell_p$ (or $\in c_0$ if $p=\infty$) for every $i=1,\dots,s$.\qed
\end{proof}

Let $E$ be a locally convex space, and let $L$ be a normed space. In the next proposition we consider the topology on $\LL(L,E)$ whose base of neighborhoods at zero operator is the family
\[
\big\{[B_L;U]: U\in\Nn_0^c(E)\big\},\quad \mbox{ where }\; [B_L;U]:=\{T\in \LL(L,E): T(B_L)\subseteq U\}.
\]

\begin{proposition} \label{p:Lp-E-operator}
Let $p\in[1,\infty)$, and let $E$ be a locally convex space. Then the map
\[
R_p: \LL(\ell_{p^\ast}^0,E) \to \ell_p^w(E)\;\; \mbox{ $($or $R_1: \LL(c_0^0,E)\to \ell_1^w(E)$ if $p=1$$)$}
\]
 defined by
\[
R_p(T):=\big(T(e_n^\ast)\big)_{n\in\w},
\]
is a linear homeomorphism. If in addition $E$ is  sequentially complete, then the same map
\[
R_p: \LL(\ell_{p^\ast},E) \to \ell_p^w(E), \;\;  (\mbox{or $R_1: \LL(c_0)\to \ell_1^w(E)$ if $p=1$}), \quad R_p(T):=\big(T(e_n^\ast)\big)_{n\in\w}
\]
is a linear homeomorphism as well.
\end{proposition}

\begin{proof}
To show that $R_p$ is well-defined, fix an arbitrary  $T\in \LL(\ell_{p^\ast}^0,E)$ (or $T\in \LL(c_0^0,E)$ if $p=1$). Then $R_p(T)=\big(T(e_n^\ast)\big)$. Since $T^\ast\in \LL(E'_\beta,\ell_p)$, for every $\chi\in E'$, we have
\begin{equation} \label{equ:Lp-E-operator-1}
\big( \langle\chi, T(e_n^\ast)\rangle\big)_{n\in\w}=\big( \langle T^\ast(\chi), e_n^\ast\rangle\big)_{n\in\w}=T^\ast(\chi)\in \ell_p ,
\end{equation}
which means that $R_p(T) \in \ell_p^w(E)$, as desired. It is clear that $R_p$ is linear and injective.

To prove that $R_p$ is also surjective, fix an arbitrary $\big( x_n\big)_{n\in\w}\in\ell_p^w(E)$. Define a linear map $T: \ell_{p^\ast}^0\to E$ (or $T:c_0^0\to E$ if $p=1$) by $T(e_n^\ast)=x_n$ for every $n\in\w$.  To prove that $T$ is continuous it suffices to show that $T$ is bounded. To this end, we prove that the set $Q:=T\big(B_{\ell^0_{p^\ast}}\big)$ (or $Q:=T\big(B_{c^0_0}\big)$  if $p=1$) is a bounded subset of $E$. Fix $\chi\in E'$. If $p>1$, the H\"{o}lder inequality implies
\[
\big| \langle\chi, T(\sum a_ie^\ast_i)\rangle \big|=\big| \sum a_i \langle\chi,x_i\rangle\big|\leq \big(\sum |a_i|^{p^\ast}\big)^{1/p^\ast} \cdot \big(\sum |\langle\chi,x_i\rangle|^p\big)^{1/p} \leq \big\| \big(\langle\chi,x_i\rangle\big)_i \|_{\ell_p},
\]
and if $p=1$, we obtain
\[
\big| \langle\chi, T(\sum a_ie^\ast_i)\rangle \big|=\big| \sum a_i \langle\chi,x_i\rangle\big|\leq \max\{ |a_i|:i\in\w\} \cdot \sum |\langle\chi,x_i\rangle|\leq \big\| \big(\langle\chi,x_i\rangle\big)_i \|_{\ell_1}.
\]
In any case, we obtain that $|\langle\chi,Q\rangle|$ is a bounded subset of $\IR$. Thus $Q$ is bounded, as desired.

Assume that $E$ is sequentially complete. Then, by Proposition \ref{p:extens-bounded}, the operator $T$ can be extended to a bounded operator from $\ell_{p^\ast}$ (or from $c_0^0$ onto $c_0$ if $p=1$) to $E$. Thus $R_p$ is a well-defined map from $\LL(\ell_{p^\ast},E)$ onto $\ell_p^w(E)$ (or onto $\ell_1^w(E)$ if $p=1$).

In both cases, by construction, we have $R_p(T)=\big( x_n\big)_{n\in\w}$ and hence $R_p$ is surjective. Thus $R_p$ is a linear isomorphism.

To show that $R_p$ is also a homeomorphism, let $U\in\Nn_0^c(E)$. Denote by $B$ the closed unit ball of $\ell_{p^\ast}^0$ (resp., $\ell_{p^\ast}$, $c_0^0$ or $c_0$). Then (\ref{equ:Lp-E-operator-1}) implies
\[
\begin{aligned}
\rho_U\big(R_p(T)\big) \leq 1 & \Leftrightarrow \sup_{\chi\in U^\circ} \big\| \big( \langle\chi,R_p(T)\rangle\big)\big\|_{\ell_p}\leq 1
\Leftrightarrow \sup_{\chi\in U^\circ} \big\| \big( \langle\chi,\big(T(e_n^\ast)\big)_{n\in\w}\rangle\big)\big\|_{\ell_p} \leq 1\\
& \Leftrightarrow \sup_{\chi\in U^\circ}\|T^\ast(\chi)\|_{\ell_p}\leq 1 \Leftrightarrow \sup_{\chi\in U^\circ} \sup_{x\in B} |\langle T^\ast(\chi),x\rangle|\leq 1 \\
&  \Leftrightarrow \sup_{x\in B} \Big( \sup_{\chi\in U^\circ}  |\langle \chi,T(x)\rangle|\Big)\leq 1 \Leftrightarrow T(B)\subseteq U^\circ \Leftrightarrow T\in [B;U].
\end{aligned}
\]
Thus $R_p$ is a homeomorphism. \qed
\end{proof}

\begin{remark} \label{rem:Rp-infty} {\em
The condition $p<\infty$ in Proposition \ref{p:Lp-E-operator} is essential. Indeed, for $p=\infty$, let $E=\ell_1$ and let $T:\ell_1^0\to E$ be the identity map. Then $R_\infty(T)=(e_n^\ast)$ does not belong to $c_0^w(E)$ (= the family of weakly null sequences in $E$) by the Schur property.\qed}
\end{remark}

Below we introduce new classes of weak barrelledness conditions which will be studied in the next Section \ref{sec:p-barrelled} and used actively in what follows.
\begin{definition} \label{def:p-barrelled} {\em
Let $p\in[1,\infty]$. A locally convex space $E$ is called
\begin{enumerate}
\item[{\rm(i)}]  {\em $p$-barrelled } if every weakly $p$-summable sequence in $E'_{w^\ast}$ is equicontinuous;
\item[{\rm(ii)}] {\em $p$-quasibarrelled } if every weakly $p$-summable sequence in $E'_\beta$ is equicontinuous.\qed
\end{enumerate}}
\end{definition}

Let $p\in[1,\infty]$, $E$ be a locally convex space, and let $T:E\to \ell_p$ (or $T:E\to c_0$ if $p=\infty$) be an operator. Then for every $x\in E$, we have (recall that $\{e^\ast_n\}_{n\in\w}$ is the canonical basis of $\ell_{p^\ast}$)
\[
T(x)=\big( \langle e_n^\ast, T(x)\rangle\big)_{n\in\w}=\big( \langle T^\ast(e_n^\ast),x\rangle\big)_{n\in\w}\in\ell_p \; \mbox{ (or $\in c_0$ if $p=\infty$)}.
\]
For every $n\in\w$, set $\chi_n:= T^\ast(e_n^\ast)$. Therefore $T$ defines a sequence $\{\chi_n\}$ in $E'$ such that
\[
T(x)=\big( \langle \chi_n,x\rangle\big)_{n\in\w}\in\ell_p \mbox{ (or $\in c_0$ if $p=\infty$) for every $x\in E$}.
\]
Analogously to the map $R_p$ defined in Lemma \ref{l:prop-p-sum}, for any locally convex vector topology $\TTT$ on $E'$  finer than $\sigma(E',E)$, we can define a map $S_p^\TTT:\LL(E,\ell_p)\to \ell_p^w(E',\TTT)$ (or $S_\infty^\TTT: \LL(E,c_0)\to \ell_\infty^w(E',\TTT)=c_0^w(E',\TTT)$ if $p=\infty$)  by
\[
S_p^\TTT(T):=\big(T^\ast(e^\ast_n)\big)_{n\in\w}.
\]
It is clear that $\ell_p^w(E',\TTT)$ can and will be considered as a vector subspace of $\ell_p^w(E'_{w^\ast})$. By this reason we shall omit the superscript $\TTT$ and write simply $S_p$, and hence we obtain the linear map
\[
S_p:\LL(E,\ell_p)\to \ell_p^w(E'_{w^\ast}) \quad \big(\mbox{or $S_\infty: \LL(E,c_0)\to \ell_\infty^w(E'_{w^\ast})=c_0^w(E'_{w^\ast})$ if $p=\infty$}\big)
\]
defined by $S_p(T):=\big(T^\ast(e^\ast_n)\big)_{n\in\w}$.

\begin{proposition} \label{p:operator-Lp}
Let $p\in[1,\infty]$, and let $E$ be a locally convex space. Then:
\begin{enumerate}
\item[{\rm(i)}] $S_p$ is a well-defined linear injection whose image contains equicontinuous sequences $($so that the sequence $\big\{T^\ast(e^\ast_n)\big\}_{n\in\w}$ is bounded in $E'_\beta$$)$; hence if $S_p$ is surjective, then $E$ is a $p$-barrelled space;
\item[{\rm(ii)}] the map $S_\infty$ is an isomorphism  if and only if $E$ is $\infty$-barrelled $(=$ $c_0$-barrelled$)$;
\item[{\rm(iii)}] if $1<p<\infty$, then the image of $S_p$ is contained in $\ell_p^w(E'_\beta)$; hence if $S_p\big( \LL(E,\ell_p)\big)= \ell_p^w(E'_\beta)$, then $E$ is a $p$-quasibarrelled space.
\end{enumerate}
\end{proposition}

\begin{proof}
(i) To show that $S_p$ is well-defined, let $T\in \LL(E,\ell_p)$ (or $T\in \LL(E,c_0)$ if $p=\infty$). Then for every $x\in E=\big(E'_{w^\ast}\big)'$, we have
\[
\big(\langle T^\ast(e^\ast_n),x\rangle\big)_{n\in\w} = \big(\langle e^\ast_n,T(x)\rangle)=T(x)\in\ell_p \; (\mbox{or  $\in c_0$ if $p=\infty$}),
\]
and hence $S_p(T)\in\ell_p^w(E'_{w^\ast})$. The injectivity of $S_p$ follows from the fact that $(e^\ast_n)$ is a  basis of $\ell_{p^\ast}$ (or $c_0$ in the case $p=1$).

To show that the image of $S_p$ contains equicontinuous sequences assume that $p<\infty$ (the case $p=\infty$ can be considered analogously). Fix $S_p(T)=(\chi_n)\in \ell_p^w(E'_{w^\ast})$, where $T\in  \LL(E,\ell_p)$ and $\chi_n=T^\ast(e^\ast_n)$ for every $n\in\w$. The continuity of $T$ implies that there is $U\in\Nn_0(E)$ such that $\|T(x)\|_{\ell_p}\leq 1$ for every $x\in U$. In particular,  for every $n\in\w$, we have
\[
| \langle \chi_n ,x\rangle| = | \langle T^\ast(e^\ast_n),x\rangle|= | \langle e^\ast_n,T(x)\rangle|\leq 1 \;\; \mbox{ for each }\; x\in U.
\]
Therefore  $\{\chi_n\}_{n\in\w}\subseteq U^\circ$, and hence the sequence $\{\chi_n\}_{n\in\w}$ is equicontinuous; it is strongly bounded by Theorem 11.3.5 of \cite{NaB}.

Assume now that the map $S_p$ is onto. Then each sequence $(\chi_n)\in \ell_p^w(E'_{w^\ast})$ (or $(\chi_n)\in c_0^w(E'_{w^\ast})$ if $p=\infty$) is equicontinuous. By definition this means that $E$ is $p$-barrelled.
\smallskip

(ii) By (i) it suffices to prove that the map $S_\infty$ is surjective if $E$ is $\infty$-barrelled. To this end, let $(\chi_n)\in c_0^w(E'_{w^\ast})$. So $(\chi_n)$ is a weak$^\ast$ null-sequence in $E'$. Define a map $T:E\to c_0$ by $T(x):= \big(\langle \chi_n, x\rangle)$. Clearly, $T$ is a well-defined linear map. We show that $T$ is continuous. Since  $E$ is $\infty$-barrelled, the sequence $(\chi_n)$ is equicontinuous and hence there is $U\in\Nn_0(E)$ such that $\chi_n\in U^\circ$ for every $x\in U$ and each $n\in\w$. Therefore, for every $x\in U$, we obtain
\[
\|T(x)\|_{c_0}= \sup_{n\in\w} |\langle \chi_n, x\rangle|\leq 1
\]
and hence $T(U)\subseteq B_{c_0}$. Therefore $T$ is continuous. Observe that
\[
\langle T^\ast(e^\ast_n),x\rangle=\langle e^\ast_n, T(x)\rangle=\langle \chi_n, x\rangle\;\; \mbox{ for every $x\in E$ and $n\in\w$},
\]
which means that $T^\ast(e^\ast_n)=\chi_n$ and hence $S_\infty(T)=(\chi_n)$, as desired.
\smallskip

(iii) 
To show that the image of $S_p$ is contained in $\ell_p^w(E'_\beta)$, let $T\in \LL(E,\ell_p)$. Fix an arbitrary $\eta\in E''$. Then $T^{\ast\ast}(\eta)\in (\ell_p)''=\ell_p$ and hence
\[
\sum_{n\in\w}\big|\langle \eta, T^\ast(e^\ast_n)\rangle\big|^p = \sum_{n\in\w}\big|\langle T^{\ast\ast}(\eta), e^\ast_n\rangle\big|^p =\big\| T^{\ast\ast}(\eta)\big\|_{\ell_p}^p,
\]
therefore $S_p(T)\in\ell_p^w(E'_{\beta})$.

If the image of $S_p$ is the whole space $\ell_p^w(E'_{\beta})$, then, by (i), each sequence $(\chi_n)\in \ell_p^w(E'_{\beta})$ is equicontinuous. By definition this means that $E$ is $p$-quasibarrelled.\qed
\end{proof}

One can naturally find some sufficient conditions on a locally convex space  $E$  for which the map $S_p:\LL(E,\ell_p)\to \ell_p^w(E'_{w^\ast})$ is surjective. It is not surprising that barrelled spaces satisfy this condition. We prove this assertion in Proposition \ref{p:p-sum-operator}.

\begin{lemma} \label{l:p-sum-graph}
Let $p\in[1,\infty]$, and let $E$ be a locally convex space. If $\{\chi_n\}_{n\in\w}$ is a weakly $p$-summable sequence in $E'_{w^\ast}$, then the linear map $T:E\to\ell_p$ $($or $T:E\to c_0$ if $p=\infty$$)$ defined by $T(x):=(\langle\chi_n,x\rangle)$ has closed graph.
\end{lemma}

\begin{proof}
Assume that a net $\big\{\big(x_i,T(x_i)\big)\big\}$ converges to a point $\big(x,(y_n)\big)$ in $E\times \ell_p$ (or $E\times c_0$ if $p=\infty$). Since the coordinate projections are continuous, it follows that  $x_i\to x$ and $\langle\chi_n,x_i\rangle\to y_n$ for every $n\in\w$. Since all $\chi_n$ are continuous and $x_i\to x$, we obtain $\langle\chi_n,x_i\rangle\to \langle\chi_n,x\rangle$ and hence $y_n=\langle\chi_n,x\rangle$ for every $n\in\w$. Therefore $\big(x,(y_n)\big)=\big(x,T(x)\big)$. Thus the graph of $T$ is closed.\qed
\end{proof}

\begin{proposition} \label{p:p-sum-operator}
Let $p\in[1,\infty]$, and let $E$ be a barrelled locally convex space. Then for every  weakly $p$-summable sequence $\{\chi_n\}_{n\in\w}$ in $E'_{w^\ast}$, the linear map $T:E\to\ell_p$  $($or $T:E\to c_0$ if $p=\infty$$)$  defined by $T(x):=(\langle\chi_n,x\rangle)$ is continuous. Consequently, the linear map $S_p$ is  a bijection.
\end{proposition}

\begin{proof}
By Lemma \ref{l:p-sum-graph}, the linear map $T$ has closed graph. Therefore, by the Closed Graph Theorem \cite[14.3.4]{NaB}, $T$ is continuous. Thus $S_p$ is surjective and hence, by (i) of Proposition \ref{p:operator-Lp}, $S_p$ is bijective.\qed
\end{proof}

\begin{corollary} \label{c:1-sum}
Let $E$ be a barrelled locally convex space, and let $\{\chi_n\}_{n\in\w}\subseteq E'_{w^\ast}$ be a weakly $1$-summable sequence. Then for every $t=(t_n)\in\ell_\infty$, the map $\eta_t:E\to \IF$ defined by $\langle \eta_t,x\rangle:=\sum_{n\in\w} t_n\cdot \langle\chi_n,x\rangle$ is continuous.
\end{corollary}

\begin{proof}
By Proposition \ref{p:p-sum-operator}, the map $T:E\to\ell_1$ defined by $T(x)=\big(\langle\chi_n,x\rangle\big)$ is continuous. It remains to note that $\langle \eta_t,x\rangle=\langle t,T(x)\rangle$.\qed
\end{proof}



Instead of to say ``weakly $p$-summable sequence in $E'_{w^\ast}$'' as in Lemma  \ref{l:p-sum-graph} and Proposition \ref{p:p-sum-operator}, we introduce a more convenient notation. For Banach spaces this notion is introduced in \cite{FZ-p}.
\begin{definition}\label{def:tvs-weak*-p-sum}{\em
Let $p\in[1,\infty]$, and let $E$ be a separated  topological vector space.
A sequence  $\{\chi_n\}_{n\in\w}$ in $E'$ is called
\begin{enumerate}
\item[$\bullet$] {\em weak$^\ast$ $p$-summable} if it is weakly $p$-summable in $E'_{w^\ast}$, i.e., if  for every $x\in E$, it follows that $(\langle\chi_n, x\rangle)_{n\in\w} \in\ell_p$ if $p<\infty$, or  $(\langle\chi_n, x\rangle)_{n\in\w} \in c_0$ if $p=\infty$;
\item[$\bullet$] {\em weak$^\ast$ $p$-convergent to $\chi\in E'$} if  $\{\chi_n-\chi\}_{n\in\w}$ is  weakly $p$-summable in $E'_{w^\ast}$;
\item[$\bullet$] {\em weak$^\ast$ $p$-Cauchy} if it is weakly $p$-Cauchy in $E'_{w^\ast}$, i.e.,  if for each pair of strictly increasing sequences $(k_n),(j_n)\subseteq \w$, the sequence  $\{\chi_{k_n}-\chi_{j_n}\}_{n\in\w}$ is weak$^\ast$ $p$-summable;
\end{enumerate}
}
\end{definition}

Note that $\ell_p^{w}(E'_{w^\ast})$ (or $c_0^{w}(E'_{w^\ast})$ if $p=\infty$) denotes the family of all weak$^\ast$ $p$-summable sequences in $E'$.  For simplicity of notations we shall write $\ell_p^{w^\ast}(E'):=\ell_p^{w}(E'_{w^\ast})$ and $c_0^{w^\ast}(E'):=c_0^{w}(E'_{w^\ast})$. Definitions \ref{def:tvs-weakly-p-sum} and \ref{def:tvs-weak*-p-sum} immediately imply the next observation.
\begin{proposition} \label{p:weak*-sum-weakly}
Let $E$ be a locally convex space. Then every weakly $p$-summable sequence in $E'_\beta$ is weak$^\ast$ $p$-summable. If in addition $E$ is semi-reflexive, then the converse is also true.
\end{proposition}
The condition of being a semi-reflexive space in Proposition \ref{p:weak*-sum-weakly} is essential as Example \ref{exa:c0-1-barrel} shows.

If $1\leq p<q\leq\infty$,  we have the following relationships between the notions introduced in Definition \ref{def:p-barrelled}:
\begin{equation} \label{equ:summable*}
\xymatrix{
{\substack{\mbox{weak$^\ast$} \\ \mbox{$1$-summable}}} \ar@{=>}[r] & {\substack{\mbox{weak$^\ast$} \\ \mbox{$p$-summable}}}  \ar@{=>}[r] & {\substack{\mbox{weak$^\ast$} \\ \mbox{$q$-summable}}}  \ar@{=>}[r] & {\substack{\mbox{weak$^\ast$ $\infty$-summable}\;  \\ \mbox{(= weak$^\ast$ null)}}} }
\end{equation}


The case when sequences from $\ell_p^{w}(E'_{w^\ast})$ are equicontinuous is of independent interest.
\begin{definition} \label{def:equ-Lp-seq} {\em
Let $p\in[1,\infty]$, and let $E$ be a locally convex space. Denote by $\mathcal{E}\ell_p^{w}(E'_{w^\ast})$ or $\mathcal{E}c_0^{w}(E'_{w^\ast})$ the family of  all {\em equicontinuous} sequences from $\ell_p^{w}(E'_{w^\ast})$  or $c_0^{w}(E'_{w^\ast})$.\qed}
\end{definition}

We shall use the next simple assertion.
\begin{proposition} \label{p:equi-p-summab}
Let $p\in[1,\infty]$, and let $E$ be a locally convex space such that $E=E_w$. Then every equicontinuous weak$^\ast$ $p$-summable sequence $\{\chi_n\}_{n\in\w}$ in $E'$ is finite-dimensional.
\end{proposition}

\begin{proof}
Since $\{\chi_n\}$ is equicontinuous and $E$ carries its weak topology, there is a finite subset $F$ of $E'$ such that $\{\chi_n\}_{n\in\w}\subseteq \big( F^\circ\big)^\circ$. Then $\{\chi_n\}_{n\in\w}\subseteq \spn(F)$ and hence $\{\chi_n\}_{n\in\w}$ is finite-dimensional.\qed 
\end{proof}






The next lemma describes weakly $p$-summable sequences and  weak$^\ast$ $p$-summable sequences in direct products and direct sums and their (strong) duals. 

\begin{lemma} \label{l:support-p-sum}
Let  $p\in[1,\infty]$, and let $\{E_i\}_{i\in I}$  be a non-empty family of locally convex spaces. Then:
\begin{enumerate}
\item[{\rm(i)}] a sequence $\{\chi_n\}_{n\in\w} \subseteq \big(\prod_{i\in I} E_i\big)'_\beta$ is weak$^\ast$  $p$-summable $($resp., weakly $p$-summable$)$ if and only if its support $F:=\supp\{\chi_n\}$ is finite and the sequence $\{\chi_{n}(i)\}_{n\in\w}$ is weak$^\ast$  $p$-summable $($resp., weakly $p$-summable$)$ in $(E_i)'_\beta$ for every $i\in F$;
\item[{\rm(ii)}] a sequence $\{x_n\}_{n\in\w} \subseteq \prod_{i\in I} E_i$  is weakly $p$-summable  if and only if the sequence $\{x_{n}(i)\}_{n\in\w}$ is weakly  $p$-summable in $E_i$ for every $i\in I$;
\item[{\rm(iii)}] a sequence $\{\chi_n\}_{n\in\w} \subseteq \big(\bigoplus_{i\in I} E_i\big)'_\beta$ is weak$^\ast$  $p$-summable $($resp., weakly $p$-summable$)$ if and only if the sequence $\{\chi_{n}(i)\}_{n\in\w}$ is weak$^\ast$  $p$-summable $($resp., weakly $p$-summable$)$ in $(E_i)'_\beta$ for every $i\in F$;
\item[{\rm(iv)}] a sequence $\{x_n\}_{n\in\w} \subseteq \bigoplus_{i\in I} E_i$  is weakly $p$-summable  if and only if its support $F:=\supp\{x_n\}$ is finite and the sequence $\{x_{n}(i)\}_{n\in\w}$ is weakly  $p$-summable in $E_i$ for every $i\in F$.
\end{enumerate}
\end{lemma}

\begin{proof}
(i) To prove the necessity we suppose for a contradiction that the support $F$ is infinite. Then there is a sequence $(n_k)$ in $\w$ such that
\[
\supp(\chi_{n_{k+1}}) \SM \bigcup_{i\leq k} \supp(\chi_{n_{i}}) \not=\emptyset.
\]
Therefore there is a one-to-one sequence $(j_k)$ in $I$ such that
\begin{equation} \label{equ:p-sequence-sum-0}
j_0\in \supp(\chi_{n_{0}}) \; \; \mbox{ and }\;\; j_{k+1} \in \supp(\chi_{n_{k+1}}) \SM \bigcup_{i\leq k} \supp(\chi_{n_{i}}).
\end{equation}
For every $k\in\w$, by induction on $k$, choose $x_{j_k}\in E_{j_k}$ such that
\begin{equation} \label{equ:p-sequence-sum}
\big\langle \chi_{n_{k+1}}(j_{k+1}), x_{j_{k+1}} \big\rangle > k+1 + \sum_{s\leq  k} \big| \big\langle \chi_{n_{k+1}}(j_{s}), x_{j_s} \big\rangle\big|.
\end{equation}
Define $\xxx=(x_i)\in \prod_{i\in I} E_i$ by $x_i=x_{j_k}$ if $i=j_k$ for some $k\in\w$, and $x_i=0$ otherwise. Then, for every $k\in\w$,  (\ref{equ:p-sequence-sum-0}) and (\ref{equ:p-sequence-sum}) imply $\langle \chi_{n_{k+1}},\xxx\rangle > k+1$ and hence the sequence $\{\chi_n\}$ is not weak$^\ast$ (resp., weakly)  $p$-summable, a contradiction. Thus $F$ is finite.

If $j\in F$ and $\xxx=(x_i)\in \prod_{i\in I} E_i$ (resp., $\xxx=(x_i)\in \big( \prod_{i\in I} E_i\big)''_\beta=\prod_{i\in I} (E_i)''_\beta$, see Proposition \ref{p:product-sum-strong}) is such that only the $j$th coordinate is non-zero, then the definition of  weak$^\ast$  (resp., weakly) $p$-summable sequences implies that the sequence $\{\chi_{n}(j)\}_{n\in\w}$ is weak$^\ast$  (resp., weakly) $p$-summable in $(E_j)'_\beta$.

The sufficiency is trivial.
\smallskip

(ii) and (iii) follow from Proposition \ref{p:product-sum-strong}.
\smallskip

(iv) The assertion is trivial if $I$ is finite. If $I$ is infinite, the assertion holds true by the case of finite $I$ and the well known fact that every bounded subset of $\bigoplus_{i\in I} E_i$ has finite support. \qed
\end{proof}


\section{$p$-(quasi)barrelled locally convex spaces} \label{sec:p-barrelled}


Let $p\in[1,\infty]$.  In the previous section we introduced new weak barrelledness conditions of being $p$-(quasi)barrelled spaces. Recall that an lcs $E$ is {\em $p$-(quasi)barrelled } if every weakly $p$-summable sequence in $E'_{w^\ast}$ (resp., $E'_\beta$) is equicontinuous. Since these spaces naturally appear in the study of geometrical properties of locally convex spaces like $V$-, $V^\ast$- or Gelfand--Phillips type properties, in this section we prove some general results and provide several examples which show the  relationships between the weak barrelledness conditions introduced in Definitions \ref{def:weak-barrel} and \ref{def:p-barrelled}. We start from the following lemma which immediately follows from the corresponding definitions and  (ii) of Lemma \ref{l:prop-p-sum}.
\begin{lemma} \label{l:inf-c0-quasibarelled}
Let $E$ be a locally convex space.
\begin{enumerate}
\item[{\rm(i)}] If $E$ is $\ell_\infty$-$($quasi$)$barrelled, then it is $\infty$-$($quasi$)$barrelled.
\item[{\rm(ii)}] If $E$ is $\infty$-quasibarrelled, then it is $c_0$-quasibarrelled; the converse is true if $E'_\beta$ has the Schur property.
\end{enumerate}
\end{lemma}
In other words, the class of $\infty$-quasibarrelled spaces lies between the class of $c_0$-barrelled spaces and the class of $c_0$-quasibarrelled spaces, and Examples \ref{exa:c0-1-barrel} and \ref{exa:c0-quasi-1quasi} below show that this disposition is strict.
It is convenient to describe the  relationships between the weak barrelledness conditions introduced in Definitions \ref{def:weak-barrel} and \ref{def:p-barrelled} in a diagram which immediately follows from definitions and Lemmas \ref{l:prop-p-sum}(ii) and \ref{l:inf-c0-quasibarelled}.
If $1\leq p<q\leq\infty$, then we have
\begin{equation} \label{equ:p-barrel}
\xymatrix{
\mbox{barrelled}  \ar@{=>}[r]\ar@{=>}[d] & {\substack{\mbox{$\infty$-barrelled}\\ \mbox{(= $c_0$-barrelled)}}}  \ar@{=>}[r]\ar@{=>}[d] & \mbox{$q$-barrelled} \ar@{=>}[r]\ar@{=>}[d] &  \mbox{$p$-barrelled} \ar@{=>}[r]\ar@{=>}[d] & \mbox{$1$-barrelled}\ar@{=>}[d] \\
{\substack{\mbox{quasi-}\\ \mbox{barrelled}}}  \ar@{=>}[r] & {\substack{\mbox{$\infty$-quasi-}\\ \mbox{barrelled}}}  \ar@{=>}[r]\ar@{=>}[d] & {\substack{\mbox{$q$-quasi-}\\ \mbox{barrelled}}}  \ar@{=>}[r] &  {\substack{\mbox{$p$-quasi-}\\ \mbox{barrelled}}}  \ar@{=>}[r] & {\substack{\mbox{$1$-quasi-}\\ \mbox{barrelled}}}  \\
&{\substack{\mbox{$c_0$-quasi-}\\ \mbox{barrelled}}} &&&
}
\end{equation}

In the next assertions and examples we clarify Diagram \ref{equ:p-barrel}. 

\begin{proposition} \label{p:l-inf-not-c0-bar}
Let $X$ be a sequentially complete barrelled space with the Glicksberg property, and let $E=\big(X', \mu(X',X)\big)$. Then:
\begin{enumerate}
\item[{\rm(i)}] $X_w= (E')_{w^\ast}$ is sequentially complete;
\item[{\rm(ii)}] $E$ is a $c_0$-barrelled space;
\item[{\rm(iii)}] if moreover $X$ is von Neumann complete, then $E$ is $\ell_\infty$-quasibarrelled if and only if every bounded subset of $X$ is relatively compact.
\end{enumerate}
\end{proposition}

\begin{proof}
First we note that the barrelledness of $X$ implies that $E'_\beta$ is the space $X$ since the polar of a bounded subset of $E$ is a barrel, see Theorem 8.8.4 of \cite{NaB}.

(i) Let $S=\{x_n\}_{n\in\w}$ be a weakly Cauchy sequence in $X$. Since $X$ has the  Glicksberg property it is a Schur space and hence $S$ is a Cauchy sequence in $X$. As $X$ is sequentially complete, $S$ converges to a point $x\in X$. Thus $x_n\to x$ also in the weak topology.

(ii) Let $S=\{x_n\}_{n\in\w}$ be a weak$^\ast$ null sequence in $E'=X$. Then $S$ is a weakly null sequence in the space $X$. By the Schur property, $x_n\to 0$ in $X$. Since $X$ is sequentially (hence also locally) complete, the closed absolutely convex hull $K:=\cacx(S)$ of $S$ is a compact subset of $X$. Therefore, by the Mackey--Arens theorem, $K^\circ$ is a neighborhood of zero in $E$. Whence $K^{\circ\circ}=K$ and hence also $S$ are equicontinuous. Thus $E$ is $c_0$-barrelled.

(iii) 
Assume that $E$ is an $\ell_\infty$-quasibarrelled space. Suppose for a contradiction that there is a bounded subset $B$ of $X$ which is not relatively compact. Since $X$ is von Neumann complete, $B$ is not precompact.  Therefore there are a sequence $S=\{x_n\}_{n\in\w}$ in $B$ and $U\in\Nn_0(X)$ such that $x_n-x_m\not\in U$ for all distinct $n,m\in\w$. Then $S$ is also bounded and non-precompact in $E'_\beta=X$. As $E$ is  $\ell_\infty$-quasibarrelled, $S$ is equicontinuous and hence there is $V\in \Nn_0(E)$ such that $S\subseteq V^\circ$.  Since, by the Alaoglu theorem, $V^\circ$ is weakly compact in $X$, the Glicksberg property of $X$ implies that $V^\circ$ is a compact subset of $X$. Therefore the sequence $S$ is precompact in $X$. But this contradicts the choice of $S$.

Conversely, assume that every bounded subset of $X$ is relatively compact. To show that $E$ is an $\ell_\infty$-quasibarrelled space, we prove that every countable bounded subset $B$ of  $E'_\beta=X$ is equicontinuous. Observe that $B^{\circ\circ}$ is bounded and closed in $X$, and hence, by assumption,  $B^{\circ\circ}$ is an absolutely convex (weakly) compact subset of $X$. By the Mackey--Arens theorem, we obtain that  $B^{\circ}= B^{\circ\circ\circ}$ is a neighborhood of zero in $E$. So  $B^{\circ\circ}$ and hence also $B$ are equicontinuous. Thus $E$ is  $\ell_\infty$-quasibarrelled. \qed
\end{proof}

Recall that a locally convex space $E$ is called {\em semi-Montel} if every bounded subset of $E$ is relatively compact, and $E$ is a {\em Montel space} if it is a barrelled semi-Montel space. Clearly, every semi-Montel space has the Glicksberg property.

\begin{corollary} \label{c:l-inf-not-c0-bar}
Let $X$ be a strict $(LF)$-space. Then $X$ is a Montel space if and only if it has  the Schur property and the space $\big(X', \mu(X',X)\big)$ is  $\ell_\infty$-quasibarrelled.
\end{corollary}

\begin{proof}
Any strict  $(LF)$-space is complete and barrelled by Theorem 4.5.7 and Proposition 11.3.1 of \cite{Jar}. By Theorem 1.2 of \cite{Gabr-free-resp}, $X$ has the Glicksberg property if and only if it has the Schur property. Now Proposition \ref{p:l-inf-not-c0-bar} applies. \qed
\end{proof}
In particular, Corollary \ref{c:l-inf-not-c0-bar} gives a new characterization of Fr\'{e}chet--Montel spaces which are studied in Section 11.6 of \cite{Jar}.

The next example shows that there are even metrizable (hence quasibarrelled) locally convex spaces which are not  $p$-barrelled for every $p\in[1,\infty]$. On the other hand, this example complements (iv) of Proposition \ref{p:operator-Lp} by showing that the condition ``$1<p<\infty$'' is not necessary for the inclusion $S_p\big( \LL(E,\ell_p)\big) \subseteq\ell_p^w(E'_{\beta})$.

\begin{example} \label{exa:c0-1-barrel}
Let $p\in[1,\infty]$, and let $E=(c_0)_p$. Then:
\begin{enumerate}
\item[{\rm(i)}] $E$ is metrizable {\rm(}hence quasibarrelled{\rm)} but not $p$-barrelled;
\item[{\rm(ii)}]  each operator $T\in \LL(E,\ell_p)$ or $T\in \LL(E,c_0)$ is finite-dimensional;
\item[{\rm(iii)}] $S_p\big( \LL(E,\ell_p)\big) = \ell_p^w(E'_{\beta})$ $($or $S_\infty\big( \LL(E,c_0)\big) = c_0^w(E'_{\beta})$ if $p=\infty$$)$;
\item[{\rm(iv)}] $\ell_p^w(E'_{\beta}) \subsetneq \ell_p^w(E'_{w^\ast})$  $($or $c_0^w(E'_{\beta}) \subsetneq c_0^w(E'_{w^\ast})$ if $p=\infty$$)$.
\end{enumerate}
\end{example}

\begin{proof}
First we prove the following claim which is used in what follows.

{\em Claim 1. $E'_\beta= \varphi$, and hence $E''=\IF^\w$.} Indeed, since $E$ is dense in the product $\IF^\w$, then $E'=\varphi$ algebraically. To show that the equality holds also topologically, let $B=\prod_{n\in\w} a_n\mathbb{D}$ (where all $a_n>0$) be a bounded subset of $\IF^\w$, so that $B^\circ$ is a  standard neighborhood of zero in $\varphi$. Fix an arbitrary $\chi\in E'$, so $\chi$ has finite support. Since $E\cap B$ is dense in $B$ and is bounded in $E$, we obtain that $\chi\in B^\circ$ if and only if $\chi\in (E\cap B)^\circ$. Now it is clear that $E'_\beta= \varphi$. The claim is proved.
\smallskip

(i) Since every equicontinuous subset of $E'$ is strongly bounded by Theorem 11.3.5 of \cite{NaB}, Claim 1 implies that the equicontinuous subsets of $E'$ are finite-dimensional. Therefore to prove that $E$ is not $p$-barrelled, it suffices to show that the infinite-dimensional sequence $S=\big\{\tfrac{e_n^\ast}{(n+1)^2} \big\}_{n\in\w}$, where $\{ e_n^\ast\big\}_{n\in\w}$ is the  standard coordinate basis of $E'=\varphi$, is  weakly $1$-summable  in $E'_{w^\ast}$. To this end, fix an arbitrary $\xxx=(x_n)\in E$. Since $x_n\to 0$ it follows that the series
\[
\sum_{n\in\w} \Big|\Big\langle\tfrac{e_n^\ast}{(n+1)^2},\xxx\Big\rangle\Big| =\sum_{n\in\w} \tfrac{|x_n|}{(n+1)^2}
\]
converges. Thus $S$ is weakly $1$-summable in $E'_{w^\ast}$, as desired.
\smallskip

(ii) Let $T\in \LL(E,\ell_p)$ (or $T\in \LL(E,c_0)$ if $p=\infty$). Take a standard neighborhood $U=E\cap (a \mathbb{D})^n \times \IF^{\w\SM n}$ ($a>0$) of zero in $E$ such that $T(U)\subseteq B_{\ell_p}$ (or $T(U)\subseteq B_{c_0}$ if $p=\infty$). Then $E\cap \big(\{0\}^n\times\IF^{\w\SM n}\big)$ is contained in the kernel of $T$. Thus $T$ is finite-dimensional.
\smallskip

(iii) First we describe the space $\ell_p^w(E'_{\beta})$ (or $c_0^w(E'_{\beta})$ if $p=\infty$). By Claim 1, $E'_\beta=\varphi$ and $E''=\IF^\w$. Since any sequence in $\ell_p^w(E'_{\beta})$  (or in $c_0^w(E'_{\beta})$ if $p=\infty$) is strongly bounded and hence finite-dimensional, Lemma \ref{l:lp-finite-dim} implies that $(\chi_n)_{n\in\w}$ belongs to $\ell_p^w(E'_{\beta})$ if and only if there is $s\in\w$ such that
\begin{equation} \label{equ:c0p-exa-1}
\chi_n= a_{0,n} e^\ast_0+\cdots+ a_{s,n} e^\ast_s \quad (n\in\w)
\end{equation}
for some scalars $a_{i,j}\in\IF$ such that
\begin{equation} \label{equ:c0p-exa-2}
(a_{i,n})_{n\in\w} \in \ell_p \;\; (\mbox{or $\in c_0$ if $p=\infty$}) \;\; \mbox{ for all }\; 0\leq i\leq s.
\end{equation}

Let $T\in\LL(E,\ell_p)$ (or $T\in\LL(E,c_0)$ if $p=\infty$). Then, by (ii), $T$ is finite-dimensional. Therefore  there are some $\mu_1,\dots,\mu_m\in E'$ and $\nu_1,\dots,\nu_m\in \ell_p$ (or $\in c_0$ if $p=\infty$) such that $T(x)=\sum_{i=1}^m \langle\mu_i,x\rangle \nu_i$. Then for every $x\in E$ and $n\in\w$, we have
\[
\langle T^\ast(e_n^\ast),x\rangle=\sum_{i=1}^m \langle\mu_i,x\rangle \cdot \langle e_n^\ast,\nu_i\rangle=\Big\langle \sum_{i=1}^m \langle e_n^\ast,\nu_i\rangle \mu_i, x \Big\rangle
\]
and hence $T^\ast(e_n^\ast)=\sum_{i=1}^m \langle e_n^\ast,\nu_i\rangle \mu_i$. Thus $\{T^\ast(e_n^\ast)\}_{n\in\w}$ is contained in the finite-dimensional subspace $\spn\{\mu_1,\dots,\mu_m\}$ of $E'_\beta= \varphi$. If
\[
\mu_i = d_{0,i} e^\ast_0+\cdots+ d_{s,i} e^\ast_s \;\; \mbox{ for } \;\; 1\leq i\leq m,
\]
then for every $n\in\w$, we obtain
\[
T^\ast(e_n^\ast) = \sum_{i=1}^m \langle e_n^\ast,\nu_i\rangle \Big(  \sum_{j=0}^s d_{j,i} e_j^\ast\Big)=\sum_{j=0}^s\Big( \sum_{i=1}^m \langle e_n^\ast,\nu_i\rangle d_{j,i}\Big) \cdot e_j^\ast.
\]
Set $a_{j,n}:=\sum_{i=1}^m \langle e_n^\ast,\nu_i\rangle d_{j,i}$ and $\chi_n:= \sum_{j=0}^s a_{j,n} e_j^\ast$. Since $\nu_i\in\ell_p$ (or $\in c_0$ if $p=\infty$), it follows that $(a_{j,n})_n\in\ell_p$ (or $\in c_0$ if $p=\infty$) for every $0\leq j\leq s$. Therefore, by (\ref{equ:c0p-exa-1}) and  (\ref{equ:c0p-exa-2}), $(\chi_n)\in \ell_p^w(E'_\beta)$  (or $\in c_0^w(E'_{\beta})$ if $p=\infty$) and hence $S_p\big( \LL(E,\ell_p)\big) \subseteq \ell_p^w(E'_{\beta})$ (or $S_\infty\big( \LL(E,c_0)\big) \subseteq  c_0^w(E'_{\beta})$).

To prove the inverse inclusion $\ell_p^w(E'_{\beta}) \subseteq S_p\big( \LL(E,\ell_p)\big)$ (or $c_0^w(E'_{\beta}) \subseteq  S_\infty\big( \LL(E,c_0)\big)$), let $(\chi_n)\in \ell_p^w(E'_\beta)$. Then we can assume  that $(\chi_n)$ has a decomposition (\ref{equ:c0p-exa-1}) such that  (\ref{equ:c0p-exa-2}) is satisfied. Define $T:E\to \ell_p$ (or $T:E\to c_0$ if $p=\infty$) by
\[
T(x):= \big( \langle\chi_n, x\rangle\big)_{n\in\w} \quad (x\in E).
\]
To show that $T$ is continuous, for $0\leq i\leq s$, we set $\nu_i:=(a_{i,n})\in \ell_p$ (or $\in c_0$ if $p=\infty$). Then, for $x=(x_n)\in E$, we obtain
\[
T(x)=\big( \langle\chi_n, x\rangle\big)_n =\big(a_{0,n} x_0+\cdots+ a_{s,n} x_s \big)_n = \sum_{i=0}^s x_i \nu_i=\sum_{i=0}^s \langle e^\ast_i, x\rangle \nu_i.
\]
Therefore $T$ is finite-dimensional and hence continuous. Since  $\langle T^\ast(e^\ast_n), x\rangle=\langle e^\ast_n, T(x)\rangle=\langle\chi_n, x\rangle$ we obtain $S_p(T)=(\chi_n)$. As  $(\chi_n)$ was arbitrary,  this proves the inverse inclusion.
\smallskip

(iv) To prove this clause it suffices to show that the weakly $1$-summable sequence $(\chi_n)_n$ in $E'_{w^\ast}$ constructed in (i) does not belong to $\ell_p^w(E'_{\beta})$. To this end, let $\eta=(1,2,\dots)\in E''=\IF^\w$. Since the series $\sum_{n\in\w}  |\langle \eta,\chi_n\rangle|=\sum_{n\in\w} \tfrac{1 }{n+1}$ diverges, we obtain $(\chi_n)\not\in \ell_p^w(E'_{\beta})$  (or $\not\in c_0^w(E'_{\beta})$ if $p=\infty$).\qed
\end{proof}

The next examples shows that there exist $c_0$-quasibarrelled spaces which are not $1$-quasibarrelled clarifying the relationships in Diagram \ref{equ:p-barrel}.
\begin{example} \label{exa:c0-quasi-1quasi}
There is a  $c_0$-quasibarrelled space which is not $1$-quasibarrelled.
\end{example}

\begin{proof}
For every $n\in\w$, let $x_n=\tfrac{1}{n+1}$ and $x_\infty=0$. Let $\mathbf{s}=\{x_n\}_{n\in\w}\cup\{x_\infty\}$ be the convergent sequence endowed with the topology induced from $\IR$, and let $E=L(\mathbf{s})$ be the free lcs over $\mathbf{s}$. Then, by Theorem 1.2 of \cite{Gabr-free-lcs}, the space $E$ is $c_0$-quasibarrelled.  To show that $E$ is not $1$-quasibarrelled, we construct a  weakly $1$-summable sequence $\{f_n\}$ in the strong dual $E'_\beta$ of $E$ which is not equicontinuous. Note that $E'_\beta$ is the Banach space $C(\mathbf{s})$ by Proposition 3.4 of  \cite{Gabr-free-lcs}. For every $n\in\w$, define $f_n\in C(\mathbf{s})$ by
\[
f_n(x)=1 \; \mbox{ if } x=x_n \; \mbox{ and } \; f_n(x)=0 \; \mbox{ otherwise}.
\]
Now, let $\mu\in C(\mathbf{s})'$ and hence $\mu=\sum_{n\leq\w} a_n \delta_{x_n}$, where $\|\mu\|=\sum_{n\leq \w} |a_n|<\infty$. Then $\sum_{n\leq\w} |\langle\mu, f_n\rangle |=\sum_{n\in\w} |a_n|<\infty$. Thus $\{f_n\}$ is weakly $1$-summable.

Let us show that $\{f_n\}$ is not equicontinuous which means that $E$ is not $1$-quasibarrelled. Indeed, suppose for a contradiction that $\{f_n\}$ is equicontinuous. Then, by Proposition 2.3 of \cite{Gabr-free-lcs}, the sequence $\{f_n\}$ is equicontinuous  as a subset of the function space $C(\mathbf{s})$. Since $f_n(x_\infty)=0$ for every $n\in\w$ and $x_n\to x_\infty$, it is easy to see that  $\{f_n\}$ is not equicontinuous at $0$, a contradiction.\qed
%
\end{proof}

The aforementioned results motivate the problem of finding  classes of locally convex spaces $E$ such that $E$ is $c_0$-$($quasi$)$barrelled if and only if it is a $p$-$($quasi$)$barrelled space for some (every) $p\in[1,\infty]$. Below we show that the class of Mackey spaces satisfies this condition for $p$-barrelledness. Let us recall that Theorem 12.1.4 of \cite{Jar} states that an lcs $E$  is $c_0$-(quasi)barrelled  if and only if  the space $\big(E',\sigma(E',E)\big)$ $\big($resp., $\big(E',\beta(E',E)\big)$$\big)$ is  locally complete. Analogously we have the following assertion.


\begin{theorem} \label{t:loc-complete-p-quasi}
Let $p\in[1,\infty]$, and let $E$ be a locally convex space.
If $E$ is a $p$-barrelled space, then $\big(E',\sigma(E',E)\big)$ is locally complete; the converse is true if $E$ is a Mackey space.
\end{theorem}

\begin{proof}
By Theorem  10.2.4 of \cite{Jar}, 
to show that  the space $E'_{w^\ast}:=\big(E',\sigma(E',E)\big)$ is  locally complete it suffices to prove that for every weak$^\ast$ $p$-summable sequence  $S=\{\chi_n\}_{n\in\w}$ in $E'_\beta$, the set $\cacx(S)$ is $\sigma(E',E)$-compact. Since $E$ is a $p$-barrelled space, the sequence $S$ and hence also $\cacx(S)$ are equicontinuous. It follows that $\cacx(S)^\circ$ is a neighborhood of zero in $E$. By the Alaoglu theorem, $\cacx(S)^{\circ\circ}=\cacx(S)$  is weak$^\ast$ compact. Thus $\big(E',\sigma(E',E)\big)$  is locally complete.

Let now $E$ be a Mackey space, and assume that $\big(E',\sigma(E',E)\big)$  is locally complete. To show that $E$ is a $p$-barrelled space, let $S=\{\chi_n\}_{n\in\w}$ be a  weak$^\ast$ $p$-summable sequence in $E'$. Since $E'_{w^\ast}$ is locally complete, Theorem   10.2.4 of \cite{Jar} 
implies that $K:=\cacx(S)$ is compact in $E'_{w^\ast}$. Since $E$ is Mackey, the Mackey--Arens theorem implies that $K^\circ$ is a neighborhood of zero in $E$. Therefore $K^{\circ\circ}=K$ and hence also $S$ are equicontinuous. Thus $E$ is a $p$-barrelled space.\qed
\end{proof}

\begin{corollary} \label{c:loc-complete-p-quasi}
A Mackey space $E$ is $c_0$-barrelled  if and only if it is $p$-barrelled for some $($every$)$  $p\in[1,\infty]$.
\end{corollary}

\begin{proof}
If $E$ is $c_0$-barrelled, then clearly $E$ is $p$-barrelled for all  $p\in[1,\infty]$. Conversely, if $E$ is $p$-barrelled for some  $p\in[1,\infty]$, then, by Theorem \ref{t:loc-complete-p-quasi}, $\big(E',\sigma(E',E)\big)$  is  locally complete. Thus $E$ is $c_0$-barrelled by  Theorem 12.1.4 of \cite{Jar}.\qed
\end{proof}

It is well known that the class of (quasi)barrelled spaces is closed under taking products, direct sums, inductive limits, Hausdorff quotients and a completion of a barrelled space is barrelled, see Proposition 11.3.1 of \cite{Jar}. Below we consider analogous properties for $p$-(quasi)barrelled spaces.
\begin{proposition} \label{p:p-bar-product}
Let  $p\in[1,\infty]$. 
\begin{enumerate}
\item[{\rm(i)}] Every Hausdorff quotient space of a $p$-$($quasi$)$barrelled space is $p$-$($quasi$)$barrelled. Consequently, complemented subspaces of
$p$-$($quasi$)$barrelled spaces are $p$-$($quasi$)$barrelled.
\item[{\rm(ii)}] The Tychonoff product of a family of locally convex spaces is $p$-$($quasi$)$barrelled if and only if each of its factor is a $p$-$($quasi$)$barrelled space.
\item[{\rm(iii)}] Countable locally convex direct sums of $p$-$($quasi$)$barrelled spaces are $p$-$($quasi$)$barrelled spaces.
\item[{\rm(iv)}]  Let $H$ be a dense subspace of an lcs $E$. If $H$ is  $p$-barrelled  then so  is $E$.
\item[{\rm(v)}] Let $H$ be a dense, large subspace of an lcs $E$. If $H$ is $p$-quasibarrelled then so  is $E$.
\end{enumerate}
\end{proposition}

\begin{proof}
(i) Let $H$ be a closed subspace of a $p$-(quasi)barrelled space $E$, and let $q:E\to E/H$ be the quotient map. To show that $E/H$ is $p$-(quasi)barrelled, let $S=\{\chi_i\}_{i\in\w}$ be a weak$^\ast$ (resp., weakly) $p$-summable sequence  $(E/H)'_\beta$. Since the adjoint map $q^\ast$ is weak$^\ast$ (resp., strongly) continuous, (iii) of Lemma \ref{l:prop-p-sum} implies that  $\{q^\ast(\chi_i)\}_{i\in\w}$ is a weak$^\ast$ (resp., weakly) $p$-summable sequence in $E'_\beta$. As $E$ is $p$-(quasi)barrelled, $\{q^\ast(\chi_i)\}_{i\in\w}$ is equicontinuous and hence there is $U\in\Nn_0(E)$ such that $\{q^\ast(\chi_i)\}_{i\in\w}\subseteq U^\circ$. Since $q$ is a quotient map, the image $q(U)$ of $U$ is a neighborhood of zero in $E/H$. Then for every $q(x)\in q(U)$ with $x\in U$ and each $i\in\w$, we have
\[
|\langle \chi_i, q(x)\rangle|=|\langle q^\ast(\chi_i), x\rangle|\leq 1
\]
which means that $S\subseteq q(U)^\circ$ and hence the sequence $S$ is equicontinuous. Thus $E/H$ is a $p$-(quasi)barrelled space.
\smallskip

(ii) The necessity follows from (i).
To prove the sufficiency, assume that $E=\prod_{i\in I} E_i$ is the product of a non-empty family $\{E_i\}_{i\in I}$ of $p$-(quasi)barrelled spaces, and let $S=\{\chi_n\}_{n\in\w}$ be a weak$^\ast$ (resp., weakly) $p$-summable sequence in $E'_\beta$. Then, by (i) of Lemma \ref{l:support-p-sum}, $S$ has finite support $F\subseteq I$ and for every $i\in F$, the sequence $S_i:=\{\chi_n(i)\}_{n\in\w}$ is weak$^\ast$ (resp., weakly) $p$-summable  in $(E_i)'_\beta$. Since $E_i$ is $p$-(quasi)barrelled, $S_i$ is equicontinuous and hence $S_i\subseteq U_i^\circ$ for some $U_i\in\Nn_0(E_i)$. Set $U:= \prod_{i\in F} \tfrac{1}{|F|} U_i \times \prod_{i\in I\SM F} E_i$. Then $U$ is a neighborhood of zero in $E$ such that for every $x=(x_i)\in U$ and each $n\in\w$, we have
\[
|\langle\chi_n, x\rangle|=\Big| \sum_{i\in F} \langle\chi_n(i),x_i\rangle\Big| \leq \sum_{i\in F} |\langle\chi_n(i),x_i\rangle|\leq 1.
\]
Therefore $S\subseteq U^\circ$ is equicontinuous. Thus $E$ is a $p$-(quasi)barrelled space.
\smallskip

(iii) Let $E=\bigoplus_{n\in \w} E_n$ be the direct sum of a sequence $\{E_n\}_{n\in\w}$ of $p$-(quasi)barrelled spaces. To show that $E$ is $p$-(quasi)barrelled, fix a weak$^\ast$ (resp., weakly) $p$-summable sequence $S=\{\chi_i\}_{i\in\w}$ in $E'_\beta$. Then, by  (i) of Lemma \ref{l:support-p-sum}, the sequence $\{\chi_i(n)\}_{i\in\w}$ is weak$^\ast$ (resp., weakly) $p$-summable  in $(E_n)'_\beta$ for every $n\in\w$. For every $n\in\w$, the $p$-(quasi)barrelledness of $E_n$ implies that $\{\chi_i(n)\}_{n\in\w}$ is equicontinuous and therefore there is $U_n\in \Nn_0(E_n)$ such that $\{\chi_i(n)\}_{n\in\w} \subseteq \tfrac{1}{2^{n+1}} U_n^\circ$. Set $U:= E\cap \prod_{n\in\w} U_n$. Then $U$ is a neighborhood of zero in $E$. For every $x=(x_n)\in U$ and each $i\in\w$, we have
\[
|\langle\chi_i, x\rangle|=\Big| \sum_{n\in \w} \langle\chi_i(n),x_n\rangle\Big| \leq \sum_{n\in \w} |\langle\chi_i(n),x_n\rangle|\leq  \sum_{n\in \w} \tfrac{1}{2^{n+1}} =1.
\]
Therefore $S\subseteq U^\circ$. Thus $S$ is equicontinuous and hence $E$ is $p$-(quasi)barrelled.
\smallskip

(iv), (v): Recall that $E'=H'$ and, in the case (v), Proposition \ref{p:large-charac} implies that $E'_\beta=H'_\beta$. Now, fix a weak$^\ast$ (resp., weakly) $p$-summable sequence $S=\{\chi_i\}_{i\in\w}$ in $E'_\beta$. Then $S$ is weak$^\ast$ (resp., weakly) $p$-summable also in $H'_\beta$. Since $E$ is $p$-barrelled (resp., $p$-quasibarrelled), $S$ is equicontinuous and hence there is $U\in \Nn_0(H)$ such that $S\subseteq U^\circ$. Since the closure $\overline{U}$ of $U$ in $E$ is a neighborhood of zero in $E$ and $S\subseteq \overline{U}^{\,\circ}$, we obtain that $S$ is equicontinuous in $E'_\beta$. Thus $E$ is $p$-(quasi)barrelled.\qed
\end{proof}

Since the strict inductive limit of a sequence $\{E_n\}_{n\in\w}$ of locally convex spaces is a quotient space of $\bigoplus_{n\in\w} E_n$, Proposition \ref{p:p-bar-product} implies

\begin{corollary} \label{p:p-bar-ind}
Strict inductive limits of sequences of  $p$-$($quasi$)$barrelled spaces are  $p$-$($quasi$)$barrelled.
\end{corollary}

\begin{remark} \label{rem:subspace-p-bar} {\em
(i) A dense subspace of a barreled space can be not $p$-barrelled. Indeed, let $X$ be a Tychonoff space containing infinite functionally bounded subsets, and let $H=C_p(X)$ and $E=B_1(X)$. Then $E$ is barrelled by Theorem 1.1 of \cite{BG-Baire-lcs}. By the Buchwalter--Schmets theorem, the dense subspace $H$ of $E$ is not barrelled. Note also that, by Proposition 2.6 of \cite{GK-DP}, $C_p(X)$ is barrelled if and only if it is $c_0$-barrelled. Since $C_p(X)$ is always quasibarrelled (hence Mackey), Corollary \ref{c:loc-complete-p-quasi} implies that $H$ is not  $p$-barrelled for all $p\in[1,\infty]$.

(ii) Closed subspaces of $p$-barrelled spaces can be not $p$-barrelled. Indeed, let $H=C_p(X)$ from (i). By Theorem 4.7 of \cite{BG-sGP}, $H$ is topologically isomorphic to a closed subspace of even a barrelled space $E$.\qed
}
\end{remark}


\section{$p$-Schur property for locally convex spaces } \label{sec:p-Schur}


Analogously to the $p$-Schur property for Banach spaces we define.

\begin{definition} \label{def:p-Schur} {\em
Let $p\in[1,\infty]$. A separated tvs $E$ is said to have a {\em $p$-Schur property} or is a {\em $p$-Schur space} if every weakly $p$-summable sequence is a null-sequence.\qed}
\end{definition}
In particular, $E$ has the Schur property if and only if it is an $\infty$-Schur space.

For further references we select the next three simple lemmas. The first one is an immediate corollary of definitions and the Hahn--Banach extension theorem.

\begin{lemma} \label{l:p-Schur-sub}
Let $p\in[1,\infty]$, and let $E$ be a locally convex space with the $p$-Schur property.
\begin{enumerate}
\item[{\rm(i)}] If $1\leq q\leq p$, then $E$ has the $q$-Schur property.
\item[{\rm(ii)}] If $H$ is a linear subspace of $E$, then $H$ is also a $p$-Schur space.
\end{enumerate}
\end{lemma}

\begin{lemma} \label{l:p-Schur-prod}
Let  $p\in[1,\infty]$, and let $\{E_i\}_{i\in I}$  be a non-empty family of locally convex spaces. Then the following assertions are equivalent:
\begin{enumerate}
\item[{\rm(i)}] $\prod_{i\in I} E_i$ is a  $p$-Schur space;
\item[{\rm(ii)}] $\bigoplus_{i\in I} E_i$  is a  $p$-Schur space;
\item[{\rm(iii)}] $E_i$  is a  $p$-Schur space for each $i\in I$.
\end{enumerate}
\end{lemma}

\begin{proof}
(i)$\Ra$(ii) Let $S=\{x_n=(x_{n,i})_{i\in I}\}_{n\in\w}$ be a weakly $p$-summable sequence in the direct sum $E:=\bigoplus_{i\in I} E_i$. Then, by (iv) of Lemma \ref{l:support-p-sum}, the sequence $\{x_{n,i}\}_{n\in\w}$ is weakly  $p$-summable in $E_i$ for every $i\in I$ and the support $F$ of the sequence $S$ is finite. Therefore, by (ii) of Lemma \ref{l:p-Schur-sub} and (i), $x_n\to 0$ in $\prod_{i\in F} E_i$, and hence $x_n\to 0$ in $E$.

The implication (ii)$\Ra$(iii) follows from (ii) of Lemma \ref{l:p-Schur-sub}.

(iii)$\Ra$(i) Let $\{x_n=(x_{n,i})_{i\in I}\}_{n\in\w}$ be a weakly $p$-summable sequence in the product $E:=\prod_{i\in I} E_i$. Then, by (ii) of Lemma \ref{l:support-p-sum}, the sequence $\{x_{n,i}\}_{n\in\w}$ is weakly  $p$-summable in $E_i$ for every $i\in I$. By the $p$-Schur property of $E_i$, we obtain that $x_{n,i}\to 0$ in $E_i$. Thus $x_n\to 0$ in $E$.\qed
\end{proof}

\begin{lemma} \label{l:p-Schur-ind}
Let  $p\in[1,\infty]$, and let $E=\SI E_n$ be the strict inductive limit of an increasing sequence $\{E_n\}_{n\in \w}$ of locally convex spaces. Then $E$ has the $p$-Schur property if and only if all spaces $E_n$ are  $p$-Schur spaces.
\end{lemma}

\begin{proof}
The necessity follows from the fact that all $E_n$ are subspaces of $E$ and (ii) of Lemma \ref{l:p-Schur-sub}. The sufficiency follows from the fact that $E$ is regular (= any bounded subset of $E$ sits in some $E_n$) and (vi) of Lemma \ref{l:prop-p-sum}.\qed
\end{proof}

The next proposition shows that the property of being a $p$-Schur space depends on $p$ and will be used for constructing other counterexamples.
Recall (see Example 2.6 of \cite{DM}) that the Banach space $\ell_p$ does not have the $p^\ast$-Schur property, where $\tfrac{1}{p}+\tfrac{1}{p^\ast}=1$. In (ii) of the next proposition we generalize this result.

\begin{proposition} \label{p:Lp-Schur}
Let $1< p<\infty$, and let $q\in[1,\infty]$. Then:
\begin{enumerate}
\item[{\rm(i)}] if $q<p^\ast$, then $\ell_p$ has the $q$-Schur property;
\item[{\rm(ii)}] if $q\geq p^\ast$, then $\ell_p^0$ does not have the $q$-Schur property;
\end{enumerate}
\end{proposition}

\begin{proof}
(i) To show that $\ell_p$ has the $q$-Schur property, let $\{x_n\}_{n\in\w}$ be a weakly $q$-summable sequence in $\ell_p$, and suppose for a contradiction that $\|x_n\|_{\ell_p}\not\to 0$. Then, passing to a subsequence if needed, we can assume that there is $C>1$ such that
\begin{equation} \label{equ:Lp-p-Schur-1}
\tfrac{1}{C} \leq \|x_n\|_{\ell_p}\leq C \quad \mbox{ for every } \; n\in\w.
\end{equation}

We claim that $\{x_n\}_{n\in\w}$ does not contain a subsequence which is Cauchy in the norm topology. Indeed, otherwise, we could find a subsequence $\{x_{n_k}\}_{k\in\w}$ of $\{x_n\}_{n\in\w}$ which converges in the norm to some $x\in\ell_p$. Therefore $x_{n_k}\to x$ in the weak topology, and hence $x=0$ because $\{x_n\}_{n\in\w}$ is weakly null. But then $\|x_{n_k}\|_{\ell_p}\to 0$ which contradicts (\ref{equ:Lp-p-Schur-1}).

Since $\{x_n\}_{n\in\w}$ is weakly null, (\ref{equ:Lp-p-Schur-1}) and Proposition 2.1.3 of \cite{Al-Kal} (see also Theorem 1.3.2 of \cite{Al-Kal}) imply that there are a basic subsequence $\{x_{n_k}\}_{k\in\w}$ of $\{x_n\}_{n\in\w}$ and a linear topological isomorphism $R: L:=\cspn\big(\{x_{n_k}\}_{k\in\w}\big)\to \ell_p$ such that
\[
R(x_{n_k})=a_k e_k \quad \mbox{ for every } \; k\in\w, \mbox{ where $a_k=\|x_{n_k}\|_{\ell_p}$},
\]
and the subspace $L$ is complemented in $\ell_p$. It follows from (\ref{equ:Lp-p-Schur-1}) that the map $Q:\ell_p\to\ell_p$ defined by $Q(\alpha_k)=(a_k\alpha_k)$ is a topological linear isomorphism. Therefore, replacing $R$ by $Q^{-1}\circ R$, without loss of generality we can assume that $a_k=1$ for every $k\in\w$, that is
\begin{equation} \label{equ:Lp-p-Schur-2}
R(x_{n_k})=e_k \quad \mbox{ for every } \; k\in\w.
\end{equation}

Since $q<p^\ast$, there exists an element $\chi=(y_n)\in \ell_{p^\ast}\SM \ell_q$. As $L$ is a subspace of $\ell_p$ we can consider $R^\ast(\chi) \in L'$ as an element of $(\ell_p)'$. Then, by (\ref{equ:Lp-p-Schur-2}), we have
\[
\sum_{k\in\w} \big| \langle R^\ast(\chi), x_{n_k}\rangle\big|^q  =\sum_{k\in\w} \big| \langle \chi, e_k\rangle\big|^q =\sum_{k\in\w} |y_n|^q =\infty
\]
which means that the sequence $\{x_{n_k}\}_{k\in\w}$ and hence also $\{x_n\}_{n\in\w}$ are not weakly $q$-summable, a contradiction. Thus $\ell_p$ has the $q$-Schur property.
\smallskip

(iii) To show that $\ell_p^0$ does not have the $q$-Schur property we consider the standard unit basis $\{e_n\}_{n\in\w}$ of $\ell_p^0$. Then, by Example \ref{exa:lp-in-lr}, $\{e_n\}_{n\in\w}$ is weakly $q$-summable. But since $\|e_n\|_{\ell_p}=1$ for all $n\in\w$, the space $\ell_p^0$ does not have the $q$-Schur property.\qed
\end{proof}

Lemma \ref{l:p-Schur-sub}(ii) and Proposition \ref{p:Lp-Schur} immediately imply

\begin{corollary} \label{c:Lp-Schur}
Let $1< p<\infty$, and let $q\geq p^\ast$. If a locally convex space $E$ contains an isomorphic copy of $\ell_p^0$, then $E$ does not have the $q$-Schur property.
\end{corollary}


\section{$(V^\ast)$ type subsets of locally convex spaces } \label{sec:V*-sets}


We start from the following generalization of the notion of a $p$-$(V^\ast)$ set in a Banach space.
\begin{definition}\label{def:tvs-V*-subset}{\em
Let $p,q\in[1,\infty]$. A non-empty subset   $A$ of a separated topological vector space $E$ is called a  {\em $(p,q)$-$(V^\ast)$ set} (resp., a {\em $(p,q)$-$(EV^\ast)$ set}) if
\[
\Big(\sup_{a\in A} |\langle \chi_n, a\rangle|\Big)\in \ell_q \; \mbox{ if $q<\infty$, } \; \mbox{ or }\;\; \Big(\sup_{a\in A} |\langle \chi_n, a\rangle|\Big)\in c_0 \; \mbox{ if $q=\infty$},
\]
for every  (resp., equicontinuous) weakly $p$-summable sequence $\{\chi_n\}_{n\in\w}$ in  $E'_\beta$. $(p,\infty)$-$(V^\ast)$ sets and $(1,\infty)$-$(V^\ast)$ sets will be called simply {\em $p$-$(V^\ast)$ sets} and {\em $(V^\ast)$ sets}, respectively. Analogously, $(p,\infty)$-$(EV^\ast)$ sets and $(1,\infty)$-$(EV^\ast)$ sets will be called {\em $p$-$(EV^\ast)$ sets} and {\em $(EV^\ast)$ sets}, respectively. \qed}
\end{definition}


In the next lemma we summarize some of basic properties of $(p,q)$-$(V^\ast)$ sets.

\begin{lemma} \label{l:V*-set-1}
Let $p,q\in[1,\infty]$, and let $(E,\tau)$ be a locally convex space. Then:
\begin{enumerate}
\item[{\rm(i)}] every $(p,q)$-$(V^\ast)$ set in $E$ is a $(p,q)$-$(EV^\ast)$ set, the converse is true if $E$ is $p$-quasibarrelled;
\item[{\rm(ii)}] every $(p,q)$-$(EV^\ast)$ set in $E$ is bounded;
\item[{\rm(iii)}] the family of all $(p,q)$-$(V^\ast)$ $($resp., $(p,q)$-$(EV^\ast)$$)$ sets in $E$ is closed under taking subsets, finite unions, finite sums, and closed absolutely convex hulls;
\item[{\rm(iv)}] if $T$ is an operator from $E$ into an lcs $L$ and $A$ is a subset of $E$, then $T(A)$ is a  $(p,q)$-$(V^\ast)$  $($resp., $(p,q)$-$(EV^\ast)$$)$ set if so is $A$;
\item[{\rm(v)}] a  subset $A$ of $E$ is a $(p,q)$-$(V^\ast)$  $($resp., $(p,q)$-$(EV^\ast)$$)$  set  if and only if every countable subset of $A$ is a  $(p,q)$-$(V^\ast)$  $($resp., $(p,q)$-$(EV^\ast)$$)$  set;
\item[{\rm(vi)}] if $p',q'\in[1,\infty]$ are such that $p'\leq p$ and $q\leq q'$, then every $(p,q)$-$(V^\ast)$ $($resp., $(p,q)$-$(EV^\ast)$$)$ set $A$ in $E$ is also a $(p',q')$-$(V^\ast)$ $($resp., $(p',q')$-$(EV^\ast)$$)$ set; in particular, any $(p,q)$-$(V^\ast)$  $($resp., $(p,q)$-$(EV^\ast)$$)$ is a $(V^\ast)$ $($resp., $(EV^\ast)$$)$ set;
\item[{\rm(vii)}] the property of being a $(p,q)$-$(V^\ast)$ set depends only on the duality $(E,E')$, i.e.,   if $\TTT$ is a locally convex vector topology on $E$ compatible with the topology $\tau$ of $E$, then $(E,\TTT)$ and $(E,\tau)$ have the same $(p,q)$-$(V^\ast)$ sets;
\item[{\rm(viii)}] if $H$ is a subspace of $E$, then $(p,q)$-$(V^\ast)$ sets of $H$ are also $(p,q)$-$(V^\ast)$  $($resp., $(p,q)$-$(EV^\ast)$$)$  sets in $E$.
\end{enumerate}
\end{lemma}

\begin{proof}
(i) and (iii) are clear, and (viii) follows from (iv) applied to the identity operator $\Id_H:H\to E$.

(ii) Let $A$ be a $(p,q)$-$(EV^\ast)$ set and suppose for a contradiction that it is unbounded. Then there is $\chi\in E'$ such that for every $n\in\w$, there exists $a_n\in A$ for which $|\langle \chi, a_n\rangle|> (n+1)^{3}$. For every $n\in\w$, set $\chi_n:=\tfrac{1}{(n+1)^{2}}\cdot \chi$. Then $\{\chi_n\}_{n\in\w}$ is an equicontinuous weakly $p$-summable sequence in  $E'_\beta$. However, since
\[
\sup_{a\in A} |\langle \chi_n, a\rangle| \geq |\langle \chi_n, a_n\rangle|> n+1 \to \infty
\]
it follows that $A$ is not a $(p,q)$-$(EV^\ast)$ set, a contradiction. Thus $A$ is bounded.

(iv) Observe that the adjoint operator $T^\ast: L'_\beta\to E'_\beta$ is continuous. 
Fix a (resp., equicontinuous) weakly $p$-summable sequence $\{\chi_n\}_{n\in\w}$ in  $L'_\beta$. Then, by Lemma \ref{l:prop-p-sum}(iii), the sequence $\{T^\ast(\chi_n)\}_{n\in\w}$ is weakly $p$-summable in  $E'_\beta$, and if $\{\chi_n\}_{n\in\w}$ is equicontinuous then so is $\{T^\ast(\chi_n)\}_{n\in\w}$. Therefore
\[
\Big(\sup_{a\in A} |\langle \chi_n, T(a)\rangle|\Big)= \Big(\sup_{a\in A} |\langle T^\ast (\chi_n), a\rangle|\Big)\in \ell_q \; \mbox{ (or } \; \in c_0 \; \mbox{ if $q=\infty$}),
\]
which means that $T(A)$ is a $(p,q)$-$(V^\ast)$ (resp., $(p,q)$-$(EV^\ast)$) set in $L$.

(v) The necessity follows from (iii). To prove the sufficiency suppose for a contradiction that $A$ is not a $(p,q)$-$(V^\ast)$ (resp., $(p,q)$-$(EV^\ast)$) set in $E$. Then there is a (resp., equicontinuous) weakly $p$-summable sequence $\{\chi_n\}_{n\in\w}$ in  $E'_\beta$ such that
\[
\Big(\sup_{a\in A} |\langle \chi_n, a\rangle|\Big)\not\in \ell_q \; \mbox{ if $q<\infty$, } \; \mbox{ or }\;\; \Big(\sup_{a\in A} |\langle \chi_n, a\rangle|\Big)\not\in c_0 \; \mbox{ if $q=\infty$}.
\]
Assume that $q<\infty$ (the case $q=\infty$ can be considered analogously). For every $n\in\w$, choose $a_n\in A$ such that $|\langle \chi_n, a_n\rangle|\geq \tfrac{1}{2} \cdot \sup_{a\in A} |\langle \chi_n, a\rangle|$. Then
\[
\sum_{n\in\w} |\langle \chi_n, a_n\rangle|^q \geq \tfrac{1}{2^q} \sum_{n\in\w} \big(\sup_{a\in A} |\langle \chi_n, a\rangle|\big)^q =\infty.
\]
Thus the countable subset $\{a_n\}_{n\in\w}$ of $A$ is not a $(p,q)$-$(V^\ast)$ (resp., $(p,q)$-$(EV^\ast)$) set in $E$, a contradiction.

(vi) Take any  (resp., equicontinuous)  weakly $p'$-summable sequence $\{\chi_n\}_{n\in\w}$ in  $E'_\beta$. Since $p'\leq p$,  $\{\chi_n\}$  is also (resp., equicontinuous)  weakly $p$-summable and hence $\big( \sup_{a\in A} |\langle \chi_n, a\rangle|\big)\in \ell_q$ (or $\in c_0$ if $q=\infty$). It remains to note that $\ell_q\subseteq \ell_{q'}$ because $q\leq q'$.

(vii) Since $\tau$ and $\TTT$ are compatible, $\Bo(E,\tau)=\Bo(E,\TTT)$ and hence $(E,\tau)'_\beta=(E,\TTT)'_\beta$. Now the assertion follows from 
the definition of $(p,q)$-$(V^\ast)$ sets.\qed
\end{proof}

\begin{notation} \label{n:V*-set} {\em
The family of all $(p,q)$-$(V^\ast)$ sets (resp. $p$-$(V^\ast)$ sets, $(p,q)$-$(EV^\ast)$ sets, $(V^\ast)$ sets etc.) of an lcs $E$ is denoted by $\mathsf{V}^\ast_{(p,q)}(E)$ (resp. $\mathsf{V}^\ast_{p}(E)$, $\mathsf{EV}^\ast_{(p,q)}(E)$, $\mathsf{V}^\ast(E)$ etc.).\qed}
\end{notation}

Below we characterize $p$-$(V^\ast)$ sets in Tychonoff products and locally convex direct sums.
\begin{proposition} \label{p:product-sum-V*-set}
Let  $p\in[1,\infty]$, and let $\{E_i\}_{i\in I}$  be a non-empty family of locally convex spaces. Then:
\begin{enumerate}
\item[{\rm(i)}] a subset $K$ of $E=\prod_{i\in I} E_i$ is a  $(p,q)$-$(V^\ast)$ set $($resp., a $(p,q)$-$(EV^\ast)$ set$)$ if and only if so are all its coordinate projections;
\item[{\rm(ii)}]  a subset $K$ of  $E=\bigoplus_{i\in I} E_i$  is a  $(p,q)$-$(V^\ast)$ set $($resp., a $(p,q)$-$(EV^\ast)$ set$)$ if and only if so are all its coordinate projections   and the support of $K$ is finite.
\end{enumerate}
\end{proposition}

\begin{proof}
The necessity follows from (iv) of Lemma \ref{l:V*-set-1} since $E_i$ is a direct summand of $E$ and, for the case (ii), from the well known fact that any bounded subset of a  locally convex  direct sum has finite support.
\smallskip


To prove the sufficiency, let $K$ be a subset of $E$ such that each projection $K_i$ of $K$ is a $(p,q)$-$(V^\ast)$ set $($resp., a $(p,q)$-$(EV^\ast)$ set$)$ in $E_i$, and, for the case (ii), $K_i=\{0\}$ for all but finitely many indices $i\in I$. We distinguish between the cases (i) and (ii).
\smallskip

(i) Take an arbitrary (resp., equicontinuous)  weakly $p$-summable sequence $\{\chi_n\}_{n\in\w}$ in $E'_\beta$, where $\chi_n=(\chi_{i,n})_{i\in I}$. Since $\{\chi_n\}$ is a bounded subset of the direct sum $E'_\beta$ (see Proposition \ref{p:product-sum-strong}), it has finite support, i.e., for some finite set $F\subseteq I$ it follows that $\chi_{i,n}=0$ for every $n\in\w$ and each $i\in I\SM F$. Observe also that if $\{\chi_n\}$ is equicontinuous, then for every $i\in F$, the sequence $\{\chi_{i,n}\}_{n\in\w}$ is also  equicontinuous as a projection of  $\{\chi_n\}$ onto the $i$th coordinate. Then
\begin{equation} \label{equ:prod-V*-1}
\sup_{x\in K} |\langle \chi_n, x\rangle|=\sup_{x\in K} \big|\sum_{i\in F} \langle \chi_{i,n}, x(i)\rangle\big|\leq \sum_{i\in F} \sup_{x(i)\in K_i} |\langle \chi_{i,n}, x(i)\rangle|.
\end{equation}
Since all $K_i$ are $(p,q)$-$(V^\ast)$ sets (resp., $(p,q)$-$(EV^\ast)$ sets) and $\{\chi_{i,n}\}_{n\in\w}$ is weakly $p$-summable by (i) of Lemma \ref{l:support-p-sum}, we have $\big( \sup_{x(i)\in K_i} |\langle \chi_{i,n}, x(i)\rangle|\big)  \in \ell_q$ (or $\in c_0$ if $q=\infty$). Therefore, by (\ref{equ:prod-V*-1}), also  $\big( \sup_{x\in K} |\langle \chi_n, x\rangle|\big) \in \ell_q$ (or $\in c_0$ if $q=\infty$). Thus $K$ is  a $(p,q)$-$(V^\ast)$ set (resp., a $(p,q)$-$(EV^\ast)$ set) in $E'$.
\smallskip

(ii) Let $F\subseteq I$ be the finite support of $K$. Take an arbitrary  (resp., equicontinuous)  weakly $p$-summable sequence $\{\chi_n\}_{n\in\w}$ in $E'_\beta$, where $\chi_n=(\chi_{i,n})_{i\in I}$ with $\chi_{i,n}\in E'_i$. If in addition $\{\chi_n\}$ is equicontinuous, then as above, for every $i\in F$, the sequence $\{\chi_{i,n}\}_{n\in\w}$ is also  equicontinuous in $E'_i$.  Then
\begin{equation} \label{equ:prod-V*-2}
\sup_{x\in K} |\langle \chi_n, x\rangle|=\sup_{x\in K} \big|\sum_{i\in F} \langle \chi_{i,n}, x(i)\rangle\big|\leq \sum_{i\in F} \sup_{x(i)\in K_i} |\langle \chi_{i,n}, x(i)\rangle|.
\end{equation}
Since all $K_i$ are $(p,q)$-$(V^\ast)$ sets (resp., $(p,q)$-$(EV^\ast)$ sets) and $\{\chi_{i,n}\}_{n\in\w}$ is (resp., equicontinuous) weakly $p$-summable by (iii) of Lemma \ref{l:support-p-sum}, we have $\big( \sup_{x(i)\in K_i} |\langle \chi_{i,n}, x(i)\rangle|\big)  \in \ell_q$ (or $\in c_0$ if $q=\infty$). Therefore, by (\ref{equ:prod-V*-2}),   $\big( \sup_{x\in K} |\langle \chi_n, x\rangle|\big) \in \ell_q$ (or $\in c_0$ if $q=\infty$) as well. Thus $K$ is  a $(p,q)$-$(V^\ast)$ set  (resp., a $(p,q)$-$(EV^\ast)$ set)  in $E'$.\qed
\end{proof}

Although the class $\mathsf{EV}^\ast_{(p,q)}(E)$ of $(p,q)$-$(EV^\ast)$ sets is a bornology by (iii) of Lemma \ref{l:V*-set-1}, for the case $q<p$ this class contains only the unique element $\{0\}$ as the next assertion shows.

\begin{proposition} \label{p:V*-q<p}
Let $E$ be a locally convex space.
\begin{enumerate}
\item[{\rm(i)}] If $1\leq q<p\leq\infty$, then $\mathsf{V}_{(p,q)}^\ast(E)=\mathsf{EV}_{(p,q)}^\ast(E)=\{0\}$.
\item[{\rm(ii)}] If $1\leq p\leq q\leq\infty$, then $\mathsf{V}_{(p,q)}^\ast(E)$ contains all finite subsets of $E$.
\end{enumerate}
\end{proposition}

\begin{proof}
(i) By (i) of Lemma  \ref{l:V*-set-1}, we have $\mathsf{V}_{(p,q)}^\ast(E)\subseteq\mathsf{EV}_{(p,q)}^\ast(E)$. So it suffices to prove that every $(p,q)$-$(EV^\ast)$ set is $\{0\}$. Suppose for a contradiction and some $A\in \mathsf{EV}_{(p,q)}^\ast(E)$ contains a non-zero element $a$. Taking into account (iii) of Lemma \ref{l:V*-set-1} we can assume that $A=\mathbb{D}\cdot a$. Since $E$ is locally convex, there is a closed subspace $\tilde E$ of $E$ such that $E=\spn(a)\oplus {\tilde E}$. Fix an arbitrary $\chi\in E'$ such that $\chi\in {\tilde E}^\perp$ and $\langle\chi,a\rangle=1$. Take  positive numbers $t,s$ such that $q<t<s<p$. For every $n\in\w$, let $\chi_n:=\tfrac{\chi}{(n+1)^{1/t}}$. It is clear that the sequence $\{\chi_n\}_{n\in\w}$ is equicontinuous. For every $\eta\in E''$, we have
\[
\sum_{n\in\w} |\langle \eta,\chi_n\rangle|^s =|\langle\eta,\chi\rangle|^s \cdot \sum_{n\in\w} \tfrac{1}{(n+1)^{s/t}} <\infty
\]
and hence the sequence $\{\chi_n\}_{n\in\w}$ is weakly $p$-summable in $E'_\beta$. On the other hand, since
\[
\sum_{n\in\w} \Big(\sup_{x\in A} |\langle\chi_n,x\rangle|\Big)^q = \sum_{n\in\w} |\langle\chi_n,a\rangle|^q=\sum_{n\in\w} \tfrac{1}{(n+1)^{q/t}}=\infty
\]
we obtain that $A$ is not a $(p,q)$-$(EV^\ast)$ set, a contradiction.
\smallskip

(ii) By (iii) of Lemma  \ref{l:V*-set-1} it suffices to show that $A=\{x\}$ is a $(p,q)$-$(V^\ast)$ set for every $x\in E$. Let $\{\chi_n\}_{n\in\w}$ be a weakly $p$-summable sequence in $E'_\beta$. Then $\big(\langle\chi_n,x\rangle\big)\in \ell_p$ (or $\in c_0$ if $p=\infty$). Since $p\leq q$ it follows that $\big(\sup_{x\in A}|\langle\chi_n,x\rangle|\big)\in \ell_q$ (or $\in c_0$ if $q=\infty$). Thus $A$ is a $(p,q)$-$(V^\ast)$ set.\qed
\end{proof}

It is natural to find some classes of subsets which are $(p,q)$-$(V^\ast)$ sets. Below, under additional assumption on an lcs $E$, we show that any precompact subset $A$ of $E$ is $p$-$(V^\ast)$ sets. 

\begin{proposition} \label{p:precompact-p-V*}
Let  $p\in[1,\infty]$, and let $E$  be a locally convex space.
\begin{enumerate}
\item[{\rm(i)}] Every precompact subset $A$ of $E$ is a $p$-$(EV^\ast)$ set.
\item[{\rm(ii)}] If $E$ is $p$-quasibarrelled, then every precompact subset $A$ of $E$ is a $p$-$(V^\ast)$ set.
\end{enumerate}
\end{proposition}

\begin{proof}
Let $S=\{\chi_n\}_{n\in\w}$ be a (resp., equicontinuous) weakly $p$-summable sequence  in $E'_\beta$. If $E$ is $p$-quasibarrelled, then $S$ is equicontinuous. Therefore in both cases we can assume that $S$ is equicontinuous. Since $S$ is weakly $p$-summable, it is a weak$^\ast$ null-sequence. Hence, by Proposition 3.9.8 of \cite{horvath}, the weak$^\ast$ topology $\sigma(E',E)$ and the topology $\tau_{pc}$ of uniform convergence on precompact subsets of $E$ coincide on $S$.  Therefore $\chi_n\to 0$ also in  $\tau_{pc}$. As $A$ is precompact, we obtain $\sup_{x\in A} |\langle \chi_n, x\rangle|\to 0$. Thus $A$ is a $p$-$(EV^\ast)$ set (resp., a $p$-$(V^\ast)$ set).\qed
\end{proof}

Setting $p=\infty$ in (ii) of Proposition \ref{p:precompact-p-V*} we obtain

\begin{corollary} \label{c:precompact-p-V*}
If $E$ is a $\infty$-quasibarrelled space, then every precompact subset of $E$ is a $(V^\ast)$ set.
\end{corollary}

It is natural to consider the next problem: {\em Characterize locally convex spaces $E$ for which every bounded set is a $(p,q)$-$(V^\ast)$ set, i.e. $\Bo(E)=\mathsf{V}_{(p,q)}^\ast(E)$.} Of course, by Proposition \ref{p:V*-q<p}, the problem has sense only in the case $p\leq q$. For the case $q=\infty$,   a complete solution of this problem is obtained in the next assertion, see also Corollary \ref{c:Bo=V*} below.

\begin{theorem} \label{t:Bo=Vp}
Let  $p\in[1,\infty]$,  and let $E$ be a locally convex space. Then $\Bo(E)=\mathsf{V}_p^\ast(E)$ if and only if $E'_\beta$ has the $p$-Schur property.
\end{theorem}

\begin{proof}
Assume that $\Bo(E)=\mathsf{V}_p^\ast(E)$. To show that $E'_\beta$ has the $p$-Schur property, fix an arbitrary weakly $p$-summable sequence $\{\chi_n\}_{n\in\w}$ in $E'_\beta$. To show that $\chi_n\to 0$ in $E'_\beta$, let $B\in\Bo(E)$. Then $B$ is a $p$-$(V^\ast)$ set in $E$, and hence
$
\lim_{n\to\infty}\sup_{b\in B} \big|\langle \chi_n, b\rangle\big|= 0
$
and therefore $\chi_n\in B^\circ$ for all sufficiently large $n\in\w$. Since $B$ was arbitrary this means that $\chi_n\to 0$ in $E'_\beta$, as desired.

Conversely, assume that $E'_\beta$ has the $p$-Schur property. Let $B\in\Bo(E)$. To show that $B$ is  a  $p$-$(V^\ast)$ set in $E$, take any weakly $p$-summable sequence $\{\chi_n\}_{n\in\w}$ in $E'_\beta$. Since $E'_\beta$ has the $p$-Schur property, we obtain that for every $\e>0$,  there is $N_\e\in\w$ such that $\chi_n \in \e B^\circ$ for all $n\geq N_\e$. Therefore $\sup_{b\in B} \big|\langle\chi_n,b\rangle\big| \leq \e$ for all $n\geq N_\e$.
As $\e$ was arbitrary we obtain that $\sup_{b\in B} \big|\langle\chi_n, b\rangle\big|\to 0$. Thus  $B$ is  a  $p$-$(V^\ast)$ set.\qed
\end{proof}

To obtain the equalities  $\mathsf{V}_{(p,q)}^\ast(E)=\mathsf{EV}_{(p,q)}^\ast(E)=\Bo(E)$ for some (in particular, function) spaces we need the next result.
Recall that an lcs $E$ is called {\em feral} if every bounded subset of $E$ is finite-dimensional. 

\begin{proposition} \label{p:strong-dual-feral}
A locally convex space $E$  carries its weak topology and is quasibarrelled if and only if $E'_\beta$ is feral.
\end{proposition}

\begin{proof}
Assume that $E$  carries its weak topology and is quasibarrelled. If $B$ is a bounded subset of $E'_\beta$, then $B$ is equicontinuous. Since $E$ carries its weak topology, there is a finite subset $F$ of $E'$  such that $B\subseteq (F^\circ)^\circ$ and hence $B$ is contained in the finite-dimensional subspace $\spn(F)$ of $E'$.

Conversely, assume that $E'_\beta$ is feral. Then every bounded subset of $E'_\beta$, being finite-dimensional, is trivially equicontinuous. Therefore $E$ is quasibarrelled. In particular, $E$ is a subspace of $E''=(E'_\beta)'_\beta$. But since all bounded subsets of $E'_\beta$ are finite-dimensional it follows that $E''$ carries its weak$^\ast$ topology. Thus $E$ carries its weak topology.\qed
\end{proof}


\begin{proposition} \label{p:feral-V*-prop}
Let $1\leq p\leq q\leq\infty$, and let $E$ be a locally convex space whose strong dual is feral. Then $\mathsf{V}_{(p,q)}^\ast(E)=\mathsf{EV}_{(p,q)}^\ast(E)=\Bo(E)$.
\end{proposition}

\begin{proof}
Since any weakly $p$-summable sequence $\{\chi_n\}_{n\in\w}$ in $E'_\beta$ is weakly bounded, the ferality of $E'_\beta$ implies that  $\{\chi_n\}$ is finite-dimensional. Then, by Lemma \ref{l:lp-finite-dim}, there are linearly independent elements $\eta_1,\dots,\eta_s\in E'$ and sequences $(a_{1,n}),\dots,(a_{s,n})\in \ell_p$ (or $\in c_0$ if $p=\infty$) such that
\[
\chi_n=a_{1,n} \eta_1 +\cdots +a_{s,n}\eta_s \;\; \mbox{ for every $n\in\w$}.
\]
Now, let $A$ be a bounded subset of $E$. Then
\[
\sup_{a\in A} |\langle\chi_n,a\rangle| \leq \sum_{i=1}^s |a_{i,n}| \cdot \sup_{a\in A} |\langle\eta_i,a\rangle|
\]
and hence the inequality $p\leq q$ implies $\big(\sup_{a\in A} |\langle\chi_n, a\rangle|\big)_n \in\ell_q$ (or $\in c_0$ if $q=\infty$). Therefore $A$ is a $(p,q)$-$(V^\ast)$ set, and hence $\Bo(E) \subseteq \mathsf{V}_{(p,q)}^\ast(E)$. As, by  Lemma \ref{l:V*-set-1}, the inclusions $\mathsf{V}_{(p,q)}^\ast(E)\subseteq \mathsf{EV}_{(p,q)}^\ast(E)\subseteq \Bo(E)$ are satisfied for every locally convex space we obtain $\mathsf{V}_{(p,q)}^\ast(E)=\mathsf{EV}_{(p,q)}^\ast(E)=\Bo(E)$.\qed
\end{proof}

Since the space $C_p(X)$ carries its weak topology and is quasibarrelled for every Tychonoff space $X$, Propositions \ref{p:strong-dual-feral} and \ref{p:feral-V*-prop} immediately imply the next assertion.

\begin{corollary} \label{c:Cp-V*-sets}
Let $1\leq p\leq q\leq\infty$, $X$ be a Tychonoff space, and let $E$ be a linear subspace of $\IF^X$ containing $C_p(X)$. Then $\mathsf{V}_{(p,q)}^\ast(E)=\mathsf{EV}_{(p,q)}^\ast(E)=\Bo(E)$. 
\end{corollary}


It is well known (see \cite[p.~477]{Emmanuele-92}) that every $(V^\ast)$ subset of a Banach space is weakly sequentially precompact. However this result is not true for general locally convex spaces.
\begin{example} \label{exa:V*-not-seq-precompact}
The space $\IR^\mathfrak{c}$ contains a uniformly bounded sequence $S=\{f_n\}_{n\in\w}$ which is a $(p,q)$-$(V^\ast)$ set for all $1\leq p\leq q\leq\infty$ but is not (weakly) sequentially precompact.
\end{example}

\begin{proof}
Consider the compact space $3^\w$ with the discrete topology, so $\IR^{3^\w}=\IR^\mathfrak{c}$. Let  $S=\{f_n\}_{n\in\w}$ be the sequence defined in Lemma \ref{l:seq-precom-precom}, i.e., $f_n$ is the projection of $3^\w$ onto the $n$th coordinate. It follows from Corollary \ref{c:Cp-V*-sets} that $S$ is  $(p,q)$-$(V^\ast)$ subset of $\IR^\mathfrak{c}$ for all $1\leq p\leq q\leq\infty$. To show that $S$ is not sequentially precompact, suppose the converse and for some infinite subset $I$ of $\w$, the sequence $\{f_n\}_{n\in I}$ is Cauchy in $\IR^\mathfrak{c}$. Choose an arbitrary infinite subset $J$ of $I$ such that $I\SM J$ is infinite. Definite $h\in 3^\w$ by
\[
h(n):=\left\{
\begin{aligned}
1, & \; \mbox{ if } n\in J,\\
0, & \; \mbox{ if } n\in \w\SM J,
\end{aligned} \right.
\]
and let $e_h:\IR^\mathfrak{c}\to \IR$ be the evaluation functional at $h$, i.e., $\langle e_h,f\rangle:=f(h)$ for every function $f\in \IR^\mathfrak{c}$. By assumption, $\{f_n\}_{n\in I}$ is Cauchy and hence there is a finite subset $F$ of $I$ such that
\[
|\langle e_h,f_i\rangle -\langle e_h,f_{j}\rangle| <\tfrac{1}{4} \;\; \mbox{ for all }\; i,j\in I\SM F.
\]
However, if $i\in I\SM (F\cup J)$ and $j\in J\SM F$ we obtain
\[
\langle e_h,f_i\rangle=f_i(h)=0 \; \mbox{ and } \; \langle e_h,f_j\rangle=f_j(h)=1
\]
and hence $|\langle e_h,f_i\rangle -\langle e_h,f_{j}\rangle|=1>\tfrac{1}{4}$, a contradiction.\qed
\end{proof}

\begin{theorem} \label{t:V*-set-precompact}
Let $1\leq p\leq q\leq\infty$. For a locally convex space $E$ the following assertions are equivalent:
\begin{enumerate}
\item[{\rm (i)}] every $(p,q)$-$(V^\ast)$ subset {\rm(}resp.,  $(p,q)$-$(EV^\ast)$-subset{\rm)} of $E$ is precompact;
\item[{\rm (ii)}] each operator $T:L\to E$ from an lcs $L$ to $E$ which transforms bounded subsets of $L$ to  $(p,q)$-$(V^\ast)$ subsets {\rm(}resp.,  $(p,q)$-$(EV^\ast)$-subset{\rm)}  of $E$, transforms bounded subsets of $L$ to precompact subsets of $E$;
\item[{\rm (iii)}] as in {\rm(ii)} with a normed space $L$.
\end{enumerate}
If in addition $E$ is locally complete, then {\rm(i)--(iii)} are equivalent to
\begin{enumerate}
\item[{\rm (iv)}] as in {\rm(ii)} with a Banach space $L$.
\end{enumerate}
\end{theorem}

\begin{proof}
(i)$\Rightarrow$(ii) Let $T:L\to E$ be an operator which transforms bounded subsets of an lcs $L$ to  $(p,q)$-$(V^\ast)$ (resp.,  $(p,q)$-$(EV^\ast)$) subsets of $E$. Let $A$ be a  bounded subset of $L$. Then $T(A)$ is a $(p,q)$-$(V^\ast)$  (resp.,  $(p,q)$-$(EV^\ast)$) subset of $E$, and hence, by (i), $T(A)$ is precompact. Thus $T$ transforms bounded subsets of $L$ to precompact  subsets of $E$.
\smallskip

(ii)$\Rightarrow$(iii) and (ii)$\Rightarrow$(iv) are trivial.
\smallskip

(iii)$\Rightarrow$(i) and (iv)$\Rightarrow$(i): Fix a $(p,q)$-$(V^\ast)$  (resp.,  $(p,q)$-$(EV^\ast)$)  subset $A$ of $E$. By Lemma \ref{l:V*-set-1}, without loss of generality we can assume that $A=A^{\circ\circ}$. Consider the normed  space $E_A$ (if $E$ is locally complete, then $E_A$ is a Banach space), and recall that the closed unit ball $B$ of  $E_A$ is exactly $A$. By Proposition \ref{p:bounded-norm}, the identity inclusion $T:E_A\to E$ is continuous and the set $T(B)=A$ is a $(p,q)$-$(V^\ast)$ set. Since any bounded subset of $E_A$ is contained in some $aB$, $a>0$, Lemma \ref{l:V*-set-1} implies that $T$ transforms bounded subsets of the normed (resp., Banach) space $E_B$ to  $(p,q)$-$(V^\ast)$   (resp.,  $(p,q)$-$(EV^\ast)$) subsets of $E$.  Therefore, by (iii) and (iv), the set $A=T(B)$ is precompact.\qed
\end{proof}

\section{$(V)$ type subsets of locally convex spaces } \label{sec:V-sets}


Below we generalize the notion of a $p$-$(V)$ subset of a Banach space given in  Definition \ref{def:small-bounded-p}.
\begin{definition}\label{def:tvs-V-subset}{\em
Let $p,q\in[1,\infty]$, and let $E$ be a separated topological vector space.  A non-empty subset $B$ of $E'$ is called a  {\em $(p,q)$-$(V)$ set} if
\[
\Big(\sup_{\chi\in B} |\langle \chi, x_n\rangle|\Big)\in \ell_q \; \mbox{ if $q<\infty$, } \; \mbox{ or }\;\; \Big(\sup_{\chi\in B} |\langle \chi, x_n\rangle|\Big)\in c_0 \; \mbox{ if $q=\infty$},
\]
for every weakly $p$-summable sequence $\{x_n\}_{n\in\w}$ in  $E$. $(p,\infty)$-$(V)$ sets and $(1,\infty)$-$(V)$ sets  will be called simply {\em $p$-$(V)$ sets} and {\em $(V)$ sets}, respectively. If in addition the set $B$ is equicontinuous, it is called a  {\em $(p,q)$-$(EV)$ set}, a {\em $p$-$(EV)$ set} or a  {\em $(EV)$ set}, respectively.\qed}
\end{definition}

\begin{lemma} \label{l:V-set-1}
Let $p,q\in[1,\infty]$, and let $(E,\tau)$ be a locally convex space. Then:
\begin{enumerate}
\item[{\rm(i)}] each $(p,q)$-$(EV)$ set in $E'$ is a strongly bounded $(p,q)$-$(V)$ set; if in addition $E$ is barrelled, then every $(p,q)$-$(V)$ set is a $(p,q)$-$(EV)$ set;
\item[{\rm(ii)}]  every $(p,q)$-$(V)$ set in $E'$ is weak$^\ast$ bounded;
\item[{\rm(iii)}] the family of all $(p,q)$-$(V)$ $($resp., $(p,q)$-$(EV)$$)$ sets in $E'$ is closed under taking subsets, finite unions, finite sums, and closed absolutely convex hulls in $E'_{w^\ast}$;
\item[{\rm(iv)}] if $T:E\to L$ is an operator to an lcs $L$ and $B$ is a $(p,q)$-$(V)$ $($resp., $(p,q)$-$(EV)$$)$ set in $L'_\beta$, then $T^\ast(B)$ is a $(p,q)$-$(V)$ $($resp., $(p,q)$-$(EV)$$)$ set in $E'_\beta$;
\item[{\rm(v)}] a  subset $B$ of $E'$ is a $(p,q)$-$(V)$ set  if and only if every countable subset of $B$ is a  $(p,q)$-$(V)$ set;
\item[{\rm(vi)}] if $p',q'\in[1,\infty]$ are such that $p'\leq p$ and $q\leq q'$, then every $(p,q)$-$(V)$ $($resp., $(p,q)$-$(EV)$$)$ set in $E'$ is also a $(p',q')$-$(V)$ $($resp., $(p',q')$-$(EV)$$)$ set; in particular, any $(p,q)$-$(V)$  $($resp., $(p,q)$-$(EV)$$)$ is a $(V)$ $($resp., $(EV)$$)$ set;
\item[{\rm(vii)}] the property of being a $(p,q)$-$(V)$ set depends only on the duality $(E,E')$, i.e.,   if $\TTT$ is a locally convex vector topology on $E$ compatible with the topology $\tau$ of $E$, then the $(p,q)$-$(V)$ sets of $(E,\TTT)'$ are exactly the $(p,q)$-$(V)$ sets of $E'$;
\item[{\rm(viii)}] every $(p,q)$-$(V)$ set in $E'$ is a $(p,q)$-$(V^\ast)$ set in $E'_{w^\ast}$; the converse is true if $E$ is barrelled;
\item[{\rm(ix)}] if $H$ is a dense subspace of $E$, then every $(p,q)$-$(V)$  $($resp., $(p,q)$-$(EV)$$)$ set in $E'$ is a $(p,q)$-$(V)$ $($resp., $(p,q)$-$(EV)$$)$ set in $H'$;
\item[{\rm(x)}] if $E$ is additionally quasibarrelled, then every  $(p,q)$-$(V^\ast)$ set in $E'_{\beta}$ is a $(p,q)$-$(V)$  set.
\end{enumerate}
\end{lemma}

\begin{proof}
(i) and (iii) are clear, and (ix) follows from (iv) applied to the identity inclusion $\Id_H:H\to E$ and the algebraic equality $H'=E'$.
\smallskip

(ii) Let $B$ be a $(p,q)$-$(V)$ set in $E'$ and suppose for a contradiction that it is not weak$^\ast$ bounded. Then there is $x\in E$ such that for every $n\in\w$, there exists $b_n\in B$ for which $|\langle b_n,x\rangle|> (n+1)^{3}$. For every $n\in\w$, set $x_n:=\tfrac{1}{(n+1)^{2}}\cdot x$. Then $\{x_n\}_{n\in\w}$ is a weakly $p$-summable sequence in  $E$. However, since
\[
\sup_{b\in B} |\langle b,x_n\rangle| \geq |\langle b_n, x_n\rangle|> n+1 \to \infty
\]
it follows that $B$ is not a $(p,q)$-$(V)$ set, a contradiction.

(iv) First we note that the adjoint operator $T^\ast: L'_\beta\to E'_\beta$ is continuous, and hence if $B\subseteq L'$ is equicontinuous then so is its image $T^\ast(B)$. Therefore it suffices to consider only the case when $B$ is a $(p,q)$-$(V)$ set in $L'$.  Fix a weakly $p$-summable sequence $\{x_n\}_{n\in\w}$ in  $E$. Then, by Lemma \ref{l:prop-p-sum}(iii), the sequence $\{T(x_n)\}$ is weakly $p$-summable in  $L$. Therefore
\[
\Big(\sup_{\chi\in B} |\langle T^\ast(\chi), x_n\rangle|\Big)= \Big(\sup_{\chi\in B} |\langle \chi, T(x_n)\rangle|\Big)\in \ell_q \; \mbox{ (or } \; \in c_0 \; \mbox{ if $q=\infty$}),
\]
which means that $T^\ast(B)$ is a $(p,q)$-$(V)$ set in $E'$.

(v) The necessity follows from (iii). To prove the sufficiency suppose for a contradiction that $B$ is not a $(p,q)$-$(V)$ set in $E$. Then there is a weakly $p$-summable sequence $\{x_n\}_{n\in\w}$ in  $E$ such that
\[
\Big(\sup_{\chi\in B} |\langle \chi, x_n\rangle|\Big)\not\in \ell_q \; \mbox{ if $q<\infty$, } \; \mbox{ or }\;\; \Big(\sup_{\chi\in B} |\langle \chi, x_n\rangle|\Big)\not\in c_0 \; \mbox{ if $q=\infty$}.
\]
Assume that $q<\infty$ (the case $q=\infty$ can be considered analogously). For every $n\in\w$, choose $\chi_n\in B$ such that $|\langle \chi_n, x_n\rangle|\geq \tfrac{1}{2} \cdot \sup_{a\in A} |\langle \chi_n, a\rangle|$. Then
\[
\sum_{n\in\w} |\langle \chi_n, x_n\rangle|^q \geq \tfrac{1}{2^q} \sum_{n\in\w} \big(\sup_{\chi\in B} |\langle \chi, x_n\rangle|\big)^q =\infty.
\]
Thus the countable subset $\{\chi_n\}_{n\in\w}$ of $B$ is not a $(p,q)$-$(V)$ set in $E'$, a contradiction.

(vi) Take any weakly $p'$-summable sequence $\{x_n\}_{n\in\w}$ in  $E$. Since $p'\leq p$,  $\{x_n\}$  is also weakly $p$-summable and hence $\big( \sup_{\chi\in B} |\langle \chi, x_n\rangle|\big)\in \ell_q$ (or $\in c_0$ if $q=\infty$). It remains to note that $\ell_q\subseteq \ell_{q'}$ because $q\leq q'$.

(vii) immediately follows from (iv) of Lemma \ref{l:prop-p-sum}.

(viii) Let $B$ be a $(p,q)$-$(V)$ set in $E'$. Then for every weakly $p$-summable sequence $\{x_n\}_{n\in\w}$ in $E$ we have $\big(\sup_{\chi\in B} |\langle \chi, x_n\rangle|\big)\in \ell_q$ (or $\in c_0$ if $q=\infty$). Observe that $\big(E'_{w^\ast}\big)'_\beta =E_\beta$. By (ii) of Lemma \ref{l:prop-p-sum}, we have  $\ell_p^w(E_\beta)\subseteq \ell_p^w(E)$. Therefore also for every weakly $p$-summable sequence $(y_n)\in\ell_p^w(E_\beta)$ we have $\big(\sup_{\chi\in B} |\langle \chi, y_n\rangle|\big)\in \ell_q$ (or $\in c_0$ if $q=\infty$). Thus $B$ is a $(p,q)$-$(V^\ast)$ set in $E'_{w^\ast}$. The last assertion follows from the equality $E_\beta=E$ if $E$ is barrelled.

(x) Let $B$ be a $(p,q)$-$(V^\ast)$ set in $E'_\beta$. To show that $B$ is also a $(p,q)$-$(V)$ set, let $\{x_n\}_{n\in\w}$ be a weakly $p$-summable sequence in $E$. Since $E$ is quasibarrelled, the canonical map $J_E:E\to E''=(E'_\beta)'_\beta$ is continuous. Therefore the sequence $\{J_E(x_n)\}_{n\in\w}$ is weakly $p$-summable in $E''$ by Lemma \ref{l:prop-p-sum}(iii). Since $B$ is a $(p,q)$-$(V^\ast)$ set in $E'_\beta$ and $\sup_{\chi\in B} |\langle J_E(x_n),\chi\rangle|=\sup_{\chi\in B} |\langle \chi, x_n\rangle|$, it follows that $\big(\sup_{\chi\in B} |\langle \chi, x_n\rangle|\big)\in \ell_q$ (or $\in c_0$ if $q=\infty$). Thus $B$ is a $(p,q)$-$(V)$ set in $E'$.\qed
\end{proof}

\begin{notation} \label{n:V-set} {\em
The family of all $(p,q)$-$(V)$ sets (resp. $p$-$(V)$ sets, $(p,q)$-$(EV)$ sets, $(V)$ sets etc.) in $E'$ of an lcs $E$ is denoted by $\mathsf{V}_{(p,q)}(E')$ (resp. $\mathsf{V}_{p}(E')$, $\mathsf{EV}_{(p,q)}(E')$, $\mathsf{V}(E')$ etc.).\qed}
\end{notation}

Now we consider $(p,q)$-$(V)$ sets in topological products and direct sums.
\begin{proposition} \label{p:product-sum-V-set}
Let  $p,q\in[1,\infty]$, and let $\{E_i\}_{i\in I}$  be a non-empty family of locally convex spaces. Then:
\begin{enumerate}
\item[{\rm(i)}] a subset $K$ of $\big(\prod_{i\in I} E_i\big)'$ is a  $(p,q)$-$(V)$ set $($resp., a  $(p,q)$-$(EV)$ set$)$ if and only if so are all its coordinate projections  and the support of $K$ is finite;
\item[{\rm(ii)}]  a subset $K$ of  $\big(\bigoplus_{i\in I} E_i\big)'$  is a  $(p,q)$-$(V)$ set if and only if so are all its coordinate projections; if $I$ is countable, then the same is true for $(p,q)$-$(EV)$ sets.
\end{enumerate}
\end{proposition}


\begin{proof}
To prove the necessity in both cases (i) and (ii), let $E=\prod_{i\in I} E_i$ or $E=\bigoplus_{i\in I} E_i$.
Fix $j\in I$, and let $K_j$ be the projection of $K$ onto the $j$th coordinate. To show that $K_j$ is a $(p,q)$-$(V)$ set (resp., a  $(p,q)$-$(EV)$ set) in $(E_j)'$, let $T_j:E_j\to E$ be the identity embedding. Then, by (iv) of Lemma  \ref{l:V-set-1}, $T_j^\ast(K)=K_j$ is a  $(p,q)$-$(V)$ set (resp., a  $(p,q)$-$(EV)$ set)  in $(E_j)'$, as desired.

%

In the case (i) we need to check also that the support of  $K$ is finite. Suppose for a contradiction that $K_i \not=\{0_i\}$ for infinitely many indices $i$. Then we can choose a sequence $\{\chi_n\}_{n\in\w}$ in $K$ such that
\[
\supp(\chi_0)\not=\emptyset \; \mbox{ and }\; \supp(\chi_{n+1})\SM \bigcup_{j\leq n}\supp(\chi_j) \not=\emptyset \mbox{ for all } n\in\w,
\]
and hence we can choose a sequence $\{i_n\}_{n\in\w}$ in $I$ such that
\begin{equation} \label{equ:V-prod-0}
i_0\in \supp(\chi_0)\; \mbox{ and } \; i_{n+1}\in \supp(\chi_{n+1})\SM \bigcup_{j\leq n}\supp(\chi_j) \; \mbox{ for all $n\in\w$.}
\end{equation}
By induction on $n$, for every $n\in\w$, choose $y_{i_n}\in E_{i_n}$ such that
\begin{equation} \label{equ:V-prod-1}
|\langle \chi_n(i_n), y_{i_n}\rangle|\geq 1+ \sum_{k<n} |\langle \chi_n(i_k), y_{i_k}\rangle|
\end{equation}
and define $x_n=(x_{i,n})\in E$ by $x_{i,n}=y_{i_n} $ if $i=i_n$, and $x_{i,n}=0$ otherwise. Since every $\chi\in E'$ has finite support, we obtain that $\langle \chi, x_{n}\rangle=0$  for all but finitely many $n\in\w$. Therefore the sequence $\{x_n\}_{n\in\w}$ is  weakly $p$-summable  in  $E$. As, by (\ref{equ:V-prod-0}) and (\ref{equ:V-prod-1}),
\[
\sup_{\chi\in K} |\langle \chi, x_{n}\rangle|\geq |\langle \chi_n, x_{n}\rangle|\geq |\langle \chi_n(i_n), y_{i_n}\rangle|-\sum_{k<n} |\langle \chi_n(i_k), y_{i_k}\rangle| \geq 1 \not\to 0
\]
it follows that  $K$ is not a $(p,q)$-$(V)$ set, a contradiction. Thus the support of $K$ is finite, as desired.
\smallskip

To prove the sufficiency, let $K$ be a subset of $E'$ such that each projection $K_i$ of $K$ is a $(p,q)$-$(V)$ set (resp., a  $(p,q)$-$(EV)$ set) in $E'_i$, and, for the case (i), $K_i=\{0\}$ for all but finitely many indices $i\in I$. We distinguish between the cases (i) and (ii).

(i) Let $F\subseteq I$ be the finite support of $K$. Take an arbitrary  weakly $p$-summable sequence $\{x_n\}_{n\in\w}$ in $E$, where $x_n=(x_{i,n})_{i\in I}$. Observe that for every $i\in I$, the sequence $\{x_{i,n}\}_{n\in\w}$ is weakly $p$-summable in $E_i$. Then
\[
\sup_{\chi\in K} |\langle \chi, x_{n}\rangle|=\sup_{\chi\in K} \big|\sum_{i\in F} \langle \chi(i), x_{i,n}\rangle\big|\leq \sum_{i\in F} \sup_{\chi(i)\in K_i} |\langle \chi(i), x_{i,n}\rangle|
\]
and hence $\big( \sup_{\chi\in K} |\langle \chi, x_{n}\rangle|\big) \in \ell_q$ (or $\in c_0$ if $q=\infty$). Thus $K$ is  a $(p,q)$-$(V)$ set in $E'$.

If all $K_i$ are equicontinuous, then also $K$ is equicontinuous since $F$ is finite. Therefore $K$ is  a $(p,q)$-$(EV)$ set in $E'$.
\smallskip

(ii) Take an arbitrary  weakly $p$-summable sequence $\{x_n\}_{n\in\w}$ in $E$, where $x_n=(x_{i,n})_{i\in I}$. Since $\{x_n\}_{n\in\w}$ is a bounded subset of $E$, it has finite support, i.e., for some finite set $F\subseteq I$ it follows that $x_{i,n}=0$ for all $n\in\w$ and each $i\in I\SM F$. Observe also that for every $i\in I$, the sequence $\{x_{i,n}\}_{n\in\w}$ is weakly $p$-summable.  Then
\[
\sup_{\chi\in K} |\langle \chi, x_{n}\rangle|=\sup_{\chi\in K} \big|\sum_{i\in F} \langle \chi(i), x_{i,n}\rangle\big|\leq \sum_{i\in F} \sup_{\chi(i)\in K_i} |\langle \chi(i), x_{i,n}\rangle|
\]
and hence $\big( \sup_{\chi\in K} |\langle \chi, x_{n}\rangle|\big) \in \ell_q$ (or $\in c_0$ if $q=\infty$). Thus $K$ is  a $(p,q)$-$(V)$ set in $E'$.

Assume that $I=\w$ is countable, and  all $K_i$ are equicontinuous, so that  $K_i\subseteq U_i^\circ$ for some $U_i\in \Nn_0(E_i)$. Set $U:= \bigoplus_{i\in\w} \tfrac{1}{2^{i+1}} U_i$. Then $U\in\Nn_0(E)$ and for every $\chi=(\chi_i)\in K$ and each $x=(x_i)\in U$, we have
\[
|\langle\chi,x\rangle|=\Big| \sum_{i\in\supp(x)} \langle\chi_i, x_i\rangle\Big| \leq \sum_{i\in\w} \tfrac{1}{2^{i+1}} =1.
\]
Therefore $K\subseteq U^\circ$ and hence $K$ is equicontinuous. Thus $K$ is  a $(p,q)$-$(EV)$ set in $E'$. \qed
\end{proof}

Analogously to $(p,q)$-$(V)$ sets we define
\begin{definition}\label{def:tvs-*V-subset}{\em
Let $p,q\in[1,\infty]$, and let $E$ be a separated tvs.  A subset $B$ of $E'$ is called a {\em weak$^\ast$ $(p,q)$-$(V)$ set}  if it is a $(p,q)$-$(V^\ast)$ set in $E'_{w^\ast}$, i.e.,
\[
\Big(\sup_{\chi\in B} |\langle \chi, x_n\rangle|\Big)\in \ell_q \; \mbox{ if $q<\infty$, } \; \mbox{ or }\;\; \Big(\sup_{\chi\in B} |\langle \chi, x_n\rangle|\Big)\in c_0 \; \mbox{ if $q=\infty$},
\]
for every weakly $p$-summable sequence $\{x_n\}_{n\in\w}$ in  $\big(E'_{w^\ast}\big)'_\beta =E_\beta$.   Weak$^\ast$ $(p,\infty)$-$(V)$ sets and weak$^\ast$ $(\infty,\infty)$-$(V)$ sets  will be called simply {\em weak$^\ast$ $p$-$(V)$ sets} and {\em weak$^\ast$ $(V)$ sets}, respectively. If in addition the set $B$ is equicontinuous, it is called a  {\em  weak$^\ast$ $(p,q)$-$(EV)$ set}, a {\em weak$^\ast$  $p$-$(EV)$ set} or a  {\em  weak$^\ast$ $(EV)$ set}, respectively.\qed}
\end{definition}

In the next lemma, for further references, we summarize the main properties of weak$^\ast$ $(p,q)$-$(V)$ sets and weak$^\ast$ $(p,q)$-$(EV)$ sets.

\begin{lemma} \label{l:*V-set-1}
Let $p,q\in[1,\infty]$, and let $(E,\tau)$ be a locally convex space. Then:
\begin{enumerate}
\item[{\rm(i)}] each  weak$^\ast$ $(p,q)$-$(EV)$ set is a  weak$^\ast$ $(p,q)$-$(V)$ set, the converse is true if $E$ is barrelled;
\item[{\rm(ii)}]  every  weak$^\ast$ $(p,q)$-$(V)$ set in $E'$ is weak$^\ast$ bounded;
\item[{\rm(iii)}] the family of all  weak$^\ast$ $(p,q)$-$(V)$ $($resp.,  weak$^\ast$ $(p,q)$-$(EV)$$)$ sets in $E'$ is closed under taking subsets, finite unions, finite sums, and closed absolutely convex hulls in $E'_{w^\ast}$;
\item[{\rm(iv)}] if $T:E\to L$ is an operator to an lcs $L$ and $B$ is a  weak$^\ast$ $(p,q)$-$(V)$ $($resp.,  weak$^\ast$  $(p,q)$-$(EV)$$)$ set in $L'$, then $T^\ast(B)$ is a  weak$^\ast$ $(p,q)$-$(V)$ set $($resp.,  weak$^\ast$ $(p,q)$-$(EV)$$)$ set in $E'$;
\item[{\rm(v)}] a  subset $B$ of $E'$ is a weak$^\ast$  $(p,q)$-$(V)$ set  if and only if every countable subset of $B$ is a  weak$^\ast$  $(p,q)$-$(V)$ set in $E'$;
\item[{\rm(vi)}] if $p',q'\in[1,\infty]$ are such that $p'\leq p$ and $q\leq q'$, then every  weak$^\ast$ $(p,q)$-$(V)$ $($resp.,  weak$^\ast$ $(p,q)$-$(EV)$$)$ set in $E'$ is also a $(p',q')$-$(V)$ $($resp.,  weak$^\ast$ $(p',q')$-$(EV)$$)$ set; in particular, any  weak$^\ast$ $(p,q)$-$(V)$ $($resp., weak$^\ast$  $(p,q)$-$(EV)$$)$ set is a  weak$^\ast$ $1$-$(V)$ $($resp.,  weak$^\ast$ $1$-$(EV)$$)$ set;
\item[{\rm(vii)}] the property of being a weak$^\ast$ $(p,q)$-$(V)$ set depends only on the duality $(E,E')$, i.e.,   if $\TTT$ is a locally convex vector topology on $E$ compatible with the topology $\tau$ of $E$, then the  weak$^\ast$ $(p,q)$-$(V)$  sets of $(E,\TTT)'$ are exactly the  weak$^\ast$ $(p,q)$-$(V)$ sets of $E'$;
\item[{\rm(viii)}] every $(p,q)$-$(V)$ set in $E'$ is a  weak$^\ast$ $(p,q)$-$(V)$  set; the converse assertion is true if $E$ is a barrelled space;
\item[{\rm(ix)}] if $H$ is a dense subspace of $E$, then every weak$^\ast$ $(p,q)$-$(V)$ set in $E'$ is a  weak$^\ast$ $(p,q)$-$(V)$  set in $H'$.
\end{enumerate}
\end{lemma}

\begin{proof}
(i) is clear and the converse assertion follows from (ii) and the barrelledness of $E$. The clause (ii) follows from (ii) of Lemma \ref{l:V*-set-1}, and (iii) follows from (iii) of Lemma \ref{l:V*-set-1} and the easy fact that subsets, finite unions, and closed absolutely convex hulls of equicontinuous subsets of $E'$ are equicontinuous. Since the adjoint map $T^\ast$ is weak$^\ast$-weak$^\ast$ continuous (Theorem 8.10.5 of \cite{NaB}), (iv) follows from (iv) of Lemma \ref{l:V*-set-1}. The clauses (v)-(vii) follow from (v)-(vii) of Lemma \ref{l:V*-set-1}. The clause (ix) follows from (iv) applied to the identity inclusion $T=\Id_H:H\to E$.
\smallskip

(viii) Let $B\subseteq E'$ be a $(p,q)$-$(V)$ set, and let $\{x_n\}_{n\in\w}$ be a  weakly $p$-summable sequence in   $\big(E'_{w^\ast}\big)'_\beta=E_\beta$. Since the strong topology $\beta(E,E')$ is finer than the original topology of $E$, it follows that $\{x_n\}_{n\in\w}$ is a  weakly $p$-summable sequence in $E$. Therefore $\big(\sup_{\chi\in B} |\langle \chi, x_n\rangle|\big)\in \ell_q$ if $q<\infty$, or $\big(\sup_{\chi\in B} |\langle \chi, x_n\rangle|\big)\in c_0$  if $q=\infty$. This means that $B$ is a weak$^\ast$ $(p,q)$-$(V)$ set in $E'$. The converse assertion is true for barrelled spaces because $E_\beta=E$.\qed
%
\end{proof}

\begin{notation} \label{n:weak*-V-set} {\em
The family of all  weak$^\ast$  $(p,q)$-$(V)$ sets (resp.  weak$^\ast$  $p$-$(V)$ sets,  weak$^\ast$ $(p,q)$-$(EV)$ sets,  weak$^\ast$  $(V)$ sets etc.) in $E'$ of an lcs $E$ is denoted by $\mathsf{V}_{(p,q)}^\ast(E'_{w^\ast})$ (resp. $\mathsf{V}_{p}^\ast(E'_{w^\ast})$, $\mathsf{EV}_{(p,q)}^\ast(E'_{w^\ast})$, $\mathsf{V}^\ast(E'_{w^\ast})$ etc.).\qed}
\end{notation}

The condition of being a barrelled space in (viii) of Lemma \ref{l:*V-set-1} is essential as the following example shows, in which the space $\varphi$ endowed with the pointwise topology induced from $\IF^\w$ is denoted by $\varphi_p$.

\begin{example} \label{exa:weak*-non-Vp-set}
Let $1\leq p\leq q \leq \infty$, and let $E=\varphi_p$. Then:
\begin{enumerate}
\item[{\rm(i)}] $E'_{w^\ast}=E$ and $E_\beta=\varphi$, in particular, $E$ is not barrelled;
\item[{\rm(ii)}] each weakly $p$-summable sequence in $E_\beta$ is finite-dimensional, so $\ell^w_p(E_\beta)\subsetneq \ell^w_p(E)$;
\item[{\rm(iii)}] every weak$^\ast$ bounded subset of $E'$ is a weak$^\ast$  $(p,q)$-$(V)$ set, so $\Bo(E'_{w^\ast}) =\mathsf{V}_{(p,q)}^\ast(E'_{w^\ast})$;
\item[{\rm(iv)}] every $(p,q)$-$(V)$ set in $E'$ is finite-dimensional, so $\mathsf{V}_{(p,q)}(E)\subsetneq\mathsf{V}_{(p,q)}^\ast(E'_{w^\ast})$.
\end{enumerate}
\end{example}

\begin{proof}
(i) Since $E$ is dense in $\IF^\w$, we have $E'=\varphi$ algebraically. Thus $E'_{w^\ast}=E$.

It is clear that any pointwise bounded subset of $E'$ is weak$^\ast$ bounded. Then the polars
\[
\Big(\varphi \cap \prod_{i\in\w} a_n \mathbb{D}\Big)^\circ \;\; \mbox{ where all }\; \; a_n>0,
\]
form a base at zero in $E_\beta$. But these sets form also a base at zero in $\varphi$. Thus $E_\beta=\varphi$.

Since $E_\beta=\varphi\not= E$, the space $E$ is not barrelled.
\smallskip

(ii) follows from (i) since any bounded subset (in particular, each weakly $p$-summable sequence) of $\varphi$ is finite-dimensional.
Since $\{e_n\}_{n\in\w}$ is a weakly $p$-summable sequence in $E$ and is infinite-dimensional, we have $\ell^w_p(E_\beta)\subsetneq \ell^w_p(E)$.
\smallskip

(iii) Let $B$ be a weak$^\ast$ bounded subset of $E'$. Let $S=\{x_n\}_{n\in\w}$ be a  weakly $p$-summable sequence in $E_\beta$. By (ii), $S$ is finite-dimensional. Therefore there are $s\in\w$ and scalars $a_{0,n},\dots,a_{s,n}$ such that
\[
x_n =a_{0,n}e_0 +\cdots+a_{s,n}e_s \;\; \mbox{ for every } \; n\in\w.
\]
Since $S$ is weakly $p$-summable it follows that $(a_{0,n}),\dots,(a_{s,n})\in \ell_p$ (or $\in c_0$ if $p=\infty$). Then
\[
\sup_{\chi\in B} |\langle\chi,x_n\rangle| \leq \sum_{i=0}^s |a_{i,n}| \cdot \sup_{\chi\in B} |\langle \chi, e_i\rangle|
\]
and hence the inequality $p\leq q$ implies $\big(\sup_{\chi\in B} |\langle\chi_n, a\rangle|\big)_n \in\ell_q$ (or $\in c_0$ if $q=\infty$). Therefore $B$ is weak$^\ast$  $(p,q)$-$(V)$ set,  and hence $\Bo(E'_{w^\ast}) =\mathsf{V}_{(p,q)}^\ast(E'_{w^\ast})$.
\smallskip

(iv) Suppose for a contradiction that there is an infinite-dimensional $(p,q)$-$(V)$ set $B$ in $E'$. Then we can find a sequence $\{\chi_n\}_{n\in\w}$ in $B$ such that
\[
m_{n+1}:= \max\big\{ \supp(\chi_{n+1})\big\} > \max\Big\{ \bigcup_{i\leq n} \supp(\chi_i)\Big\}.
\]
For every $n\in\w$, choose $a_n\in\IF$ such that
$a_n \cdot \chi_n (m_n)>n$ and set $x_n=a_n \cdot e_{m_n}\in E$. It is clear that the sequence $\{x_n\}_{n\in\w}$ is weakly $p$-summable in $E$. However, since
$
\sup_{\chi\in B} |\langle\chi, x_n\rangle| \geq |\langle\chi_n, x_n\rangle|>n \not\to 0
$
it follows that $B$ is not a $(p,q)$-$(V)$ set, a contradiction.

Finally, since every  $(p,q)$-$(V)$ set in $E'$ is a  weak$^\ast$ $(p,q)$-$(V)$  set by (viii) of Lemma \ref{l:*V-set-1}, the strict inclusion $\mathsf{V}_{(p,q)}(E)\subsetneq\mathsf{V}_{(p,q)}^\ast(E'_{w^\ast})$ follows from (iii).\qed
\end{proof}

In Proposition \ref{p:V*-q<p} we showed that $\mathsf{V}_{(p,q)}^\ast(E)=\mathsf{EV}_{(p,q)}^\ast(E)=\{0\}$ for all $1\leq q<p\leq\infty$. An analogous result holds also for $V$-type sets.

\begin{proposition} \label{p:V-q<p}
Let $E$ be a locally convex space.
\begin{enumerate}
\item[{\rm(i)}] If $1\leq q<p\leq\infty$, then  $\mathsf{V}_{(p,q)}(E')=\mathsf{EV}_{(p,q)}(E')=\mathsf{V}_{(p,q)}(E'_{w^\ast})=\mathsf{EV}_{(p,q)}(E'_{w^\ast})=\{0\}$.
\item[{\rm(ii)}] If $1\leq p\leq q\leq\infty$, then every finite subset of $E'$ belongs to $\mathsf{V}_{(p,q)}(E')\cap\mathsf{EV}_{(p,q)}(E')\cap\mathsf{V}_{(p,q)}(E'_{w^\ast})\cap\mathsf{EV}_{(p,q)}(E'_{w^\ast})$.
\end{enumerate}
\end{proposition}

\begin{proof}
(i) Taking into account (i) of Lemma \ref{l:V-set-1} and (i) and (viii) of Lemma \ref{l:*V-set-1}, it suffices to prove that every weak$^\ast$ $(p,q)$-$(V)$  set contains only the zero. Suppose for a contradiction and some $B\in \mathsf{V}_{(p,q)}(E'_{w^\ast})$ contains a non-zero element $b$. Taking into account (iii) of Lemma \ref{l:*V-set-1} we can assume that $B=\mathbb{D}\cdot b$. Since $E'_{w^\ast}$ is locally convex, there is a closed subspace $H$ of $E'_{w^\ast}$ such that $E'_{w^\ast}=\spn(b)\oplus H$. Fix an arbitrary $z\in E_\beta$ such that $x\in H^\perp$ and $\langle b,z\rangle=1$. Take  positive numbers $t,s$ such that $q<t<s<p$. For every $n\in\w$, let $x_n:=\tfrac{z}{(n+1)^{1/t}}$. For every $\eta\in (E_\beta)'$, we have
\[
\sum_{n\in\w} |\langle \eta,x_n\rangle|^s =|\langle\eta,z\rangle|^s \cdot \sum_{n\in\w} \tfrac{1}{(n+1)^{s/t}} <\infty
\]
and hence the sequence $\{x_n\}_{n\in\w}$ is weakly $p$-summable in $E_\beta$. On the other hand, since
\[
\sum_{n\in\w} \Big(\sup_{\chi\in B} |\langle\chi,x_n\rangle|\Big)^q = \sum_{n\in\w} |\langle b,x_n\rangle|^q=\sum_{n\in\w} \tfrac{1}{(n+1)^{q/t}}=\infty
\]
we obtain that $B$ is not a  weak$^\ast$ $(p,q)$-$(V)$ set, a contradiction.
\smallskip

(ii) Taking into account (i) and (iii) of Lemma \ref{l:V-set-1} and (i) and (viii) of Lemma \ref{l:*V-set-1}, it suffices to prove that the set $B=\{\chi\}$ is a $(p,q)$-$(EV)$ set for every $\chi\in E'$. Evidently, $B$ is equicontinuous. Fix a weakly $p$-summable sequence  $\{x_n\}_{n\in\w}$ in $E$. Then $\big(\langle \chi,x_n\rangle\big)\in \ell_p$ (or $\in c_0$ if $p=\infty$). Since $p\leq q$ it follows that $\big(\sup_{\chi\in B}|\langle \chi,x_n\rangle|\big)\in \ell_q$ (or $\in c_0$ if $q=\infty$). Thus $B$ is a $(p,q)$-$(EV)$ set.\qed
\end{proof}



\section{$V$ types of  weak barrelledness conditions} \label{sec:V-bar}


We start from the following general approach to weak barrelledness conditions in locally convex spaces.
\begin{definition} \label{def:BB-barrelled-cond} {\em
Let $E$ be a locally convex space, and let $\BB$ be a bornology of weak$^\ast$ bounded subsets of $E'$ such that $\bigcup\BB=E'$. Then the space $E$ is said to be
\begin{enumerate}
\item[$\bullet$] {\em $\BB$-$($quasi$)$barrelled} if every (resp., strongly bounded) set $B$ in $E'_\beta$ such that $B\in \BB$ is equicontinuous;
\item[$\bullet$] {\em $\aleph_0$-$\BB$-}({\em quasi}){\em barrelled} if every  (resp., strongly bounded) set  $B$ in $E'_\beta$ such that  $B\in \BB$  and  which is the countable union of equicontinuous subsets is also equicontinuous;
\item[$\bullet$] {\em $\ell_\infty$-$\BB$-$($quasi$)$barrelled} if every  (resp., strongly bounded)  sequence  $B$ in $E'_\beta$ such that $B\in \BB$ is equicontinuous;
\item[$\bullet$] {\em $c_0$-$\BB$-$($quasi$)$barrelled} if every weak$^\ast$ (resp., strongly) null-sequence  $B$ in $E'_\beta$ such that $B\in \BB$ is equicontinuous.\qed
\end{enumerate}  }
\end{definition}
If $\BB$ is the family $\Bo(E'_{w^\ast})$ of all weak$^\ast$ bounded subsets of $E'$, we obtain the classical weak barrelledness conditions from Definition \ref{def:weak-barrel}; in this case the letter $\BB$ will be omitted. By Lemma \ref{l:V-set-1} and  Proposition \ref{p:V-q<p}, we know that the family $\mathsf{V}_{(p,q)}(E')$ is a bornology such that $\bigcup \mathsf{V}_{(p,q)}(E')=E'$ if and only if $1\leq p\leq q\leq\infty$. In this case $\mathsf{V}_{(p,q)}(E')$ is also saturated and defines weak barrelledness conditions. As usual, if $q=\infty$ we shall omit the second subscript and called $V_{(p,\infty)}$-$($quasi$)$barrelled spaces by $V_{p}$-$($quasi$)$barrelled spaces etc. Analogously, $V_{(\infty,\infty)}$-$($quasi$)$barrelled spaces will be called  $V$-$($quasi$)$barrelled etc. Observe that $E$ is $V_{(p,q)}$-barrelled if and only if $\mathsf{V}_{(p,q)}(E')=\mathsf{EV}_{(p,q)}(E')$.



The next proposition shows the relationships between the classical weak barrelledness conditions from Definition \ref{def:weak-barrel} and the introduced ones in Definition \ref{def:BB-barrelled-cond} for the family $\BB=\mathsf{V}_{(p,q)}(E')$, it is immediately follows from  the inclusion $\mathsf{V}_{(p,q)}(E') \subseteq \Bo(E'_{w^\ast})$.
\begin{proposition} \label{p:Vp-barrelled}
Let $1\leq p\leq q\leq\infty$, and let $E$ be a locally convex space. Then:
\begin{enumerate}
\item[{\rm(i)}] If $E$ is $($quasi$)$barrelled, then it is $V_{(p,q)}$-$($quasi$)$barrelled.
\item[{\rm(ii)}] If $E$ is $\aleph_0$-$($quasi$)$barrelled, then it is $\aleph_0$-$V_{(p,q)}$-$($quasi$)$barrelled.
\item[{\rm(iii)}]  If $E$ is $\ell_\infty$-$($quasi$)$barrelled, then it is $V_{(p,q)}$-$($quasi$)$barrelled.
\item[{\rm(iv)}]  If $E$ is $c_0$-$($quasi$)$barrelled, then it is $c_0$-$V_{(p,q)}$-$($quasi$)$barrelled.
\end{enumerate}
\end{proposition}

It is naturally to define classes of locally convex spaces between the class of $\ell_\infty$-quasibarrelled spaces and the class of $c_0$-quasibarrelled spaces, and between the class of  $\ell_\infty$-$V_{(p,q)}$-quasibarrelled  and the class of  $c_0$-$V_{(p,q)}$-quasibarrelled spaces.

\begin{definition} \label{def:strict-Vp-barrelled} {\em
Let $1\leq p\leq q\leq\infty$, and let $E$ be a locally convex space. Then $E$ is called {\em strictly $c_0$-quasibarrelled} (resp., {\em strictly $c_0$-$V_{(p,q)}$-quasibarrelled}) if every weakly null sequence in the strong dual $E'_\beta$ (which is a $V_{(p,q)}$-set, respectively) is equicontinuous.\qed}
\end{definition}
Clearly, strictly $c_0$-quasibarrelled spaces are exactly $\infty$-quasibarrelled spaces.

One can naturally ask whether the implications in Proposition \ref{p:Vp-barrelled} 
are not reversible. In Corollary \ref{c:Cp-Vp-barrelled} below we show that there are $V_{(p,q)}$-barrelled spaces which are not $c_0$-barrelled.
\begin{theorem} \label{t:Cp-V}
Let $1\leq p\leq q\leq\infty$, $X$ be a Tychonoff space, and let $E$ be a linear subspace of the product $\IF^X$ containing $C_p(X)$. Then every $(p,q)$-$(V)$ set in   $E'$  is finite-dimensional. Consequently, $E$ is $V_{(p,q)}$-barrelled.
\end{theorem}

\begin{proof}
First we prove the following claim.

{\em Claim 1. A sequence $\{f_n\}_{n\in\w}$ in $C_p(X)$ is weakly $p$-summable if and only if $\big(f_n(x)\big)\in\ell_p$ (or $\in c_0$ if $p=\infty$) for every $x\in X$.}
Indeed, if $\{f_n\}_{n\in\w}$ is  weakly $p$-summable, then $\big(\langle\delta_x, f_n\rangle\big)\in\ell_p$ (or $\in c_0$ if $p=\infty$)  for every $x\in X$, where $\delta_x$ denotes the Dirac measure at $x$ defined by $\langle\delta_x,f\rangle=f(x)$. This proves the necessity. The sufficiency follows from the fact that every $\chi\in C_p(X)'$ is a finite linear combination of Dirac measures $\delta_x$. The claim is proved.
\smallskip

By (vi) of Lemma \ref{l:V-set-1}, every  $(p,q)$-$(V)$ set is a $(V)$ set. Therefore it suffices to prove that all $(V)$ sets  are  finite-dimensional. So, let $B$ be a $(V)$ set in $E'$ (recall that $E'=L(X)$ algebraically). Then, for every weakly $1$-summable sequence $\{f_n\}$ in $E$ we have
\begin{equation} \label{equ:p-V-feral-1}
\lim_{n\to\infty} \sup_{\chi\in B} |\langle\chi, f_n\rangle|=0.
\end{equation}
To show that $B$ is finite-dimensional it suffices to prove that the support $\supp(B)$ of $B$ is finite.

Suppose for a contradiction that $\supp(B)$ is infinite. Then there is a sequence $\{\chi_n\}_{n\in\w}$ in $B$ such that $\chi_0\not=0$ and
\[
\supp(\chi_{n+1})\SM \bigcup_{i\leq n} \supp(\chi_{n}) \not=\emptyset.
\]
For every $n\in\w$, choose a point $x_n\in \supp(\chi_n)$ such that $x_{n+1} \not\in  \bigcup_{i\leq n} \supp(\chi_{n})$.
By Lemma 11.7.1 of \cite{Jar} and passing to a subsequence if needed, one can easily find a sequence $\{U_n\}_{n\in\w}$ of open sets in $X$ such that
\begin{equation} \label{equ:p-V-feral-2}
U_n\cap \supp(\chi_n)=\{x_n\}, \; U_{n+1}\cap  \bigcup_{i\leq n} \supp(\chi_{n})=\emptyset \; \mbox{ and } U_n\cap U_m=\emptyset
\end{equation}
for all distinct $n,m\in\w$. For every $n\in\w$, choose a  function $f_n\in C_p(X)\subseteq E$ such that
\begin{equation} \label{equ:p-V-feral-3}
f_n(X\SM U_n)=\{0\}\;\; \mbox{  and }\;\; f_n(x_n)=\tfrac{1}{|\chi_n(x_n)|}.
\end{equation}
Since $\{x:f_n(x)\cdot f_m(x)\not=0\}\subseteq U_n\cap U_m=\emptyset$  for all distinct $n,m\in\w$, it follows that the series $\sum_{n\in\w} |f_n(x)|$ has at most one nonzero summand and therefore, by Claim 1, the sequence $\{f_n\}$ is weakly $1$-summable in $E$. On the other hand, (\ref{equ:p-V-feral-2}) and (\ref{equ:p-V-feral-3}) imply
\[
\sup_{\chi\in B} |\langle \chi, f_n\rangle| \geq |\langle \chi_n, f_n\rangle| =\big| \chi_n(x_n)\cdot f_n(x_n)\big|=1 \not\to 0,
\]
which contradicts the choice of $B$. Thus $\supp(B)$ is finite.\qed
\end{proof}

\begin{corollary} \label{c:Cp-Vp-barrelled}
Let $1\leq p\leq q\leq\infty$, and let $\alpha$ be a countable ordinal. Then for every Tychonoff space $X$, the function spaces $C_p(X)$ and $B_\alpha(X)$ are $V_{(p,q)}$-barrelled. Consequently, if $X$ contains an infinite functionally bounded subset, then the space $C_p(X)$ is $V_{(p,q)}$-barrelled but not $c_0$-barrelled.
\end{corollary}

\begin{proof}
If $X$ contains an infinite functionally bounded subset, then, by the Buchwalter--Schmets theorem, the space $C_p(X)$ is not barrelled, and hence, by Proposition 2.6 in \cite{GK-DP}, it is not $c_0$-barrelled.\qed 
\end{proof}



Below we consider some categorical properties of $V_{(p,q)}$-(quasi)barrelled spaces.
\begin{proposition} \label{p:Vp-bar-property} 
Let $1\leq p\leq q\leq\infty$. Let us say that a locally convex space $E$ is $\AAA$-$($quasi$)$barrelled  if it is either a $V_{(p,q)}$-$($quasi$)$barrelled space, an $\aleph_0$-$V_{(p,q)}$-$($quasi$)$barrelled space, an $\ell_\infty$-$V_{(p,q)}$-$($quasi$)$barrelled space, or a $c_0$-$V_{(p,q)}$-$($quasi$)$barrelled space. 
\begin{enumerate}
\item[{\rm(i)}] Every Hausdorff quotient space of an $\AAA$-$($quasi$)$barrelled space is $\AAA$-$($quasi$)$barrelled. Consequently, complemented subspaces of $\AAA$-$($quasi$)$barrelled spaces are $\AAA$-$($quasi$)$barrelled.
\item[{\rm(ii)}] The Tychonoff product of a family of locally convex spaces is $\AAA$-$($quasi$)$barrelled if and only if each of its factor is an $\AAA$-$($quasi$)$barrelled space.
\item[{\rm(iii)}] A countable locally convex direct sum is $\AAA$-$($quasi$)$barrelled  if and only if all its summands are $\AAA$-$($quasi$)$barrelled spaces.
\item[{\rm(iv)}]  Let $H$ be a dense subspace of an lcs $E$. If $H$ is $\AAA$-$($quasi$)$barrelled then so  is $E$.
\end{enumerate}
\end{proposition}

\begin{proof}
We consider only the $V_{(p,q)}$-(quasi)barrelled case. Other cases can be considered analogously taking into account the next well known facts:
\begin{enumerate}
\item[(1)] continuous linear images of equicontinuous sets are equicontinuous,
\item[(2)] if $q:E\to E/H$ is a quotient map, then $q^\ast:(E/H)'_\beta \to E'_\beta$ is injective;
\item[(3)] if $T:E\to L$ is an operator, then $T^\ast$ is weak$^\ast$ and strongly continuous (Theorems 8.10.5 and 8.11.3 of \cite{NaB}).
\end{enumerate}

(i) Let $H$ be a closed subspace of a $V_{(p,q)}$-(quasi)barrelled space $E$, and let $q:E\to E/H$ be the quotient map. To show that $E/H$ is $V_{(p,q)}$-(quasi)barrelled, let $B\in \mathsf{V}_{(p,q)}\big((E/H)'\big)$ be a (resp., strongly bounded) set  in $(E/H)'_\beta$. By (iv) of Lemma \ref{l:V-set-1}, $q^\ast(B)$ is a $(p,q)$-$(V)$ set in $E'_\beta$. If in addition $B$ is bounded in $(E/H)'_\beta$, then the strong continuity of $q^\ast$ implies that  also $q^\ast(B)$ is bounded in $E'_\beta$.  As $E$ is $V_{(p,q)}$-(quasi)barrelled, $q^\ast(B)$  is equicontinuous and hence there is $U\in\Nn_0(E)$ such that $q^\ast(B)\subseteq U^\circ$. Since $q$ is a quotient map, the image $q(U)$ is a neighborhood of zero in $E/H$. Then for every $q(x)\in q(U)$ and each $b\in B$, we obtain $|\langle b,q(x)\rangle|=|\langle q^\ast(b),x\rangle|\leq 1$. Therefore $B$ is equicontinuous. Thus $E/H$ is a $V_{(p,q)}$-(quasi)barrelled space.

\smallskip

(ii) Let $E=\prod_{i\in I} E_i$. If $E$ is a $V_{(p,q)}$-(quasi)barrelled space then so are all factors $E_i$ by the clause (i). Conversely, assume that all spaces $E_i$ are $V_{(p,q)}$-(quasi)barrelled. Let $B\in \mathsf{V}_{(p,q)}(E')$.  Then, by Proposition \ref{p:product-sum-V-set}, $B$ has finite support $F\subseteq I$ and for every $i\in I$, the  projection $B_i$ of $B$ onto $E'_i$ is a  $(p,q)$-$(V)$ set. As $E_i$ is $V_{(p,q)}$-(quasi)barrelled, the set $B_i$ is equicontinuous. Since $B\subseteq \prod_{i\in F} B_i$ it follows that also $B$ is equicontinuous. Thus $E$ is a  $V_{(p,q)}$-(quasi)barrelled space.
\smallskip

(iii) Let $E=\bigoplus_{n\in \w} E_n$ be the direct sum a sequence $\{E_n\}_{n\in\w}$ of locally convex spaces. If $E$ is  $V_{(p,q)}$-(quasi)barrelled then, by (i), so are all its summands $E_n$. Conversely, assume that all spaces $E_n$ are  $V_{(p,q)}$-(quasi)barrelled.  Let $B\in \mathsf{V}_{(p,q)}(E')$.  Then, by Proposition \ref{p:product-sum-V-set},  for every $n\in \w$, the  projection $B_n$ of $B$ onto $E'_n$ is a  $(p,q)$-$(V)$ set.  As $E_n$ is $V_{(p,q)}$-(quasi)barrelled, the set $B_n$ is equicontinuous  and therefore there is $U_n\in \Nn_0(E_n)$ such that $B_n \subseteq \tfrac{1}{2^{n+1}} U_n^\circ$. Set $U:= E\cap \prod_{n\in\w} U_n$. Then $U$ is a neighborhood of zero in $E$ and $B\subseteq \prod_{n\in\w} B_n \subseteq U^\circ$. Thus $B$ is equicontinuous and hence $E$ is a $V_{(p,q)}$-(quasi)barrelled space.
\smallskip

(iv) Recall that $E'=H'$, and let $\Id_H: H\to E$ be the identity map. Let $B\in \mathsf{V}_{(p,q)}(E')$. Then, by (ix) of Lemma \ref{l:V-set-1}, $B=\Id_H^\ast(B)$ is a $(p,q)$-$(V)$ set in $H'$ (and it is  bounded in $H'_\beta$ if so is $B$ in $E'_\beta$). Since $H$ is  $V_{(p,q)}$-(quasi)barrelled, $B$ is equicontinuous. Take $U\in\Nn_0(H)$ such that $B\subseteq U^\circ$. Then $\overline{U}^{\,E}$ is a neighborhood of zero in $E$ such that $B\subseteq \big(\overline{U}^{\,E}\big)^\circ$. Hence $B$ is equicontinuous. Thus  $E$ is a $V_{(p,q)}$-(quasi)barrelled space.\qed
\end{proof}

\begin{corollary}  \label{c:Vp-bar-property}
In the notation of Proposition \ref{p:Vp-bar-property}, if $E=\SI E_n$ is the strict inductive limit of a sequence $\{E_n\}_{n\in\w}$ of $\AAA$-$($quasi$)$barrelled spaces, then also $E$ is $\AAA$-$($quasi$)$barrelled.
\end{corollary}

\begin{proof}
Recall that $E$ is a quotient space of $\bigoplus_{n\in\w} E_n$, and Proposition \ref{p:Vp-bar-property} applies.\qed
\end{proof}

\section{$V$ type properties for locally convex  spaces} \label{sec:property-Vp}


Below we introduce natural generalizations of the property $V_p$ for Banach spaces. Proposition \ref{p:V-q<p} shows that there is a sense to consider only the case $1\leq p\leq q\leq\infty$.
\begin{definition}\label{def:property-Vp}{\em
Let $1\leq p\leq q\leq\infty$. A locally convex space $E$ is said to have
\begin{enumerate}
\item[$\bullet$] a {\em property $V_{(p,q)}$} (a {\em property $EV_{(p,q)}$}) if every $(p,q)$-$(V)$ (resp., $(p,q)$-$(EV)$) set in $E'_\beta$ is relatively weakly compact;
\item[$\bullet$] a {\em property $sV_{(p,q)}$} (a {\em property $sEV_{(p,q)}$}) if every $(p,q)$-$(V)$ (resp., $(p,q)$-$(EV)$) set in $E'_\beta$ is relatively weakly sequentially compact;
\item[$\bullet$] a {\em weak property $sV_{(p,q)}$} or {\em property $wsV_{(p,q)}$}  (resp., {\em weak property $sEV_{(p,q)}$} or {\em property $wsEV_{(p,q)}$}) if each $(p,q)$-$(V)$-subset (resp.,  $(p,q)$-$(EV)$-subset) of $E'_\beta$ is weakly sequentially precompact.
\end{enumerate}
In the cases $(p,\infty)$ and  $(1,\infty)$ we shall say that $E$ has the {\em property} $V_p$, $EV_p$, $sV_p$, $sEV_p$,  $wsV_p$,  $wsEV_p$ or the {\em property} $V$, $EV$, $sV$, $sEV$,  $wsV$,  $wsEV$, respectively.\qed }
\end{definition}

The following lemma summarizes some relationships between $V_{(p,q)}$ type properties. 
\begin{lemma} \label{l:property-Vp-Vq}
Let $1\leq p\leq q\leq\infty$, and let $(E,\tau)$ be a locally convex space.
\begin{enumerate}
\item[{\rm(i)}] If $E$ has the property $V_{(p,q)}$ $($resp., $sV_{(p,q)}$$)$, then $E$  has the property $EV_{(p,q)}$ $($resp., $sEV_{(p,q)}$$)$. The converse is true if $E$ is $V_{(p,q)}$-barrelled.
\item[{\rm(ii)}]  If $p\leq p'\leq q'\leq q$ and $E$ has the property $V_{(p,q)}$ $($resp., $sV_{(p,q)}$, $EV_{(p,q)}$ or $sEV_{(p,q)}$$)$, then $E$  has the property $V_{(p',q')}$ $($resp., $sV_{(p',q')}$, $EV_{(p',q')}$ or $sEV_{(p',q')}$$)$.
\item[{\rm(iii)}]  If $E$ is $V_{(p,q)}$-barrelled and  $E'_\beta$ is weakly angelic (for example, $E'_\beta$ is a strict $(LF)$ space), then $E$ has the property $V_{(p,q)}$ if and only if it has the property $sV_{(p,q)}$ if and only if it  has the property $EV_{(p,q)}$ if and only if it  has the property $sEV_{(p,q)}$.
\item[{\rm(iv)}]  If $E$ has the property $sV_{(p,q)}$ $($resp., $sEV_{(p,q)}$$)$, then $E$ has  the property $wsV_{(p,q)}$ $($resp., $wsEV_{(p,q)}$$)$; the converse assertion is true if $E'_\beta$ is weakly sequentially complete.
\item[{\rm(v)}] If $\TTT$ is a locally convex vector topology on $E$ compatible with $\tau$, then the spaces $(E,\TTT)$ and $(E,\tau)$ have the property $V_{(p,q)}$ $($resp., $sV_{(p,q)}$$)$ simultaneously.
\end{enumerate}
\end{lemma}

\begin{proof}
(i) Since, by (i) of Lemma \ref{l:V-set-1}, every $(p,q)$-$(EV)$  set is also a $(p,q)$-$(V)$, it is clear that the property $V_{(p,q)}$ (resp., $sV_{(p,q)}$) is stronger than the the property $EV_{(p,q)}$ $($resp., $sEV_{(p,q)}$). Assume that $E$ is a $V_{(p,q)}$-barrelled space. Then, by Proposition \ref{p:Vp-barrelled}, $\mathsf{V}_{(p,q)}(E)=\mathsf{EV}_{(p,q)}(E)$. Thus $E$ has  the property $V_{(p,q)}$ (resp., $sV_{(p,q)}$) if and only if it has  the property $EV_{(p,q)}$ $($resp., $sEV_{(p,q)}$).

(ii) follows from (vi) of Lemma \ref{l:V-set-1} which states that every $(p',q')$-$(V)$ (resp., $(p',q')$-$(EV)$) set is also a $(p,q)$-$(V)$ (resp., $(p,q)$-$(EV)$) set.

(iii) follows from the equality $\mathsf{V}_{(p,q)}(E)=\mathsf{EV}_{(p,q)}(E)$ (since $E$ is $V_{(p,q)}$-barrelled) and the well known fact that in angelic spaces the property of being a relatively compact set is equivalent to the property of being a relatively sequentially compact set. It remains to note that any strict $(LF)$ space is weakly angelic by Lemma \ref{l:angelic-strict-LF}.

(iv) follows from the corresponding definitions and Lemma \ref{l:seq-p-comp}.

(v) follows from (vii) of Lemma  \ref{l:V-set-1} and the equality $(E,\TTT)'_\beta=E'_\beta$.\qed
\end{proof}


\begin{example}
Let $1<p<\infty$, and let $E=\ell_p$. Then every $(p^\ast,p^\ast)$-$(V)$ subset of $E'_\beta$ is relatively compact and hence $E$ has  the property $V_{(p^\ast,p^\ast)}$. Consequently, every $(p,p)$-$(V^\ast)$ subset of $E$ is relatively compact and hence $E$ has  the property $V_{(p,p)}^\ast$.
\end{example}

\begin{proof}
Let $B$ be a $(p^\ast,p^\ast)$-$(V)$ subset of $E'_\beta=\ell_{p^\ast}$. By Example \ref{exa:lp-in-lr}, the standard unit basis $\{e_n\}_{n\in\w}$ of $E$ is weakly $p^\ast$-summable. Therefore the series $\sum_{n\in\w} \big(\sup_{\chi\in B} |\langle\chi,e_n\rangle|\big)^{p^\ast}$ converges and hence
\[
\sup\Big\{ \sum_{n=m}^\infty |a_n|^{p^\ast}: \chi=(a_n)\in B\Big\} \leq  \sum_{n=m}^\infty \Big(\sup_{\chi=(a_n)\in B} |a_n|\Big)^{p^\ast}=\sum_{n=m}^\infty \big(\sup_{\chi\in B} |\langle\chi,e_n\rangle|\big)^{p^\ast}\to 0
\]
as $m\to\infty$. Thus, by Proposition \ref{p:compact-ell-p}, $B$ is relatively compact. The second assertion follows from the reflexivity of $E$.\qed
\end{proof}

By Corollary 2 of \cite{Pelcz-62}, the finite product of Banach spaces with the property $V$ has the property $V$. Below we generalize this result.
\begin{proposition} \label{p:product-sum-V}
Let  $1\leq p\leq q\leq\infty$, and let $\{E_i\}_{i\in I}$  be a non-empty family of locally convex spaces. Then:
\begin{enumerate}
\item[{\rm(i)}] $E=\prod_{i\in I} E_i$ has the property $V_{(p,q)}$ $($resp., $EV_{(p,q)}$, $sV_{(p,q)}$, $sEV_{(p,q)}$, $wsV_{(p,q)}$ or $wsEV_{(p,q)}$$)$ if and only if all factors $E_i$ have the same property;
\item[{\rm(ii)}] $E=\bigoplus_{i\in I} E_i$ has the property $V_{(p,q)}$  if and only if all $E_i$ have the property $V_{(p,q)}$;
\item[{\rm(iii)}] if $I$ is countable, then  $E=\bigoplus_{i\in I} E_i$ has the property $sV_{(p,q)}$  $($resp., $sEV_{(p,q)}$, $wsV_{(p,q)}$ or $wsEV_{(p,q)}$$)$ if and only if all summands $E_i$ have  the same property.
\end{enumerate}
\end{proposition}

\begin{proof}
To prove the necessity, let $E$ have the property $V_{(p,q)}$ (resp., the property  $EV_{(p,q)}$, $sV_{(p,q)}$, $sEV_{(p,q)}$, $wsV_{(p,q)}$ or $wsEV_{(p,q)}$). Fix $j\in I$, and let $K_j$ be a $(p,q)$-$(V)$ (resp., $(p,q)$-$(EV)$) set in $(E_j)'$. Since, by Proposition \ref{p:product-sum-strong}, $(E_j)'_\beta$ is a direct summand of  $E'_\beta$, to show that $K_j$ is relatively weakly  compact (resp., relatively weakly sequentially compact or weakly sequentially precompact) in $(E_j)'_\beta$ it suffices to show that the set $K:=K_j\times \prod_{i\in I\SM \{j\}} \{0_i\}$ is relatively  weakly  compact (resp., relatively weakly sequentially compact or  weakly sequentially precompact) in $E'_\beta$. In all cases (i)--(iii), by Proposition \ref{p:product-sum-V-set}, $K$ is a $(p,q)$-$(V)$  (resp., $(p,q)$-$(EV)$) set in $E'$. Since $E$ has the property $V_{(p,q)}$  (resp., the property $EV_{(p,q)}$, $sV_{(p,q)}$, $sEV_{(p,q)}$, $wsV_{(p,q)}$ or $wsEV_{(p,q)}$), it follows that $K$ is relatively  weakly  compact (resp., relatively weakly sequentially compact or  weakly sequentially precompact)  in $E'_\beta$, as desired.

To prove the sufficiency, assume that all spaces $E_i$ have the property $V_{(p,q)}$ (resp., the property  $EV_{(p,q)}$, $sV_{(p,q)}$, $sEV_{(p,q)}$, $wsV_{(p,q)}$ or $wsEV_{(p,q)}$), and let  $K$ be a $(p,q)$-$(V)$ (resp., $(p,q)$-$(EV)$) set in $E'_\beta$.
Taking into account Lemmas \ref{l:pr-rsc} and \ref{l:pr-rsc-2} and that the weak topology of a product is the product of weak topologies (see Theorem 8.8.5 of \cite{Jar}), to show that $K$ is relatively weakly  compact  (resp., relatively weakly sequentially compact or weakly sequentially precompact)  it suffices to prove that
\begin{itemize}
\item[(1)] for every $i\in I$, the projection $K_i$ of $K$ onto the $i$th coordinate is relatively weakly  compact (resp., weakly sequentially  compact or weakly sequentially precompact) in $(E_i)'_\beta$, and
\item[(2)] in the case (i) when $E$ is the direct product, $K_i=\{0\}$ for all but finitely many indices  $i\in I$.
\end{itemize}
But both conditions (1) and (2) are satisfied by Lemma \ref{l:pr-rsc-2} and Proposition \ref{p:product-sum-V-set}.\qed
\end{proof}





Generalizing Corollary 1 of \cite{Pelcz-62}, it is proved in Corollary 2.5 of \cite{LCCD} that a quotient space of a Banach space with the property $V_p$ has the property $V_p$, as well. The following proposition generalizes this result.
\begin{proposition} \label{p:prop-V-quotient}
Let $1\leq p\leq q\leq\infty$, and let $H$ be a closed subspace of a locally convex space $E$ such that the quotient map $q:E\to E/H$ is  almost bounded-covering. If $E$ has the property  $V_{(p,q)}$ $($resp., $EV_{(p,q)}$, $sV_{(p,q)}$, $sEV_{(p,q)}$, $wsV_{(p,q)}$ or $wsEV_{(p,q)}$$)$, then also its quotient space $E/H$ has the same property.
\end{proposition}

\begin{proof}
Recall that, by Theorem  8.12.1 of \cite{NaB},  the image $q^\ast\big[(E/H)'\big]$ of the adjoint map $q^\ast$ is the annihilator $H^\perp$ of $H$.
Hence $q^\ast\big[(E/H)'\big]$ is weakly closed in $E'_\beta$. Since $q$ is  almost bounded-covering, Proposition \ref{p:bounded-covering} implies that $q^\ast$ is an embedding of $(E/H)'_\beta$ onto the weakly closed subspace $H^\perp$ of $E'_\beta$. Now, let $B$ be a $(p,q)$-$(V)$ (resp., $(p,q)$-$(EV)$)  set in $(E/H)'_\beta$. Then, by (iv) of Lemma \ref{l:V-set-1}, $q^\ast(B)$ is a $(p,q)$-$(V)$ (resp., $(p,q)$-$(EV)$) set in $E'_\beta$. Therefore, by the  property  $V_{(p,q)}$ (resp., $EV_{(p,q)}$, $sV_{(p,q)}$, $sEV_{(p,q)}$, $wsV_{(p,q)}$ or $wsEV_{(p,q)}$)  of $E$, the set $q^\ast(B)$ is a relatively weakly compact (resp., relatively weakly sequentially compact or weakly sequentially precompact) subset of  $H^\perp$. As $q^\ast$ is a homeomorphism onto $H^\perp$ it follows that the set $B$ is relatively weakly compact (resp., relatively weakly sequentially compact or weakly sequentially precompact)  in $(E/H)'_\beta$. Thus $E/H$ has   the property  $V_{(p,q)}$ (resp., $EV_{(p,q)}$, $sV_{(p,q)}$, $sEV_{(p,q)}$, $wsV_{(p,q)}$ or $wsEV_{(p,q)}$).\qed
\end{proof}

\begin{corollary} \label{c:s-ind-Vp}
Let $1\leq p\leq q\leq\infty$, and let  $E=\SI E_n$ be the strict inductive limit of a sequence $\{ (E_n,\tau_n)\}_{n\in\w}$ of locally convex spaces such that $E_n$ is a closed proper subspace of $E_{n+1}$ for every $n\in\w$. If all spaces $E_n$ have the property  $V_{(p,q)}$ $($resp., $EV_{(p,q)}$, $sV_{(p,q)}$, $sEV_{(p,q)}$, $wsV_{(p,q)}$ or $wsEV_{(p,q)}$$)$, then also $E$ has the same property.
\end{corollary}

\begin{proof}
First we note that $E=H/L$ for some closed subspace $L$ of the direct locally convex sum $H=\bigoplus_{n\in\w} E_n$. By Proposition \ref{p:product-sum-V}, $H$ has the property  $V_{(p,q)}$ (resp., $EV_{(p,q)}$, $sV_{(p,q)}$, $sEV_{(p,q)}$, $wsV_{(p,q)}$ or $wsEV_{(p,q)}$).  By Theorem 4.5.5 of \cite{Jar}, 
any bounded set of $E$ is contained in some $E_n$ and is bounded in $E_n$. Therefore, the quotient map $q:H\to H/L=E$ is bounded-covering and hence, by Proposition \ref{p:prop-V-quotient}, $E$ has the property  $V_{(p,q)}$ (resp., $EV_{(p,q)}$, $sV_{(p,q)}$, $sEV_{(p,q)}$, $wsV_{(p,q)}$ or $wsEV_{(p,q)}$).\qed
\end{proof}

We shall use also the following notion. 
\begin{definition} \label{def:P-compatible-top-w} {\em
Let $\mathcal{P}$ be a property defined on every topological vector space. Two separated vector topologies $\tau$ and $\TTT$ on a vector space $E$ are called {\em weakly $\PPP$-compatible} if the spaces $(E,\tau)_w$ and $(E,\TTT)_w$ have the same sets with $\PPP$.\qed}
\end{definition}

\begin{proposition} \label{p:Vp-dense}
Let $1\leq p\leq q\leq\infty$, and let $H$ be  a dense subspace  of a locally convex space $E$.
\begin{enumerate}
\item[{\rm(i)}] Assume that $\mathcal{P}$ is the property of being  a relatively compact subset,  and let $\beta(E',H)$ and $\beta(E',E)$ be weakly $\mathcal{P}$-compatible. If $H$ has the property $V_{(p,q)}$ $($or $EV_{(p,q)}$$)$, then  $E$ has the property $V_{(p,q)}$  $($resp., $EV_{(p,q)}$$)$.
\item[{\rm(ii)}] Assume that $\mathcal{P}$ is the property of being  a relatively sequentially compact subset, and let $\beta(E',H)$ and $\beta(E',E)$ be weakly $\mathcal{P}$-compatible. If  $H$ has the property $sV_{(p,q)}$ $($or $sEV_{(p,q)}$$)$, then $E$ has the property $sV_{(p,q)}$  $($resp., $sEV_{(p,q)}$$)$.
\item[{\rm(ii)}] Assume that $\mathcal{P}$ is the property of being  a sequentially precompact subset, and let $\beta(E',H)$ and $\beta(E',E)$ be weakly $\mathcal{P}$-compatible. If  $H$ has the property $wsV_{(p,q)}$ $($or $wsEV_{(p,q)}$$)$, then $E$ has the property $wsV_{(p,q)}$  $($resp., $wsEV_{(p,q)}$$)$.
\end{enumerate}
\end{proposition}

\begin{proof}
(i) Let $B$ be a $(p,q)$-$(V)$ (or $(p,q)$-$(EV)$) set in $E'$. Then, by (ix) of Lemma \ref{l:V-set-1}, the set $B$ is also a $(p,q)$-$(V)$ (or $(p,q)$-$(EV)$) set in $H'=E'$. Since $H$ has the property $V_{(p,q)}$ (or $EV_{(p,q)}$), $B$ is relatively weakly compact in $H'_\beta$. As $\beta(E',H)$ and $\beta(E',E)$ are weakly $\mathcal{P}$-compatible, it follows that $B$  is relatively weakly compact in $E'_\beta$ and hence $E$ has  the property $V_{(p,q)}$  (or $EV_{(p,q)}$).

(ii) and (iii) can be proved analogously.\qed
\end{proof}

As an immediate consequence of Propositions \ref{p:Vp-dense} and \ref{p:large-charac} we obtain the following assertion.
\begin{corollary} \label{c:Vp-large-subspace}
Let $1\leq p\leq q\leq\infty$, and let $H$ be  a large subspace  of a locally convex space $E$. If  $H$ has the property $V_{(p,q)}$ $($resp., $EV_{(p,q)}$, $sV_{(p,q)}$, $sEV_{(p,q)}$, $wsV_{(p,q)}$ or $wsEV_{(p,q)}$$)$, then also $E$ has  the same property.
\end{corollary}


To show that closed subspaces of spaces with the property $V_{(p,q)}$ may not have the property $V_{(p,q)}$, we notice the following assertion.
\begin{theorem} \label{t:Cp-V-p}
Let $1\leq p\leq q\leq\infty$, and let $\alpha$ be a countable ordinal. Then for every Tychonoff space $X$, the function spaces $C_p(X)$ and $B_\alpha(X)$ have the properties $V_{(p,q)}$, $EV_{(p,q)}$, $sV_{(p,q)}$, $sEV_{(p,q)}$, $wsV_{(p,q)}$ and $wsEV_{(p,q)}$.
\end{theorem}

\begin{proof}
By Theorem \ref{t:Cp-V}, every $(p,q)$-$(V)$ set (and hence also each $(p,q)$-$(EV)$ set) in $E'$ is finite-dimensional and the assertion follows.\qed
\end{proof}

\begin{corollary} \label{c:V-p-closed}
Let $1\leq p\leq q\leq\infty$. For every locally convex space $E$ there is a bounded-finite, zero-dimensional, paracompact Tychonoff space $X$ such that $E_w$ is topologically isomorphic to a closed subspace of the barrelled space $C_p(X)$ with the properties $V_{(p,q)}$, $EV_{(p,q)}$, $sV_{(p,q)}$, $sEV_{(p,q)}$, $wsV_{(p,q)}$ and $wsEV_{(p,q)}$.
\end{corollary}

\begin{proof}
The existence of an embedding formulated in the corollary was proved in Theorem 4.7 of \cite{BG-sGP}. Now Theorem \ref{t:Cp-V-p} applies.\qed
\end{proof}

\begin{remark} \label{rem:V-p-closed} {\em
Corollary \ref{c:V-p-closed} shows that the properties $V_{(p,q)}$, $EV_{(p,q)}$, $sV_{(p,q)}$, and $sEV_{(p,q)}$ are not (closely) hereditary. Indeed, take for example a Banach space $E$ without the property $V$. Then, by (ii) and (v) of Lemma \ref{l:property-Vp-Vq}, $E_w$ also has no any of the properties $V_{(p,q)}$, $EV_{(p,q)}$, $sV_{(p,q)}$, and $sEV_{(p,q)}$. It remains to note that, by Corollary \ref{c:V-p-closed}, $E_w$ is topologically isomorphic to a closed subspace of the barrelled space $C_p(X)$ with all properties $V_{(p,q)}$, $EV_{(p,q)}$, $sV_{(p,q)}$, and $sEV_{(p,q)}$. \qed}
\end{remark}

\begin{remark} \label{rem:RDP-sV} {\em
Following Grothendieck \cite{Grothen}, a Banach space $X$ has the {\em reciprocal Dunford--Pettis property} ($RDP$ property) if any completely continuous (=Dunford--Pettis) operator $T:X\to Y$, were $Y$ is an arbitrary Banach space, is weakly compact. It was shown by Leavelle \cite{Leavelle} that a Banach space $X$ has $RDP$ property if and only if every $\infty$-$(V)$ subset of $E'_\beta$ is relatively weakly compact, i.e., if and only if $X$ has the property $V_\infty$. \qed 
}
\end{remark}

More generally, we have the following.
\begin{remark} \label{rem:RDPp-sVp} {\em
Let $p\in[1,\infty]$, $X$ be a Banach space, and let $X^\ast=X'_\beta$ be the Banach dual of $X$. Following Ghenciu \cite{Ghenciu-pGP}, a subset $B$ of $X^\ast$ is called a {\em weakly-$p$-$L$-set} if $\big(\sup_{\chi\in B} |\langle\chi,x_n\rangle|\big)\in c_0$ for every weakly $p$-summable sequence $\{x_n\}_{n\in\w}$ in $X$. Therefore, by definition, weakly-$p$-$L$-sets in $X^\ast$ are exactly $p$-$(V)$ subsets of $X^\ast$. Further, following \cite{Ghenciu-pGP}, the Banach space $X$ has the {\em reciprocal Dunford--Pettis property of order $p$} or $RDP_p$ (resp., the  {\em weak reciprocal Dunford--Pettis property of order $p$} or $wRDP_p$) if each  weakly-$p$-$L$-subset of $X^\ast$ is relatively weakly sequentially compact (resp., weakly sequentially precompact). Therefore, by definition, for $1\leq p\leq\infty$, we obtain that a Banach space $X$ has
\begin{enumerate}
\item[$\bullet$] the $RDP_p$ if and only if it has the property $sV_p$;
\item[$\bullet$] the $wRDP_p$ if and only if it has the property $wsV_p$.\qed
\end{enumerate}
}
\end{remark}


\section{$V^\ast$ type properties for locally convex spaces} \label{sec:V*-property}


Below we extend and generalize the property $V^\ast_p$ for Banach spaces to locally convex spaces. Recall that, by Proposition \ref{p:V*-q<p},  $\mathsf{V}_{(p,q)}^\ast(E)=\mathsf{EV}_{(p,q)}^\ast(E)=\{0\}$ for all $1\leq q<p\leq\infty$. Therefore we consider only the case $1\leq p\leq q\leq\infty$.

\begin{definition}\label{def:property-Vp*}{\em
Let $1\leq p\leq q\leq\infty$. A locally convex space $E$ is said to have
\begin{enumerate}
\item[$\bullet$] a {\em property $V^\ast_{(p,q)}$} (a {\em property $EV^\ast_{(p,q)}$}) if every $(p,q)$-$(V^\ast)$ (resp., $(p,q)$-$(EV^\ast)$) set in $E$ is relatively weakly compact;
\item[$\bullet$] a {\em property $sV^\ast_{(p,q)}$} (a {\em property $sEV^\ast_{(p,q)}$}) if every $(p,q)$-$(V^\ast)$ (resp., $(p,q)$-$(EV^\ast)$) set in $E$ is relatively weakly sequentially compact;
\item[$\bullet$] a {\em weak property $sV_{(p,q)}^\ast$} or a {\em property $wsV_{(p,q)}^\ast$}  (resp., a {\em weak property $sEV_{(p,q)}^\ast$} or a {\em property $wsEV_{(p,q)}^\ast$}) if each $(p,q)$-$(V^\ast)$-subset (resp.,  $(p,q)$-$(EV^\ast)$-subset) of $E$ is weakly sequentially precompact.
\end{enumerate}
In the case when $q=\infty$ we shall omit the subscript $q$ and say that $E$ has the  {\em property $V^\ast_{p}$} etc., and in the case when $q=\infty$ and $p=1$ we shall say that $E$ has the  {\em property $V^\ast$} etc.\qed  }
\end{definition}
Note that for Banach spaces the the property $wsV^\ast$ coincides with the property {\em weak $(V^\ast)$} introduced by Saab and Saab in \cite[p.~529]{Saab-Saab}.


Some relationships between the introduced notions are given in the following lemma. 
\begin{lemma} \label{l:property-Vp*-Vq*}
Let $1\leq p\leq q\leq\infty$, and let $(E,\tau)$ be a locally convex space.
\begin{enumerate}
\item[{\rm(i)}] If $E$ has the property $EV_{(p,q)}^\ast$ $($resp., $sEV_{(p,q)}^\ast$ or $wsEV_{(p,q)}^\ast$$)$, then $E$  has the property $V_{(p,q)}^\ast$ $($resp., $sV_{(p,q)}^\ast$ or $wsV_{(p,q)}^\ast$$)$; the converse is true if $E$ is $p$-quasibarrelled.
\item[{\rm(ii)}]  If $p\leq p'\leq q'\leq q$ and $E$ has the property $V_{(p,q)}^\ast$ $($resp., $EV_{(p,q)}^\ast$, $sV_{(p,q)}^\ast$, $sEV_{(p,q)}^\ast$, $wsV_{(p,q)}^\ast$ or $wsEV_{(p,q)}^\ast$$)$, then $E$  has the property $V_{(p',q')}^\ast$ $($resp., $EV_{(p',q')}^\ast$, $sV_{(p',q')}^\ast$, $sEV_{(p',q')}^\ast$ or $wsV_{(p',q')}^\ast$$)$.
\item[{\rm(iii)}]  If $E$ is weakly angelic {\rm(}for example, $E$ is a strict $(LF)$ space{\rm)}, then $E$ has the property $V_{(p,q)}^\ast$ $($resp., $EV_{(p,q)}^\ast$$)$ if and only if it has the property $sV_{(p,q)}^\ast$ $($resp., $sEV_{(p,q)}^\ast$$)$.
\item[{\rm(iv)}] If $\TTT$ is a locally convex vector topology on $E$ compatible with $\tau$, then $E$ has the property  $V_{(p,q)}^\ast$ $($resp., $sV_{(p,q)}^\ast$ or $wsV_{(p,q)}^\ast$$)$ if and only if $(E,\TTT)$ has the same property.
\item[{\rm(v)}] If $E$ has the property $sV_{(p,q)}^\ast$, then $E$ has the property $wsV_{(p,q)}^\ast$; the converse is true if $E$ is weakly sequentially complete.
\end{enumerate}
\end{lemma}

\begin{proof}
(i), (ii) and (iv) follow from (i), (vi) and (vii) of Lemma \ref{l:V*-set-1}, respectively. The clause (iii) follows from the fact that in angelic spaces the classes of (relatively)  sequentially compact sets and (relatively) compact sets coincide and Lemma \ref{l:angelic-strict-LF}. Finally, the clause (v) follows from the corresponding definitions and Lemma \ref{l:seq-p-comp}.\qed
\end{proof}

\begin{proposition} \label{p:product-sum-V*}
Let  $1\leq p\leq q\leq\infty$, and let $\{E_i\}_{i\in I}$  be a non-empty family of locally convex spaces. Then:
\begin{enumerate}
\item[{\rm(i)}] $E=\prod_{i\in I} E_i$ has the property $V_{(p,q)}^\ast$ $($resp., $EV_{(p,q)}^\ast$$)$ if and only if all spaces $E_i$ have the same property;
\item[{\rm(ii)}] $E=\bigoplus_{i\in I} E_i$ has the property $V_{(p,q)}^\ast$  $($resp., $EV_{(p,q)}^\ast$, $sV_{(p,q)}^\ast$, $sEV_{(p,q)}^\ast$, $wsV_{(p,q)}^\ast$ or $wsEV_{(p,q)}^\ast$$)$  if and only if all spaces  $E_i$ have  the same property;
\item[{\rm(iii)}] if $I=\w$ is countable, then $E=\prod_{i\in \w} E_i$ has the property $sV_{(p,q)}^\ast$ $($resp., $sEV_{(p,q)}^\ast$, $wsV_{(p,q)}^\ast$ or $wsEV_{(p,q)}^\ast$$)$ if and only if all spaces $E_i$ have the same property.
\end{enumerate}
\end{proposition}

\begin{proof}
To prove the necessity, let $E$ have the property $V_{(p,q)}^\ast$ (resp., the property  $EV_{(p,q)}^\ast$, $sV_{(p,q)}^\ast$, $sEV_{(p,q)}^\ast$, $wsV_{(p,q)}^\ast$ or $wsEV_{(p,q)}^\ast$). Fix $j\in I$, and let $A_j$ be a $(p,q)$-$(V^\ast)$ (resp., $(p,q)$-$(EV^\ast)$) set in $E_j$. Since $E_j$ is a direct summand of  $E$, to show that $A_j$ is relatively weakly  compact (resp., relatively weakly  sequentially compact or weakly sequentially precompact) in $E_j$ it suffices to show that the set $A:=A_j\times \prod_{i\in I\SM \{j\}} \{0_i\}$ is relatively  weakly  compact (resp. relatively weakly  sequentially compact or  weakly sequentially precompact)  in $E$. In all cases (i)--(iii), by (iv) of Lemma \ref{l:V*-set-1}, $A$ is a $(p,q)$-$(V^\ast)$  (resp., $(p,q)$-$(EV^\ast)$) set in $E$ as the image of $A_j$ under the canonical embedding of $E_j$ into $E$. Since $E$ has the property $V_{(p,q)}^\ast$ (resp., the property  $EV_{(p,q)}^\ast$, $sV_{(p,q)}^\ast$, $sEV_{(p,q)}^\ast$, $wsV_{(p,q)}^\ast$ or $wsEV_{(p,q)}^\ast$) and $E_j$ is weakly closed in $E$, it follows that $A$ is relatively  weakly compact (resp., relatively weakly  sequentially compact or  weakly sequentially precompact) in $E$, as desired.

To prove the sufficiency, assume that all spaces $E_i$ have the property $V_{(p,q)}^\ast$ (resp., the property  $EV_{(p,q)}^\ast$, $sV_{(p,q)}^\ast$, $sEV_{(p,q)}^\ast$, $wsV_{(p,q)}^\ast$ or $wsEV_{(p,q)}^\ast$), and let  $A$ be a $(p,q)$-$(V^\ast)$ (resp., $(p,q)$-$(EV^\ast)$) set in $E$.
Taking into account
\begin{enumerate}
\item Lemmas \ref{l:pr-rsc} and \ref{l:pr-rsc-2},
\item for the cases (i) and (iii), the fact that the weak topology of a product is the product of weak topologies (see Theorem 8.8.5 of \cite{Jar}),
\item for the case (ii), the fact that the support of $A$ is finite since $A$ is bounded,
\end{enumerate}
to show that $A$ is relatively weakly  compact (resp., relatively weakly  sequentially compact or  weakly sequentially precompact) it suffices to prove that for every $i\in I$, the projection $A_i$ of $A$ onto the $i$th coordinate is relatively weakly  compact (resp., relatively weakly  sequentially compact or  weakly sequentially precompact) in $E_i$. But this condition is satisfied because,  by (iv) of Lemma \ref{l:V*-set-1},  the projection $A_i$ is a $(p,q)$-$(V^\ast)$ (resp., $(p,q)$-$(EV^\ast)$) set in $E_i$ and hence it is relatively weakly  compact (resp., relatively weakly  sequentially compact or  weakly sequentially precompact)  in $E_i$ by the corresponding property of $E_i$.\qed
\end{proof}

\begin{proposition} \label{p:semirefl-Vp*}
Let $1\leq p\leq q\leq\infty$, and let $E$ be a locally convex space.
\begin{enumerate}
\item[{\rm(i)}] If $E$ is semi-reflexive, then $E$ has the property $V_{(p,q)}^\ast$ and the property $EV_{(p,q)}^\ast$. If in addition $E$ is weakly angelic, then $E$ has also the properties $sV_{(p,q)}^\ast$ and $sEV_{(p,q)}^\ast$.
\item[{\rm(ii)}] If $E$ has the property $V_p^\ast$ and $E'_\beta$ has the $p$-Schur property, then $E$ is semi-reflexive. If in addition $E$ is quasibarrelled, then $E$ is a reflexive space.
\item[{\rm(iii)}] If $E$ is reflexive, then it is a $($quasi$)$barrelled space with the properties $V_{(p,q)}^\ast$ and $V_{(p,q)}$.
\end{enumerate}
\end{proposition}

\begin{proof}
(i) Since $(p,q)$-$(V^\ast)$ sets and $(p,q)$-$(EV^\ast)$ sets are bounded (Lemma \ref{l:V*-set-1}), the assertion follows from Theorem 15.2.4 of \cite{NaB} which states that an lcs is semi-reflexive if and only if every its closed bounded subset is weakly compact. If additionally $E$ is weakly angelic, the space $E$ has also the properties $sV_{(p,q)}^\ast$ and $sEV_{(p,q)}^\ast$ by (iii) of Lemma \ref{l:property-Vp*-Vq*}.

(ii) By Theorem \ref{t:Bo=Vp}, the $p$-Schur property of $E'_\beta$ implies the equality $\Bo(E)=\mathsf{V}_p^\ast(E)$. Therefore, by the property $V_p^\ast$ of $E$, every bounded subset of $E$ is   relatively weakly compact. Thus, by Theorem 15.2.4 of \cite{NaB},  $E$ is semi-reflexive. The last assertion follows from the fact that a semi-reflexive quasibarrelled space is reflexive, see Proposition 11.4.2 of \cite{Jar}.

(iii) Assume that $E$ is reflexive. Then, by Proposition 11.4.2 of \cite{Jar}, $E$ is a (quasi)barrelled space whose closed disks are weakly compact. Taking into consideration that the class $\mathsf{V}_{(p,q)}^\ast(E)$ is saturated (see (iii) of Lemma \ref{l:V*-set-1}) it follows that $E$ has the property  $V_{(p,q)}^\ast$.

Since, by Proposition 11.4.5 of \cite{Jar}, the space $E'_\beta$ is also reflexive it follows that any closed disk in $E'_\beta$ is weakly compact. Taking also into account that any $(p,q)$-$(V)$ set is strongly bounded (because $E$ is semi-reflexive), it follows that each $(p,q)$-$(V)$ set in $E'_\beta$ is relatively weakly compact. Thus $E$ has  the property $V_{(p,q)}$.\qed
\end{proof}

\begin{remark} \label{rem:p-Schur-not-nec} {\em
In (ii) of Proposition \ref{p:semirefl-Vp*} the condition on $E'_\beta$ to have the $p$-Schur property is not necessary. Indeed, consider the reflexive Banach space $E=\ell_p$, $1<p<\infty$. Then, by (iii) of Proposition \ref{p:semirefl-Vp*}, the space $E$ has the property $V_p^\ast$. However, by Proposition \ref{p:Lp-Schur}, the space $E'_\beta=\ell_{p^\ast}$ does not have the $p$-Schur property.\qed}
\end{remark}

The clauses (i) and (ii) of Proposition \ref{p:semirefl-Vp*} immediately imply
\begin{corollary} \label{c:semi-ref-p-Schur}
Let  $p\in[1,\infty]$, and let $E$ be a locally convex space such that $E'_\beta$ has the $p$-Schur property. Then $E$ is semi-reflexive if and only if it has the property $V_p^\ast$.
\end{corollary}

For a subset $A$ of a topological space $X$, we denote by $[A]^{s}$ the {\em sequential closure} of $A$, that is
\[
[A]^{s}=\{ x\in X: \mbox{ there is } \{a_n\}_{n\in\w} \subseteq A \mbox{ such that } a_n \to x\}.
 \]
Clearly, $A\subseteq [A]^{s} \subseteq \cl(A)$ and $A$ is called {\em sequentially closed} if $A=[A]^{s}$. Evidently, every closed subset is sequentially closed. If the converse is true, i.e., if every sequentially closed subset of $X$ is closed, then the space $X$ is called a {\em sequential space}.

Generalizing Proposition 5 of \cite{Pelcz-62}, it is shown in Corollary 2.15 of \cite{LCCD} that every closed subspace of a Banach space with the property $V_p^\ast$ has the property $V_p^\ast$. The clause (i) of the following proposition generalizes this result.

\begin{proposition} \label{p:subspace-Vp*}
Let $1\leq p\leq q\leq\infty$, and let $L$ be a subspace of a locally convex space $E$.
\begin{enumerate}
\item[{\rm(i)}] If $L$ is closed in $E$ and $E$ has  the property $V_{(p,q)}^\ast$ $($resp., $EV_{(p,q)}^\ast$$)$, then also $L$ has the same property.
\item[{\rm(ii)}] If $L_w$ is sequentially closed in $E_w$ and $E$ has  the property $sV_{(p,q)}^\ast$ $($resp., $sEV_{(p,q)}^\ast$$)$, then also $L$ has the same property.
\item[{\rm(iii)}]  If $E$ has  the property $wsV_{(p,q)}^\ast$ $($resp., $wsEV_{(p,q)}^\ast$$)$,  then also $L$ has the same property.
\end{enumerate}
\end{proposition}

\begin{proof}
Let $A$ be a $(p,q)$-$(V^\ast)$  (resp., $(p,q)$-$(EV^\ast)$) set in $L$. Then, by  (viii) of Lemma \ref{l:V*-set-1}, $A$ is  a $(p,q)$-$(V^\ast)$  (resp., $(p,q)$-$(EV^\ast)$) set also in $E$.

(i) By the property $V_{(p,q)}^\ast$ (resp., $EV_{(p,q)}^\ast$) of $E$, the set $A$ is relatively weakly  compact in $E$. Since, by Corollary 8.7.3 of \cite{Jar}, $L_w$ is a closed subspace of $E_w$ it follows that $A$ is relatively weakly  compact in $L$. Thus $L$ has the property $V_{(p,q)}^\ast$ (resp., $EV_{(p,q)}^\ast$).

(ii) By the property $sV_p^\ast$ (resp., $sEV_{(p,q)}^\ast$) of $E$, the set  $A$ is relatively weakly  sequentially compact in $E$.  Since $L_w$ is a sequentially closed subspace of $E_w$ it follows that $A$ is relatively weakly sequentially compact in $L$. Thus $L$ has  the property $sV_{(p,q)}^\ast$ (resp., $sEV_p^\ast$).

(iii)  By the property $wsV_{(p,q)}^\ast$ (resp., $wsEV_{(p,q)}^\ast$) of $E$, the set  $A$ is weakly  sequentially precompact in $E$. Let $\{a_n\}_{n\in\w}$ be a sequence in $A$. Take a subsequence $\{a_{n_k}\}_{k\in\w}$ of $\{a_n\}_{n\in\w}$ which is Cauchy in $E_w$. Since $L_w$ is a subspace of $E_w$ it follows that $\{a_{n_k}\}_{k\in\w}$  is Cauchy also in $L_w$. Therefore $A$ is weakly sequentially precompact in $L$. Thus $L$ has  the property $wsV_{(p,q)}^\ast$ (resp., $wsEV_p^\ast$).\qed
\end{proof}

\begin{corollary} \label{c:s-ind-Vp*}
Let  $p\in[1,\infty]$, and let  $E=\SI E_n$ be the strict inductive limit of a sequence $\{ (E_n,\tau_n)\}_{n\in\w}$ of locally convex spaces such that $E_n$ is a closed proper subspace of $E_{n+1}$ for every $n\in\w$. If the space  $E=\SI E_n$ has  the property $V_{(p,q)}^\ast$ $($resp., $EV_{(p,q)}^\ast$, $sV_{(p,q)}^\ast$, $sEV_{(p,q)}^\ast$, $wsV_{(p,q)}^\ast$ or $wsEV_{(p,q)}^\ast$$)$, then all spaces $E_n$ have the same property.
\end{corollary}

%
We do not know whether the converse assertion in Corollary \ref{c:s-ind-Vp*} holds true.

\begin{corollary} \label{c:emb-c0-V*}
If a locally convex space  $E$  contains a closed subspace which is topologically isomorphic to $c_0$, then $E$ does not have the property $V_p^\ast$ for every $p\in[1,\infty]$.
\end{corollary}

\begin{proof}
Since $(c_0)'_\beta=\ell_1$ has the Schur property, Corollary \ref{c:semi-ref-p-Schur} implies that the Banach space $c_0$ does not have the property $V_p^\ast$ for every $p\in[1,\infty]$. Now (i) of Proposition \ref{p:subspace-Vp*} applies.\qed
\end{proof}



\begin{theorem} \label{t:Cp-V*-prop}
Let $1\leq p\leq q\leq\infty$, $X$ be a Tychonoff space, and let $E$ be a vector subspace of the product $\IF^X$ containing $C_p(X)$. Then $E$ has the property $EV^\ast_{(p,q)}$ if and only if it has the property $V^\ast_{(p,q)}$ if and only if  $E$ is quasi-complete if and only if $E$ is reflexive.
\end{theorem}

\begin{proof}
First we note that, by Corollary \ref{c:Cp-V*-sets}, we have the following $\mathsf{V}_{(p,q)}^\ast(E)=\mathsf{EV}_{(p,q)}^\ast(E)=\Bo(E)$. Now it is clear that $E$ has the property $EV^\ast_{(p,q)}$ if and only if it has the property $V^\ast_{(p,q)}$.

Assume that $E$ has the property $V^\ast_{(p,q)}$.  As we noticed above, $\mathsf{V}_{(p,q)}^\ast(E)=\Bo(E)$. Then the property $V^\ast_{(p,q)}$ of $E$ implies that the closure of any bounded subset $A$ of $E$ is weakly compact and hence compact (recall that $E$ carries its weak topology being a subspace of $\IF^X$). 
Thus $E$ is quasi-complete.

Assume that $E$ is quasi-complete. Then the closure $\overline{A}$ of every bounded subset $A$ of $E$ is complete. Since $E$ carries its weak topology, $\overline{A}$ is precompact. Being complete and precompact  $\overline{A}$  is compact, see Theorem 3.4.2 of \cite{NaB}. Thus $E$ has the property $V^\ast_{(p,q)}$.

Since $C_p(X)$ is quasibarrelled, the space $E$ is also quasibarrelled. Therefore, by Proposition 11.4.2 of \cite{Jar}, $E$ is quasi-complete if and only if it is reflexive.\qed
\end{proof}

\begin{corollary} \label{c:Cp-V*-prop}
Let $1\leq p\leq q\leq\infty$, and let $X$ be a Tychonoff space. Then $C_p(X)$ has the property $EV^\ast_{(p,q)}$ if and only if it has the property $V^\ast_{(p,q)}$ if and only if  $X$ is discrete.
\end{corollary}

\begin{proof}
By Theorem \ref{t:Cp-vNc},  
$C_p(X)$ is quasi-complete if and only if $X$ is discrete. Now Theorem \ref{t:Cp-V*-prop} applies.\qed
\end{proof}


\begin{definition} \label{def:complete-g} {\em
Let $\A$ be a saturated bornology of bounded subsets in a locally convex space $E$. The space $E$ is called
\begin{enumerate}
\item[$\bullet$] {\em $\A$-quasi-complete} if every closed set in $\A$ is complete;
\item[$\bullet$] {\em separably $\A$-quasi-complete} if every closed separable set in $\A$ is complete;
\item[$\bullet$] {\em $\A$-von Neumann complete} if every closed precompact set in $\A$ is compact;
\item[$\bullet$] {\em separably  $\A$-von Neumann complete} if every closed separable precompact set in $\A$ is compact;
\item[$\bullet$] {\em $\A$-sequentially complete} if every  sequence in $\A$ which is Cauchy converges;
\item[$\bullet$] {\em $\A$-locally complete} if the closed absolutely convex hull of a null sequence belonging to $\A$ is complete;
\item[$\bullet$] {\em  $\A$-convex compactness property} ({\em$\A$-ccp} for short) if  the closed absolutely convex hull of a compact subset from $\A$ is compact. \qed
\end{enumerate} }
\end{definition}
In the case when $\A=\Bo(E)$ we shall omit $\Bo(E)$ and obtain the classical completeness type properties considered in Section \ref{sec:completeness}.  


The next proposition generalizes Proposition 4 of \cite{Pelcz-62} and Corollary 2.16 of \cite{LCCD}.

\begin{proposition} \label{p:Vp*-duality}
Let $1\leq p\leq q\leq\infty$, and let $E$ be a quasibarrelled locally convex space. Then:
\begin{enumerate}
\item[{\rm(i)}] if $E$ has the property $V_{(p,q)}$ $($resp., $sV_{(p,q)}$$)$, then $E'_\beta$ has the property  $V_{(p,q)}^\ast$  $($resp., $sV_{(p,q)}^\ast$$)$;
\item[{\rm(ii)}] if $E$ is weakly $V_{(p,q)}^\ast$-quasi-complete and $E'_\beta$ is quasibarrelled with the property  $V_{(p,q)}$  $($resp., $sV_{(p,q)}$$)$, then $E$ has the property $V_{(p,q)}^\ast$  $($resp., $sV_{(p,q)}^\ast$$)$.
\end{enumerate}
\end{proposition}

\begin{proof}
(i) Let $B$ be a $(p,q)$-$(V^\ast)$ set in $E'_\beta$. Then, by (x) of Lemma \ref{l:V-set-1}, $B$ is a $(p,q)$-$(V)$ set. Hence, by the property $V_{(p,q)}$ (resp., $sV_{(p,q)}$) of $E$, $B$ is relatively weakly  compact (resp., relatively weakly sequentially compact)  in $E'_\beta$. Thus $E'_\beta$ has the property  $V_{(p,q)}^\ast$  (resp., $sV_{(p,q)}^\ast$).

(ii)  Let $A$ be a $(p,q)$-$(V^\ast)$ in $E$. By Lemma \ref{l:V*-set-1}, without loss of generality we can assume that $A$ is  closed and  absolutely convex. Since $E$ is quasibarrelled, the canonical map $J_E:E\to E''=(E'_\beta)'_\beta$ is an embedding, and hence, by Lemma \ref{l:V*-set-1}, the set $J_E(A)$ is a $(p,q)$-$(V^\ast)$ set in $E''$. By (i), the space $E''$ has the property   $V_{(p,q)}^\ast$  (resp., $sV_{(p,q)}^\ast$). Therefore the set  $J_E(A)$ is relatively weakly compact  (resp., relatively weakly sequentially compact)   in $E''$. Since $E$ is weakly $V_{(p,q)}^\ast$-quasi-complete and $A$ is closed and absolutely convex, we obtain that $A$ is complete in $E_w$ and hence $J_E(A)$ is complete in $(E'')_w$. So $J_E(A)$ is weakly compact  (resp., weakly sequentially compact)  in $E''$. Since $J_E$ is an embedding, the set $A$ is weakly compact  (resp., relatively weakly sequentially compact)  in $E$. Thus $E$ has the property $V_{(p,q)}^\ast$  (resp., $sV_{(p,q)}^\ast$).\qed
\end{proof}
The converse assertions to both items (i) and (ii) in Proposition \ref{p:Vp*-duality} are not true in general even for Banach spaces, see \cite{Saab-Saab-86}.

The following assertion generalizes Proposition 6 of \cite{Pelcz-62} which states that a Banach space with the property $V^\ast$ is weakly sequentially complete.
\begin{proposition} \label{p:sV*-wsc}
Let $E$ be a barrelled locally convex space. If $E$ has the property $sV^\ast$, then $E$ is weakly sequentially complete.
\end{proposition}

\begin{proof}
Fix an arbitrary weakly Cauchy sequence $S=\{x_n\}_{n\in\w}$ in $E$. We show that it weakly converges. To this end, fix a weakly $1$-summable sequence $\{\chi_n\}_{n\in\w}$ in $E'_\beta$. For every $t=(t_n)\in\ell_\infty$, consider the map $\eta_t:E\to \IF$ defined by $\langle \eta_t,x\rangle:=\sum_{n\in\w} t_n\cdot \langle\chi_n,x\rangle$. By Corollary \ref{c:1-sum}, $\eta_t\in E'$. Since $S$ is weakly Cauchy, it follows that for every $t=(t_n)\in\ell_\infty$ there is the limit
\begin{equation} \label{equ:sV*-wsc-1}
\lim_{i\to\infty} \sum_{n\in\w} t_n\cdot \langle\chi_n,x_i\rangle =\lim_{i\to\infty} \langle \eta_t,x_i\rangle.
\end{equation}
As $\{\chi_n\}_{n\in\w}$ is weakly $1$-summable, the series $\sum_{n\in\w} |\langle\chi_n,x\rangle|$ converges for every $x\in E$, that is $\big(\langle\chi_n,x\rangle\big)_{n\in\w}\in\ell_1$. Hence (\ref{equ:sV*-wsc-1}) means that the sequence $\big\{\big(\langle\chi_n,x_i\rangle\big)_{n\in\w}\big\}_{i\in\w}$ is weakly Cauchy in $\ell_1$. Since $\ell_1$ is weakly sequentially complete, there is $z=(z_n)\in\ell_1$ such that $z-(\langle\chi_n,x_i\rangle)$ is weakly null in $\ell_1$. By the Schur property of $\ell_1$ we have
\begin{equation} \label{equ:sV*-wsc-2}
\lim_{i\to\infty} \sum_{n\in\w} \big|z_n - \langle\chi_n,x_i\rangle\big| =0.
\end{equation}
We claim that $\lim_{n\to\infty} \sup_{i\in\w} |\langle\chi_n,x_i\rangle|= 0$. Indeed, otherwise, taking into account that  $\{\chi_n\}_{n\in\w}$ is weak$^\ast$ null, we could find strictly increasing sequences $(n_k), (i_k)\subseteq \w$ and $\e>0$ such that $ |\langle\chi_{n_k},x_{i_k}\rangle|\geq 2\e$. Since $z_n\to 0$ it follows that $\sum_{n\in\w} \big|z_n - \langle\chi_n,x_{i_k}\rangle\big|>\e$ for each sufficiently large $k\in\w$. But this contradicts (\ref{equ:sV*-wsc-2}). This proves the claim.

Since $\{\chi_n\}_{n\in\w}$ was arbitrary, the claim means that the sequence $S$ is a $V^\ast$ set in $E$. By the property $sV^\ast$ of $E$ we obtain that $S$ has a weakly convergent subsequence. As $S$ is weakly Cauchy it follows that $S$ weakly converges to some element of $E$. Thus $E$ is  weakly sequentially complete.\qed
\end{proof}

The clause (i) of Proposition \ref{p:Vp*-duality} and Proposition \ref{p:sV*-wsc} immediately imply
\begin{corollary} \label{c:V-E'-wsc}
Let a quasibarrelled space $E$ be such that $E'_\beta$ is barrelled. If $E$ has the property $sV$, then $E'_\beta$ is weakly sequentially complete.
\end{corollary}



\begin{remark} \label{rem:RDP*-sV*} {\em
Let us recall that a subset $A$ of a Banach space $X$ is called a {\em Dunford--Pettis set} (a $DP$ {\em set}) 
if each weakly null sequence $\{\chi_n\}_{n\in\w}$ in the Banach dual $X^\ast=X'_\beta$ tends to $0$ uniformly on $A$, i.e. $\lim_{n\to\infty} \sup_{x\in A} |\langle\chi_n,x\rangle|=0$. In other words, $DP$ sets in $X$ are exactly $\infty$-$(V^\ast)$ subsets of $X$. Following Bator, Lewis and Ochoa \cite{BLO}, the space $X$ has the {\em $RDP^\ast$ property} if each $DP$ subset of $X$ is relatively weakly sequentially compact. Therefore $X$ has the $RDP^\ast$ property if and only if it has the $sV^\ast_\infty$ property. The notion of the $RDP^\ast$ property was generalized by Ghenciu \cite{Ghenciu-pGP}. If $p\in [1,\infty)$, a subset $A$ of $X$ is called {\em weakly-$p$-Dunford--Pettis set} if for every weakly $p$-summable sequence $\{\chi_n\}_{n\in\w}$ in $X^\ast$, it follows $\lim_{n\to\infty} \sup_{x\in A} |\langle\chi_n,x\rangle|=0$.  By definition, we obtain that weakly-$p$-Dunford--Pettis sets in $X$ are exactly $p$-$(V^\ast)$ subsets of $X$. Following \cite{Ghenciu-pGP}, the space $X$ has the {\em $RDP^\ast_p$ property} if every weakly-$p$-Dunford--Pettis subset of $X$ is relatively weakly sequentially compact. We summarize this discussion as follows: {\em if $p\in [1,\infty]$, a Banach space $X$ has the $RDP^\ast_p$ property if and only if it has the property $sV^\ast_p$}.\qed}
\end{remark}

In Proposition 7 of \cite{Pelcz-62},   Pe{\l}czy\'{n}ski showed that a Banach space $X$ is reflexive if and only if it has both properties $V$ and $V^\ast$.
The $p$-version of this remarkable result is not true in general. Chen, Ch\'{a}vez-Dom\'{\i}nguez, and Li proved in \cite{LCCD} that the classical non-reflexive James space $J$ has both property $V_2$ and property $V^\ast_2$. These results motivates the next problem.

\begin{problem}
Let $E$ be a (Baire, barrelled, quasibarrelled etc.) locally convex space. Is it true that $E$ is reflexive if and only if it has the property  $V$ and the property  $V^\ast$?
\end{problem}


\section{ Property $p$-$(u)$ for locally convex spaces} \label{sec:property-u}


The following definition naturally generalizes Pe{\l}czy\'{n}ski's property $(u)$.

\begin{definition} \label{def:property-u}{\em
Let $p\in[1,\infty)$. A locally convex space $E$ is said to have a {\em property $p$-$(u)$} if for every weakly Cauchy sequence $\{y_n\}_{n\in\w}\subseteq E$ there exists a weakly $p$-summable sequence $\{x_n\}_{n\in\w}$ in $E$ such that the sequence $\big\{y_n- \sum_{j\leq n} x_j\big\}$ is weakly null. If $p=1$ we say that $E$ has the {\em property $(u)$}.\qed}
\end{definition}

\begin{remark} \label{rem:p-u-not-infty} {\em
We do not consider the case $p=\infty$ because the condition of Definition \ref{def:property-u} holds for every lcs $E$. Indeed, let $\{y_n\}_{n\in\w}\subseteq E$ be a weakly Cauchy sequence.  Set $x_0=y_0$, and for $n\geq 1$, we set $x_n=y_n-y_{n-1}$. As $\{y_n\}_{n\in\w}$ is weakly Cauchy, the sequence $\{x_n\}_{n\in\w}$ is weakly null (= weakly  $\infty$-summable). It remains to note that $y_n= \sum_{j\leq n} x_j$ for every $n\in\w$.\qed}
\end{remark}

Since every weakly $p$-summable sequence is weakly $q$-summable for $p\leq q$ we obtain
\begin{proposition} \label{p:p-u-q-u}
If $1\leq p<q<\infty $ and a locally convex space $E$ has the property $p$-$(u)$, then $E$ has the property $q$-$(u)$.
\end{proposition}

\begin{proposition} \label{p:wsc-u}
If a locally convex space $E$ is weakly sequentially complete, then it has the property $p$-$(u)$ for every $p\in[1,\infty)$.
\end{proposition}

\begin{proof}
Let $\{y_n\}_{n\in\w}$ be a weakly Cauchy sequence in $E$. Since $E_w$ is sequentially complete, there is $y\in E$ such that $y_n\to y$ in the weak topology. Set $x_0=y$ and $x_n=0$ for every positive natural number $n$. It is clear that $\{x_n\}_{n\in\w}$ is weakly $p$-summable and $\big\{y_n-\sum_{j\leq n} x_j\big\}$ is weakly null.\qed
\end{proof}

Propositions \ref{p:sV*-wsc} and \ref{p:wsc-u} imply the next result.
\begin{corollary} \label{c:sV*-p-u}
Let $E$ be a barrelled locally convex space. If $E$ has the property $sV^\ast$, then it has the property $p$-$(u)$ for every $p\in[1,\infty)$.
\end{corollary}

\begin{remark} \label{rem:p-u-no-quotient} {\em
The quotient space of a space with the property $p$-$(u)$ may not have the property $p$-$(u)$. Indeed, let $E$ be a Banach space without the property $(u)$. It is well known that $E$ is a quotient space of $\ell_1(\Gamma)$ for some infinite set $\Gamma$. Since $\ell_1(\Gamma)$  has the Schur property, it is weakly sequentially complete. Therefore, by Proposition \ref{p:wsc-u}, $\ell_1(\Gamma)$ has  the property $p$-$(u)$.\qed}
\end{remark}

Let $p\in[1,\infty]$, and let $E$ be a locally convex space. We denote by $B_1^p(E)$ the family of all $x^{\ast\ast}\in  E''$ which are weak$^\ast$ $p$-limits of weakly $p$-Cauchy sequences in $E$; if $p=\infty$ we write simply $B_1(E)$. If $x^{\ast\ast}\in  E''$  and $\{x_j\}_{j\in\w}$ is a sequence in $E$ such that the sequence $\big\{\sum_{j\leq n} x_j\big\}$ weak$^\ast$ $p$-converges to $x^{\ast\ast}$, we write $x^{\ast\ast}=\sum^{\ast p} x_j$; if $p=\infty$ we write simply $x^{\ast\ast}=\sum^{\ast} x_j$.

\begin{proposition} \label{p:p-u-B1}
Let $p\in[1,\infty)$, and let $E$ be a locally convex space. Then $E$ has the property $p$-$(u)$ if and only if for every $x^{\ast\ast}\in  B_1(E)$, there is a weakly $p$-summable sequence $\{x_j\}_{j\in\w}$  in $E$ such that $ x^{\ast\ast}=\sum^{\ast} x_j$.
\end{proposition}

\begin{proof}
Assume that $E$ has the property $p$-$(u)$, and let $x^{\ast\ast}\in  B_1(E)$. Take a weakly Cauchy sequence $\{y_n\}_{n\in\w}\subseteq E$ which weak$^\ast$ converges to $x^{\ast\ast}$. By the  property $p$-$(u)$, there is  a weakly $p$-summable sequence $\{x_j\}_{j\in\w}$  in $E$ such that the sequence $\big\{y_n- \sum_{j\leq n} x_j\big\}$ is weakly null. Then
\[
x^{\ast\ast}-\sum_{j\leq n} x_j= \big( x^{\ast\ast}-y_n \big)+ \Big( y_n - \sum_{j\leq n} x_j \Big)\to 0 \quad\mbox{ as } \; n\to\infty
\]
in the weak$^\ast$ topology on $E''$. Thus $ x^{\ast\ast}=\sum^{\ast} x_j$.

Conversely, assume that for every $x^{\ast\ast}\in  B_1(E)$, there is a weakly $p$-summable sequence $\{x_j\}_{j\in\w}$  in $E$ such that $ x^{\ast\ast}=\sum^{\ast } x_j$. Let $S=\{y_n\}_{n\in\w}$ be a weakly Cauchy sequence in $E$. Since $S$ is a bounded subset of $E$, the polar $S^\circ$ is a neighborhood of zero in $E'_\beta$. By the Alaoglu theorem, the polar $S^{\circ\circ}$ of $S^\circ$ in $E''$ is a weak$^\ast$ compact set. Therefore $S$ has a cluster point $x^{\ast\ast}\in S^{\circ\circ}$. Since $S$ is weakly Cauchy, $x^{\ast\ast}$ is the weak$^\ast$ limit of $S$. Hence $x^{\ast\ast}\in B_1(E)$ and therefore $ x^{\ast\ast}=\sum^{\ast } x_j$ for some  weakly $p$-summable sequence $\{x_j\}_{j\in\w}$  in $E$. Whence
\[
y_n-\sum_{j\leq n} x_j= \big( y_n -x^{\ast\ast} \big)+ \Big( x^{\ast\ast} - \sum_{j\leq n} x_j \Big)\to 0 \quad\mbox{ as } \; n\to\infty
\]
in the weak topology on $E$. Thus $E$ has the property $p$-$(u)$.\qed
\end{proof}

\begin{proposition} \label{p:product-sum-u}
If  $p\in[1,\infty)$ and $\{E_i\}_{i\in I}$ is a non-empty family of locally convex spaces, then
\begin{enumerate}
\item[{\rm(i)}] $E=\prod_{i\in I} E_i$ has the property $p$-$(u)$ if and only if all spaces $E_i$ have the property $p$-$(u)$;
\item[{\rm(ii)}] $E=\bigoplus_{i\in I} E_i$ has the property $p$-$(u)$  if and only if all spaces  $E_i$ have  the  property $p$-$(u)$.
\end{enumerate}
\end{proposition}

\begin{proof}
Assume that $E$  has the property $p$-$(u)$. Fix $i\in I$, and let $\{y_n^i\}_{n\in\w}\subseteq E_i$ be a weakly Cauchy sequence. If $T_i:E_i\to E$ is the identity inclusion, then $\{y_n=T_i(y_n^i)\}_{n\in\w}$ is a weakly Cauchy sequence in $E$. By the property $p$-$(u)$ of $E$, there exists a weakly $p$-summable sequence $S=\{x_n\}_{n\in\w}$ in $E$ such that the sequence $\big\{y_n- \sum_{j\leq n} x_j\big\}$ is weakly null. Then, by (iii) of Lemma \ref{l:prop-p-sum},  the projection $\{x_n(i)\}_{n\in\w}$ of $S$ onto $E_i$ is  a weakly $p$-summable sequence in $E_i$. Then for  every $\chi\in (E_i)'$ and  considering $\chi$ as an element of $E'$, we obtain
\[
\Big\langle \chi, y_n^i- \sum_{j\leq n} x_j(i) \Big\rangle=\Big\langle \chi, y_n- \sum_{j\leq n} x_j\Big\rangle\to 0 \quad\mbox{ as } \; n\to\infty.
\]
Hence the sequence $\big\{y_n^i- \sum_{j\leq n} x_j(i)\big\}$ is weakly null. Thus $E_i$  has the property $p$-$(u)$.

Conversely, assume that all spaces $E_i$  have the property $p$-$(u)$. Let $S=\{y_n\}_{n\in\w}\subseteq E$ be  a weakly Cauchy sequence. In the case (ii), let $F$ be the finite support of $S$. Then for every $i\in I$,  the projection $\{P_i(y_n)\}_{n\in\w}$ of $S$ onto $E_i$ is weakly Cauchy in $E_i$. Since $E_i$ has the property $p$-$(u)$, there exists a weakly $p$-summable sequence $\{x_n^i\}_{n\in\w}$ in $E_i$ such that the sequence $\big\{P_i(y_n)- \sum_{j\leq n} x_j^i\big\}$ is weakly null. For every $n\in\w$, define $x_n:=(x_n^i)_{i\in I}$ and observe that the obtained sequence $\{x_n\}_{n\in\w}$ is weakly $p$-summable in $E$ by Lemma \ref{l:support-p-sum} (in the case (ii), the support of $\{x_n\}_{n\in\w}$ is also $F$). If $\chi=(\chi_i)\in E'$ and, in the case (i), $G$ is the finite support of $\chi$, then
\[
\Big\langle \chi, y_n- \sum_{j\leq n} x_j\Big\rangle=\sum_{i\in F\cup G} \Big\langle \chi_i, P_i(y_n)- \sum_{j\leq n} x_j^i\Big\rangle \to 0 \quad\mbox{ as } \; n\to\infty.
\]
Hence the sequence $\big\{y_n- \sum_{j\leq n} x_j\big\}$ is weakly null. Thus $E$ has  the property $p$-$(u)$. \qed
\end{proof}

\begin{proposition} \label{p:p-u-comp}
Let $p\in[1,\infty)$, and let $(E,\tau)$ be a locally convex space. If $\TTT$ is a locally convex topology on $E$ compatible with $\tau$, then $(E,\tau)$ has the  property $p$-$(u)$ if and only if $(E,\TTT)$ has the  property $p$-$(u)$.
\end{proposition}

\begin{proof}
The assertion immediately follows from the fact that the notions of being weakly null, weakly Cauchy and weakly $p$-summable depend only on the duality $(E,E')$.\qed
\end{proof}

We finish this section with the several open problems.
\begin{problem}
For every distinct pairs $(p,q)$ and $(p',q')$ with $p\leq q$ and $p'\leq q'$, find  (Banach, Fr\'{e}chet, barrelled, etc.) locally convex spaces $E$ with the property $V_{(p,q)}$ but without the property $V_{(p',q')}'$ {\rm(}analogously, for the properties $EV_{(p,q)}$, $V_{(p,q)}^\ast$, $EV_{(p,q)}^\ast$ and $p$-$(u)${\rm)}.
\end{problem}

\begin{problem}
Characterize Tychonoff spaces $X$ for which the space $\CC(X)$ has one of the  property $V_{(p,q)}$ {\rm(}$EV_{(p,q)}$, $V_{(p,q)}^\ast$, $EV_{(p,q)}^\ast$ or $p$-$(u)${\rm)}.
\end{problem}

\begin{problem}
Characterize Tychonoff spaces $X$ for which the free lcs $L(X)$ over $X$ has  one of the property $V_{(p,q)}$ {\rm(}$EV_{(p,q)}$, $V_{(p,q)}^\ast$, $EV_{(p,q)}^\ast$ or $p$-$(u)${\rm)}.
\end{problem}

\begin{problem}
Characterize Tychonoff spaces $X$ for which the space $C_p(X)$ has the property $p$-$(u)$.
\end{problem}


\section{Weakly sequentially $p$-(pre)compact and $p$-convergent operators} \label{sec:p-convergent}


In what follows we shall use intensively the following   $p$-versions of weakly compact-type properties, which extend the corresponding notions in the case of Banach spaces defined in \cite{CS} and \cite{Ghenciu-pGP}.

\begin{definition} \label{def:weak-p-compact} {\em
Let $p\in[1,\infty]$. A subset   $A$ of a separated tvs $E$ is called
\begin{enumerate}
\item[$\bullet$] ({\em relatively}) {\em weakly sequentially $p$-compact} if every sequence in $A$ has a weakly $p$-convergent  subsequence with limit in $A$ (resp., in $E$);
\item[$\bullet$] {\em weakly  sequentially $p$-precompact} if every sequence from $A$ has a  weakly $p$-Cauchy subsequence.\qed
\end{enumerate} }
\end{definition}
Note that if $E$ is a Banach space, weakly  sequentially $p$-precompact sets of $E$ are known as {\em conditionally weakly $p$-compact sets}, see Definition 2.5 of \cite{CCDL}.

\begin{lemma} \label{l:image-p-seq-com}
Let $p\in[1,\infty]$, $E$ and $L$ be locally convex spaces, and let $T:E\to L$ be a weak-weak continuous linear map. If $A$ is a $($relatively$)$ weakly sequentially $p$-compact set {\rm(}resp., a weakly sequentially $p$-precompact set{\rm)} in $E$, then the image  $T(A)$ is a $($relatively$)$ weakly sequentially $p$-compact {\rm(}resp.,  weakly sequentially $p$-precompact{\rm)}  set in $L$.
\end{lemma}

\begin{proof}
We consider only the case when $A$ is a (relatively) weakly sequentially $p$-compact set, the case when $A$ is weakly sequentially $p$-precompact can be considered analogously. Let $\{T(a_n)\big\}_{n\in\w}$ be a sequence in $T(A)$. Since  $A$ is (relatively) weakly sequentially $p$-compact, the sequence $\{a_n\}$ has a subsequence $\big\{a_{n_k}\big\}_{k\in\w}$ which weakly $p$-converges to some $a\in A$ (resp., $a\in E$). Let now $\chi\in L'$. Since $T$ is weakly continuous, Theorem 8.10.3 of \cite{NaB} implies that $T^\ast(\chi)\in E'$ and we have
\[
\big( \langle \chi, T(a_{n_k})-T(a)\rangle\big)_{k\in\w}= \big( \langle T^\ast(\chi), a_{n_k}-a\rangle\big)_{k\in\w}\in \ell_p \; (\mbox{or $\in c_0$ if $p=\infty$}).
\]
Therefore $T(a_{n_k})$  weakly $p$-converges to  $T(a)$. Thus $T(A)$ is a $($relatively$)$ weakly sequentially $p$-compact set in $L$. \qed
\end{proof}

Below we extend the notion of $p$-convergent operators between Banach spaces defined by Castillo and S\'{a}nchez \cite{CS} to the general case of  all separated topological vector spaces. Weakly   $p$-precompact operators between Banach spaces were defined by Ghenciu \cite{Ghenciu-pGP}. 

\begin{definition}\label{def:lcs-p-oper}{\em
Let $p\in[1,\infty]$. A linear map $T:E\to L$ between separated topological vector spaces $E$ and $L$ is said to be
\begin{enumerate}
\item[$\bullet$]  {\em $p$-convergent} if it sends weakly $p$-summable sequences in $E$ into null-sequences in $L$;
\item[$\bullet$]  {\em weakly sequentially $p$-compact} if  for some $U\in \Nn_0(E)$, the set $T(U)$ is a relatively weakly  sequentially $p$-compact subset of $L$;
\item[$\bullet$]   {\em weakly  sequentially $p$-precompact} if  for some $U\in \Nn_0(E)$, the set $T(U)$ is   weakly  sequentially $p$-precompact in $L$;
\item[$\bullet$]   {\em unconditionall converging} if $T$ sends $wuC$ series into unconditionally convergent series.    \qed
\end{enumerate}}
\end{definition}

To characterize $p$-convergent operators we need the next lemma.

\begin{lemma} \label{l:p-conver-p-Cauchy}
Let $p\in[1,\infty]$, $E$ and $L$ be locally convex spaces, and let $T:E\to L$  be a $p$-convergent linear map. If $\{x_n\}_{n\in\w}\subseteq E$ is weakly $p$-Cauchy, then the sequence $\big\{T(x_n)\big\}_{n\in\w}$ is Cauchy in $L$. Consequently, if $L$ is sequentially complete, then $\big\{T(x_n)\big\}_{n\in\w}$ converges in $L$.
\end{lemma}

\begin{proof}
Take  a pair of strictly increasing sequences $(k_n),(j_n)$ in $\w$. Then $\{x_{k_n}-x_{j_n}\}$ is weakly $p$-summable. Since $T$ is $p$-convergent, we obtain $T(x_{k_n})-T(x_{j_n}) \to 0$ in $L$, and hence $\big\{T(x_n)\big\}_{n\in\w}$ is Cauchy in $L$.\qed
\end{proof}

For $1\leq p<\infty$ and operators between Banach spaces the next proposition is Theorem 2.6 of \cite{CCDL}.
\begin{proposition} \label{p:p-convergent-s}
Let $p\in[1,\infty]$, and let $T$ be a weak-weak sequentially continuous linear map from a locally convex space $E$ to a sequentially complete locally convex space $L$. Then the following assertions are equivalent:
\begin{enumerate}
\item[{\rm(i)}] $T$ is $p$-convergent;
\item[{\rm(ii)}] $T(A)$ is relatively sequentially compact in $L$ for each weakly sequentially $p$-precompact subset $A$ of $E$;
\item[{\rm(iii)}] $T(A)$ is sequentially precompact in $L$ for any weakly sequentially $p$-precompact subset $A$ of $E$.
\end{enumerate}
\end{proposition}

\begin{proof}
(i)$\Ra$(ii) Let $A$ be a weakly sequentially $p$-precompact subset of $E$, and let $S=\{a_n\}_{n\in\w}$ be a sequence in $A$. Then $S$ has a weakly $p$-Cauchy subsequence $\{a_{n_k}\}_{k\in\w}$. Then, by Lemma \ref{l:p-conver-p-Cauchy},  the sequence $\big\{T(a_{n_k})\big\}_{k\in\w}$  converges to some element $x\in L$.  Thus $T(A)$ is relatively sequentially compact in $L$.

(ii)$\Ra$(iii) is trivial.

(iii)$\Ra$(i) Let $\{x_n\}_{n\in\w}$ be a weakly $p$-summable sequence in $E$. Then $\{x_n\}_{n\in\w}$ is  weakly $p$-Cauchy in $E$, and hence $\{T(x_n)\}_{n\in\w}$ is sequentially precompact in $L$. Since $L$ is sequentially complete, Lemma \ref{l:seq-p-comp} implies that $\{T(x_n)\}_{n\in\w}$ is relatively sequentially compact in $L$. Therefore, any subsequence $S$ of $\{T(x_n)\}_{n\in\w}$ has a subsequence $\{T(x_{n_k})\}_{k\in\w}$ which converges to some element $z\in L$. On the other hand, since $T$ is weak-weak sequentially continuous and $x_n\to 0$ in the weak topology, we obtain that $T(x_n)\to 0$ in the weak topology of $L$. Therefore $z=0$. Since $S$ was arbitrary, it follows that $T(x_n)\to 0$ in the space $L$. Thus $T$ is $p$-convergent.\qed
\end{proof}

The problem to find pairs $(E,L)$ of locally convex spaces for which $\LL(E,L)$ contains only $p$-convergent operators is of independent interest.
Below we consider partial cases when this problem has a positive solution.
\begin{proposition} \label{p:p-operator-1}
Let $E$ and $L$ be locally convex spaces.
\begin{enumerate}
\item[{\rm(i)}] If $p\in[1,\infty]$ and $E$ or $L$ has the $p$-Schur property, then  each operator $T\in\LL(E,L)$ is $p$-convergent.
\item[{\rm(ii)}] Assume that $L$ has the Schur property $($for example, $L$ carries its weak topology$)$. Then  for every $p\in[1,\infty]$, each operator $T\in\LL(E,L)$ is $p$-convergent.
\end{enumerate}
\end{proposition}
\begin{proof}
(i)  If $E$ has the $p$-Schur property, the assertion is trivial.  If $L$ has the $p$-Schur property, the assertion follows from (iii) of  Lemma \ref{l:prop-p-sum} which states that continuous images of weakly $p$-summable sequences are weakly $p$-summable.

(ii) If $L$ has the Schur property, it has the $p$-Schur property for every $p\in[1,\infty]$. Now (i) applies.\qed
\end{proof}

A sufficient condition of being a $p$-convergent operator is given in the next assertion.

\begin{proposition} \label{p:p-convergent-suf}
Let $p\in[1,\infty)$, $E$ and $L$ be  locally convex spaces, and let $T:E\to L$ be a weak-weak sequentially continuous linear map. Assume that one of the following conditions holds:
\begin{enumerate}
\item[{\rm(i)}]  for any operator $S:\ell_{p^\ast}^0 \to E$ $($or $S:c_0^0\to E$ if $p=1$$)$, the linear map $T\circ S$ is sequentially precompact;
\item[{\rm(ii)}] $E$ is sequentially complete and for any operator $S:\ell_{p^\ast} \to E$ $($or $S:c_0\to E$ if $p=1$$)$, the linear map $T\circ S$ is sequentially precompact.
\end{enumerate}
Then $T$ is $p$-convergent.
\end{proposition}

\begin{proof}
Let $T\circ S$ be sequentially precompact for any operator $S$ from (i) or (ii). To show that $T$ is $p$-convergent, fix an arbitrary weakly $p$-summable sequence  $\{x_n\}_{n\in\w}$ in $E$. We have to prove that $T(x_n)\to 0$ in $L$. Suppose for a contradiction that $T(x_n)\not\to 0$ in $L$. Then we can find a neighborhood $U$ of zero in $L$ and a subsequence $\{x_{n_i}\}_{i\in\w}$ of $\{x_n\}$ such that $T(x_{n_i})\not\in U$ for every $i\in\w$. Since $\{x_{n_i}\}_{i\in\w}$ is also  weakly $p$-summable, without loss of generality we can assume that $x_n\not\in U$ for every $n\in\w$.

By  Proposition \ref{p:Lp-E-operator}, there is $S\in\LL(\ell_{p^\ast}^0,E)$ (resp., $S\in\LL(\ell_{p^\ast},E)$ or $S\in\LL(c_0^0,E)$, $S\in\LL(c_0,E)$ if $p=1$) such that $S(e_n^\ast)=x_n$ for every $n\in\w$. Since $T\circ S$ is  sequentially precompact from a normed space and $\{e_n^\ast\}_{n\in\w}$ is bounded, the sequence $\{TS(e_n^\ast)\}_{n\in\w}$ has a Cauchy subsequence $\{TS(e_{n_k}^\ast)\}_{k\in\w}$. Denote by $D$ the co-extension of $TS$ into the completion $\overline{L}$ of the space $L$. Since the sequence $\{TS(e_{n_k}^\ast)\}_{k\in\w}$ is also Cauchy in the complete space $\overline{L}$, it converges to some  point $\overline{x}\in \overline{L}$. Observe that $\{e_n^\ast\}_{n\in\w}$ is a weakly null-sequence in $\ell_{p^\ast}$ (or in $c_0$ if $p=1$), and hence the weak-weak sequential continuity of $D$ implies that $D(e_n^\ast)=TS(e_n^\ast)\to 0$ in the weak topology of $\overline{L}$. Since also $TS(e_{n_k}^\ast)\to \overline{x}$ in the weak topology on $\overline{L}$, we obtain that $\overline{x}=0$, and hence $T(x_{n_k})=TS(e_{n_k}^\ast)\to 0$ in the original topology of $L$. But this contradicts the choice of $\{x_n\}$. \qed
\end{proof}

For $p\in[1,\infty]$, the $p$-version of the weak sequential completeness  of Banach spaces is introduced in Definition 2.17 of \cite{LCCD}. Below we extend this notion to the class of all locally convex spaces.

\begin{definition} \label{def:weak-sec-p-comp} {\em
Let $p\in[1,\infty]$. A locally convex space $E$ is called {\em weakly sequentially $p$-complete} if every weakly $p$-Cauchy sequence is weakly $p$-convergent.\qed}
\end{definition}
So weakly sequentially $\infty$-complete spaces are exactly weakly sequentially complete. In particular, $\ell_1$ is weakly sequentially $\infty$-complete. Even a more general assertion holds true.

\begin{proposition} \label{p:Banach-weakly-p-complete}
If a Banach space $E$ has the Schur property, then $E$ is weakly sequentially $p$-complete for every $p\in[1,\infty]$.
\end{proposition}

\begin{proof}
Let $S= \{x_n\}_{n\in\w}$ be a weakly $p$-Cauchy sequence in $E$. Then $S$ is weakly Cauchy and hence, by the Schur property, $S$ is Cauchy in $E$. Since $E$ is complete, there is $x\in E$ such that $\|x_n-x\|\to 0$. Choose a strictly increasing sequence $(n_k)$ in $\w$ such that
\[
\|x_{n_k}-x\| <\tfrac{1}{2^k} \;\; \mbox{ for every }\; k\in\w.
\]
Then $\{x_{n_k}\}_{k\in\w}$ is weakly $p$-converges to $x$ since (we consider the case $p<\infty$)
\[
\sum_{k\in\w} |\langle\chi,x_{n_k}-x \rangle|^p \leq \sum_{k\in\w} \|\chi\| \cdot \tfrac{1}{2^k}<\infty
\]
for every $\chi\in E'$. Thus,  by Lemma \ref{l:p-Cauchy-conv}, $S$ weakly $p$-converges to $x$.\qed
\end{proof}

In Corollary \ref{c:Lp-p-complete} below we shall show that for every $1<p<\infty$, the space $\ell_p$ is weakly sequentially $p^\ast$-complete.

The next assertion will play a crucial role to characterize $p$-convergent operators.

\begin{proposition} \label{p:p-convergent-nes}
Let $p\in[1,\infty)$, and let $E$ be a locally convex space.
\begin{enumerate}
\item[{\rm(i)}] Each operator $S:\ell_{p^\ast}^0 \to E$ $($or $S:c_0^0\to E$ if $p=1$$)$ is weakly sequentially $p$-precompact.
\item[{\rm(ii)}] If $p>1$, then for each operator $S:\ell_{p^\ast} \to E$, the set $S(B_{\ell_{p^\ast}})$ is weakly sequentially $p$-compact in $E$; in particular, $S$  is weakly sequentially $p$-compact.
\item[{\rm(iii)}] If $p=1$, then every operator $S:c_0\to E$  is weakly sequentially $1$-precompact. If additionally $E$ is weakly sequentially $1$-complete, then for each operator $S:c_0\to E$, the set $S(B_{c_0})$ is a weakly sequentially $1$-compact subset of $E$; in particular, $S$  is weakly sequentially $p$-compact.
\end{enumerate}
\end{proposition}

\begin{proof}
To avoid an unnecessary complication with notations and taking into account that $p<\infty$, below we set $\ell_{1^\ast}^0:= c_0^0$ and $\ell_{1^\ast}:= c_0$.  We show that the set $S(B_{\ell_{p^\ast}^0})$ (resp., $S(B_{\ell_{p^\ast}})$) is weakly sequentially $p$-precompact (resp., weakly sequentially $p$-compact) in $E$. To this end, let $\{x_n\}_{n\in\w}$ be an arbitrary sequence in $B_{\ell_{p^\ast}^0}$ (resp., in $B_{\ell_{p^\ast}}$). To prove the proposition we have to find a weakly $p$-Cauchy (resp., weakly $p$-convergent in $S(B_{\ell_{p^\ast}})$) subsequence  of $\{S(x_n)\}_{n\in\w}$. We shall consider $\ell_{p^\ast}^0$ as a dense subspace of $\ell_{p^\ast}$. We distinguish between the next three possible cases.
\smallskip

{\em Case 1. The sequence $\{x_n\}_{n\in\w}$ contains a subsequence $\{x_{n_k}\}_{k\in\w}$ which is Cauchy in the norm topology.} Passing to a subsequence if needed, we can assume that $x_n$ norm converges to some element $x\in B_{\ell_{p^\ast}}$ and moreover, that $\|x_n-x\|_{\ell_{p^\ast}} <\tfrac{1}{2^n}$ for every $n\in\w$. Then for every $n\in\w$ and each $\chi\in E'$, we have
\[
\big| \langle \chi, S(x_n)-S(x)\rangle\big|=\big| \langle S^\ast(\chi), x_n-x\rangle\big|\leq \|S^\ast(\chi)\| \cdot \tfrac{1}{2^n}.
\]
Let $(k_n),(j_n)\subseteq \w$ be two strictly increasing sequences. Then for every $\chi\in E'$, we obtain
\[
\big| \langle\chi, S(x_{k_n})-S(x_{j_n})\rangle\big|\leq\big| \langle\chi, S(x_{k_n})-S(x)\rangle\big| +\big| \langle\chi, S(x_{j_n})-S(x)\rangle\big|\leq \|S^\ast(\chi)\|\cdot (\tfrac{1}{2^{k_n}}+\tfrac{1}{2^{j_n}})
\]
and hence $\{S(x_n)\}_{n\in\w}$ is weakly $p$-Cauchy (resp., $S(x_n)$ weakly $p$-converges to $S(x)\in S(B_{\ell_{p^\ast}})$) and we are done.
\smallskip

{\em Case 2.  The sequence $\{x_n\}_{n\in\w}$ contains a subsequence $\{x_{n_k}\}_{k\in\w}$ weakly convergent in $B_{\ell_{p^\ast}}$ which does not contain a norm-Cauchy subsequence.}
\smallskip

In this case, passing to a subsequence if needed, without loss of generality we can assume that $\{x_n\}_{n\in\w}$ weakly converges to some element $x\in B_{\ell_{p^\ast}}$. Since it does not contain a norm-Cauchy subsequence, there is $\e>0$ such that $\|x_n-x\|_{\ell_{p^\ast}} \geq\e>0$ for every $n\in\w$.
Then, by Proposition 2.1.3 of \cite{Al-Kal} (see also Theorem 1.3.2 of \cite{Al-Kal}), 
there are a basic subsequence $\{x_{n_k}-x\}_{k\in\w}\subseteq \ell_{p^\ast}$ of $\{x_n-x\}_{n\in\w}$ and a linear topological isomorphism $R:\cspn\big(\{x_{n_k}-x\}_{k\in\w}\big)\to \ell_{p^\ast}$ such that (recall that $(e_k^\ast)$ is the standard unit basis of $\ell_{p^\ast}$)
\begin{equation} \label{equ:p-convergent-1}
R(x_{n_k}-x)= a_k e^\ast_k \quad \mbox{ for every $k\in\w$, where }\; a_k=\|x_{n_k}-x\|_{\ell_{p^\ast}},
\end{equation}
and such that $V:=\cspn\big(\{x_{n_k}-x\}_{k\in\w}\big)$ is complemented in $\ell_{p^\ast}$. Note that
\[ 
\e \leq a_k=\|x_{n_k}-x\|_{\ell_{p^\ast}}\leq 2 \quad \mbox{ for every $k\in\w$},
\] 
and therefore the map $Q:\ell_{p^\ast}\to\ell_{p^\ast}$ defined by $Q(\xi_k):= (a_k \xi_k)$ is a topological linear isomorphism. Hence, replacing $R$ by $Q^{-1}\circ R$, without loss of generality we can assume that $a_k=1$ for every $k\in\w$.

\smallskip

{\em Subcase 2.1. Let $S:\ell_{p^\ast}^0 \to E$  be an operator.}
\smallskip

Since the sequence $\{x_n\}$ was arbitrary, to show that $S$  is weakly sequentially $p$-precompact it is sufficient to prove that the subsequence $\{S(x_{n_k})\}_{k\in\w}$ is weakly $p$-Cauchy in $E$. To this end, 
it suffices to show that $\big\{S\big(x_{n_{k_j}}-x_{n_{k_{j+1}}}\big)\big\}_{k\in\w}$ is weakly $p$-summable for every strictly increasing sequence $(k_j)$ in $\w$.

Let $H_{(k_j)}:=\spn\big\{x_{n_{k_j}}-x_{n_{k_{j+1}}}\big\}_{j\in\w}\subseteq V$, $R_{(k_j)}$ be the restriction of $R$ onto $H_{(k_j)}$, and let $X_{(k_j)}:=R_{(k_j)}\big(H_{(k_j)}\big)\subseteq \spn\{e^\ast_{k_{j}}\}_{j\in\w}$. Then $R_{(k_j)}$ is a linear topological isomorphism. Observe also that $H_{(k_j)}\subseteq \ell_{p^\ast}^0$. Let $\Id_{(k_j)}: H_{(k_j)}\to \ell_{p^\ast}^0$ be the identity embedding. Then we obtain the following two sequences
\[
\xymatrix{
X_{(k_j)}  \ar@{->}[r]^{R_{(k_j)}^{-1}} & H_{(k_j)}  \ar@{->}[r]^{\Id_{(k_j)}} & \ell_{p^\ast}^0  \ar@{->}[r]^S & E
}
\;\mbox{ and } \;
\xymatrix{
E'_\beta   \ar@{->}[r]^{S^\ast} & \ell_{p}  \ar@{->}[r]^{\Id_{(k_j)}^\ast}  & \hspace{2mm} \big(H_{(k_j)}\big)'_\beta \ar@{->}[r]^{\big(R_{(k_j)}^{-1}\big)^\ast} & \big(X_{(k_j)}\big)'_\beta.}
\]
Fix now an arbitrary $\chi\in\ E'$ and set $\eta:= \big(R_{(k_j)}^{-1}\big)^\ast\circ \Id_{(k_j)}^\ast \circ S^\ast(\chi)\in \big(X_{(k_j)}\big)'_\beta$. Since $X_{(k_j)}$ is a subspace of $W:=\spn\{e^\ast_{k_{j}}\}_{j\in\w}$, we can consider $\eta$ as an element $(b_j)$ of $\ell_p=W'_\beta$. By (\ref{equ:p-convergent-1}), we have $R_{(k_j)}\big(x_{n_{k_j}}-x_{n_{k_{j+1}}}\big)= e^\ast_{k_{j}} - e^\ast_{k_{j+1}}$ and hence
\[
\begin{aligned}
\langle\chi, S\big(x_{n_{k_j}}-x_{n_{k_{j+1}}}\big)\rangle & = \langle S^\ast(\chi), \Id_{(k_j)}\big( x_{n_{k_j}}-x_{n_{k_{j+1}}}\big)\rangle =
\Big\langle \big(R_{(k_j)}^{-1}\big)^\ast\circ \Id_{(k_j)}^\ast \circ S^\ast(\chi), e^\ast_{k_{j}} - e^\ast_{k_{j+1}}\Big\rangle \\
& = \big\langle  (b_j), e^\ast_{k_{j}} - e^\ast_{k_{j+1}}\big\rangle= b_j- b_{j+1},
\end{aligned}
\]
and hence
\[
\sum_{j\in\w} \big|\langle\chi, S\big(x_{n_{k_j}}-x_{n_{k_{j+1}}}\big)\rangle\big|^p = \sum_{j\in\w} \big| b_j- b_{j+1} \big|^p <\infty.
\]
Thus the sequence $\big\{S\big(x_{n_{k_j}}-x_{n_{k_{j+1}}}\big)\big\}_{k\in\w}\subseteq E$ is weakly $p$-summable, as desired.
\smallskip

{\em  Subcase 2.2. Let $S:\ell_{p^\ast} \to E$  be an operator.}
\smallskip

Since the sequence $\{x_n\}$ was arbitrary, to show that $S\big(B_{\ell_{p^\ast}}\big)$  is weakly sequentially $p$-compact it is sufficient to prove that the sequence $\{S(x_{n_k}-x)\}_{k\in\w}\subseteq E$ is weakly $p$-summable, and hence $S(x_{n_k})$ weakly $p$-converges to $S(x)$. To this end, we denote by $\Id_V:V\to \ell_{p^\ast}$ the identity embedding. Then we obtain the following sequences
\[
\xymatrix{
\ell_{p^\ast} \ar@{->}[r]^{R^{-1}} & V  \ar@{->}[r]^{\Id_V} & \ell_{p^\ast} \ar@{->}[r]^{S}& E
}
\;\mbox{ and } \;
\xymatrix{
E'_\beta  \ar@{->}[r]^{S^\ast} & \ell_p \ar@{->}[r]^{\Id_V^\ast} & V'_\beta \ar@{->}[r]^{(R^{-1})^\ast} & \ell_p.}
\]
Take an arbitrary $\chi\in E'$.  Set $(b_k):= (R^{-1})^\ast\circ \Id_V^\ast\circ S^\ast (\chi)\in (\ell_{p^\ast})'=\ell_p$. Then (\ref{equ:p-convergent-1}) implies
\[
\sum_{k\in\w} \big|\langle\chi, S(x_{n_k}-x)\rangle\big|^p  = \sum_{k\in\w} \Big|\big\langle S^\ast(\chi), \Id_V\circ R^{-1}(e^\ast_{k})\big\rangle\Big|^p  = \sum_{k\in\w} \Big|\Big\langle (b_k), e^\ast_k\Big\rangle\Big|^p  <\infty.
\]
Thus the sequence $\{S(x_{n_k}-x)\}_{k\in\w}\subseteq E$ is weakly $p$-summable. 
\smallskip


{\em Case 3.  The sequence $\{x_n\}_{n\in\w}$ does not contain a weakly convergent subsequence.}
\smallskip

Observe first that this case is possible only if $p=1$. (Indeed,  if $1<p<\infty$, then the space $\ell_{p^\ast}$ is separable and reflexive, and hence $B_{\ell_{p^\ast}}$ is a metric compact space in the weak topology. Thus $\{x_n\}_{n\in\w}$ contains weakly convergent (in $\ell_{p^\ast}$) subsequences.) So, let $S:c_0^0 \to E$ (or $S:c_0\to E$) be an operator.

By assumption, the norm closure $\overline{\{x_n\}_{n\in\w}}^{\,\|\cdot\|}$ does not contain $0$. Therefore, by the Kadets--Pe{\l}czy\'{n}ski Theorem 1.5.6 of \cite{Al-Kal}, either (1) the weak closure $K$ of $\{x_n\}_{n\in\w}$ is weakly compact in $c_0$ and fails to contain $0$, or (2) $\{x_n\}_{n\in\w}$ contains a basic subsequence.
However, the case (1) is impossible. Indeed, since  $c_0$ is a separable metrizable spaces, by \cite{Mich}, $K$ is cosmic and hence metrizable. Therefore, the sequence $\{x_n\}_{n\in\w}$ must contain a weakly convergent subsequence which is impossible by the assumption of Case 3.
Hence passing to a subsequence we can assume that $\{x_n\}_{n\in\w}$ is a basic sequence in $c_0$.

Considering $c_0$ as a subspace of $\IF^\w$, the bounded sequence $\{x_n\}_{n\in\w}$  has a convergent subsequence. Therefore, without loss of generality we can and will assume that $\{x_n\}_{n\in\w}$ pointwise converges to a point $z\in \IF^\w$.
Let $n_0=0$. Since $x_{n_0}\in c_0$, there is $N_0\in\w$ such that $x_{n_0}(i)<1$ for every $i>N_0$. As $x_n\to z$ in $\IF^\w$, there is $n_1>n_0$ such that $|x_{n_1}(i)-z(i)|<\tfrac{1}{4}$ for every $i\leq N_0$. Since $x_{n_1}\in c_0$, there is $N_1>N_0$ such that $|x_{n_1}(i)|<\tfrac{1}{4}$ for every $i>N_1$. Continuing this process we can find a subsequence $\{x_{n_k}\}_{n\in\w}$ of $\{x_n\}_{n\in\w}$ which satisfies the following property: there is a strictly increasing sequence $(N_k)\subseteq \w$ such that for every $k\in\w$, the following inequalities hold
\begin{equation} \label{equ:c0-E}
\begin{aligned}
& \big|x_{n_{k+1}}(i)- z(i)\big| <\tfrac{1}{4^{k+1}} \quad \mbox{ for every } \; i\leq N_k,\\
& \big|x_{n_{k+1}}(i)\big| <\tfrac{1}{4^{k+1}} \quad \mbox{ for every } \; i> N_{k+1}.
\end{aligned}
\end{equation}
Then for every $s>k>0$ and each $\eta\in \ell_1=(c_0^0)'$, (\ref{equ:c0-E}) implies (recall that all $x_n\in B_{c_0}$)
\begin{equation} \label{equ:c0-E-1}
\begin{aligned}
\big| \langle \eta, x_{n_{k}}-x_{n_{s}}\rangle \big| & \leq \sum_{i\leq N_{k-1}} |\eta(i)|\cdot \tfrac{2}{4^{k}} + \sum_{i=N_{k-1}+1}^{N_{s+1}}  |\eta(i)|\cdot 2 + \sum_{i>N_{s+1}}  |\eta(i)|\cdot\tfrac{2}{4^{k}}\\
& \leq \|\eta\|_{\ell_1} \cdot \tfrac{2}{4^{k}} + 2 \sum_{i=N_{k-1}+1}^{N_{s+1}}  |\eta(i)|. 
\end{aligned}
\end{equation}

We claim that  that the sequence $\{S(x_{n_k})\}_{n\in\w}$ is weakly $1$-Cauchy. To this end, fix an arbitrary strictly increasing sequence $(k_j)\subseteq \w$. For every $\chi\in E'$, let $\eta:= S^\ast(\chi)\in\ell_1$. Taking into account that $k_{j+1}+1\leq k_{j+3}-1$ and hence $N_{k_{j+1}+1}<N_{k_{j+3}-1}+1$,  (\ref{equ:c0-E-1}) implies
\[
\sum_{j\in\w} \big| \big\langle \chi, S\big(x_{n_{k_j}}-x_{n_{k_{j+1}}}\big)\big\rangle \big| \leq \sum_{j\in\w}  \tfrac{2}{4^{k_j}}\cdot \|\eta\|_{\ell_1}+\sum_{j\in\w} 2 \sum_{i=N_{k_j-1}+1}^{N_{k_{j+1}+1}}  |\eta(i)|< \|\eta\|_{\ell_1} \cdot\Big( 6+\sum_{j\in\w}  \tfrac{2}{4^{k_j}}\Big)< \infty
\]
Therefore the sequence $\{S(x_{n_k})\}_{n\in\w}$ is weakly $1$-Cauchy. Thus $S\big(B_{c_0^0}\big)$ (or $S\big(B_{c_0}\big)$) is weakly sequentially $1$-precompact.

If additionally $E$ is weakly sequentially $1$-complete and $S:c_0\to E$,  the sequence $\{S(x_{n_k})\}_{n\in\w}$ weakly $1$-converges to some point of $E$. Thus $S\big(B_{c_0}\big)$  is weakly sequentially $1$-compact.\qed
\end{proof}

\begin{corollary} \label{c:L0p-p-precompact}
Let $p\in[1,\infty)$, and let $\Gamma$ be an infinite set. Then:
\begin{enumerate}
\item[{\rm(i)}] the identity operators $\Id_{\ell_{p^\ast}^0}:\ell_{p^\ast}^0(\Gamma) \to \ell_{p^\ast}^0(\Gamma)$ and $\Id_{\ell_{p^\ast}}:\ell_{p^\ast}(\Gamma) \to \ell_{p^\ast}(\Gamma)$ $($or $\Id_{c_0^0}:c_0^0(\Gamma)\to c_0^0(\Gamma)$ and $\Id_{c_0}:c_0(\Gamma)\to c_0(\Gamma)$ if $p=1$$)$ are weakly sequentially $p$-precompact;
\item[{\rm(ii)}] {\rm(Proposition 1.4 of \cite{CS})} if $p>1$, then  the identity  operator $\Id_{\ell_{p^\ast}}:\ell_{p^\ast}(\Gamma) \to \ell_{p^\ast}(\Gamma)$ is weakly sequentially $p$-compact.
\end{enumerate}
\end{corollary}

\begin{proof}
(i) follows from Proposition \ref{p:p-convergent-nes} in which $E=\ell_{p^\ast}^0$, $c_0^0$, $\ell_{p^\ast}^0$ or $c_0$, respectively, and the fact that any sequence sits in a direct summand which is isomorphic to $E$.

(ii) follows from (ii) of Proposition \ref{p:p-convergent-nes} and the fact from (i).\qed
\end{proof}

\begin{remark} \label{rem:c0-not-p-compact} {\em
The condition $p\not= 1$ in (ii) of Corollary \ref{c:L0p-p-precompact} is essential. Indeed, for $p=1$ the sequence $\{e_0+\cdots+e_n\}_{n\in\w}\subseteq B_{c_0}$ (which has no weakly convergent subsequences) does not have weakly $1$-convergent subsequences. Therefore $\Id_{c_0}$ is not  weakly sequentially $1$-compact.\qed}
\end{remark}

\begin{corollary} \label{c:Lp-p-complete}
For every $1<p<\infty$, the space $\ell_p$ is weakly sequentially $p^\ast$-complete.
\end{corollary}

\begin{proof}
By (ii) of Proposition \ref{p:p-convergent-nes} applied to the identity operator $\Id_{\ell_p}:\ell_p\to \ell_p$, we obtain that $B_{\ell_p}$ is a weakly sequentially $p^\ast$-compact subset of $\ell_p$. Now, let $S=\{x_{n}\}_{n\in\w}$ be a  weakly $p^\ast$-Cauchy sequence in $\ell_p$. Since $S$ is bounded, without loss of generality  we can assume that $S\subseteq B_{\ell_p}$. Then $S$ has a subsequence $\{x_{n_k}\}_{k\in\w}$ which weakly $p^\ast$-converges to some point $x\in B_{\ell_p}$. Since $S$ is weakly $p^\ast$-Cauchy, Lemma \ref{l:p-Cauchy-conv} implies that also $S$ weakly $p^\ast$-converges to $x$. Thus $\ell_p$ is weakly sequentially $p^\ast$-complete.\qed
\end{proof}

Recall (see Definition 2.1 of \cite{CS-Saks}) that a Banach space $X$ belongs to the class $W_p$, $p\in[1,\infty]$, if any bounded sequence admits a weakly $p$-convergent subsequence. We extend this notion as follows.

\begin{definition} \label{def:lcs-Wp} {\em
Let $p\in[1,\infty]$. A locally convex space $E$ is said to {\em belong  to the class $W_p$} if any bounded subset of $E$ is relatively weakly sequentially $p$-compact.\qed}
\end{definition}

The next assertion immediately follows from Definitions \ref{def:weak-p-compact} and \ref{def:lcs-Wp}.
\begin{proposition} \label{p:p-Schur-Wp}
Let $p\in[1,\infty]$, and let a locally convex space $E$ belong  to the class $W_p$. If $E$ has the $p$-Schur property, then every bounded subset of $E$ is relatively sequentially compact.
\end{proposition}

\begin{corollary}  \label{c:p-Schur-Wp}
Let $1\leq p<\infty$. Then $\ell_p\in W_p$ if and only if $p\geq 2$.
\end{corollary}

\begin{proof}
Assume that $1< p<2$. Then $p<p^\ast$, and hence, by Proposition \ref{p:Lp-Schur}, $\ell_p$ has the $p$-Schur property. Recall also that $\ell_1$ has the Schur property. Since $B_{\ell_p}$ is not compact, $\ell_p$ does not belong to $W_p$ for every $1\leq p<2$.

Assume that $2\leq p<\infty$.  By (ii) of Corollary \ref{c:L0p-p-precompact}, $B_{\ell_p}$ is weakly sequentially $p^\ast$-compact. Since $1<p^\ast\leq 2$, we have $p^\ast\leq p$. Therefore any weakly $p^\ast$-convergent sequence is also weakly $p$-convergent. Thus $B_{\ell_p}$ is weakly sequentially $p$-compact and hence $\ell_p\in W_p$.\qed
\end{proof}

The following theorem generalizes and extends the corresponding assertion for Banach spaces, see \cite[p.~45]{CS} and Proposition 13 of \cite{Ghenciu-pGP}.

\begin{theorem} \label{t:p-convergent-1}
Let $p\in[1,\infty)$. For any locally convex spaces $E$ and $L$ and each  weak-weak sequentially continuous linear map $T:E\to L$,  the following assertions are equivalent:
\begin{enumerate}
\item[{\rm(i)}] $T$ is $p$-convergent;
\item[{\rm(ii)}] $T\circ S$ is a sequentially precompact operator for any operator $S:\ell_{p^\ast}^0 \to E$ $($or $S:c_0^0\to E$ if $p=1$$)$; moreover the set $T\circ S(B_{\ell_{p^\ast}^0})$ {\rm(}or $T\circ S(B_{c_0^0})${\rm)} is sequentially precompact in $E$.
\end{enumerate}
If in addition $E$ is sequentially complete, then {\rm(i)} and {\rm(ii)} are equivalent to the following
\begin{enumerate}
\item[{\rm(iii)}] $T\circ S$ is a sequentially precompact operator for any operator $S:\ell_{p^\ast} \to E$ $($or $S:c_0\to E$ if $p=1$$)$; moreover, the set $T\circ S(B_{\ell_{p^\ast}})$ {\rm(}or $T\circ S(B_{c_0})${\rm)} is sequentially precompact in $E$.
\end{enumerate}
If $1<p<\infty$ and $E$ is  sequentially complete, then {\rm(i)}--{\rm(iii)}  are equivalent to the following
\begin{enumerate}
\item[{\rm(iv)}] $T\circ S$ is a sequentially compact operator for any operator $S:\ell_{p^\ast} \to E$; moreover,  the set $T\circ S(B_{\ell_{p^\ast}})$ is sequentially compact in $E$.
\end{enumerate}
If $p=1$, $E$ and $L$ are sequentially complete, and $E$ is weakly sequentially $1$-complete, then {\rm(i)}--{\rm(iii)}  are equivalent to
\begin{enumerate}
\item[{\rm(v)}] $T\circ S$ is a sequentially compact operator for any operator $S:c_0 \to E$; moreover,   the set $T\circ S(B_{c_0})$ is sequentially compact in $E$.
\end{enumerate}
\end{theorem}

\begin{proof}
(i)$\Ra$(ii) and (i)$\Ra$(iii): Let $\{x_n\}_{n\in\w}$ be an arbitrary sequence in $B_{\ell_{p^\ast}^0}$ (resp., $B_{c_0^0}$, $B_{\ell_{p^\ast}}$ or $B_{c_0}$).  By Proposition \ref{p:p-convergent-nes},  any operator $S:\ell_{p^\ast}^0 \to E$ (resp., $S:c_0^0\to E$, $S:\ell_{p^\ast} \to E$ or $S:c_0\to E$) is weakly sequentially $p$-precompact. Therefore there is a subsequence $\{x_{n_k}\}_{k\in\w}$ of $\{x_n\}_{n\in\w}$ such that for every increasing sequence $(k_j)$ in $\w$, the sequence  $\big\{S\big(x_{n_{k_j}}-x_{n_{k_{j+1}}}\big)\big\}_{k\in\w}\subseteq E$ is weakly $p$-summable. Since $T$ is $p$-convergent it follows that $TS\big(x_{n_{k_j}}-x_{n_{k_{j+1}}}\big)\to 0$ in the space $L$. As the sequence $(k_j)$ was arbitrary, we obtain that the sequence $\{TS(x_{n_k})\}$ is Cauchy in $L$. Therefore  $TS$ is sequentially precompact. Since any sequentially precompact subset is bounded, it follows that the linear map $TS$ is bounded and hence continuous.
\smallskip

(ii)$\Ra$(i) and (iii)$\Ra$(i) follow from  Proposition \ref{p:p-convergent-suf}.
\smallskip

Assume that $1<p<\infty$ and $E$ is sequentially complete.
\smallskip

(i)$\Ra$(iv) Let $\{x_n\}_{n\in\w}$ be an arbitrary sequence in $B_{\ell_{p^\ast}}$. Then, by (ii) of Proposition \ref{p:p-convergent-nes},  there is a subsequence $\{x_{n_k}\}_{k\in\w}$ of $\{x_n\}_{n\in\w}$ and  $x\in B_{\ell_{p^\ast}}$ such that the sequence $\big\{S\big(x_{n_{k}}-x\big)\big\}_{k\in\w}\subseteq E$ is weakly $p$-summable. Since $T$ is $p$-convergent it follows that $TS\big(x_{n_{k}}-x\big)\to 0$ in $E$, and hence $TS\big(x_{n_{k}}\big) \to TS(x)\in TS(B_{\ell_{p^\ast}})$. Therefore $TS(B_{\ell_{p^\ast}})$ is a sequentially compact subset of $L$. As above since any sequentially compact subset is bounded, we obtain that the linear map $TS$ is bounded and hence continuous. In particular, $TS$ is a sequentially compact operator.
\smallskip

(iv)$\Ra$(iii) is trivial.
\smallskip

Assume that $p=1$, $E$ and $L$ are sequentially complete, and $E$ is weakly sequentially $1$-complete.
\smallskip

(i)$\Ra$(v) Let $\{x_n\}_{n\in\w}$ be an arbitrary sequence in $B_{c_0}$. Then, by (iii) of Proposition \ref{p:p-convergent-nes}, there are  a subsequence $\{x_{n_k}\}_{k\in\w}$ of $\{x_n\}_{n\in\w}$ and  $x\in B_{\ell_{p^\ast}}$ such that $S(x_{n_k})$ weakly $1$-converges to $S(x)$. Then the $1$-convergence of $T$ implies $TS\big(x_{n_{k}}\big) \to TS(x)\in TS(B_{\ell_{p^\ast}})$. Therefore the set $T\circ S(B_{c_0})$ is sequentially compact in $E$, and hence, as above, $TS$ is a sequentially compact operator.
\smallskip

(v)$\Ra$(iii) is trivial. \qed
\end{proof}

Since in angelic spaces the class of relatively sequentially compact sets coincides with the class of relatively compact sets, Theorem \ref{c:p-convergent-1} implies the following assertion.
\begin{corollary} \label{c:p-convergent-1}
Let $p\in[1,\infty)$, $E$ be a sequentially complete  locally convex space which is weakly $1$-complete if $p=1$,  and let $L$  be an angelic sequentially complete locally convex space $($for example $L$ is a strict $(LF)$-space$)$. Then for each  weak-weak sequentially continuous linear map $T:E\to L$,  the following assertions are equivalent:
\begin{enumerate}
\item[{\rm(i)}] $T$ is $p$-convergent;
\item[{\rm(ii)}] $T\circ S$ is sequentially precompact for any operator $S:\ell_{p^\ast}^0 \to E$ $($or $S:c_0^0\to E$ if $p=1$$)$.
\item[{\rm(iii)}] $T\circ S$ is sequentially precompact for any operator $S:\ell_{p^\ast} \to E$ $($or $S:c_0\to E$ if $p=1$$)$.
\item[{\rm(iv)}] $T\circ S$ is sequentially compact for any operator $S:\ell_{p^\ast} \to E$ $($or $S:c_0\to E$ if $p=1$$)$.
\item[{\rm(v)}] $T\circ S$ is compact for any operator $S:\ell_{p^\ast} \to E$ $($or $S:c_0\to E$ if $p=1$$)$.
\end{enumerate}
\end{corollary}

The conclusion of Theorem \ref{t:p-convergent-1} is not true in general for the case $p=\infty$ as the following example shows.

\begin{example} \label{exa:p=inf-convergent-1}
Let $p=\infty$, $E=L=\ell_1$, and let $T:E\to L$ and $S:\ell_1^0\to E$ be the identity maps. Then $T$ is $\infty$-convergent, but $T\circ S$ is not even weakly sequentially precompact.
\end{example}

\begin{proof}
Every weakly $\infty$-summable sequence is weakly null, and hence the Schur property of $E$ implies that $T$ is $\infty$-convergent. The Schur property implies also that the sequence $\{ TS(e_n)\}_{n\in\w}$ does not have a weakly Cauchy subsequence. Thus $T\circ S$ is not weakly sequentially precompact.\qed
\end{proof}

The condition of being sequentially complete in Theorem \ref{t:p-convergent-1} is also essential.
We denote by $\mathbf{s}=\{0\}\cup\{\tfrac{1}{n}: n\in\NN\}$ the convergent sequence with the topology induced from $\IR$.
\begin{example} \label{exa:p=1-convergent-1}
Let $p=1$, $E=L=C_p(\mathbf{s})$, $T:E\to L$ be the identity map, and let $S:c_0\to E$ be the identity inclusion defined by $S(\aaa):=f_\aaa$, where $\aaa=(a_n)_{n\in\w}\in c_0$ and
\[
f_\aaa(0)=0 \; \mbox{ and } \; f_\aaa\big(\tfrac{1}{n}\big):=a_{n-1} \; \mbox{ for }\; n\geq 1.
\]
Then $E$ is not  sequentially  complete, $T$ is $1$-convergent, but $T\circ S$ is not sequentially compact.
\end{example}

\begin{proof}
The space $E=C_p(\mathbf{s})$ is not  sequentially  complete by Theorem \ref{t:Cp-lc}. The operator $T$ is $1$-convergent by Proposition \ref{p:p-operator-1}. To show that $TS$ is not sequentially compact, consider the sequence $\{\aaa_n\}_{n\in\w}$ in $B_{c_0}$ where $\aaa_n:=(\underbrace{1,\dots,1}_n,0,\dots)$. Then $TS(\aaa_n)=f_{\aaa_n}$ pointwise converges to the function which is discontinuous at $0$. Therefore the sequence $\big\{TS(\aaa_n)\big\}_n$ does not have a convergent subsequence. Thus $TS(B_{c_0})$ is not a relatively sequentially compact subset of $L$.\qed
\end{proof}



Below we show that in general one cannot replace $\ell^0_{p^\ast}$ in (ii) of Theorem \ref{t:p-convergent-1} by $\ell_{p^\ast}$ without the condition on $E$ being  sequentially  complete.  

\begin{example} \label{exa:p=1-convergent-3}
Let $E=L=\ell_2^0$, $T:E\to L$ be the identity map. Then $E$ is not  sequentially  complete, $T$ is not $2$-convergent, but $TS$ is finite-dimensional for every $S\in\LL(\ell_2,E)$.
\end{example}

\begin{proof}
The normed space $\ell_p^0$ is not even locally complete because it is not complete, see Corollary 5.1.9 of \cite{PB}. To show that $T$ is not $2$-convergent, observe that the canonical basis $\{e_n\}_{n\in\w}$ of $E$ is weakly $2$-summable because for every $\chi\in E'=\ell_2$ we have $(\langle\chi, e_n\rangle)_{n\in\w}=\chi\in \ell_2$. However, $T(e_n)=e_n\not\to 0$ in $L$. Thus $T$ is not $2$-convergent. The last assertion follows from the next claim.

{\em Claim 1: For every $p\in[1,\infty)$, each operator $S\in\LL(\ell_p,\ell_p^0)$ has finite rank.}


Indeed, since $\ell_p^0=\bigcup_{n\in\w} \spn(e_0,\dots,e_n)$, we obtain $\ell_p=\bigcup_{n\in\w} S^{-1}\big(\spn(e_0,\dots,e_n)\big)$. By the Baire property of $\ell_p$, there is $m\in\w$ such that the closed subspace $S^{-1}\big(\spn(e_0,\dots,e_m)\big)$ of $\ell_p$ has a non-empty interior. Thus $\ell_p =S^{-1}\big(\spn(e_0,\dots,e_m)\big)$, and hence $S(\ell_p)$ is finite-dimensional.\qed
\end{proof}

Now we characterize locally convex spaces with the $p$-Schur property.

\begin{proposition} \label{p:P-Schur}
Let $p\in[1,\infty]$, and let $E$ be a locally convex space. Consider the following assertions:
\begin{enumerate}
\item[{\rm(i)}] the identity map $\Id_{E}:  E\to E$ is $p$-convergent, i.e., $E$ has the $p$-Schur property;
\item[{\rm(ii)}] every $($relatively$)$ weakly sequentially $p$-compact subset $A$ of $E$ is $($resp., relatively$)$ sequentially compact in $E$;
\item[{\rm(iii)}]  every   weakly sequentially $p$-precompact subset $A$ of $E$ is sequentially precompact in $E$.
\end{enumerate}
Then {\rm(i)$\LRa$(ii)$\Ra$(iii)}. If in addition $E$ is sequentially complete, then the implication {\rm(iii)$\Ra$(i)} is also satisfied.
\end{proposition}

\begin{proof}
(i)$\Ra$(ii) Assume that $E$ has the $p$-Schur property, and let $A$ be a (relatively) weakly sequentially $p$-compact subset of $E$. Fix an arbitrary sequence $\{a_n\}_{n\in\w}$ in $A$. Then it has a subsequence $\{a_{n_k}\}_{k\in\w}$ weakly $p$-converging  to an element $a\in A$ (resp., $a\in E$). Since $E$ has the $p$-Schur property, we obtain that $a_{n_k}\to a$ in $E$. Thus $A$ is (resp., relatively) sequentially compact in $E$.

(ii)$\Ra$(i) Assume that every (relatively) weakly sequentially $p$-compact subset $A$ of $E$ is (resp., relatively) sequentially compact in $E$. Take an arbitrary weakly $p$-summable sequence $S=\{x_n\}_{n\in\w}$ in $E$. Then $S\cup\{ 0\}$ is a weakly sequentially $p$-compact set. By assumption, $S$ is (resp., relatively)  sequentially compact in $E$. Now suppose for a contradiction that $x_n\not\to 0$ in $E$. Passing to a subsequence if needed, without loss of generality we can assume that there is $U\in\Nn_0(E)$ such that $x_n\not\in U$ for every $n\in\w$. Since $S$ is relatively  sequentially compact in $E$, there is a subsequence $\{x_{n_k}\}_{k\in\w}$ of $\{x_n\}_{n\in\w}$ which converges (in $E$) to some element $x$. Clearly, $x\not=0$ and $x_{n_k}\to x$ weakly. On the other hand, since $S$ is weakly null, $\{x_{n_k}\}_{k\in\w}$ is also weakly null and hence it can converge only to zero. Whence $x=0$, a contradiction. Therefore $S$ is a null-sequence. Thus $E$ has the $p$-Schur property.

(i)$\Ra$(iii)  Assume that $E$ has the $p$-Schur property, and let $A$ be  a weakly sequentially $p$-precompact subset of $E$. Let $\{a_n\}_{n\in\w}$ be a sequence in $A$. Then it has  a  weakly $p$-Cauchy subsequence $\{a_{n_k}\}_{k\in\w}$. Therefore for every sequence $(k_j)$ in $\w$, the sequence $\{a_{n_{k_j}}-a_{n_{k_{j+1}}}\}_{k\in\w}$ is weakly $p$-summable. By the $p$-Schur property, $a_{n_{k_j}}-a_{n_{k_{j+1}}}\to 0$ in $E$. But this means that $\{a_{n_k}\}_{k\in\w}$ is a Cauchy sequence in $E$. Thus $A$ is sequentially precompact.

(iii)$\Ra$(i) Assume that $E$ is sequentially complete and every weakly sequentially $p$-precompact subset $A$ of $E$ is sequentially precompact in $E$. Take an arbitrary weakly $p$-summable sequence $S=\{x_n\}_{n\in\w}$ in $E$. Then $S$ is weakly sequentially $p$-precompact and hence sequentially precompact in $E$. Suppose for a contradiction that $x_n\not\to 0$ in $E$. Then, passing to a subsequence if needed, we assume that there is an open $U\in\Nn_0(E)$ such that $x_n\not\in U$ for every $n\in\w$. Since $S$ is sequentially precompact, there is a Cauchy subsequence $\{x_{n_k}\}_{k\in\w}$ of $\{x_n\}_{n\in\w}$. As $E$ is sequentially complete, there is $x\in E\SM U$ such that $x_{n_k}\to x$.  Since $S$ is weakly null, $\{x_{n_k}\}_{k\in\w}$ is also weakly null and hence it can converge only to zero. As $x\not= 0$, we obtain a contradiction. Thus $x_n\to 0$ in $E$, and hence $E$ has the $p$-Schur property.\qed
\end{proof}

Theorem \ref{t:p-convergent-1} applied to the identity map $T=\Id_{E}:  E\to E$ immediately implies the following operator  characterization of the $p$-Schur property.
\begin{corollary} \label{c:p-conver-p-Schur}
Let $p\in[1,\infty)$, and let $E$ be a locally convex space. Then the following assertions are equivalent:
\begin{enumerate}
\item[{\rm(i)}] the identity operator $\Id_{E}:  E\to E$ is $p$-convergent, i.e.,  $E$ has the $p$-Schur property;
\item[{\rm(ii)}] any operator $S:\ell_{p^\ast}^0 \to E$ $($or $S:c_0^0\to E$ if $p=1$$)$ is sequentially precompact.
\end{enumerate}
If $E$ is sequentially complete, then {\rm(i)} and {\rm(ii)} are equivalent to the following
\begin{enumerate}
\item[{\rm(iii)}] any operator $S:\ell_{p^\ast} \to E$ $($or $S:c_0\to E$ if $p=1$$)$ is sequentially precompact.
\end{enumerate}
If $1<p<\infty$ and $E$ is  sequentially complete, then {\rm(i)}--{\rm(iii)}  are equivalent to the following
\begin{enumerate}
\item[{\rm(iv)}] any operator $S:\ell_{p^\ast} \to E$ is a sequentially compact.
\end{enumerate}
If $p=1$ and  $E$ is  sequentially complete and weakly sequentially $1$-complete, then {\rm(i)}--{\rm(iii)}  are equivalent to
\begin{enumerate}
\item[{\rm(v)}] any operator $S:c_0 \to E$ is sequentially compact.
\end{enumerate}
\end{corollary}

\begin{remark} {\em
It is natural to ask whether $T^{\ast\ast}$ of a $p$-convergent operator is also $p$-convergent. A negative answer to this question is given in Theorem 2.7 of \cite{CCDL}.\qed}
\end{remark}


\section{Pe{\l}czy\'{n}ski's type sets and $p$-convergent operators} \label{sec:small-bound-p-conv}


Let $E$ and $L$ be locally convex spaces, and let $T:E\to L$ be an operator.
According to Lemma \ref{l:V*-set-1}, 
if $A$ is a $(p,q)$-$(V^\ast)$ set then so is its image $T(A)$, and, by Lemmas  \ref{l:V-set-1} and \ref{l:*V-set-1}, if $B$ is a $(p,q)$-$(V)$ set or a weak$^\ast$ $(p,q)$-$(V)$ set in $E'$, then so is $T^\ast(B)$.  Taking into account that $A$ is a bounded subset of $E$ and $B$ is a weak$^\ast$ bounded subset of $E'$, these results motivate the following problems:
\begin{enumerate}
\item[(1)] Characterize those operators $T$ which map {\em all } bounded  sets into $(p,q)$-$(V^\ast)$ sets.
\item[(2)] Characterize those operators $T$ such that $T^\ast$ maps {\em all }  weak$^\ast$ bounded sets or strongly bounded sets into  $(p,q)$-$(V)$ sets or weak$^\ast$ $(p,q)$-$(V)$ sets.
\end{enumerate}

The purpose of this section is to give complete answers to problems (1) and (2) for the $(p,\infty)$-case in the general case when $E$ and $L$ are locally convex spaces. We are interested in this special case also because it is dually connected with $p$-convergent operators, see in particular  Theorems \ref{t:bounded-to-p-V*} 
and \ref{t:p-convergent-2}.

If $1\leq p<\infty$, a characterization of operators $T$ between Banach spaces for which $T^\ast$ is $p$-convergent was obtained by Ghenciu \cite{Ghenciu-pGP}. Below we generalize this result.

\begin{theorem} \label{t:bounded-to-p-V*}
Let $p\in[1,\infty]$, and let $T:E\to L$ be an operator between  locally convex spaces  $E$ and $L$. Then the following assertions are equivalent:
\begin{enumerate}
\item[{\rm(i)}] for every $B\in\Bo(E)$, the image $T(B)$ is a $p$-$(V^\ast)$ set in $L$;
\item[{\rm(ii)}] $T^\ast:L'_\beta \to E'_\beta$ is $p$-convergent.
\end{enumerate}
If $1\leq p<\infty$, then {\em(i)} and {\em(ii)} are equivalent to
\begin{enumerate}
\item[{\rm(iii)}] $T^\ast\circ S$ is sequentially precompact for any operator $S:\ell_{p^\ast}^0 \to L'_\beta$ $($or $S:c_0^0 \to L'_\beta$ if $p=1$$)$.
\end{enumerate}
If $1< p<\infty$ and  $L'_\beta$ is sequentially complete, then {\rm(i)-(iii)} are equivalent to the following
\begin{enumerate}
\item[{\rm(iv)}] $T^\ast\circ S$ is sequentially compact for any operator $S:\ell_{p^\ast} \to L'_\beta$;
\item[{\rm(v)}] $T^\ast\circ S$ is sequentially precompact for any operator $S:\ell_{p^\ast} \to L'_\beta$.
\end{enumerate}
If $p=1$, $E'_\beta$ and $L'_\beta$ are sequentially complete and $L'_\beta$ is weakly sequentially $1$-complete, then {\rm(i)-(iii)} are equivalent to the following
\begin{enumerate}
\item[{\rm(vi)}] $T^\ast\circ S$ is sequentially compact for any operator $S:c_0 \to L'_\beta$;
\item[{\rm(vii)}] $T^\ast\circ S$ is sequentially precompact for any operator $S:c_0 \to L'_\beta$.
\end{enumerate}

\end{theorem}

\begin{proof}
(i)$\Ra$(ii) Let $\{\chi_n\}_{n\in\w}$ be a weakly $p$-summable sequence in $L'_\beta$. We show that $T^\ast(\chi_n)\to 0$ in $E'_\beta$. To this end, fix an arbitrary $B\in\Bo(E)$. Since $T(B)$ is a  $p$-$(V^\ast)$ set in $L$ we have
\[
\sup_{b\in B} \big|\langle T^\ast (\chi_n), b\rangle\big|= \sup_{b\in B} \big|\langle\chi_n, T(b)\rangle\big| \to 0 \; \mbox{ as }\; n\to\infty,
\]
and hence $T^\ast (\chi_n)\in B^\circ$ for all sufficiently large $n\in\w$. Since $B$ was arbitrary this means that $T^\ast (\chi_n)\to 0$ in $E'_\beta$, as desired.

(ii)$\Ra$(i) Let $B\in\Bo(E)$. To show that $T(B)$ is  a  $p$-$(V^\ast)$ set in $L$, take any weakly $p$-summable sequence $\{\chi_n\}_{n\in\w}$ in $L'_\beta$. For every $\e>0$, the polar $\e B^\circ =\big(\tfrac{1}{\e}B\big)^\circ$ is a neighborhood of zero in $E'_\beta$. Since $T^\ast$ is  $p$-convergent, we have $T^\ast(\chi_n)\to 0$ in $E'_\beta$ and hence there is $N_\e\in\w$ such that $T^\ast(\chi_n) \in \e B^\circ$ for all $n\geq N_\e$. Therefore
\[
\sup_{b\in B} \big|\langle\chi_n, T(b)\rangle\big| = \sup_{b\in B} \big|\langle T^\ast (\chi_n), b\rangle\big| \leq \e  \; \mbox{ for all }\; n\geq N_\e.
\]
As $\e$ was arbitrary it follows that $\sup_{b\in B} \big|\langle\chi_n, T(b)\rangle\big|\to 0$. Thus  $T(B)$ is  a  $p$-$(V^\ast)$ set.
\smallskip

All other assertions and equivalences follow from Theorem \ref{t:p-convergent-1} applied to $E_1=L'_\beta$, $L_1=E'_\beta$ and $T_1=T^\ast$.\qed
\end{proof}

Theorem \ref{t:bounded-to-p-V*} applied to the identity map $T=\Id_{E}:  E \to E$ immediately implies the next corollary in which the equivalence (i)$\Leftrightarrow$(ii) is proved in Theorem \ref{t:Bo=Vp}.

\begin{corollary} \label{c:Bo=V*}
Let  $p\in[1,\infty]$,  and let $E$ be a locally convex space. Then the following conditions are equivalent
\begin{enumerate}
\item[{\rm(i)}] every bounded subset of $E$ is a $p$-$(V^\ast)$ set $\big($i.e., $\Bo(E)=\mathsf{V}^\ast_p(E)$$\big)$;
\item[{\rm(ii)}] the identity map $\Id_{E'}:  E'_\beta \to E'_\beta$ is $p$-convergent, i.e., $E'_\beta$ has the $p$-Schur property.
\end{enumerate}
If $1\leq p<\infty$, then {\em(i)} and {\em(ii)} are equivalent to
\begin{enumerate}
\item[{\rm(iii)}] any operator $S:\ell_{p^\ast}^0 \to E'_\beta$ $($or $S:c_0^0 \to E'_\beta$ if $p=1$$)$ is sequentially precompact.
\end{enumerate}
If $1< p<\infty$ and  $E'_\beta$ is sequentially complete, then {\rm(i)-(iii)} are equivalent to the following
\begin{enumerate}
\item[{\rm(iv)}] any operator $S:\ell_{p^\ast} \to E'_\beta$  is sequentially compact;
\item[{\rm(v)}] any operator $S:\ell_{p^\ast} \to E'_\beta$ is sequentially precompact.
\end{enumerate}
If $p=1$, $E'_\beta$ is sequentially complete and weakly sequentially $1$-complete, then {\rm(i)-(iii)} are equivalent to the following
\begin{enumerate}
\item[{\rm(vi)}] any operator $S:c_0 \to E'_\beta$ is sequentially compact;
\item[{\rm(vii)}] any operator $S:c_0 \to E'_\beta$ is sequentially precompact.
\end{enumerate}
\end{corollary}


Let  $1\leq p\leq q\leq\infty$, and let $E$ be a locally convex space. According to Lemma \ref{l:V*-set-1} the family $\mathsf{V}^\ast_{(p,q)}(E)$ of all $(p,q)$-$(V^\ast)$ sets in  $E$ is a saturated bornology containing all finite subsets of $E$. Therefore one can naturally define:

\begin{definition} \label{def:V-L-topology} {\em
Let $1\leq p\leq q\leq\infty$, and let $E$ be a locally convex space. Denote by $V^\ast_{(p,q)}(E',E)$  the polar topology on $E'$ of uniform convergence on $(p,q)$-$(V^\ast)$ sets. Set $V^\ast_{p}(E',E)=V^\ast_{(p,\infty)}(E',E)$ and $V^\ast(E',E)=V^\ast_{(\infty,\infty)}(E',E)$.\qed}
\end{definition}
Since the family $\mathsf{V}^\ast_{(p,q)}(E)$ depends only on the duality, the topology $V^\ast_{(p,q)}(E',E)$ is a dual topology. For further references we select the next simple lemma.

\begin{proposition} \label{p:V*-topology}
Let  $1\leq p\leq q\leq\infty$, and let $E$ be a locally convex space. Then:
\begin{enumerate}
\item[{\rm(i)}] $\sigma(E',E)\subseteq V^\ast_{(p,q)}(E',E) \subseteq \beta(E',E)$,
\item[{\rm(ii)}] $V^\ast_{(p,q)}(E',E) \subseteq \mu(E',E)$ if and only if every $(p,q)$-$(V^\ast)$ set $A$ in $E$ is relatively weakly compact. Consequently, $E$ has the property $V^\ast_{(p,q)}$ if and only if the topology $V^\ast_{(p,q)}(E',E)$ is compatible with the duality $(E,E')$.
\end{enumerate}
\end{proposition}

\begin{proof}
(i) is clear, 
and
(ii) follows from the Mackey--Arens theorem and the fact that $\mathsf{V}^\ast_{(p,q)}(E)$ is a saturated bornology. \qed
\end{proof}

The next result naturally complements Theorem \ref{t:bounded-to-p-V*}.
\begin{proposition} \label{p:T*-p-convergent}
Let $1\leq p\leq q\leq\infty$, and let $T:E\to L$ be an operator between locally convex spaces  $E$ and $L$. Then:
 \begin{enumerate}
\item[{\rm(i)}] the adjoint map $T^\ast: \big( L',V^\ast_{(p,q)}(L',L)\big) \to \big(E', V^\ast_{(p,q)}(E',E)\big)$ is continuous;
\item[{\rm(ii)}] the adjoint map $T^\ast: L'_\beta \to \big(E', V^\ast_{p}(E',E)\big)$ is $p$-convergent.
\end{enumerate}
\end{proposition}

\begin{proof}
(i) To show that $T^\ast$ is continuous, let $A^\circ$ be a standard $V^\ast_{(p,q)}(E',E)$-neighborhood of zero, where $A$ is a $(p,q)$-$(V^\ast)$ set in $E$. By Lemma \ref{l:V*-set-1}, $T(A)$ is a $(p,q)$-$(V^\ast)$ set in $L$. Then for every $\eta\in T(A)^\circ$ and each $a\in A$, we have $|\langle T^\ast(\eta),a\rangle|= |\langle \eta,T(a)\rangle|\leq 1$ and hence $T^\ast\big(T(A)^\circ\big) \subseteq A^\circ$. Thus $T^\ast$ is continuous.

(ii) Let $\{\chi_n\}_{n\in\w}$ be a weakly $p$-summable sequence in $L'_\beta$. To show that $T^\ast(\chi_n)\to 0$ in $\big(E', V^\ast_p(E',E)\big)$, fix an arbitrary $B\in\mathsf{V}^\ast_p(E)$. By Lemma \ref{l:V*-set-1}, $T(B)$ is a  $p$-$(V^\ast)$ set in $L$ and hence
\[
\lim_{n\to\infty}\sup_{b\in B} \big|\langle T^\ast (\chi_n), b\rangle\big|= \lim_{n\to\infty}\sup_{b\in B} \big|\langle\chi_n, T(b)\rangle\big| = 0 .
\]
Therefore $T^\ast (\chi_n)\in B^\circ$ for all sufficiently large $n\in\w$. Since $B$ was arbitrary this means that $T^\ast (\chi_n)\to 0$ in $V^\ast_p(E',E)$, as desired.\qed
\end{proof}

Below we give sufficient and necessary conditions on operators to be $p$-convergent.
\begin{theorem} \label{t:p-convergent-2}
Let $p\in[1,\infty]$, $E$ and $L$ be locally convex spaces, and let $T\in\LL(E,L)$.
\begin{enumerate}
\item[{\rm(i)}] If  $T^\ast(B)$ is a $p$-$(V)$ set in $E'$ for every $B\in\Bo(L'_\beta)$, then $T$ is $p$-convergent. The converse assertion is true if $L$ is quasibarrelled.
\item[{\rm(ii)}] If $T^\ast(B)$ is a $p$-$(V)$ set in $E'$  for every $B\in\Bo(L'_{w^\ast})$, then $T$ is $p$-convergent.  The converse assertion is true if $L$ is barrelled.
\end{enumerate}
\end{theorem}

\begin{proof}
Assume that $T^\ast(B)$ is a $p$-$(V)$ set in $E'$ for every $B\in\Bo(L'_\beta)$ (resp., $B\in\Bo(L'_{w^\ast})$). Fix an arbitrary  weakly $p$-summable sequence $\{x_n\}_{n\in\w}$ in $E$. We have to show that $T(x_n)\to 0$ in $L$. To this end, fix $U\in\Nn_0^c(L)$. Then the polar $B:=U^\circ$ is strongly (hence weak$^\ast$) bounded by Theorem~11.3.5 of \cite{NaB}, and hence $T^\ast(B)$ is a $p$-$(V)$ set in $E'$. Consequently, by the definition of $p$-$(V)$ sets, we have
\[
\lim_{n\to\infty} \sup_{\chi\in B} |\langle \chi, T(x_n)\rangle|=\lim_{n\to\infty} \sup_{\chi\in B} |\langle T^\ast(\chi), x_n\rangle|= 0.
\]
In  particular, for all sufficiently large $n\in\w$, it follows that $|\langle \chi, T(x_n)\rangle|\leq 1$ for every $\chi\in B=U^\circ$. This means that $T(x_n)\in U^{\circ\circ}=U$ for all  sufficiently large $n\in\w$. Thus $T(x_n)\to 0$ in $L$ and hence $T$ is $p$-convergent.
\smallskip

Conversely, assume additionally that $L$ is quasibarrelled (resp., barrelled) and let $T$ be $p$-convergent. Let $B\in\Bo(L'_\beta)$ (resp., $B\in\Bo(L'_{w^\ast})$).  We have to show that
\[
\lim_{n\to\infty} \sup_{\chi\in B} |\langle T^\ast(\chi), x_n\rangle|=\lim_{n\to\infty} \sup_{\chi\in B} |\langle \chi, T(x_n)\rangle|=0
\]
for every weakly $p$-summable sequence $\{x_n\}_{n\in\w}$ in $E$. To this end, fix $\e>0$. Since $T$ is $p$-convergent, $T(x_n)\to 0$ in $L$. As $L$ is quasibarrelled (resp., barrelled) the set $B$ is equicontinuous. 
Take $U\in \Nn_0(L)$ such that $B\subseteq U^\circ$. Choose $N_\e\in\w$ such that $T(x_n)\in \e U$ for every $n\geq N_\e$. Then for every $n\geq N_\e$, we obtain
\[
\sup_{\chi\in B} |\langle \chi, T(x_n)\rangle|\leq \sup\big\{\sup_{\chi\in B} |\langle \chi, y\rangle|: y\in \e U\big\} \leq \e.
\]
Thus $ \sup_{\chi\in B} |\langle T^\ast(\chi), x_n\rangle|\to 0$, as desired.\qed
\end{proof}

\begin{corollary} \label{c:p-converg-adjoint-weakly-compact}
Let $p\in[1,\infty]$, $E$ be a locally convex space with the property $V_p$ {\rm(}resp., the property  $sV_p${\rm)}, and let $L$ be a quasibarrelled space. If an operator $T:E\to L$ is $p$-convergent, then the adjoint $T^\ast: L'_\beta\to E'_\beta$ maps bounded sets of $ L'_\beta$ into relatively weakly {\rm(}resp., sequentially{\rm)} compact sets of $E'_\beta$. Consequently, if $L$ is a normed space and  $T$ is $p$-convergent, then $T^\ast$ is weakly compact {\rm(}resp., weakly sequentially compact{\rm)}.
\end{corollary}

Now we summarize characterizations of the $p$-Schur property of an lcs $E$ and its strong dual $E'_\beta$.
\begin{corollary} \label{c:Bo=Vp}
Let  $p\in[1,\infty]$,  and let $E$ be a locally convex space. Then:
\begin{enumerate}
\item[{\rm(i)}] if $E$ is quasibarrelled, then $\Bo(E'_\beta)\subseteq \mathsf{V}_p(E')$ if and only if $E$ has the $p$-Schur property;
\item[{\rm(ii)}] if $E$ is barrelled, then $\Bo(E'_{w^\ast})=\mathsf{V}_p(E')$ if and only  $E$ has the $p$-Schur property;
\item[{\rm(iii)}] $\Bo(E)=\mathsf{V}_p^\ast(E)$ if and only if $E'_\beta$ has the $p$-Schur property.
\end{enumerate}
\end{corollary}

\begin{proof}
(i) follows from Theorem \ref{t:p-convergent-2} applied  to the identity map $\Id_E:E\to E$.

(ii) By Theorem \ref{t:p-convergent-2} applied  to the identity map $\Id_E:E\to E$, we obtain that $\Bo(E'_{w^\ast})\subseteq \mathsf{V}_p(E')$ if and only  $E$ has the $p$-Schur property. It remains to note that $\mathsf{V}_p(E')\subseteq \Bo(E'_{w^\ast})$ always by (ii) of Lemma \ref{l:V-set-1}.

(iii) is Theorem \ref{t:Bo=Vp}.\qed
\end{proof}

Below we give a useful construction of $p$-convergent operators from $\ell_1^0$ into locally convex spaces.

\begin{proposition} \label{p:seq-p-convergent}
Let  $p\in[1,\infty]$, $\{x_n\}_{n\in\w}$ be a bounded sequence in a locally convex space $(E,\tau)$, 
and let $T:\ell_1^0 \to E$ be a linear map defined by
\[
T(a_0 e_0+\cdots+a_ne_n):=a_0 x_0+\cdots+ a_n x_n \quad (n\in\w, \; a_0,\dots,a_n\in\IF),
\]
where $\{e_n\}$ is the canonical unit basis of $\ell_1^0$. Then $T$ is continuous. Moreover, if $E$ is locally complete, then $T$ can be extended to a continuous operator from $\ell_1$ to $E$. In any case,  if $\{x_n\}_{n\in\w}$ is a $p$-$(V^\ast)$ set, then $T^\ast: E'_\beta\to\ell_\infty$ is $p$-convergent.
\end{proposition}

\begin{proof}
Denote by $A$ the closed absolutely convex hull of $\{x_n\}_{n\in\w}$. Then $A$ is a bounded subset of $E$ (in the cases when  $\{x_n\}_{n\in\w}$ is a $p$-$(V^\ast)$ set, Lemma  \ref{l:V*-set-1} 
implies that  $A$ is also a $p$-$(V^\ast)$ set). 
Consider the normed space $(E_{A},\|\cdot\|)$ and observe that 
$\|x_n\|\leq 1$ for every $n\in\w$. Therefore, if $|a_0|+\cdots+|a_n|\leq 1$ we obtain
\[
\big\|T(a_0 e_0+\cdots+a_ne_n)\big\|= \|a_0 x_0+\cdots+ a_n x_n\|\leq 1
\]
and hence $T:\ell_1^0 \to (E_{A},\|\cdot\|)$ is continuous. In the case when $E$ is locally complete, by Proposition \ref{p:bounded-norm},  $(E_{A},\|\cdot\|)$ is a Banach space and hence $T$ can be extended to a continuous operator from $\ell_1$ to $E_A$. Since, by Proposition \ref{p:bounded-norm}, the norm topology  on $E_A$ is finer than the original topology $\tau{\restriction}_{E_A}$ it follows that $T:\ell_1^0 \to E$ (or $T:\ell_1 \to E$ if $E$ is locally complete) is continuous.

Assume now that  $\{x_n\}_{n\in\w}$ is a $p$-$(V^\ast)$ set.
To show that the adjoint operator $T^\ast$ is $p$-convergent, let $\{\chi_n\}_{n\in\w}$ be a weakly  $p$-summable sequence in $E'_\beta$. Since $A$ is a $p$-$(V^\ast)$ set, by definition, we have $\lim_{n\to\infty} \sup_{x\in A} |\langle\chi_n,x\rangle|=0$. Therefore
\[
\| T^\ast(\chi_n)\|_{\ell_\infty}= \sup_{k\in\w} \big|\langle T^\ast(\chi_n),e_k\rangle\big|= \sup_{k\in\w} \big|\langle \chi_n,x_k\rangle\big| \leq \sup_{x\in A} \big|\langle \chi_n,x\rangle\big| \to 0\;\; \mbox{ as } \; n\to\infty.
\]
Thus $\big\{T^\ast(\chi_n)\big\}_{n\in\w}$ is a null sequence in $\ell_\infty$, and hence $T^\ast$ is $p$-convergent.\qed
\end{proof}

We know  that $p$-$(V^\ast)$ sets are bounded. Below we give an operator characterization of those spaces $E$ in which the $p$-$(V^\ast)$ sets have a stronger topological property than just being bounded as, for example, being weakly sequentially (pre)compact. The following theorem extends Theorem 15 of \cite{Ghenciu-pGP}.

\begin{theorem}  \label{t:p-V*-precompact}
Let $p\in[1,\infty]$, and let $(E,\tau)$ be a locally convex space. Then the following assertions are equivalent:
\begin{enumerate}
\item[{\rm(i)}] for every normed space $Y$, if $T\in\LL(Y,E)$ is such that the adjoint operator $T^\ast:E'_\beta\to Y'_\beta$ is $p$-convergent, then $T$ is weakly sequentially precompact $($resp., weakly sequentially compact, sequentially precompact or sequentially compact$)$;
\item[{\rm(ii)}] the same as {\rm(i)} with $Y=\ell_1^0$;
\item[{\rm(iii)}] every $p$-$(V^\ast)$ subset of $E$ is weakly sequentially precompact $($resp., relatively weakly sequentially compact,  sequentially precompact or relatively  sequentially compact$)$.
\end{enumerate}
If additionally $E$ is locally complete, then {\em (i)-(iii)} are equivalent to
\begin{enumerate}
\item[{\rm(iv)}]  the same as {\rm(i)} with $Y=\ell_1$.
\end{enumerate}
\end{theorem}

\begin{proof}
(i)$\Ra$(ii) is clear.

(ii)$\Ra$(iii) Let $A$ be a $p$-$(V^\ast)$ subset of $E$.  To show that $A$ is weakly sequentially precompact (resp., relatively weakly sequentially compact, sequentially precompact or relatively sequentially compact), take an arbitrary sequence $\{x_n\}_{n\in\w}$ in $A$ and define $T:\ell_1^0 \to E$ by
\[
T(a_0 e_0+\cdots+a_ne_n):=a_0 x_0+\cdots+ a_n x_n \quad (n\in\w, \; a_0,\dots,a_n\in\IF).
\]
Then, by Proposition \ref{p:seq-p-convergent}, $T$ is continuous and such that $T^\ast: E'_\beta\to\ell_\infty$ is $p$-convergent.

By the assumption of (ii), we obtain that $T$ is weakly sequentially precompact (resp., weakly sequentially compact,  sequentially precompact or sequentially compact). By definition this means that the sequence $\{x_n\}=\big\{T(e_n)\big\}$ has a  subsequence which is weakly Cauchy  (resp., weakly convergent in $E$, Cauchy in $E$  or convergent in $E$). Thus $A$ is a weakly sequentially precompact (resp., relatively  weakly sequentially compact,  sequentially precompact or relatively  sequentially compact) subset of $L$, as desired.

(iii)$\Ra$(i) To prove this implication it suffices to show that $T(B_Y)$ is weakly sequentially precompact (resp.,  relatively weakly sequentially compact,  sequentially precompact or relatively sequentially compact) in $E$ under the assumption that $T^\ast$ is $p$-convergent. To this end, we note that, by the equivalence (i)$\Leftrightarrow$(ii) of Theorem \ref{t:bounded-to-p-V*}, the $p$-convergence of $T^\ast$ and the boundedness of $B_Y$ imply that the set $T(B_Y)$ is a $p$-$(V^\ast)$ set, and hence, by (iii), $T(B_Y)$ is  weakly sequentially precompact (resp.,  relatively weakly sequentially compact,  sequentially precompact or  relatively sequentially compact), as desired.
\smallskip

Assume now that $E$ is locally complete. Then the implication (i)$\Ra$(iv) is evident. To prove the implication (iv)$\Ra$(ii), let $T\in\LL(\ell_1^0,E)$ be such that $T^\ast:E'_\beta\to\ell_\infty$ is $p$-convergent. Then $\{T(e_n)\}_{n\in\w}$ is a bounded subset of $E$, and hence, by Proposition \ref{p:bounded-norm},  the space $E_B$ is Banach, where $B$ is the closed absolutely convex hull of $\{T(e_n)\}_{n\in\w}$. Since $T\big( B_{\ell_1^0} \big) \subseteq B_{E_B}$ the operator $T:\ell_1^0\to E_B$ is continuous. As $E_B$ is complete, $T$ can be extended to an operator $\bar T$ from $\ell_1$ to $E_B$. Since, by Proposition \ref{p:bounded-norm}, the topology of $E_B$ is stronger than the original topology on $E_B$ it follows that $\bar T$ belongs to $\LL(\ell_1,E)$. Observe that $(\bar T)^\ast=T^\ast$ because $\ell_1^0$ is dense in $\ell_1$, and hence $(\bar T)^\ast$ is $p$-convergent. Therefore $\bar T$ is weakly sequentially precompact (resp., weakly sequentially compact,  sequentially precompact or sequentially compact), and hence so is the operator $T$. \qed
\end{proof}

For the weakly sequentially precompact case and the weakly sequentially compact case Theorem \ref{t:p-V*-precompact} can be reformulated as an operator characterization of the property $wsV^\ast_p$ and the property $sV^\ast_p$, respectively.
\begin{theorem}  \label{t:p-V*-precompact-V*}
Let $p\in[1,\infty]$, and let $(E,\tau)$ be a locally convex space. Then the following assertions are equivalent:
\begin{enumerate}
\item[{\rm(i)}] for every normed space $Y$, if $T\in\LL(Y,E)$ is such that the adjoint operator $T^\ast:E'_\beta\to Y'_\beta$ is $p$-convergent, then $T$ is weakly sequentially precompact $($resp., weakly sequentially compact$)$;
\item[{\rm(ii)}] the same as {\rm(i)} with $Y=\ell_1^0$;
\item[{\rm(iii)}] the space $E$ has the property $wsV^\ast_p$ $($resp.,  the property $sV^\ast_p$$)$.
\end{enumerate}
If additionally $E$ is locally complete, then {\em (i)-(iii)} are equivalent to
\begin{enumerate}
\item[{\rm(iv)}]  the same as {\rm(i)} with $Y=\ell_1$.
\end{enumerate}
\end{theorem}

The following proposition complements Theorem \ref{t:p-V*-precompact}.

\begin{proposition} \label{p:weak-sDPp}
Let $p\in[1,\infty]$, and let $E$ be a locally convex space. Then the following assertions are equivalent:
\begin{enumerate}
\item[{\rm(i)}] each relatively weakly sequentially $p$-compact set in $E$ is a $p$-$(V^\ast)$ set;
\item[{\rm(ii)}] if $\{x_n\}_{n\in\w}\subseteq E$ is weakly $p$-summable and $\{\chi_n\}_{n\in\w}$ is weakly $p$-summable in $E'_\beta$, then
\[
\lim_{n\to\infty} \langle\chi_n,x_n\rangle=0.
\]
\end{enumerate}
\end{proposition}

\begin{proof}
(i)$\Ra$(ii) The sequence $\{x_n\}_{n\in\w}$ is trivially weakly sequentially $p$-compact. Therefore, by (i), $\{x_n\}_{n\in\w}$ is a $p$-$(V^\ast)$ set. Therefore
\[
\lim_{n\to\infty} |\langle\chi_n,x_n\rangle| \leq \lim_{n\to\infty} \sup_{k\in\w} |\langle\chi_n,x_k\rangle|=0.
\]

(ii)$\Ra$(i) Let $A$ be a relatively weakly sequentially $p$-compact set in $E$, and suppose for a contradiction that $A$ is not a $p$-$(V^\ast)$ set. Then there is a weakly $p$-summable sequence $\{\chi_n\}_{n\in\w}$  in $E'_\beta$ such that $\lim_{n\to\infty} \sup_{x\in A} |\langle\chi_n,x\rangle| \not= 0$. Passing to a subsequence if needed, we can assume that there are $a>0$  and a sequence $\{x_n\}_{n\in\w}$ in $A$ such that
\[
 |\langle\chi_n,x_n\rangle|\geq a \quad \mbox{ for every } n\in\w.
\]
Since $A$ is relatively weakly sequentially $p$-compact and once more passing to a subsequence, we can assume that $\{x_n\}_{n\in\w}$ weakly $p$-converges to some point $x\in E$. Therefore $\{x_n-x\}_{n\in\w}$ is weakly $p$-summable. As $\{\chi_n\}_{n\in\w}$ is weakly $p$-summable, it follows $\lim_{n\in\w}\langle\chi_n,x\rangle=0$. Therefore, by (ii), we obtain
\[
a \leq |\langle\chi_n,x_n\rangle|\leq |\langle\chi_n,x_n-x\rangle|+|\langle\chi_n,x\rangle|\to 0,
\]
a contradiction. \qed
\end{proof}

Let $E$ and $Y$ be locally convex spaces, and let $T:Y\to E$ be an operator. If $B$ is a $p$-$(V)$ set in $E'$, then by (iv) of Lemma \ref{l:V-set-1}, its image $T^\ast(B)$ is a $p$-$(V)$ set in $Y'$. Analogously to Theorems \ref{t:bounded-to-p-V*} 
and \ref{t:p-convergent-2} it is natural to describe some classes of operators $T$ for which the image $T^\ast(B)$ has some additional (topological) properties as, for example,  being a relatively (weakly) sequentially [pre]compact set. Below we obtain some necessary and sufficient conditions.

 We start from the following construction of $p$-convergent operators into $\ell_\infty$ or $c_0$.

\begin{lemma} \label{l:p-operator-L-inf}
Let $E$ be a locally convex space, $\{\chi_n\}_{n\in\w}\subseteq E'$ be a weak$^\ast$ bounded $($resp., weak$^\ast$ null$)$ sequence,  and let $S:E\to \ell_\infty$ $($resp., $S:E\to c_0$$)$ be a  linear map defined  by
\[
S(x):=\big(\langle\chi_n,x\rangle\big)_{n\in\w} \quad (x\in E).
\]
Then:
\begin{enumerate}
\item[{\rm(i)}] $S$ is continuous if and only if $\{\chi_n\}_{n\in\w}$ is equicontinuous.
\item[{\rm(ii)}] if $p\in [1,\infty]$, then $S$ is $p$-convergent if and only if $\{\chi_n\}_{n\in\w}$ is a $p$-$(V)$ set.
\end{enumerate}
\end{lemma}

\begin{proof}
(i) Assume that $S$ is continuous. Then for every $\e>0$, there is $V\in\Nn_0(E)$ such that $S(V)\subseteq \e B_{\ell_\infty}$ (resp., $S(V)\subseteq \e B_{c_0}$).  This means that $|\langle\chi_n,x\rangle|\leq \e$ for every $n\in\w$ and each $x\in V$. Thus $\{\chi_n\}_{n\in\w}$ is equicontinuous. Conversely, assume that $\{\chi_n\}_{n\in\w}$ is equicontinuous. Then  there is $U\in\Nn_0(E)$ such that $\{\chi_n\}_{n\in\w}\subseteq U^\circ$. Then $S(U)\subseteq B_{\ell_\infty}$  (resp., $S(U)\subseteq B_{c_0}$), and hence $S$ is continuous.

(ii) Assume that $S$ is $p$-convergent. To show that the sequence $\{\chi_n\}_{n\in\w}$ is a $p$-$(V)$ set, fix a weakly $p$-summable sequence $\{x_n\}_{n\in\w}$ in $E$. Since $S$ is $p$-convergent, $S(x_n)\to 0$ in $\ell_\infty$. Therefore
\[
\sup_{n\in\w} \big| \langle\chi_n,x_k\rangle\big|=\big\| S(x_k)\big\|_{\ell_\infty}\to 0 \quad \mbox{ as }\;\; k\to\infty,
\]
and hence $\{\chi_n\}_{n\in\w}$ is a $p$-$(V)$ set. Conversely, assume that $\{\chi_n\}_{n\in\w}$ is a $p$-$(V)$ set. To show that $S$ is $p$-convergent, let $\{x_n\}_{n\in\w}$ be a weakly $p$-summable sequence in $E$. Then, by the definition of $p$-$(V)$ sets, we have
\[
\lim_{k\to\infty} \big\| S(x_k)\big\|_{\ell_\infty} = \lim_{k\to\infty} \sup_{n\in\w} \big| \langle\chi_n,x_k\rangle\big|=0
\]
and hence $S$ is $p$-convergent.\qed
\end{proof}

Now we consider an important case when $Y$ is a normed space.

\begin{proposition} \label{p:image-p-V-set}
Let $p\in[1,\infty]$, $E$ be an $\ell_\infty$-$V_p$-barrelled space,  and let $T:Y\to E$ be a weakly sequentially $p$-precompact operator from a normed space $Y$ to $E$. Then for every $p$-$(V)$ set $B$ in $E'$, its image $T^\ast(B)$ is relatively {\rm(}sequentially{\rm)} compact in the Banach space $Y'_\beta$.
\end{proposition}

\begin{proof}
To prove that $T^\ast(B)$ is relatively compact in the Banach space $Y'_\beta$, we have to show that  for every sequence $\{\chi_n\}_{n\in\w}$ in $B$, $\{T^\ast(\chi_n)\}_{n\in\w}$ has a convergent (in $Y'_\beta$) subsequence. Since $E$ is $\ell_\infty$-$V_p$-barrelled and $\{\chi_n\}_{n\in\w}$ is a $p$-$(V)$ set, the sequence $\{\chi_n\}_{n\in\w}$ is equicontinuous. Therefore, by Lemma \ref{l:p-operator-L-inf}, the linear map $S:E\to \ell_\infty$ defined by $S(x):=\big(\langle\chi_n,x\rangle\big)_{n\in\infty}$ is a $p$-convergent operator.

We show that the operator $ST: Y\to \ell_\infty$ is compact. Indeed, let $\{x_n\}_{n\in\w}\subseteq B_Y$. Since $T$ is weakly sequentially $p$-precompact, the sequence $\big\{T(x_n)\big\}$ contains a weakly $p$-Cauchy subsequence $\big\{T(x_{n_k})\big\}_k$. As $S$ is $p$-convergent, Lemma \ref{l:p-conver-p-Cauchy} implies that  $ST(x_{n_k})$ converges to some element of $\ell_\infty$. Thus $ST(B_Y)$ has compact closure in $\ell_\infty$, and hence $ST$ is compact.

By the Schauder theorem (for its formulation, see Theorem \ref{t:Schauder} below), the adjoint map $T^\ast S^\ast$ is also compact. Now, if $\{e^\ast_n\}$ is the standard unit basis of $\ell_1\subseteq (\ell_\infty)'$, then
\[
\langle S^\ast(e_n^\ast),x\rangle=\langle e_n^\ast,S(x)\rangle=\langle\chi_n,x\rangle \quad\mbox{ for every } \; x\in E,
\]
and hence $S^\ast(e_n^\ast)=\chi_n$. Therefore the sequence $\big\{T^\ast(\chi_n)\big\}_{n\in\w}= \big\{T^\ast S^\ast(e^\ast_n)\big\}_{n\in\w}$ has a convergent subsequence, as desired. \qed 
\end{proof}

\begin{corollary}  \label{c:image-p-V*-set}
Let $p\in[1,\infty]$, $E$ be a locally convex space such that $E'_\beta$ is an $\ell_\infty$-$V_p$-barrelled space,  and let $T:E\to Y$ be an operator into a Banach space $Y$  whose adjoint operator $T^\ast: Y'_\beta\to E'_\beta$ is weakly sequentially $p$-precompact. Then for every $p$-$(V^\ast)$ set $A\subseteq E$, its image $T(A)$ is relatively compact in $Y$.
\end{corollary}

\begin{proof}
Since $A$ is a $p$-$(V^\ast)$ set, the set $J_E(A)$ is  a $p$-$(V)$ set in $E''$ (recall that $J_E:E\to E''$ denotes the canonical inclusion). Therefore, by Proposition  \ref{p:image-p-V-set}, $T^{\ast\ast}\big(J_E(A)\big)$ is relatively compact in the Banach space $Y''$. Since $J_Y$ is an embedding of $Y$ onto the closed subspace $J_Y(Y)$ of $Y''$ and $T^{\ast\ast}\big(J_E(A)\big)=J_Y\big(T(A)\big)\subseteq J_Y(Y)$, we obtain that $J_Y\big(T(A)\big)$ is relatively compact in $J_Y(Y)$. Thus $T(A)$ is relatively compact in $Y$.\qed
\end{proof}

Below we apply the obtained results to give necessary and sufficient conditions for subsets of $E$ and $E'$ to be a $p$-$(V^\ast)$ set or a $p$-$(V)$ set, respectively.
The next result generalizes Corollary 20 of \cite{Ghenciu-pGP}.

\begin{theorem}  \label{t:p-V*-set}
Let $1<p<\infty$, and let $E$ be a quasibarrelled space such that $E'_\beta$ is an  $\ell_\infty$-$V_p$-barrelled space. Then for a bounded subset $A$ of $E$ the following assertions are equivalent:
\begin{enumerate}
\item[{\rm(i)}] $A$ is a $p$-$(V^\ast)$ set;
\item[{\rm(ii)}] $T(A)$ is relatively compact whenever $Y$ is a Banach space and $T:E\to Y$ is an operator whose adjoint operator $T^\ast: Y'_\beta\to E'_\beta$ is weakly sequentially $p$-precompact;
\item[{\rm(iii)}] $T(A)$ is relatively compact in $\ell_p$ for every operator $T: E\to \ell_{p}$.
\end{enumerate}
\end{theorem}

\begin{proof}
(i)$\Ra$(ii) follows from Corollary \ref{c:image-p-V*-set}.

(ii)$\Ra$(iii) Let $T\in \LL(E,\ell_p)$. Then $T^\ast\in \LL(\ell_{p^\ast},E'_\beta)$. By (ii) of Corollary \ref{c:L0p-p-precompact}, the identity operator $\Id_{\ell_{p^\ast}}$ of $\ell_{p^\ast}$ is weakly sequentially $p$-compact. Hence the operator $T^\ast=T^\ast\circ \Id_{\ell_{p^\ast}}$ is weakly sequentially $p$-precompact. Therefore, by (ii), $T(A)$ is relatively compact in $\ell_{p}$.

(iii)$\Ra$(i) Let $\{\chi_n\}_{n\in\w}$ be a weakly $p$-summable sequence in $E'_\beta$. Since $E'_\beta$ is quasi-complete by Proposition 11.2.3 of \cite{Jar}, we apply Proposition \ref{p:Lp-E-operator} to find an operator $S\in\LL(\ell_{p^\ast},E'_\beta)$ such that $S(e_n^\ast)=\chi_n$ for every $n\in\w$. Set $T:=S^\ast\circ J_E$, where $J_E:E\to E''$ is the canonical map. Since $E$ is quasibarrelled,  $J_E$ is continuous, and hence $T$ is an operator from $E$ into $\ell_p$. By assumption of (iii), we obtain that $T(A)$ is relatively compact in $\ell_p$. In particular,
by Proposition \ref{p:compact-ell-p}, we have
\[
\sup_{a\in A} |\langle\chi_n, a\rangle|=\sup_{a\in A} |\langle J_E(a), S(e_n^\ast)\rangle|=\sup_{a\in A} |\langle T(a), e_n^\ast\rangle|\to 0 \; \mbox{ as } n\to\infty.
\]
Thus $A$  is a $p$-$(V^\ast)$ set.\qed
\end{proof}

The case $p=1$ was considered by Emmanuele in \cite{Emmanuele-V}. For Banach spaces the next theorem was proved in Theorem 1.1  of \cite{Emmanuele-V}.

\begin{theorem}  \label{t:1-V*-set}
Let $E$ be a barrelled space. Then a bounded subset $A$ of $E$ is  a $(V^\ast)$ set if and only if $T(A)$ is relatively compact in $\ell_1$ for every operator $T: E\to \ell_{1}$.
\end{theorem}

\begin{proof}
Assume that $A$ is  a $(V^\ast)$ set. Then, by (iv) of Lemma \ref{l:V*-set-1}, $T(A)$ is a $(V^\ast)$ subset of $\ell_1$. Since, by Proposition 9 of \cite{Pelcz-62}, $\ell_1$ has the property $V^\ast$ it follows that $T(A)$ is relatively weakly compact. Therefore, by the Schur property, $T(A)$ is a relatively compact subset of $\ell_1$. Conversely, assume that $T(A)$ is relatively compact in $\ell_1$ for every operator $T: E\to \ell_{1}$. Let $\{\chi_n\}_{n\in\w}$ be an arbitrary weakly $1$-summable sequence in $E'_\beta$. Then, by Proposition \ref{p:p-sum-operator}, the linear map  $T:E\to\ell_p$  defined by $T(x):=(\langle\chi_n,x\rangle)$ is continuous. By assumption, $T(A)$ is relatively compact. Therefore, by Proposition \ref{p:compact-ell-p}, we obtain
$
\lim_{n\to\infty} \sup_{a\in A} |\langle\chi_n,x\rangle|=0.
$
Thus $A$ is a $(V^\ast)$ set.\qed
\end{proof}

In Corollary 1.3 of \cite{Bombal} Bombal proved  that every relatively weakly sequentially compact subset of a Banach space is a $(V^\ast)$ set. Below we generalize this result.

\begin{corollary} \label{c:ws-precom-V*}
Let $E$ be a barrelled space. Then every weakly sequentially precompact subset $A$ of $E$ is a  $(V^\ast)$ set.
\end{corollary}

\begin{proof}
Let $T:E\to\ell_1$ be an operator. Then $T(A)$ is  weakly sequentially precompact in $\ell_1$. Since $\ell_1$ is weakly sequentially complete and has the Schur property, $T(A)$ is relatively compact in $\ell_1$. Thus, by Theorem \ref{t:1-V*-set},  $A$ is a  $(V^\ast)$ set.\qed
\end{proof}

Following \cite{Gabr-free-resp}, we shall say that a sequence $A=\{ a_n\}_{n\in\w}$ of an lcs $E$ is {\em equivalent to the unit basis $\{ e_n: n\in\w\}$ of $\ell_1$} or is an {\em $\ell_1$-sequence} if there exists a linear topological isomorphism $R$ from the closure $\overline{\spn}(A)$ of $\spn(A)$ onto a subspace of $\ell_1$ such that $R(a_n)=e_n$ for every $n\in\w$ (we do not assume that $\overline{\spn}(A)$ is complete). The following corollary generalizes  Theorem 1.2 of \cite{Emmanuele-V}.

\begin{corollary} \label{c:1-V*-set}
Let $E$ be a barrelled space satisfying the following condition:
\begin{enumerate}
\item[$(\dagger)$] if a bounded subset $A$ of $E$ is not relatively weakly sequentially compact, then there is an  $\ell_1$-sequence  $\{ a_n\}_{n\in\w}\subseteq A$ such that $\overline{\spn}\big(\{ a_n\}_{n\in\w}\big)$ is complemented in $E$.
\end{enumerate}
Then $E$ has the property $sV^\ast$.
\end{corollary}

\begin{proof}
Suppose for a contradiction that there is a $(V^\ast)$ subset $A$ of $E$ which is not  relatively weakly sequentially compact. Then, by $(\dagger)$, there are a sequence $\{ a_n\}_{n\in\w}\subseteq A$ and a linear topological isomorphism $R$ from $\overline{\spn}(A)$ onto a subspace of  $\ell_1$ such that $R(a_n)=e_n$ for every $n\in\w$ and the subspace $\overline{\spn}\big(\{ a_n\}_{n\in\w}\big)$ is complemented in $E$. Let $P$ be a projection from $E$ onto $\overline{\spn}\big(\{ a_n\}_{n\in\w}\big)$. Then $R\circ P:E\to \ell_1$ is an operator such that $R\circ P(a_n)=e_n$ for every $n\in\w$. By Theorem  \ref{t:1-V*-set}, the canonical unit basis $\{ e_n\}_{n\in\w}$ of $\ell_1$ is relatively compact, a contradiction.\qed
\end{proof}

\begin{corollary} \label{c:rwsc-pV*}
Let $p\geq 1$ and a locally convex space $E$ satisfy one of the following conditions
\begin{enumerate}
\item[{\rm(i)}] $1<p<2$ and $E$ is a quasibarrelled space such that $E'_\beta$ is an  $\ell_\infty$-$V_p$-barrelled space,
\item[{\rm(ii)}] $p=1$ and $E$ is barrelled.
\end{enumerate}
Then every relatively weakly sequentially $p$-compact subset of $E$ is a $p$-$(V^\ast)$ set.
\end{corollary}

\begin{proof}
Let $A$ be a  relatively weakly sequentially $p$-compact subset of $E$. Then, by Lemma \ref{l:image-p-seq-com}, $T(A)$ is a  relatively weakly sequentially $p$-compact subset of $\ell_p$ for every operator $T:E\to\ell_p$. If $1<p<2$, we obtain $p<p^\ast$ and hence, by Proposition \ref{p:Lp-Schur}, $\ell_p$ has the $p$-Schur property. Recall that $\ell_1$ has the Schur property. Therefore in both cases (i) and (ii), the set $T(A)$ is relatively sequentially compact in $\ell_p$. Thus, by Theorems \ref{t:p-V*-set} and \ref{t:1-V*-set}, $T(A)$ is a $p$-$(V^\ast)$ subset of $E$.\qed
\end{proof}

\begin{remark} {\em
The condition $1\leq p<2$ in Corollary \ref{c:rwsc-pV*} is essential even for Banach spaces. Indeed, let $p=2$ and $E=\ell_2$. Then the standard unit basis $S=\{e_n\}_{n\in\w}$ of $E$ is weakly $2$-summable (indeed, if $\chi=(a_n)\in E'=E$, then $\sum_{n\in\w} |\langle\chi,e_n\rangle|^2 =\sum_{n\in\w} a_n^2 <\infty$). Therefore $S$ is a relatively weakly sequentially $2$-compact subset of $E$. However, for $T=\Id_E$, the set $T(S)=S$ is not relatively compact in $E$. Thus, by Theorem \ref{t:p-V*-set}, $S$ is not a $2$-$(V^\ast)$ set.\qed}
\end{remark}

For the case when $Y=\ell_{p^\ast}^0$ and $p\not=\infty$ we can reverse Proposition \ref{p:image-p-V-set} as follows.

\begin{proposition} \label{p:image-p-V-set-2}
Let $p\in[1,\infty)$, $E$ be a locally convex space,  and let $B$ be a subset of $E'$. Assume that one of the following conditions holds:
\begin{enumerate}
\item[{\rm(i)}]  $T^\ast(B)$ is relatively compact in $\ell_{p}$ for every $T\in\LL(\ell_{p^\ast}^0, E)$ $($or  $T\in\LL(c_0^0, E)$ if $p=1$$)$;
\item[{\rm(ii)}]  $E$ is  sequentially  complete and $T^\ast(B)$ is relatively compact in $\ell_{p}$ for every $T\in\LL(\ell_{p^\ast}, E)$ $($or  $T\in\LL(c_0, E)$ if $p=1$$)$.
\end{enumerate}
Then $B$ is a $p$-$(V)$ set.
\end{proposition}

\begin{proof}
Let $\{x_n\}_{n\in\w}$ be a weakly $p$-summable sequence in $E$. Then, by  Proposition \ref{p:Lp-E-operator}, there is $T\in\LL(\ell_{p^\ast}^0,E)$ or  $T\in\LL(c_0^0, E)$ (resp., $T\in\LL(\ell_{p^\ast},E)$ or  $T\in\LL(c_0, E)$ if $E$ is  sequentially  complete) such that $T(e_n^\ast)=x_n$ for every $n\in\w$. Since $T^\ast(B)$ is relatively compact in $\ell_p$, Proposition \ref{p:compact-ell-p} implies
\[
 \sup_{\chi\in B} |\langle\chi, x_n\rangle|= \sup_{\chi\in B} |\langle T^\ast(\chi),e_n^\ast\rangle| \to 0,
\]
which means that $B$ is a $p$-$(V)$ set in $E'$, as desired.\qed
\end{proof}

The next theorem generalizes Theorem 19 of \cite{Ghenciu-pGP}, for the case $1<p<\infty$, Theorem 3.10 of \cite{CCDL} and, for the case $p=1$, Proposition 5 of \cite{Cilia-Em}.
\begin{theorem}  \label{t:p-V-set}
Let $1\leq p<\infty$, and let $E$ be an $\ell_\infty$-$V_p$-barrelled space. Then for a weak$^\ast$ bounded subset $B$ of $E'$ the following assertions are equivalent:
\begin{enumerate}
\item[{\rm(i)}] $B$ is a $p$-$(V)$ set;
\item[{\rm(ii)}] for every normed space $Y$ and each weakly sequentially $p$-precompact operator $T:Y\to E$, the set $T^\ast(B)$ is relatively compact in the Banach space $Y'_\beta$;
\item[{\rm(iii)}] $T^\ast(B)$ is relatively compact in $\ell_{p}$ for every $T\in\LL(\ell_{p^\ast}^0, E)$ $($or $T\in\LL(c_0^0, E)$ if $p=1$$)$.
\end{enumerate}
If additionally $E$ is  sequentially  complete, then {\em (i)-(iii)} are equivalent to
\begin{enumerate}
\item[{\rm(iv)}]  $T^\ast(B)$ is relatively compact in $\ell_{p}$  for every $T\in\LL(\ell_{p^\ast}, E)$ $($or $T\in\LL(c_0, E)$ if $p=1$$)$.
\end{enumerate}
\end{theorem}

\begin{proof}
(i)$\Ra$(ii)  follows from Proposition \ref{p:image-p-V-set}.

(ii)$\Ra$(iii) Let $T\in\LL(\ell_{p^\ast}^0, E)$ (or $T\in\LL(c_0^0, E)$ if $p=1$). By (i) of Corollary \ref{c:L0p-p-precompact}, the identity operator $\Id_{\ell_{p^\ast}^0}$ (or $\Id_{c_0^0}$ if $p=1$) is weakly sequentially $p$-precompact, and hence, by Lemma \ref{l:image-p-seq-com}, also $T=T\circ \Id_{\ell_{p^\ast}^0}$ (or $T=T\circ \Id_{c^0_0}$  if $p=1$)  is weakly sequentially $p$-precompact. By (ii), $T^\ast(B)$ is relatively compact in $\ell_{p}$.

(iii)$\Ra$(i) follows from (i) of Proposition \ref{p:image-p-V-set-2}.

Assume now that $E$ is additionally sequentially  complete.

(iii)$\Ra$(iv) Let $T\in\LL(\ell_{p^\ast}, E)$ (or $T\in\LL(c_0,E)$ if $p=1$). Denote by $S$ the restriction of $T$ onto $\ell_{p^\ast}^0$ (or onto $c_0^0$ if $p=1$). Since $\ell_{p^\ast}^0$ (resp., $c_0^0$) is dense in $\ell_{p^\ast}$ (resp., $c_0$) it follows that $T^\ast=S^\ast$. Therefore, by (iii), $T^\ast(B)=S^\ast(B)$ is relatively compact in $\ell_{p}$.

(iv)$\Ra$(i) follows from (ii) of Proposition \ref{p:image-p-V-set-2}.\qed
\end{proof}


\section{Gantmacher's property for locally convex spaces and an operator characterization of Pe{\l}czy\'{n}ski's  property $sV_p$} \label{sec:oper-sVp}


The aim of this section is to give an operator characterization of the property $sV_p$ and an independent and short proof of Pe{\l}czy\'{n}ski's Theorem \ref{t:Pel-C(K)}. First we recall some well known constructions, we use standard notations from \cite{Jar}.

Let $E$ be a locally convex space. Then every $U\in \Nn_{0}^c(E)$ defines a seminorm $q_U$ by
\[
q_U(x):=\inf \{ \lambda>0: x\in\lambda U\} \quad (x\in E).
\]
Set $N(U):= q_U^{-1}(0)=\bigcap_{n\in\NN} \tfrac{1}{n} U$. Observe that $N(U)$ is a closed subspace of $E$ such that $U=U+N(U)$.
Therefore $q_U(x)$ defines a norm on the quotient vector space
\[
E_{(U)}:=E/N(U) \quad \mbox{ by } \;\; \|x+N(U)\|_U :=q_U(x) \;\; (x\in E).
\]
We denote by $\Phi_U$ the quotient linear map $E\to E_{(U)}=E/N(U)$.
Observe that $U=\{x\in E:q_U(x)\leq 1\}$.
It follows that $\Phi_U(U)= B_{E_{(U)}}$ and hence $\Phi_U$ is continuous. Moreover, the co-restriction of the adjoint map $\Phi_U^\ast$ onto $E'_{U^\circ}:=(E')_{U^\circ}$
which is also denoted by $\Phi_U^\ast$ is an  isometric isomorphism of $\big(E_{(U)}\big)'_\beta$ onto $E'_{U^\circ}$. In particular, $E'_{U^\circ}$ is a Banach space.

The next assertion generalizes Proposition 17.1.2 of \cite{Jar} and gives a characterization of compact-type operators defined in Definition \ref{def:operators}. Recall (see Section \ref{sec:completeness}) that the topology $\TTT_\AAA$ of $E'_\AAA :=(E', \TTT_\AAA)$ is the polar topology on $E'$ of uniform convergence on the members of $\AAA$.

\begin{proposition} \label{p:compact-type-operator}
Let $E$ and $L$ be locally convex spaces, $\AAA(L)\subseteq \Bo(L)$  be a bornology on $L$, and let $T\in\LL(E,L)$. Consider the following assertions:
\begin{enumerate}
\item[{\rm(i)}] $T\in\mathcal{LA}(E,L)$;
\item[{\rm(ii)}]  there are $U\in \Nn_0^{c}(E)$ and $T_U\in \mathcal{LA}(E_{(U)},L)\cap \LL(E_{(U)},L_{T(U)})$  such that $T=T_U \circ \Phi_U$;
\item[{\rm(iii)}] there is $U\in \Nn_0^{c}(E)$ such that the adjoint operator $T^\ast$ maps $L'_{\AAA}$ continuously into the Banach space $E'_{U^\circ}$; in particular, the set $T^\ast\big( T(U)^\circ\big)\subseteq E'$ is equicontinuous.
\end{enumerate}
Then {\rm (i)$\LRa$(ii)$\Ra$(iii)}. The implication {\rm (iii)$\Ra$(i)} is satisfied if $\AAA(L)$ is saturated.
\end{proposition}

\begin{proof}
(i)$\Ra$(ii) Let $T\in\mathcal{LA}(E,L)$. Choose $U\in  \Nn_0^{c}(E)$ such that $T(U)\in \AAA(L)$. Then $T(U)$ is a bounded absolutely convex subset of $L$, and hence for every $x\in N(U)$, we obtain $T(x)\in \bigcap_{n\in\NN} \tfrac{1}{n} T(U)=\{0\}$. Therefore $T$ induces a linear map $T_U: E_{(U)}\to L$.
Observe that $T_U$ is continuous as a linear map from  $E_{(U)}$ into $L_{T(U)}$ because $T_U\big( B_{E_{(U)}}\big) \subseteq T(U)$ and $T(U)$ is contained in the closed unit ball of $L_{T(U)}$. Since, by Proposition \ref{p:bounded-norm}, the norm topology of $L_{T(U)}$  is finer than the induced original topology on $L$ it follows that $T_U$ is continuous. The inclusion $T_U\big( B_{E_{(U)}}\big) \subseteq T(U)$ and the property (c) of the bornology $\AAA(L)$ imply  also that $T_U\in\mathcal{LA}(E_{(U)},L)$. It remains to note that $T=T_U \circ \Phi_U$.
\smallskip


The implication (ii)$\Ra$(i) immediately follows from the equality $T_U\big( B_{E_{(U)}}\big) = T(U)$ (which holds because $\Phi_U(U)= B_{E_{(U)}}$).
\smallskip

(ii)$\Ra$(iii) Assume that there are $U\in \Nn_0^{c}(E)$ and an operator $T_U\in \mathcal{LA}(E_{(U)},L)$  such that $T=T_U \circ \Phi_U$. Then $T^\ast= \Phi_U^\ast \circ T_U^\ast$. As we noticed above, the co-restriction of $ \Phi_U^\ast$ is an isometric isomorphism of $\big(E_{(U)}\big)'_\beta$ onto $E'_{U^\circ}$. Observe that $B_{E_{(U)}}=U+N(U)$, where the right hand side contains conjugate classes of $x\in U$. Therefore, by the properties (b)-(c) of  the bornology $\AAA(L)$  and since $T_U\in\mathcal{LA}(E_{(U)},L)$, we obtain  $T_U\big(U+N(U)\big)\in\AAA(L)$. Hence, by the definition of the topology $\TTT_\AAA$ on $L'$, the polar $W:=T_U\big(U+N(U)\big)^\circ$ is a neighborhood of zero in $L'_{\AAA}$. Now, for every $x\in U$ and each $\chi\in W$, we obtain
\[
|\langle T^\ast(\chi),x\rangle|=|\langle \chi, T_U\circ \Phi_U(x)\rangle|=|\langle \chi, T_U\big(x+N(U)\big)\rangle|\leq 1
\]
which means that $T^\ast(W)\subseteq U^\circ=B_{E'_{U^\circ}}$. Thus $T^\ast:L'_{\AAA}\to E'_{U^\circ}$ is continuous. Finally, since $T(U)^\circ= T_U\big(U+N(U)\big)^\circ=W$ and $T^\ast\big(T(U)^\circ\big)=T^\ast(W) \subseteq U^\circ$ the set $T^\ast\big(T(U)^\circ\big)$ is equicontinuous.
%
\smallskip


(iii)$\Ra$(i) Assume additionally that $\AAA(L)$ is saturated, and let $T^\ast:L'_{\AAA}\to E'_{U^\circ}$ be continuous. Then there is $A\in\AAA$ such that $T^\ast(A^\circ)\subseteq U^\circ$. This means that for every $x\in U$ and each $\chi\in A^\circ$, we have $|\langle\chi,T(x)\rangle|= |\langle T^\ast(\chi),x\rangle|\leq 1$ and hence $T(U) \subseteq A^{\circ\circ}$. As $\AAA(L)$ is saturated, we have  $A^{\circ\circ}\in \AAA(L)$. Therefore, by (c) of the bornology $\AAA(L)$, it follows $T(U)\in\AAA(L)$. Thus $T\in \mathcal{LA}(E,L)$.\qed
\end{proof}

The next notion is motivated by (iii) of Proposition \ref{p:compact-type-operator}.

\begin{definition} \label{def:operator-equi} {\em
Let $E$ and $L$ be locally convex spaces. A linear mapping $T:L\to E'$ is called {\em equicontinuous} if for some $U\in \Nn_0(L)$, the set $T(U)\subseteq E'$ is equicontinuous.\qed}
\end{definition}

Let $E$ be a locally convex space. If $U\in\Nn_0(E)$, then, by Theorem 11.3.5 of \cite{NaB}, $U^\circ$ is strongly bounded and hence, by Proposition  \ref{p:bounded-norm}, the canonical inclusion $E'_{U^\circ} \to E'_\beta$ is continuous. Therefore there exists a {\em finest} locally convex topology denoted by $\TTT_n$ on the dual space $E'$ such that for every $U\in\Nn_0(E)$, the canonical inclusion $E'_{U^\circ} \to (E',\TTT_n)$ is continuous. Clearly, $\beta(E',E)\subseteq \TTT_n$. In the case when $E$ is a normed space, the continuity of $E'_{B_E^\circ}=E'_\beta \to (E',\TTT_n)$ implies that $\TTT_n\subseteq\beta(E',E)$, and hence  $\TTT_n=\beta(E',E)$ is a Banach norm topology (which motivates the subscript $n$ in the notation $\TTT_n$).

\begin{theorem} \label{t:operator-compact-type}
Let $E$ and $L$ be locally convex spaces, and assume that $\AAA(L)\subseteq \Bo(L)$ is a saturated bornology on  $L$. Then $\LL(E,L)=\mathcal{LA}(E,L)$ if and only if for every $T\in\LL(E,L)$,  the adjoint mapping $T^\ast:L'_{\AAA}\to E'_\beta$ is equicontinuous and hence continuous.
\end{theorem}

\begin{proof}
Taking into account that $\beta(E',E) \subseteq \TTT_n$, the necessity is proved in Proposition \ref{p:compact-type-operator}. To prove the sufficiency, let $T\in\LL(E,L)$. Since $T^\ast$ is  equicontinuous, there are $A\in\AAA$ and $U\in\Nn_0^c(E)$ such that $T^\ast(A^\circ)\subseteq U^\circ$. This means that for every $x\in U$ and each $\chi\in A^\circ$, we have $|\langle\chi,T(x)\rangle|= |\langle T^\ast(\chi),x\rangle|\leq 1$ and hence $T(U) \subseteq A^{\circ\circ}$. As $\AAA(L)$ is saturated, we have  $A^{\circ\circ}\in \AAA(L)$. Therefore $T(U)\in\AAA(L)$ and hence $T\in \mathcal{LA}(E,L)$.

To show that $T^\ast$ is also continuous, fix a bounded subset $B$ of $E$. As  $T^\ast$ is  equicontinuous, there are $A\in\AAA$ and $U\in\Nn_0^c(E)$ such that $T^\ast(A^\circ)\subseteq U^\circ$.  Since, by  Theorem 11.3.5 of \cite{NaB}, $U^\circ$ is strongly bounded, there is $\lambda>0$ such that $U^\circ \subseteq \lambda B^\circ$, and hence $T^\ast(\tfrac{1}{\lambda} A^\circ) \subseteq \tfrac{1}{\lambda} U^\circ \subseteq  B^\circ$. Thus $T^\ast:L'_{\AAA}\to E'_\beta$ is  continuous.\qed
\end{proof}

Below we consider two cases when the adjoint operator has some of compact-type properties.
First we recall the next two well known and  important results in Banach space theory.
\begin{theorem}[Schauder] \label{t:Schauder}
Let $T:E\to L$ be an operator between Banach spaces $E$ and $L$. Then $T$ is compact if and only if its adjoint $T^\ast$ is compact.
\end{theorem}
\begin{theorem}[Gantmacher] \label{t:Gantmacher}
Let $T:E\to L$ be an operator between Banach spaces $E$ and $L$. Then $T$ is weakly compact if and only if its adjoint $T^\ast$ is weakly compact.
\end{theorem}

The next proposition is a version of the ``direct part'' of Schauder's Theorem \ref{t:Schauder}. 
\begin{proposition} \label{p:Schauder-nes}
Let  $T$ be a precompact operator from a locally convex space $E$ to a normed space $L$. Then $T^\ast(B_{L'_\beta})$ is a compact subset of $ E'_\beta$, and hence the adjoint operator $T^\ast: L'_\beta \to E'_\beta$ is compact.
\end{proposition}

\begin{proof}
Consider the bornology $\mathcal{PC}(L)$ of precompact subsets of $L$.  Since $T$ is precompact, Proposition \ref{p:compact-type-operator} implies that for some $U\in\Nn_0^c(E)$,  the map $T^\ast: L'_{\mathcal{PC}(L)}\to E'_{U^\circ}$ is continuous. As $U^\circ$ is strongly bounded by Theorem 11.3.5 of \cite{NaB}, Proposition \ref{p:bounded-norm} implies that the norm topology of $E'_{U^\circ}$ is stronger than the topology induced from $E'_\beta$. Therefore the map $T^\ast: L'_{\mathcal{PC}(L)}\to E'_\beta$ is continuous. By the Alaoglu--Bourbaki Theorem 8.5.2 of \cite{Jar}, the dual unit closed ball $B_{L'_\beta}=(B_L)^\circ$ is compact in $L'_{\mathcal{PC}(L)}$. Thus $T^\ast(B_{L'_\beta})$ is a compact subset of $ E'_\beta$, as desired.\qed
\end{proof}

In what follows we shall use the following remark without mentioning.

\begin{remark} {\em
Let $T$ be an operator from an lcs $E$ to a Banach space $L$. Then, by the (weak) angelicity of $L$, the operator $T$ is (weakly) sequentially compact if and only if it is (weakly) compact.\qed}
\end{remark}

The next lemma immediately follows from the corresponding definitions. 
\begin{lemma} \label{l:Gantmacher-nes}
Let $T$ be an operator from a normed space $E$ to a Banach space $L$. Then $T$ is compact {\rm(}resp., sequentially compact, weakly compact or weakly sequentially compact{\rm)} if and only if so is its unique extension $\bar T: \overline{E}\to L$ to the completion $\overline{E}$ of $E$.
\end{lemma}

The following assertion is a version of the necessity in Gantmacher's Theorem \ref{t:Gantmacher}.
\begin{proposition} \label{p:Gantmacher-nes}
Let $T$ be an operator from a locally convex space $E$ to a Banach space $L$. If $T$ is compact {\rm(}resp.,  weakly compact{\rm)}, then its adjoint operator $T^\ast: L'_\beta\to E'_\beta$ is compact and sequentially compact {\rm(}resp.,  weakly compact and weakly  sequentially compact{\rm)}.
\end{proposition}

\begin{proof}
Applying Proposition \ref{p:compact-type-operator} to the family $\AAA(L)=\mathcal{RC}(L)$ (resp., $\mathcal{RWC}(L)$) we obtain that there are $U\in \Nn_0^c(E)$ and an operator $T_U\in \mathcal{LA}(E_{(U)},L)$ such that $T=T_U\circ \Phi_U$. Since $E_{(U)}$ is a normed space, Lemma \ref{l:Gantmacher-nes} implies that the extension $\overline{T_U}$ of $T_U$ onto the completion $\overline{E_{(U)}}$ of $E_{(U)}$ is a compact (resp., weakly compact) operator. Note also that $L'_\beta$ is a Banach space and hence it is weakly angelic. Therefore, by the Schauder  Theorem \ref{t:Schauder} and Gantmacher's Theorem \ref{t:Gantmacher}, the adjoint operator $\overline{T_U}^{\,\ast}=T_U^\ast$ is also compact  and sequentially compact (resp.,  weakly compact and weakly  sequentially compact), and hence so is $T^\ast=\Phi_U^\ast \circ \overline{T_U}^{\,\ast}$.\qed
\end{proof}



Using Proposition \ref{p:Gantmacher-nes} we give below a sufficient condition on an lcs $E$ to have the property $sV_p$, cf. Theorem \ref{t:p-V*-precompact-V*}.
\begin{proposition} \label{p:sVp-sufficient}
Let  $p\in[1,\infty]$, and let $E$ be an $\ell_\infty$-$V_p$-barrelled space. If each $p$-convergent operator $T:E\to \ell_\infty$ is weakly compact, then $E$ has the property $sV_p$.
\end{proposition}

\begin{proof}
Let $B\subseteq E'$ be a $p$-$(V)$ set. We prove that $B$ is relatively weakly sequentially compact. To this end, fix an arbitrary sequence $S=\{\chi_n\}_{n\in\w}$ in $B$. By (iii) of Lemma \ref{l:V-set-1}, $S$ is also a $p$-$(V)$ subset of $E'$. Consider the linear map $T:E\to \ell_\infty$ defined by
\[
T(x):= \big( \langle\chi_n,x\rangle\big)_{n\in\w} \quad (x\in E).
\]
Since $E$ is $\ell_\infty$-$V_p$-barrelled, the sequence $S$ is equicontinuous. Therefore, by (i) of Lemma \ref{l:p-operator-L-inf}, $T$ is continuous. As $S$ is a a $p$-$(V)$ set, (ii) of Lemma \ref{l:p-operator-L-inf} implies that $T$ is $p$-convergent. Hence, by assumption, $T$ is weakly sequentially compact. By Proposition \ref{p:Gantmacher-nes}, the adjoint operator $T^\ast: \ell_\infty'\to E'_\beta$ is also weakly sequentially compact.

For every $n\in\w$, let $e_n^\ast$ be the $n$th coordinate functional on $\ell_\infty$, and observe that $T^\ast(e_n^\ast)=\chi_n$ (indeed, if $x\in E$, then $\langle T^\ast(e_n^\ast),x\rangle=\langle e_n^\ast,T(x)\rangle=\langle\chi_n,x\rangle$). Since $T^\ast$ is weakly sequentially compact and all $e_n^\ast$ belong to the closed unit ball of $\ell_\infty'$, there is a subsequence $\{e^\ast_{n_k}\}_{k\in\w}$ of $\{e^\ast_{n}\}_{n\in\w}$ such that the sequence $\{T^\ast(e_{n_k}^\ast)=\chi_{n_k}\}_{k\in\w}$ weakly converges to some functional $\chi\in E'_\beta$. Thus $B$ is relatively weakly sequentially compact in $E'_\beta$.\qed
\end{proof}

Being motivated by Schauder's Theorem \ref{t:Schauder} and Gantmacher's Theorem \ref{t:Gantmacher} one can naturally introduce the following notions.
\begin{definition} \label{def-Schauder-Gant} {\em
A locally convex space $E$ is said to have
\begin{enumerate}
\item[$\bullet$]  {\em the {\rm(}sequential{\rm)} Schauder property} if for every Banach space $L$, an operator $T:E\to L$ is compact if and only if its adjoint $T^\ast: L'_\beta\to E'_\beta$ is (resp., sequentially) compact;
\item[$\bullet$]  {\em the strong {\rm(}sequential{\rm)} Schauder property} if for every lcs $L$, an operator $T:E\to L$ is (resp., sequentially) compact if and only if so is its adjoint $T^\ast: L'_\beta\to E'_\beta$;
\item[$\bullet$]  {\em the {\rm(}sequential{\rm)} Gantmacher property} if for every Banach space $L$, an operator $T:E\to L$ is weakly compact if and only if $T^\ast: L'_\beta\to E'_\beta$ is weakly (resp., sequentially) compact;
\item[$\bullet$]  {\em the strong {\rm(}sequential{\rm)} Gantmacher property} if for every lcs $L$, an operator $T:E\to L$ is weakly  (resp., sequentially) compact if and only if so is $T^\ast: L'_\beta\to E'_\beta$.\qed
\end{enumerate}}
\end{definition}

The introduced notions are well-behaved under taking products.
\begin{proposition} \label{p:product-S-Gant}
Let $\{E_i\}_{i\in I}$ be a non-empty family of locally convex spaces. Then the product $E=\prod_{i\in I} E_i$ has the Schauder property {\rm(}resp., the sequential Schauder property, the Gantmacher property  or the sequential Gantmacher property{\rm)} if and only if all $E_i$ have the same property.
\end{proposition}

\begin{proof}
Assume that $E$ has the Schauder property (resp., the sequential Schauder property, the Gantmacher property  or the sequential Gantmacher property). Fix $j\in I$, and let $T_j:E_j\to L$ be an operator from $E_j$ to a Banach space $L$. If $T_j$ is  compact (resp., sequentially compact,  weakly compact or weakly  sequentially compact), then,  by Proposition \ref{p:Gantmacher-nes}, the adjoint operator $T_j^\ast: L'_\beta\to (E_j)'_\beta$ has the same property. Conversely, assume that $T_j^\ast: L'_\beta\to (E_j)'_\beta$ is  compact (resp., sequentially compact,  weakly compact or weakly  sequentially compact). Let $0_j: \prod_{i\in I\SM\{j\}} E_i\to L$ be the zero operator, and let $I_j:(E_j)'_\beta \to E'_\beta$ be the identity embedding onto the $j$coordinate. Define an operator $T:E\to L$ by $T=T_j\times 0_j$. Then for every $(x_i)\in E$ and each $\chi\in L'$, we have
\[
\langle T^\ast(\chi), (x_i)\rangle=\langle\chi,\big(T(x_i)\big)_{i\in I}\rangle=\langle \chi,T_j(x_j)\rangle =\langle T_j^\ast(\chi), x_j\rangle=\big\langle \big( I_j\circ T_j^\ast\big)(\chi), (x_i)\big\rangle,
\]
which means that $ T^\ast=I_j\circ T_j^\ast$. Since $ T_j^\ast$ is  compact (resp., sequentially compact,  weakly compact or weakly  sequentially compact) so is the operator $T^\ast$. By  the Schauder property (resp., the sequential Schauder property, the Gantmacher property  or the sequential Gantmacher property) of $E$, the operator $T=T_j\times 0_j$ is  compact (resp., sequentially compact,  weakly compact or weakly  sequentially compact), and hence so is $T_j$. Thus $E_j$ has  the Schauder property (resp., the sequential Schauder property, the Gantmacher property  or the sequential Gantmacher property).

Assume now that all spaces $E_i$ have  the Schauder property (resp., the sequential Schauder property, the Gantmacher property  or the sequential Gantmacher property). Let $T$ be an operator from $E$ to a Banach space $L$. If $T$ is   compact (resp., sequentially compact,  weakly compact or weakly  sequentially compact), then,  by Proposition \ref{p:Gantmacher-nes}, so is also the adjoint operator $T^\ast$. Conversely, assume that $T^\ast: L'_\beta\to E'_\beta$ is compact (resp., sequentially compact,  weakly compact or weakly  sequentially compact). By Lemma \ref{l:product-normed}, there is a finite subset $F$ of $I$ such that $\prod_{i\in I\SM F} E_i$ is in the kernel of $T$. Therefore to show that $T$ is  compact (resp., sequentially compact,  weakly compact or weakly  sequentially compact) it suffices to assume that $I$ has only two elements. So, consider an operator $T:E=H\times Z\to L$ such that $T^\ast: L'_\beta \to H'_\beta \times Z'_\beta$ is  compact (resp., sequentially compact,  weakly compact or weakly  sequentially compact). Let $I_H:H\to H\times Z$ and $I_Z:Z\to H\times Z$ be the coordinate embeddings. Then
\begin{equation} \label{equ:Gantmacher-1}
T(h,z)=T(h,0)+T(0,z)=T\circ I_H(h) +T\circ I_Z(z).
\end{equation}
Since $T^\ast$ is  compact (resp., sequentially compact,  weakly compact or weakly  sequentially compact) so are $I_H^\ast \circ T^\ast$ and $I_Z^\ast \circ T^\ast$. Since $H$ and $Z$ have  the Schauder property (resp., the sequential Schauder property, the Gantmacher property  or the sequential Gantmacher property), it follows that the operators $T\circ I_H:H\to L$ and $T\circ I_Z:Z\to L$ are  compact (resp., sequentially compact,  weakly compact or weakly  sequentially compact). Thus, by (\ref{equ:Gantmacher-1}), also the operator $T$ has the same property. \qed
\end{proof}

It is well known that any locally convex space embeds into the product of a family of Banach spaces. This fact and Schauder's Theorem \ref{t:Schauder}, Gantmacher's Theorem \ref{t:Gantmacher} and Proposition \ref{p:product-S-Gant} immediately imply the following corollary.
\begin{corollary} \label{c:embed-S-Gant}
Any locally convex space embeds into a locally convex space with the Schauder property {\rm(}resp., the sequential Schauder property, the Gantmacher property  or the sequential Gantmacher property{\rm)}.
\end{corollary}



\begin{proposition} \label{p:sVp-nessecity}
Let  $p\in[1,\infty]$, $E$ be a locally convex space with the sequential Gantmacher property, and let $\Gamma$ be an infinite set. If $E$ has the property $sV_p$, then each $p$-convergent operator $T:E\to \ell_\infty(\Gamma)$ is weakly {\rm(}sequentially{\rm)} compact.
\end{proposition}

\begin{proof}
Let $T:E\to \ell_\infty(\Gamma)$ be a $p$-convergent operator. To show that $T$ is weakly sequentially compact, by the sequential Gantmacher property, it suffices to prove that the adjoint operator $T^\ast$ is weakly sequentially compact. Let  $\{x_i\}_{i\in\w}$   be an arbitrary weakly $p$-summable sequence in $E$, and let $B$ be the closed unit ball of $\ell'_\infty(\Gamma)$. Since $T$ is $p$-convergent, $T(x_i)\to 0$ in $\ell_\infty(\Gamma)$. Therefore
\[
\sup_{\eta\in B} \big| \langle T^\ast(\eta), x_i\rangle\big| = \sup_{\eta\in B} \big| \langle \eta, T(x_i)\rangle\big|\leq \|T(x_i)\| \to 0,
\]
which means that $T^\ast(B)$ is a $p$-$(V)$ set in $E'$. By the property $sV_p$, we obtain that $T^\ast(B)$ is a relatively weakly sequentially compact subset of $E'_\beta$. Thus $T^\ast$ is weakly sequentially compact.\qed
\end{proof}

Joining Proposition \ref{p:sVp-sufficient} and Proposition \ref{p:sVp-nessecity} we obtain the following operator characterization of the property $sV_p$ which generalizes Pe{\l}czy\'{n}ski's operator characterization of the property $V$ for Banach spaces, see  \cite{Pelcz-62}.

\begin{theorem} \label{t:sVp-charac}
Let  $p\in[1,\infty]$, and let $E$ be an $\ell_\infty$-$V_p$-barrelled space with the sequential Gantmacher property. Then the following assertions are equivalent:
\begin{enumerate}
\item[{\rm(i)}] for every Banach space $L$, each $p$-convergent operator $T:E\to L$ is weakly  compact;
\item[{\rm(ii)}] each $p$-convergent operator $T:E\to \ell_\infty$ is weakly  compact;
\item[{\rm(iii)}] $E$ has the property $sV_p$.
\end{enumerate}
\end{theorem}

\begin{proof}
The implication (i)$\Ra$(ii) is trivial, and (ii)$\Ra$(iii) follows from Proposition \ref{p:sVp-sufficient}.

(iii)$\Ra$(i) It is well known that there is an infinite set $\Gamma$ such that the Banach space $L$ embeds into $\ell_\infty(\Gamma)$. Let $R:L\to \ell_\infty(\Gamma)$ be an embedding. Then $R\circ T$ is also a $p$-convergent operator. Therefore, by Proposition \ref{p:sVp-nessecity},  $R\circ T$ is weakly compact. Since $R(L)$ is a closed subspace of $\ell_\infty(\Gamma)$, it is also weakly closed. It  follows that if $U\in \Nn_0(E)$ is such that $R\circ T(U)$ is relatively weakly compact in $\ell_\infty(\Gamma)$, then $T(U)$ is relatively weakly compact in $L$. Thus $T$  is a weakly compact operator.\qed
\end{proof}




Below we provide a short proof of Pe{\l}czy\'{n}ski's Theorem \ref{t:Pel-C(K)}.

\begin{theorem}[\cite{Pelcz-62}] \label{t:sVp-C(K)}
For every $p\in[1,\infty]$ and each compact space $K$, the Banach space $C(K)$ has the property $V_p$.
\end{theorem}

\begin{proof}
By Gantmacher's Theorem \ref{t:Gantmacher} and Theorem \ref{t:sVp-charac}, it suffice to prove that each $p$-convergent operator $T:C(K)\to \ell_\infty$ is weakly compact. Suppose for a contradiction that there is a $p$-convergent operator $T:C(K)\to \ell_\infty$ which is not weakly compact. Since $T$ is not weakly compact, Pe{\l}czy\'{n}ski's Theorem 5.5.3 of \cite{Al-Kal} implies that $C(K)$ has a subspace $H$ isomorphic to $c_0$ such that the restriction $T{\restriction}_H$ is an isomorphism. Identifying $H$ with $c_0$, we observe that the standard unit basis $\{e_n\}_{n\in\w}$ of $c_0$ is weakly $p$-summable (indeed, if $x=(a_n)\in \ell_1$, then $\sum_{n\in\w} \big|\langle x,e_n\rangle\big|^p=\sum_{n\in\w} |a_n|^p <\infty$). Since $e_n\not\to 0$ in $c_0$ and $T{\restriction}_H$ is an isomorphism, it follows that $T(e_n)\not\to 0$ in $L$. Hence $T$ is not $p$-convergent, a contradiction.\qed
\end{proof}





\section{$(q,p)$-convergent operators and characterizations of Pe{\l}czy\'{n}ski's type sets} \label{sec:oper-qp}


In Section \ref{sec:small-bound-p-conv} we obtained operator characterizations of $p$-$(V^\ast)$ sets and $p$-$(V)$ sets using weakly sequentially $p$-precompact operators, see Theorems \ref{t:p-V*-set} and \ref{t:p-V-set}. The  purpose of this section is to find operator characterizations of $(p,q)$-$(V^\ast)$ sets and $(p,q)$-$(V)$ sets. For this we recall first some definitions.

Let $L$ be a Banach space and $p\in[1,\infty]$. Recall that a sequence $\{x_n\}_{n\in\w}$ in $L$ is called {\em strongly $p$-summable} if $\big(\|x_n\|)\in\ell_p$ (or $\big(\|x_n\|)\in c_0$ if $p=\infty$). The family of all strongly $p$-summable sequences in $L$ is denoted by $\ell_p^s(L)$ (or by $c_0^s(L)$ if $p=\infty$). Recall also that $\ell_p^s(L)$ and $c_0^s(L)$ are Banach spaces under pointwise operations and the norm defined by
\[
\big\|(x_n)\big\|^{strong}_p:= \Big( \sum_{n\in\w} \|x_n\|^p\Big)^{1/p}\;\; \mbox{ and }\;\; \big\|(x_n)\big\|^{strong}_\infty:=\sup_{n\in\w} \big\|x_n\big\| ,
\]
respectively.

The notion of strongly $p$-summable sequences in Banach spaces can be naturally extended to the general case as follows, see \S~19.4 of \cite{Jar}: a sequence $\{x_n\}_{n\in\w}$ in an lcs $L$ is called {\em strongly $p$-summable} if $\big(q_U(x_n)\big)\in\ell_p$ (or $\big(q_U(x_n)\big)\in c_0$ if $p=\infty$) for every $U\in\Nn_0^c(E)$, where as usual $q_U$ denotes the gauge functional of $U$.
Since $U\subseteq V$ implies $q_U\geq q_V$,  for Banach spaces, this notion of strong $p$-summability coincides with the above mentioned usual notion.

Below we introduce a new class of operators which allows to characterize $(p,q)$-$(V^\ast)$ sets and $(p,q)$-$(V)$ sets.
\begin{definition} \label{def:qp-summable} {\em
Let $1\leq p\leq q\leq \infty$, and let $E$ and $L$ be locally convex spaces. A linear map $T:E\to L$ is called {\em $(q,p)$-convergent} if it sends weakly $p$-summable sequences in $E$ to strongly $q$-summable sequences in $L$. \qed}
\end{definition}

\begin{remark} \label{rem:qp-convergent}{\em
(i) Our terminology is based on the evident fact that an operator $T:E\to L$ is $(\infty,p)$-convergent if and only if it is $p$-convergent and the important notion of $(q,p)$-summing operators between Banach spaces (see \cite{DJT}, \cite{Pietsch} and Section \ref{sec:p-summing} below).

(ii) The condition $p\leq q$ in Definition \ref{def:qp-summable} is important because if $q<p$, then only the zero linear map can be $(q,p)$-convergent. Indeed, if $T\not=0$, take $(\lambda_n)\in\ell_p\SM \ell_q$ (or $\in c_0\SM \ell_q$ if $p=\infty$) and $U\in\Nn_0^c(L)$ such that $q_U\big(T(x)\big)\not=0$ for some $x\in E$. Then the sequence $\{x_n=\lambda_n x\}_{n\in\w}$ is weakly $p$-summable. However, the series $\sum_{i\in\w} q_U\big(T(x_i)\big)^q= q_U\big(T(x)\big)^q \sum_{i\in\w} |\lambda_n|^q$ diverges.\qed}
\end{remark}

Using $(p,q)$-$(V)$ sets we can characterize $(q,p)$-convergent operators into normed spaces as follows.
\begin{theorem} \label{t:qp-convergent-pq-V}
Let $1\leq p\leq q\leq \infty$, $E$ be a locally convex space, and let $T$ be an operator from $E$ to a normed space $L$. Then $T$ is $(q,p)$-convergent if and only if $T^\ast(B_{L'})$ is a $(p,q)$-$(EV)$ subset of $E'$.
\end{theorem}

\begin{proof}
First we observe that for every $x\in E$, we have
\[
\|T(x)\|=\sup_{\chi\in B_{L'}} |\langle \chi,T(x)\rangle| = \sup_{\chi\in B_{L'}} |\langle T^\ast(\chi),x\rangle|.
\]
Now, for every weakly $p$-summable sequence $\{x_n\}_{n\in\w}$ in  $E$, we obtain that $\big(\|T(x_n)\|)_n\in\ell_q$ (or $\in c_0$ if $q=\infty$) if and only if
\[
\Big(\sup_{\chi\in B_{L'}} |\langle T^\ast(\chi),x_n\rangle|\Big) \in \ell_q \; \; \;\; \big(\mbox{or } \in c_0 \; \mbox{ if $q=\infty$}\big).
\]
Taking into account that $B_{L'}$ and hence also $T^\ast(B_{L'})$ are equicontinuous, it follows that $\big(\|T(x_n)\|)_n\in\ell_q$ (or $\in c_0$ if $q=\infty$) if and only if $T^\ast(B_{L'})$ is a $(p,q)$-$(EV)$ subset of $E'$.\qed
\end{proof}
The next corollary complements Corollary \ref{c:p-converg-adjoint-weakly-compact}.

\begin{corollary} \label{c:qp-converg-adjoint-weakly-compact}
Let $1\leq p\leq q\leq \infty$, $E$ be a locally convex space with the property $EV_{(p,q)}$ {\rm(}resp., the property $sEV_{(p,q)}${\rm)}, and let $T$ be an operator from $E$ to a normed space $L$. If $T$ is $(q,p)$-convergent, then $T^\ast:L'_\beta\to E'_\beta$ is weakly {\rm(}resp., sequentially{\rm)} compact. The converse is true if every weakly compact, absolutely convex and equicontinuous subset of $E'_\beta$ is a $(p,q)$-$(EV)$ set.
\end{corollary}

\begin{proof}
By Theorem \ref{t:qp-convergent-pq-V}, the absolutely convex set $T^\ast(B_{L'})$ is a $(p,q)$-$(EV)$ subset of $E'$. Therefore, by the property $EV_{(p,q)}$ (resp., the property $sEV_{(p,q)}$), $T^\ast(B_{L'})$ is a relatively weakly (resp., sequentially) compact subset of $E'_\beta$, as desired.

Conversely, assume that  every weakly compact, absolutely convex and equicontinuous subset of $E'_\beta$ is a $(p,q)$-$(EV)$ set and $T^\ast(B_{L'})$ is a relatively weakly (resp., sequentially) compact subset of $E'_\beta$. Since $T^\ast(B_{L'})$ is also equicontinuous we obtain that its closure $\cl_{E'_\beta}\big(T^\ast(B_{L'})\big)$ and hence also $T^\ast(B_{L'})$ are $(p,q)$-$(EV)$ sets. Thus, by Theorem \ref{t:qp-convergent-pq-V}, $T$ is  $(q,p)$-convergent.\qed
\end{proof}

The following assertion is dual to Theorem \ref{t:qp-convergent-pq-V}.

\begin{theorem} \label{t:qp-convergent-pq-V-adj}
Let $1\leq p\leq q\leq \infty$, $E$ be a locally convex space, and let $T$ be an operator from a normed space $L$ to $E'_\beta$. Then $T(B_L)$ is a $(p,q)$-$(V)$ subset of $E'$ if and only if the linear map $T^\ast{\restriction}_E$ from $E$ to $L'_\beta$ is $(q,p)$-convergent.
\end{theorem}

\begin{proof}
First we note that for every $x\in E$, we have (recall that $J_E:E\to E''$ denotes the canonical map)
\begin{equation} \label{equ:qp-convergent-pq-V-adj-1}
\|T^\ast(x)\|_{L'_\beta}=\|T^\ast(J_E(x))\|_{L'_\beta}=\sup_{y\in B_L} |\langle T^\ast(J_E(x)),y\rangle|= \sup_{y\in B_L} |\langle T(y),x\rangle|.
\end{equation}
Let now $\{x_n\}_{n\in\w}$ be a weakly $p$-summable sequence in $E$. By (\ref{equ:qp-convergent-pq-V-adj-1}), we have
$\big(\|T^\ast(x_n)\|_{L'_\beta}\big)_{n\in\w} =\big( \sup_{y\in B_L} |\langle T(y),x_n\rangle|\big)_{n\in\w}$. Then the theorem follows from the definition of $(p,q)$-$(V)$ sets and  the definition of $(q,p)$-convergent linear maps.\qed
\end{proof}

The next theorem characterizes barrelled lcs for which $(p,q)$-$(V)$ subsets of $E'$ have additional compact-type topological property.

\begin{theorem} \label{t:qp-convergent-V-compact}
Let $1\leq p\leq q\leq \infty$, let  $E$ be a {\rm(}resp., an $\aleph_0$-barrelled{\rm)} locally convex space. Then the following assertions are equivalent:
\begin{enumerate}
\item[{\rm(i)}] if $L$ is a normed space and $T:E\to L$ is a $(q,p)$-convergent operator, then $T^\ast:L'_\beta\to E'_\beta$ is weakly sequentially compact {\rm(}resp., sequentially compact,  weakly sequentially $p$-compact or  weakly sequentially $p$-precompact{\rm)};
\item[{\rm(ii)}] the same as {\rm(i)} with $L=\ell_\infty$;
\item[{\rm(iii)}] each $(p,q)$-$(EV)$ {\rm(}resp., $(p,q)$-$(V)${\rm)} subset of $E'$ is relatively weakly sequentially compact {\rm(}resp.,  relatively sequentially compact,  relatively  weakly sequentially $p$-compact or  weakly sequentially $p$-precompact{\rm)} in $E'_\beta$.
\end{enumerate}
\end{theorem}

\begin{proof}
(i)$\Ra$(ii) is clear.

(ii)$\Ra$(iii) Let $B$ be a $(p,q)$-$(EV)$ (resp., $(p,q)$-$(V)$) subset of $E'$. Fix an arbitrary sequence $S=\{\chi_n\}_{n\in\w}$ in $B$. We claim that $S$ is equicontinuous. This clear if $B$ be a $(p,q)$-$(EV)$ set. Assume that $E$ is $\aleph_0$-barrelled.  Since $B$ is weak$^\ast$ bounded (see (ii) of Lemma \ref{l:V-set-1}), the  $\aleph_0$-barrelleddeness of $E$ implies that $S$ is equicontinuous. 
Therefore we can apply Lemma \ref{l:p-operator-L-inf} to get that the linear map $T:E\to \ell_\infty$ defined  by
\[
T(x):=\big(\langle\chi_n,x\rangle\big)_{n\in\w} \quad (x\in E).
\]
is continuous.

Let now $\{x_n\}_{n\in\w}$ be a weakly $p$-summable sequence in $E$. Since $\{\chi_n\}_{n\in\w}$ is also  a $(p,q)$-$(V)$ set, we obtain
\begin{equation} \label{equ:qp-convergent-V-compact-1}
\big(\sup_{i\in\w} |\langle\chi_i,x_n\rangle|\big)_{n\in\w}\in\ell_q \quad (\mbox{or $\in c_0$ if $q=\infty$}).
\end{equation}
Observe that $\|T(x_n)\|_{\ell_\infty}=\sup_{i\in\w} |\langle\chi_i,x_n\rangle|$ for every $n\in\w$. Therefore (\ref{equ:qp-convergent-V-compact-1}) implies
\[
\big(\|T(x_n)\|_{\ell_\infty}\big)_{n\in\w}\in\ell_q \quad (\mbox{or $\in c_0$ if $q=\infty$}),
\]
which means that $T$ is a $(q,p)$-convergent operator. Therefore, by (ii), the adjoint operator $T^\ast$ is weakly sequentially compact (resp., sequentially compact,  weakly sequentially $p$-compact or  weakly sequentially $p$-precompact).

For every $n\in\w$, if $e_n^\ast$ is the $n$th coordinate functional of $\ell_\infty$ and $x\in E$, then $\langle T^\ast(e_n^\ast),x\rangle=\langle e_n^\ast,T(x)\rangle=\langle\chi_n,x\rangle$ and hence $T^\ast(e_n^\ast)=\chi_n$. Therefore the sequence $\{\chi_n\}_{n\in\w}$ has a subsequence which weakly converges (resp., converges, weakly sequentially $p$-converges or weakly sequentially  $p$-Cauchy). Thus $B$ is relatively weakly sequentially compact (resp.,
relatively sequentially compact,  relatively  weakly sequentially $p$-compact or  weakly sequentially $p$-precompact).
\smallskip

(iii)$\Ra$(i) Let $T:E\to L$ be a $(q,p)$-convergent operator to a normed space $L$. Then, by Theorem \ref{t:qp-convergent-pq-V}, $T^\ast(B_{L'})$ is a $(p,q)$-$(EV)$ subset of $E'$ (and hence also a $(p,q)$-$(V)$ set). Now (iii) implies that $T^\ast$ is a weakly sequentially compact (resp., sequentially compact,  weakly sequentially $p$-compact or  weakly sequentially $p$-precompact) operator.\qed
\end{proof}

\begin{theorem} \label{t:qp*-convergent-compact}
Let $1\leq p\leq q\leq \infty$, let  $E$ be a locally convex space. Then the following assertions are equivalent:
\begin{enumerate}
\item[{\rm(i)}] if $L$ is a normed space and $T:L\to E'_\beta$ is an operator such that $T^\ast{\restriction}_E: E \to L'_\beta$ is a $(q,p)$-convergent linear map, then $T$ is weakly sequentially compact {\rm(}resp., sequentially compact,  weakly sequentially $p$-compact or  weakly sequentially $p$-precompact{\rm)};
\item[{\rm(ii)}] the same as {\rm(i)} with $L=\ell_1^0$;
\item[{\rm(iii)}] each bounded  subset of $E'_\beta$ which is a $(p,q)$-$(V)$ set is relatively weakly sequentially compact {\rm(}resp.,  relatively sequentially compact,  relatively  weakly sequentially $p$-compact or  weakly sequentially $p$-precompact{\rm)}.
\end{enumerate}
Moreover, if $E'_\beta$ is locally complete, then {\rm(i)-(iii)} are equivalent to
\begin{enumerate}
\item[{\rm(iv)}] the same as {\rm(i)} with $L=\ell_1$.
\end{enumerate}
\end{theorem}

\begin{proof}
(i)$\Ra$(ii) and (i)$\Ra$(iv) are clear.

(ii)$\Ra$(iii) and (iv)$\Ra$(iii): Let $B$ be a bounded $(p,q)$-$(V)$ subset of $E'_\beta$. Fix an arbitrary sequence $S=\{\chi_n\}_{n\in\w}$ in $B$, so $S$ is a bounded subset of $E'_\beta$. Therefore, by Proposition \ref{p:seq-p-convergent}, the linear map $T:\ell_1^0 \to E'_\beta$ (or $T:\ell_1 \to E'_\beta$ if $E'_\beta$ is locally complete) defined by
\[
T(a_0 e_0+\cdots+a_ne_n):=a_0 \chi_0+\cdots+ a_n \chi_n \quad (n\in\w, \; a_0,\dots,a_n\in\IF).
\]
is continuous. For every $n\in\w$ and each $x\in E$, we have
$
\langle T^\ast{\restriction}_E(x),e_n\rangle=\langle T(e_n),x\rangle=\langle\chi_n,x\rangle
$
and hence $T^\ast{\restriction}_E(x)=\big(\langle\chi_n,x\rangle\big)_n\in\ell_\infty$. In particular, $\|T^\ast{\restriction}_E(x)\|_{\ell_\infty}=\sup_{n\in\w} |\langle\chi_n,x\rangle|$.

Let now $\{x_n\}_{n\in\w}$ be a weakly $p$-summable sequence in $E$. Since $B$ and hence also $S$ are $(p,q)$-$(V)$ sets we obtain $\big(\|T^\ast{\restriction}_E(x_n)\|_{\ell_\infty}\big)=\big(\sup_{i\in\w} |\langle\chi_i,x_n\rangle|\big)\in\ell_q$ (or $\in c_0$ if $q=\infty$). Therefore $T^\ast{\restriction}_E$ is $(q,p)$-convergent, and hence, by (ii) or (iv),  $T$ is weakly sequentially compact (resp., sequentially compact,  weakly sequentially $p$-compact or  weakly sequentially $p$-precompact). Therefore $S$ has a weakly convergent (resp., convergent, weakly $p$-convergent, or weakly $p$-Cauchy) subsequence, as desired.
\smallskip

(iii)$\Ra$(i)  Let $T:L\to E'_\beta$ be an operator from a normed space such that $T^\ast{\restriction}_E: E \to L'_\beta$ is a $(q,p)$-convergent linear map. Then, by Theorem \ref{t:qp-convergent-pq-V-adj}, $T(B_L)$ is a $(p,q)$-$(V)$ set and hence it is relatively weakly sequentially compact (resp.,  relatively sequentially compact,  relatively  weakly sequentially $p$-compact or  weakly sequentially $p$-precompact). Thus $T$ is weakly sequentially compact (resp., sequentially compact,  weakly sequentially $p$-compact or  weakly sequentially $p$-precompact).\qed
\end{proof}

We need the following result which is similar to Theorem \ref{t:qp-convergent-pq-V-adj}.
\begin{theorem} \label{t:qp-convergent-pq-V-1}
Let $1\leq p\leq q\leq \infty$, $E$ be a locally convex space, and let $T$ be an operator from a normed space $L$ to $E$. Then $T(B_L)$ is a $(p,q)$-$(V^\ast)$ subset of $E$ if and only if  $T^\ast:E'_\beta \to L'_\beta$ is $(q,p)$-convergent.
\end{theorem}

\begin{proof}
First we note that for every $\chi\in E'$, we have
\begin{equation} \label{equ:qp-convergent-pq-V-1}
\|T^\ast(\chi)\|_{L'_\beta}=\sup_{y\in B_L} |\langle T^\ast(\chi),y\rangle|= \sup_{y\in B_L} |\langle \chi,T(y)\rangle|.
\end{equation}
Let $\{\chi_n\}_{n\in\w}$ be a weakly $p$-summable sequence in $E'_\beta$. Then, by (\ref{equ:qp-convergent-pq-V-1}), we have
$\big(\|T^\ast(\chi_n)\|_{L'_\beta}\big)_{n\in\w} =\big( \sup_{y\in B_L} |\langle \chi_n, T(y)\rangle|\big)_{n\in\w}$. Now the theorem follows from the definition of $(p,q)$-$(V^\ast)$ sets and  the definition of $(q,p)$-convergent operators.\qed
\end{proof}

\begin{theorem} \label{t:qp*-convergent-L}
Let $1\leq p\leq q\leq \infty$, and let  $E$ be a locally convex space. Then the following assertions are equivalent:
\begin{enumerate}
\item[{\rm(i)}] if $L$ is a normed space and $T:L\to E$ is an operator such that $T^\ast: E'_\beta \to L'_\beta$ is $(q,p)$-convergent, then $T$ is weakly sequentially compact {\rm(}resp., sequentially compact,  weakly sequentially $p$-compact or  weakly sequentially $p$-precompact{\rm)};
\item[{\rm(ii)}] the same as {\rm(i)} with $L=\ell_1^0$;
\item[{\rm(iii)}] each $(p,q)$-$(V^\ast)$ subset of $E$ is relatively weakly sequentially compact {\rm(}resp.,  relatively sequentially compact,  relatively  weakly sequentially $p$-compact or  weakly sequentially $p$-precompact{\rm)}.
\end{enumerate}
Moreover, if $E$ is locally complete, then {\rm(i)-(iii)} are equivalent to
\begin{enumerate}
\item[{\rm(iv)}] the same as {\rm(i)} with $L=\ell_1$.
\end{enumerate}
\end{theorem}

\begin{proof}
(i)$\Ra$(ii) and (i)$\Ra$(iv) are clear.

(ii)$\Ra$(iii) and (iv)$\Ra$(iii): Let $A$ be a $(p,q)$-$(V^\ast)$ subset of $E$. Fix an arbitrary sequence $S=\{x_n\}_{n\in\w}$ in $A$, so $S$ is a bounded subset of $E$. Therefore, by Proposition \ref{p:seq-p-convergent}, the linear map $T:\ell_1^0 \to E$ (or $T:\ell_1 \to E$ if $E$ is locally complete) defined by
\[
T(a_0 e_0+\cdots+a_ne_n):=a_0 x_0+\cdots+ a_n x_n \quad (n\in\w, \; a_0,\dots,a_n\in\IF).
\]
is continuous. For every $n\in\w$ and each $\chi\in E'$, we have
$
\langle T^\ast(\chi),e_n\rangle=\langle \chi,T(e_n)\rangle=\langle\chi,x_n\rangle
$
and hence $T^\ast(\chi)=\big(\langle\chi,x_n\rangle\big)_n\in\ell_\infty$. In particular, $\|T^\ast(\chi)\|_{\ell_\infty}=\sup_{n\in\w} |\langle\chi,x_n\rangle|$.

Let now $\{\chi_n\}_{n\in\w}$ be a weakly $p$-summable sequence in $E'_\beta$. Since $A$ and hence also $S$ are $(p,q)$-$(V^\ast)$ sets we obtain $\big(\|T^\ast(\chi_n)\|_{\ell_\infty}\big)=\big(\sup_{i\in\w} |\langle\chi_n,x_i\rangle|\big)\in\ell_q$ (or $\in c_0$ if $q=\infty$). Therefore $T^\ast$ is $(q,p)$-convergent, and hence, by (ii) or (iv),  $T$ is weakly sequentially compact (resp., sequentially compact,  weakly sequentially $p$-compact or  weakly sequentially $p$-precompact). Therefore $S=\{T(e_n)\}_{n\in\w}$ has a weakly convergent (resp., convergent, weakly $p$-convergent, or weakly $p$-Cauchy) subsequence, as desired.
\smallskip

(iii)$\Ra$(i)  Let $T:L\to E$ be an operator from a normed space such that $T^\ast: E'_\beta \to L'_\beta$ is a $(q,p)$-convergent operator. Then, by Theorem \ref{t:qp-convergent-pq-V-1}, $T(B_L)$ is a $(p,q)$-$(V^\ast)$ set and hence it is relatively weakly sequentially compact (resp.,  relatively sequentially compact,  relatively  weakly sequentially $p$-compact or  weakly sequentially $p$-precompact). Thus $T$ is weakly sequentially compact (resp., sequentially compact,  weakly sequentially $p$-compact or  weakly sequentially $p$-precompact).\qed
\end{proof}

To obtain another characterization of $(p,q)$-$(V^\ast)$ sets we need the next assertion.
For a set $A$, we denote by $\ell^0_p(A)$ the subspace of the Banach space $\ell_p(A)$ consisting of all vectors with finite support.

\begin{proposition} \label{p:L1-0-E}
Let  $A$ be a bounded subset of a locally convex space $(E,\tau)$, $B=\cacx(A)$, and let $i_B:(E_B,\|\cdot\|)\to E$ be the identity operator. Let $T_A:\ell_1^0(A) \to (E_B,\|\cdot\|)$ be the identity inclusion  defined by
\[
T_A(\lambda_0 a_0+\cdots+\lambda_n a_n):=\lambda_0 a_0+\cdots+\lambda_n a_n \quad (n\in\w, \; \lambda_0,\dots,\lambda_n\in\IF, \; a_0,\dots,a_n\in A).
\]
Then $T_A$ is bounded. Moreover, if $E$ is locally complete, then $T_A$ can be extended to a continuous operator ${\tilde T}_A$ from $\ell_1(A)$ to $(E_B,\|\cdot\|)$.
In particular, $i_B\circ T_A: \ell_1^0(A)\to E$ {\rm(}or  $i_B\circ {\tilde T}_A: \ell_1(A)\to E$ if $E$ is locally complete{\rm)} is bounded and hence continuous.
\end{proposition}

\begin{proof}
First we recall that, by Proposition \ref{p:bounded-norm}, $i_B$ is indeed an operator.

By the definition of $(E_{B},\|\cdot\|)$ we observe that $\|b\|\leq 1$ for every $b\in B$. Therefore, if $|\lambda_0|+\cdots+|\lambda_n|\leq 1$ and $a_0,\dots,a_n\in A$, we obtain
\[
\big\|T_A(\lambda_0 a_0+\cdots+\lambda_n a_n)\big\|= \|\lambda_0 a_0+\cdots+\lambda_n a_n\|\leq 1
\]
and hence $T:\ell_1^0(A) \to (E_{B},\|\cdot\|)$ is bounded. 

Assume that $E$ is locally complete. Then, by Proposition 5.1.6 of \cite{PB}, the space $E_B$ is a Banach space. Therefore the operator $T:\ell_1^0(A) \to (E_{B},\|\cdot\|)$  can be extended to an operator ${\tilde T}_A:\ell_1(A) \to (E_{B},\|\cdot\|)$. \qed 
\end{proof}


\begin{theorem} \label{t:V*-p-summing}
Let $1\leq p\leq q\leq\infty$. Then a subset $A$ of a locally convex space $E$ is a $(p,q)$-$(V^\ast)$ set if and only if the adjoint operator $T_A^\ast: E'_\beta\to \ell_\infty(A)$ is $(q,p)$-convergent.
\end{theorem}

\begin{proof}
By Proposition \ref{p:L1-0-E},  every bounded subset $A$ of $E$ defines an operator $T_A:\ell_1^0(A) \to E$ by
\[
T_A(\lambda_0 a_0+\cdots+\lambda_n a_n):=\lambda_0 a_0+\cdots+\lambda_n a_n \quad (n\in\w, \; \lambda_0,\dots,\lambda_n\in\IF, \; a_0,\dots,a_n\in A).
\]
Observe that for each $\chi\in E'$, the $a$th coordinate $T_A^\ast(\chi)(a)$ of $T_A^\ast(\chi)$ is
\[
T_A^\ast(\chi)(a) = \langle T_A^\ast(\chi),a\rangle=\langle\chi,T_A(a)\rangle=\langle\chi,a\rangle,
\]
and hence
\begin{equation} \label{equ:p-summing-5}
\|T^\ast_A(\chi)\|_{\ell_\infty(A)}=\sup_{a\in A} |T^\ast_A(\chi)(a)|=\sup_{a\in A}|\langle\chi,a\rangle|
\end{equation}

Now, by definition, a subset $A$ of $E$ is a  $(p,q)$-$(V^\ast)$ set if and only if $\big(\sup_{a\in A}|\langle\chi_n,a\rangle|\big)\in \ell_q$ (or $\in c_0$ if $q=\infty$) for every weakly $p$-summable sequence $\{\chi_n\}_{n\in\w}$ in $E'_\beta$, and hence, by  (\ref{equ:p-summing-5}), if and only if $\big(\|T^\ast_A(\chi_n)\|_{\ell_\infty(A)}\big)\in \ell_q$ (or $\in c_0$ if $q=\infty$) for every weakly $p$-summable sequence $\{\chi_n\}_{n\in\w}$ in $E'_\beta$, i.e., $T^\ast_A$  is a $(q,p)$-convergent operator.\qed
\end{proof}


\section{$(q,p)$-summing operators and applications} \label{sec:p-summing}


Let $1\leq p\leq q\leq \infty$.
Recall (see Definition \ref{def:qp-summable}) that an operator $T$ from an lcs $E$ to an lcs $L$ is {\em $(q,p)$-convergent} if it sends weakly $p$-summable sequences in $E$ to strongly $q$-summable sequences in $L$. Using this class of operators we characterized in Theorem \ref{t:V*-p-summing} $(p,q)$-$(V^\ast)$ subsets of $E$. Therefore it is natural to find some classes of operators from $E$ to $L$ which are  $(q,p)$-convergent. If $E$ and $L$ are Banach spaces it is well known (see \cite{DJT}) that the class of $(q,p)$-convergent operators coincides with the class of $(q,p)$-summing operators (a detailed proof of a more general result is given in Corollary \ref{c:qp-summable-Banach} below). In this section we define  $(q,p)$-summing operators between arbitrary locally convex spaces, show that the most important results for Banach spaces (as the Pietsch Domination Theorem and the Pietsch factorization Theorem) remain true  and show that any $(q,p)$-summing operator is $(q,p)$-convergent.

First we recall the definition of $(q,p)$-summing operators between Banach spaces which plays an essential role in many areas of Banach Space Theory, see the classical books \cite{DJT} and \cite{LT-1}.
\begin{definition} \label{def:p-summing-Banach} {\em
Let $1\leq p\leq q<\infty$. An operator $T:X\to Y$ between Banach spaces is called {\em $(q,p)$-summing} if there is a constant $C>0$ such that for every $n\in\w$ and each $x_0,\dots,x_n\in X$, we have
\begin{equation} \label{equ:Pietsch-00}
\Big( \sum_{i=0}^n \|T(x_i)\|^q\Big)^{1/q} \leq C \cdot \sup\Big\{ \Big( \sum_{i=0}^n  |\langle\chi,x_i\rangle|^p\Big)^{1/p}: \chi\in B_{X'}\Big\}.
\end{equation}
The least $C$ for which the inequality (\ref{equ:Pietsch-00}) always holds is denoted by $\pi_{q,p}(T)$. We shall write $\prod_{q,p}(X,Y)$ for the family of all $(q,p)$-summing operators from $X$ into $Y$. If $q=p$ we shall say that $T$ is {\em $p$-summing} and set $\prod_{p}(X,Y):=\prod_{p,p}(X,Y)$ and $\pi_{p}(T)=\pi_{p,p}(T)$.\qed}
\end{definition}

Since the both sides in (\ref{equ:Pietsch-00}) are homogenous, one can assume that all vectors $x_0,\dots,x_n$ belongs to the closed unit ball $B_X$ of $X$. Observe also that $B_X$ is an absolutely convex closed neighborhood of $0\in X$ (that is $B_X\in \Nn_0^c(X)$) and  $B_{X'}=B_X^\circ$. Therefore the notion of $(q,p)$-summing operator between Banach spaces is defined by the very specific neighborhood  $B_X$ of zero (or by the dual pair $\big(B_X,B_X^\circ\big)$).
Taking into account these remarks one can naturally define $(q,p)$-summing operators between locally convex spaces as follows.

\begin{definition} \label{def:U-qp-summing} {\em
Let $1\leq p\leq q<\infty$, $E$ and $L$ be locally convex spaces, and let $U$ be an absolutely convex closed subset of $E$. An operator $T:E\to L$ is called {\em $(q,p)$-summing with respect to $U$} or {\em $U$-$(q,p)$-summing} if for every $V\in\Nn_0^c(L)$ there is a constant $C_V>0$ such that for every $n\in\w$ and each $x_0,\dots,x_n\in U$, we have
\begin{equation} \label{equ:Pietsch}
\Big( \sum_{i=0}^n q_V\big(T(x_i)\big)^q\Big)^{1/q} \leq C_V \cdot \sup\Big\{ \Big( \sum_{i=0}^n  |\langle\chi,x_i\rangle|^p\Big)^{1/p}: \chi\in U^\circ\Big\},
\end{equation}
where $q_V$ is the gauge of $V$. The least $C_V$ for which the inequality (\ref{equ:Pietsch}) always holds is denoted by $\pi_{q,p}^{U,V}(T)$.  We denote by $\prod_{q,p}^U(E,L)$ the family of all $U$-$(q,p)$-summing operators from $E$ to $L$. If $q=p$ we shall say that $T$ is {\em $U$-$p$-summing} and set $\prod_{p}^U(E,L):=\prod_{p,p}^U(E,L)$ and $\pi_{p}^{U,V}(T):=\pi_{p,p}^{U,V}(T)$.\qed}
\end{definition}

In the case when $L$ is a Banach space Definition \ref{def:U-qp-summing} can be formulated in a simpler form (this follows from that fact that if $V\subseteq V'$, then $q_V\geq q_{V'}$).

\begin{definition} \label{def:U-p-summing} {\em
Let $1\leq p\leq q<\infty$, $E$ be a locally convex space, $L$ be a Banach space, and let $U$ be an absolutely convex closed subset of $E$. An operator $T:E\to L$ is called {\em $(q,p)$-summing with respect to $U$} or {\em $U$-$(q,p)$-summing} if there is a constant $C>0$ such that for every $n\in\w$ and each $x_0,\dots,x_n\in U$, we have
\begin{equation} \label{equ:Pietsch-0}
\Big( \sum_{i=0}^n \|T(x_i)\|^q\Big)^{1/q} \leq C \cdot \sup\Big\{ \Big( \sum_{i=0}^n  |\langle\chi,x_i\rangle|^p\Big)^{1/p}: \chi\in U^\circ\Big\}.
\end{equation}
The least $C$ for which the inequality (\ref{equ:Pietsch-0}) always holds is denoted by $\pi_{q,p}^U(T)$. We denote by $\prod_{q,p}^U(E,L)$ the family of all $U$-$(q,p)$-summing operators from $E$ to $L$. If $q=p$ we shall say that $T$ is {\em $U$-$p$-summing} and set $\prod_{p}^U(E,L):=\prod_{p,p}^U(E,L)$ and $\pi_{p}^U(T):=\pi_{p,p}^U(T)$.\qed}
\end{definition}


We emphasize that the notion of $U$-$(q,p)$-summing operators indeed depends on $U$ (or, more precisely, on the dual pair $(U,U^\circ)$) even for Banach spaces as the following example shows.

\begin{example} \label{exa:p-summing-U}
Let $1\leq p<\infty$, and let $(\lambda_n)\in\ell_2$ be such that $0<\lambda_n<1$ for every $n\in\w$. Let $E=L=\ell_2$, and let $U:=E \cap \prod_{n\in\w} \big[-\tfrac{1}{\lambda_n^2},\tfrac{1}{\lambda_n^2}\big]$. Then the diagonal operator
\[
D_\lambda:E \to E \mbox{ defined by } D_\lambda(a_n):=(\lambda_n a_n)
\]
is $B_{E}$-$p$-summing {\rm(}i.e., $D_\lambda$ is $p$-summing in the sense of Definition \ref{def:p-summing-Banach}{\rm)} but it is not $U$-$p$-summing.
\end{example}

\begin{proof}
The operator $D_\lambda$ is $B_{E}$-$p$-summing by Theorem 2.b.4 of \cite{LT-1}. Observe that $\tfrac{1}{\lambda_k}e_k\in U$ for all $k\in\w$, where $\{e_n\}_{n\in\w}$ denotes the standard unit basis of $\ell_2$. It is clear that if $\chi=(b_n)\in U^\circ$, then  $|b_n|\leq \lambda_n^2$. Therefore  the right hand side of (\ref{equ:Pietsch-0}), applied to $n=0$ and $x_0^k=\tfrac{1}{\lambda_k}e_k$, is equal to
\[
C\cdot \sup\{ |\langle\chi,x_0^k\rangle|: \chi\in U^\circ\}\leq C\cdot |\lambda_k| \to 0\;\; (k\to\infty).
\]
On the other hand, since $\|D_\lambda(x_0^k)\|=\|e_k\|=1$ for every $k\in\w$, we see that the operator $D_\lambda$ is  not $U$-$p$-summing.\qed
\end{proof}

Taking into consideration Theorem \ref{t:V*-p-summing}, the case when $L$ is a Banach space is of independent interest.
Let $1\leq p\leq q<\infty$, $E$ be a locally convex space,  $L$ be a Banach space, and let $T:E\to L$ be an operator. Then the {\em associative linear map} ${\hat T}: \ell_p[E] \to \ell_q^s(L)$ is defined by
\[
{\hat T}\big((y_n)_{n\in\w}\big):= \big((T(y_n))_{n\in\w}\big),\quad (y_n)_{n\in\w}\in \ell_p[E].
\]
The question when ${\hat T}$ is well-defined is considered in (i) of Proposition \ref{p:qp-summing-norm} below.

Let $E$ be a Banach space. Then Proposition 2.1 of \cite{DJT} states that if $p=q$, then $T$ is $p$-summing if and only if ${\hat T}$ is continuous, and Proposition 17.2.3 of \cite{Pietsch} states that $T$ is $(q,p)$-summing if and only if it is  $(q,p)$-convergent. Below we extend these results. Recall that
for every lcs $E$ and $U\in\Nn_0^c(E)$, the seminorm $\rho_U:\ell_p^w(E)\to\IR$ is defined by $\rho_U\big( (x_n)_{n\in\w}\big):=\sup_{\chi\in U^\circ} \big\|\big( \langle\chi,x_n\rangle \big)_n\big\|_{\ell_p}$.


\begin{proposition} \label{p:qp-summing-norm}
Let $1\leq p\leq q<\infty$, $E$ be a locally convex space,  $L$ be a Banach space, and let $T:E\to L$ be an operator.
\begin{enumerate}
\item[{\rm(i)}] $T$ is $U$-$(q,p)$-summing for some $U\in\Nn_0^c(E)$ if and only if ${\hat T}$ is well-defined and continuous.
\item[{\rm(ii)}] If ${\hat T}$ is continuous, then  $T$ is $(q,p)$-convergent.
\item[{\rm(iii)}] If the space $\ell_p[E]$ is barrelled and  $T$ is $(q,p)$-convergent, then ${\hat T}$ is continuous.
\end{enumerate}
\end{proposition}

\begin{proof}
(i) Assume that $T$ is a $U$-$(q,p)$-summing operator for some $U\in\Nn_0^c(E)$. To prove that ${\hat T}$ is  well-defined and continuous it suffices to show that
\[
\big\| \big((T(y_n))_{n\in\w}\big)\|_{\ell_q} \leq \pi_{q,p}^U
\]
for every  $(y_n)\in \ell_p[E]$  such that $\rho_U\big((y_n)_{n\in\w}\big)\leq 1$. 


Now, let $\{x_n\}_{n\in\w}$ be a weakly $p$-summable sequence in $E$ such that $\rho_U\big((x_n)_{n\in\w}\big)\leq 1$. Since $T$ is $U$-$(q,p)$-summing, for every $n\in\w$, we have
\begin{equation} \label{equ:qp-summing-norm-1}
\Big( \sum_{i=0}^n \|T(x_i)\|^q\Big)^{1/q} \leq \pi_{q,p}^U \cdot \sup\Big\{ \Big( \sum_{i=0}^n  |\langle\chi,x_i\rangle|^p\Big)^{1/p}: \chi\in U^\circ\Big\}.
\end{equation}
Then
\begin{equation} \label{equ:qp-summing-norm-2}
\begin{aligned}
\big\|\big(T(x_n)\big)_{n\in\w}\|_{\ell_q} & =\sup_{n\in\w} \Big( \sum_{i=0}^n \|T(x_i)\|^q\Big)^{1/q}\leq \pi_{q,p}^U \cdot \sup_{n\in\w}\; \sup\Big\{ \Big( \sum_{i=0}^n  |\langle\chi,x_i\rangle|^p\Big)^{1/p}: \chi\in U^\circ\Big\}\\
& =\pi_{q,p}^U \cdot \sup_{ \chi\in U^\circ} \; \sup_{n\in\w}\Big( \sum_{i=0}^n  |\langle\chi,x_i\rangle|^p\Big)^{1/p}=\pi_{q,p}^U \cdot  \sup_{\chi\in U^\circ} \big\|\big( \langle\chi,x_n\rangle \big)_n\big\|_{\ell_p}\leq \pi_{q,p}^U,
\end{aligned}
\end{equation}
as desired.
\smallskip

Conversely, assume that ${\hat T}$ is well-defined and  continuous. Choose $U\in\Nn_0^c(E)$ and $C>0$ such that
\begin{equation} \label{equ:qp-summing-norm-3}
\big\| \big(T(y_n)\big)_{n\in\w}\|_{\ell_q} \leq C \;\; \mbox{
for every  $(y_n)\in \ell_p[E]$ with $\rho_U\big((y_n)_{n\in\w}\big)\leq 1$.}
\end{equation}
We show that  $T$ is $U$-$(q,p)$-summing with $\pi_{q,p}^U \leq C$. To this end, let $F=\{x_0,\dots,x_n\}\subseteq U$ be arbitrary. For every $i>n$, set $x_i:=0$. Then $(x_n)_{n\in\w}\in \ell_p[E]$. Set $a:=\rho_U\big((x_n)_{n\in\w}\big)$ and consider two possible cases.
\smallskip

{\em Case 1. Assume that $a=0$.} By the definition of $\rho_U$ this means that  $\langle\chi,x_i\rangle=0$ for every $0\leq i\leq n$ and each $\chi\in U^\circ$. Therefore $\spn(F)\subseteq U^{\circ\circ}=U$ and $\rho_U\big((b_nx_n)_{n\in\w}\big)=0$ for each sequence $(b_n)\in\IF^\w$. Then (\ref{equ:qp-summing-norm-3}) implies that $F\subseteq \ker(T)$, and hence
\[
\Big( \sum_{i=0}^n \|T(x_i)\|^q\Big)^{1/q}=0 = C \cdot \sup\Big\{ \Big( \sum_{i=0}^n  |\langle\chi,x_i\rangle|^p\Big)^{1/p}: \chi\in U^\circ\Big\}.
\]
\smallskip

{\em Case 2.  Assume that  $a>0$.} Then $\rho_U\big((\tfrac{x_n}{a})_{n\in\w}\big)=1$ and hence, by (\ref{equ:qp-summing-norm-3}), we obtain
\[
\tfrac{1}{a}\Big( \sum_{i=0}^n \|T(x_i)\|^q\Big)^{1/q}=  \big\| \big(T(\tfrac{x_n}{a})\big)_{n\in\w}\big\|_{\ell_q}\leq C,
\]
or
\[
\Big( \sum_{i=0}^n \|T(x_i)\|^q\Big)^{1/q}\leq C\cdot a =C\cdot \sup\Big\{ \Big( \sum_{i=0}^n  |\langle\chi,x_i\rangle|^p\Big)^{1/p}: \chi\in U^\circ\Big\}.
\]
\smallskip

Cases 1 and 2 show that   $T$ is a $U$-$(q,p)$-summing operator with $\pi_{q,p}^U \leq C$.
\smallskip

(ii) Assume that ${\hat T}$ is continuous. Choose $U\in\Nn_0^c(E)$ and $C>0$ such that
\[
\big\| \big(T(y_n)\big)_{n\in\w}\|_{\ell_q} \leq C \;\; \mbox{
for every  $(y_n)\in \ell_p[E]$ with $\rho_U\big((y_n)_{n\in\w}\big)\leq 1$.}
\]
To show that  $T$ is $(q,p)$-convergent, take an arbitrary weakly $p$-summable sequence $S=\{x_n\}_{n\in\w}$ in $E$. Since $S$ is bounded, there is $\lambda>0$ such that $S\subseteq \lambda U$. Then, by Cases 1 and 2 in the proof of Lemma \ref{l:topology-L^w}, we obtain $\rho_U\big(S)\leq \lambda$. Therefore $\big\| \big(T(y_n)\big)_{n\in\w}\|_{\ell_q} \leq C\lambda$. Thus $T$ is $(q,p)$-convergent.
\smallskip

(iii) Let $\ell_p[E]$ be a barrelled  space, and assume that $T$ is a $(q,p)$-convergent operator. Then we can define a linear map ${\hat T}: \ell_p[E] \to \ell_q^s(L)$ by
\[
{\hat T}\big((y_n)_{n\in\w}\big):= \big(T(y_n)\big)_{n\in\w},\quad (y_n)_{n\in\w}\in \ell_p[E].
\]

Observe that ${\hat T}$ has a closed graph $\Gamma_{{\hat T}}$ because if a net $\big((y_n^\alpha)_{n\in\w},(T(y_n^\alpha))_{n\in\w}\big)$ converges to a point $\big((z_n)_{n\in\w}, (t_n)_{n\in\w}\big)\in \ell_p[E] \times \ell_p^s(L)$, then:
\begin{enumerate}
\item[(1)] $(y_n^\alpha)_{n\in\w}\to (z_n)_{n\in\w}$ in $\ell_p[E]$ and hence $y_n^\alpha\to z_n$ in $E$ for every $n\in\w$, and
\item[(2)]  $(T(y_n^\alpha))_{n\in\w}\to (t_n)_{n\in\w}$ in the Banach space $\ell_p^s(L)$ and hence $T(y_n^\alpha)\to t_n$ in $L$ for every $n\in\w$.
\end{enumerate}
Therefore $t_n=T(z_n)$ for every $n\in\w$, and hence $\big((z_n)_{n\in\w}, (t_n)_{n\in\w}\big)=\big((z_n)_{n\in\w}, (T(z_n))_{n\in\w}\big)$. Thus $\Gamma_{{\hat T}}$ is closed.

Since the space $\ell_p[E]$ is barrelled, the Closed Graph Theorem  14.3.4 of \cite{NaB} implies that ${\hat T}$ is continuous.\qed
\end{proof}

\begin{remark} {\em
If $q<p$, then, by Remark \ref{rem:qp-convergent}, only the zero linear map is $(q,p)$-convergent. Therefore, by Proposition \ref{p:qp-summing-norm}, there is sense to consider $(q,p)$-summable operators only if $q\geq p$.\qed}
\end{remark}

Since, by Corollary \ref{c:L[E]-Banach}, $\ell_p[E]$ is a Fr\'{e}chet space (hence barrelled) if so is $E$,  the next result immediately follows from Proposition \ref{p:qp-summing-norm}. 
\begin{corollary} \label{c:qp-summable-Banach}
Let $1\leq p\leq q<\infty$, $E$ be a locally convex space, $L$ be Banach spaces, and let $T\in\LL(E,L)$. If the space $\ell_p[E]$ is barrelled {\rm(}if, for example, $E$ is a Fr\'{e}chet space{\rm)}, then the following assertions are equivalent:
\begin{enumerate}
\item[{\rm(i)}] $T$ is $U$-$(q,p)$-summable for some $U\in\Nn_0^c(E)$;
\item[{\rm(ii)}] ${\hat T}: \ell_p[E] \to \ell_q^s(L)$ is continuous;
\item[{\rm(iii)}] $T$ is $(q,p)$-convergent.
\end{enumerate}
\end{corollary}



In spite of Example \ref{exa:p-summing-U}, we show below that the most important general results concerning $(q,p)$-summing operators proved in  Sections 2 and 10 of \cite{DJT} remain true also in the general case of Definition \ref{def:U-qp-summing} and  Definition \ref{def:U-p-summing}. We start from the following extension of Proposition 2.3 of \cite{DJT} which states that every finite rank operator between Banach spaces is $p$-summing.
Let $E$ and $L$ be locally convex spaces, and let $H$ be a subspace of $E'$. We shall say that a finite rank operator $T(x)=\sum_{i=0}^n \langle\eta_i,x\rangle \cdot y_i$ from $E$ to $L$ {\em has co-support in $H$} if $\eta_0,\dots,\eta_n\in H$.

\begin{proposition} \label{p:p-summing-vector}
Let $1\leq p\leq q<\infty$, $E$ and $L$ be locally convex spaces, and let $U$ be an absolutely convex closed subset of $E$. Then:
\begin{enumerate}
\item[{\rm(i)}] $\prod_{q,p}^U(E,L)$ is a vector subspace of $\LL(E,L)$;
\item[{\rm(ii)}] $\prod_{p}^U(E,L)$ contains all finite rank operators with co-support in $\spn(U^\circ)$.
\end{enumerate}
\end{proposition}

\begin{proof}
(i) Let $V\in\Nn_0^c(L)$, and let $T,S\in \prod_{q,p}^U(E,L)$ with constants $C_V^T$ and $C_V^S$ in (\ref{equ:Pietsch}). Then for every non-zero $\lambda\in\IF$, (\ref{equ:Pietsch}) is satisfied for the operator $\lambda T$ with the constant $|\lambda|\cdot C_V^T$ and hence $\lambda T\in \prod_{q,p}^U(E,L)$. Set $C_V:=C_V^T+C_V^S$. Then for every $n\in\w$ and $x_0,\dots,x_n\in U$, (\ref{equ:Pietsch}) implies
\[
\begin{aligned}
\Big( \sum_{i=0}^n q_V\big(T(x_i)+S(x_i)\big)^q\Big)^{1/q} &\leq \Big( \sum_{i=0}^n q_V\big(T(x_i)\big)^q\Big)^{1/q}+\Big( \sum_{i=0}^n q_V\big(S(x_i)\big)^q\Big)^{1/q}\\
& \leq C_V \cdot \sup\Big\{ \Big( \sum_{i=0}^n  |\langle\chi,x_i\rangle|^p\Big)^{1/p}: \chi\in U^\circ\Big\}.
\end{aligned}
\]
Therefore $T+S\in \prod_{q,p}^U(E,L)$, and hence $\prod_{q,p}^U(E,L)$ is a vector space.

(ii) Since every finite rank operator $T$ with co-support in $\spn(U^\circ)$ is a finite sum of operators of the form $F(x):=\langle\eta,x\rangle \cdot y$, where $x\in E$ and $\eta\in \spn(U^\circ)$, by (i), it suffices to check that $F\in \prod_p^U(E,L)$. To this end, let $V\in\Nn_0^c(L)$ and let $a>0$ be such that $a\eta\in U^\circ$. Set $C_V:=\frac{q_V(y)}{a}$. Then for every $n\in\w$ and each $x_0,\dots,x_n\in U$, we have
\[
\begin{aligned}
\Big( \sum_{i=0}^n q_V\big(F(x_i)\big)^p\Big)^{1/p} & =\frac{q_V(y)}{a}\cdot \Big( \sum_{i=0}^n |\langle a\eta,x_i\rangle|^p\Big)^{1/p} \\
& \leq C_V \cdot \sup\Big\{ \Big( \sum_{i=0}^n  |\langle\chi,x_i\rangle|^p\Big)^{1/p}: \chi\in U^\circ\Big\}.
\end{aligned}
\]
Thus $F\in \prod_p^U(E,L)$.\qed
\end{proof}

\begin{remark} \label{rem:finite-rank-notp-sum} {\em
The condition ``to have co-support in $\spn(U^\circ)$'' in Proposition \ref{p:p-summing-vector} is essential. Indeed, let $E=\IR^\w$, $L=\ell_1$,  $U=[-1,1]\times \IR^{\w\SM\{0\}}$, and let $\eta=(0,1,0,\dots)\in\varphi=E'$.  We show that the operator $T=\langle\eta,\cdot\rangle y$, where $y$ is a non-zero vector of $L$, does not belong to $\prod_p^U(E,L)$. It is easy to see that $U^\circ=[-1,1]\oplus\bigoplus_{\w\SM\{0\}}\{0\}$ and hence $\eta\not\in \spn(U^\circ)$. Let $n=0$ and $x_0:=(0,1,0,\dots)\in U$. Then $\|T(x_0)\|=\|\langle\eta,x_0\rangle y\|=\|y\|\not=0$, but
$
\sup\big\{   |\langle\chi,x_0\rangle|: \chi\in U^\circ\big\}=0.
$
Thus (\ref{equ:Pietsch-0}) is not satisfied and hence  $T\not\in \prod_p^U(E,L)$.\qed
}
\end{remark}

The next result for Banach spaces is proved in Theorem 2.4 of \cite{DJT}, we propose its direct proof.
\begin{theorem}[\protect{Ideal Property}] \label{t:ideal-p-summing}
Let $1\leq p\leq q<\infty$, $E_0$, $E$, $L$ and $L_0$ be locally convex spaces, $U$ be an absolutely convex closed subset of $E$, and let $S:E_0\to E$, $T:E\to L$ and $R:L\to L_0$ be operators. If $T$ is $U$-$(q,p)$-summing, then $RTS$ is $V$-$(q,p)$-summing, where $V=\big(S^\ast(U^\circ)\big)^\circ$.
\end{theorem}

\begin{proof}
Since $U^\circ$ is absolutely convex so is $S^\ast(U^\circ)$. Therefore, by the Bipolar Theorem 8.3.8 of \cite{NaB}, $V^\circ$ is the weak$^\ast$ closure of  $S^\ast(U^\circ)$, and hence for each $n\in\w$ and $x_0,\dots,x_n\in V$, we have
\begin{equation} \label{equ:pq-summing-1}
\begin{aligned}
\sup\Big\{ \Big( \sum_{i=0}^n  |\langle\eta,x_i\rangle|^p\Big)^{1/p}: \eta\in V^\circ\Big\} & =\sup\Big\{ \Big( \sum_{i=0}^n  |\langle\eta,x_i\rangle|^p\Big)^{1/p}: \eta\in S^\ast(U^\circ)\Big\}\\
& = \sup\Big\{ \Big( \sum_{i=0}^n  |\langle\chi,S(x_i)\rangle|^p\Big)^{1/p}: \chi\in U^\circ\Big\}.
\end{aligned}
\end{equation}
Observe that if $x\in V$ and  $\chi\in U^\circ$, then $\eta=S^\ast(\chi)\in V^{\circ}$  and
\[
|\langle\chi,S(x)\rangle|=|\langle\eta,x\rangle|\leq 1
\]
which means that $S(V)\subseteq U^{\circ\circ}= U$.

Fix an arbitrary $W_0\in\Nn_0^c(L_0)$ and set $W:=R^{-1}(W_0)$. Then $W\in\Nn_0^c(L)$, and for every  $y\in L$, we have
\[
q_W(y):=\inf\{\lambda>0:y\in\lambda W\}=\inf\{\lambda>0:R(y)\in\lambda W_0\}=q_{W_0}R(y)
\]
and hence $q_W=q_{W_0}R$. Therefore, for each $n\in\w$ and $x_0,\dots,x_n\in V$,  the $U$-$(q,p)$-summability of $T$, applied to $W$, implies
\[
\Big( \sum_{i=0}^n q_{W_0}\big(RTS(x_i)\big)^q\Big)^{1/q} \leq  C_{W}\cdot \sup\Big\{ \Big( \sum_{i=0}^n  |\langle\chi,S(x_i)\rangle|^p\Big)^{1/p}: \chi\in U^\circ\Big\}.
\]
Hence, by (\ref{equ:pq-summing-1}), for every $n\in\w$ and each $x_0,\dots,x_n\in V$, we obtain
\[
\Big( \sum_{i=0}^n q_{W_0}\big(RTS(x_i)\big)^q\Big)^{1/q}\leq C_{W}\cdot \sup\Big\{ \Big( \sum_{i=0}^n  |\langle\eta,x_i\rangle|^p\Big)^{1/p}: \eta\in V^\circ\Big\}
\]
which means that $RTS$ is $V$-$(q,p)$-summing.\qed
\end{proof}

The proof of the next result is similar to the proof of Theorem 10.4 of \cite{DJT}.
\begin{theorem}[\protect{Inclusion Property}] \label{t:p-summing-inclusion}
Let $1\leq p_j\leq q_j<\infty$ ($j=1,2)$ satisfy
\[
p_1\leq p_2, \;\; q_1\leq q_2,\;\; \mbox{ and }\;\; \tfrac{1}{p_1}-\tfrac{1}{q_1}\leq \tfrac{1}{p_2}-\tfrac{1}{q_2}.
\]
Let $E$ and $L$ be locally convex spaces, and let $U$ be an absolutely convex closed subset of $E$. Then $\prod_{q_1,p_1}^U(E,L)\subseteq \prod_{q_2,p_2}^U(E,L)$.
\end{theorem}

\begin{proof}
Let $T\in \prod_{q_1,p_1}^U(E,L)$, and let $V\in\Nn_0^c(L)$. Take $n\in\w$ and $x_0,\dots,x_n\in U$, and consider the next two cases.
\smallskip

{\em Case 1. Assume that $p_1=p_2=p$}. Then the inequalities $q_1\leq q_2$ and $\|x\|_{\ell_{q_2}}\leq \|x\|_{\ell_{q_1}}$ ($x\in\ell_{q_1}$) imply
\[
\Big( \sum_{i=0}^n q_V\big(T(x_i)\big)^{q_2}\Big)^{\tfrac{1}{q_2}} \leq \Big( \sum_{i=0}^n q_V\big(T(x_i)\big)^{q_1}\Big)^{\tfrac{1}{q_1}} \leq C_V \cdot \sup\Big\{ \Big( \sum_{i=0}^n  |\langle\chi,x_i\rangle|^p\Big)^{1/p}: \chi\in U^\circ\Big\}.
\]
Thus $T\in \prod_{q_2,p_2}^U(E,L)$.
\smallskip

{\em Case 2. Assume that $p_1 < p_2$}. Then the inequality $\tfrac{1}{p_1}-\tfrac{1}{q_1}\leq \tfrac{1}{p_2}-\tfrac{1}{q_2}$ implies  $q_1 <q_2$. Therefore we can define numbers $1<p,q<\infty$ by
\[
\tfrac{1}{p}:=\tfrac{1}{p_1}-\tfrac{1}{p_2} \;\; \mbox{ and }\;\; \tfrac{1}{q}:=\tfrac{1}{q_1}-\tfrac{1}{q_2}.
\]
For every $i\leq n$, set $\lambda_i:=q_V\big(T(x_i)\big)^{q_2/q}$, so that $q_V\big(T(x_i)\big)^{q_2}= q_V\big(T(\lambda_i x_i)\big)^{q_1}$. Set $\lambda=\lambda_0+\cdots+\lambda_n$. Since the case $\lambda=0$ is trivial,  we assume that $\lambda>0$. Then $\tfrac{\lambda_i}{\lambda} x_i\in U$ because $U$ is absolutely convex. Since $T$ is $U$-$(q,p)$-summing,  we obtain
\[
\Big( \sum_{i=0}^n q_V\big(T(x_i)\big)^{q_2}\Big)^{1/q_1}=\lambda\cdot\Big( \sum_{i=0}^n q_V\big(T(\tfrac{\lambda_i}{\lambda} x_i\big)^{q_1}\Big)^{1/q_1} \leq C_V \cdot \sup_{\chi\in U^\circ} \Big( \sum_{i=0}^n  \lambda_i^{p_1}\cdot |\langle\chi,x_i\rangle|^{p_1}\Big)^{1/p_1}.
\]
Since $q_2 >q_1$, we apply H\"{o}lder's inequality to the conjugate numbers  $\tfrac{q}{p_1}$ and $\tfrac{q}{q-p_1}$   and $a_i=\lambda_i^{q_1}$ and $b_i=|\langle\chi,x_i\rangle|^{q_1}$ in the right hand side to get
\[
\begin{aligned}
\Big( \sum_{i=0}^n q_V\big(T(x_i)\big)^{q_2}\Big)^{\tfrac{1}{q_1}} & \leq C_V\cdot \Big( \sum_{i=0}^n  \lambda_i^{p_1\cdot \tfrac{q}{p_1}} \Big)^{\tfrac{1}{p_1}\cdot \tfrac{p_1}{q}}\cdot \sup\Big\{ \Big( \sum_{i=0}^n  |\langle\chi,x_i\rangle|^{p_1\cdot \tfrac{q}{q-p_1}}\Big)^{\tfrac{1}{p_1} \cdot \tfrac{q-p_1}{q}}: \chi\in U^\circ\Big\} \\
& = C_V\cdot \Big( \sum_{i=0}^n q_V\big(T(x_i)\big)^{\tfrac{q_2}{q}\cdot q }\Big)^{\tfrac{1}{q}}\cdot \sup\Big\{ \Big( \sum_{i=0}^n  |\langle\chi,x_i\rangle|^{\tfrac{p_1q}{q-p_1}}\Big)^{\tfrac{q-p_1}{p_1q}}: \chi\in U^\circ\Big\}\\
& =C_V\cdot \Big( \sum_{i=0}^n q_V\big(T(x_i)\big)^{q_2}\Big)^{\tfrac{1}{q_1}-\tfrac{1}{q_2}}\cdot \sup\Big\{ \Big( \sum_{i=0}^n  |\langle\chi,x_i\rangle|^{\tfrac{p_1q}{q-p_1}}\Big)^{\tfrac{q-p_1}{p_1q}}: \chi\in U^\circ\Big\}.
\end{aligned}
\]
Dividing both sides by the second factor of the right hand side and taking into account that
\[
\tfrac{q-p_1}{p_1q}= \tfrac{1}{p_1}-\tfrac{1}{q}=\tfrac{1}{p_1}-\tfrac{1}{q_1}+\tfrac{1}{q_2}\leq \tfrac{1}{p_2}
\]
and hence $\tfrac{p_1q}{q-p_1}\geq p_2$, we obtain
\[
\Big( \sum_{i=0}^n q_V\big(T(x_i)\big)^{q_2}\Big)^{1/q_2} \leq C_V\cdot \sup\Big\{ \Big( \sum_{i=0}^n  |\langle\chi,x_i\rangle|^{p_2}\Big)^{1/p_2}: \chi\in U^\circ\Big\}
\]
which means that $T$ is $U$-$(q_2,p_2)$-summing.\qed
\end{proof}

The next theorem extends Theorem 2.8 and Proposition 10.2 of \cite{DJT}.
\begin{theorem}[\protect{Injectivity of $\prod_{q,p}$}] \label{t:injectivity-p-summing}
Let $1\leq p\leq q<\infty$, $E$, $L$ and $L_0$ be locally convex spaces, and let $U$ be an absolutely convex closed subset of $E$. If $I:L\to L_0$ is an embedding, then an operator $T:E\to L$ is $U$-$(q,p)$-summing if and only if so is $I\circ T:E\to L_0$.
\end{theorem}

\begin{proof}
The necessity follows from the ideal property Theorem \ref{t:ideal-p-summing} applied to $E_0=E$ and $S=\Id_E$. To prove the sufficiency, assume that  $I\circ T:E\to L_0$  is $U$-$(q,p)$-summing. Fix an arbitrary $V\in\Nn_0^c(L)$. Since $I$ is an embedding, there is $W\in\Nn_0^c(L_0)$ such that $W\cap I(L)=I(V)$. Since $I\circ T:E\to L_0$  is $U$-$(q,p)$-summing, there is $C_W>0$ such that for every $n\in\w$ and $x_0,\dots,x_n\in U$, we have
\[
\Big( \sum_{i=0}^n q_W\big(I\circ T(x_i)\big)^q\Big)^{1/q} \leq C_W \cdot \sup\Big\{ \Big( \sum_{i=0}^n  |\langle\chi,x_i\rangle|^p\Big)^{1/p}: \chi\in U^\circ\Big\},
\]
Observe that for every $y\in L$, we have
\[
q_V(y)=\inf\{\lambda>0:y\in\lambda V\}=\inf\{\lambda>0:I(y)\in\lambda W\}= q_W\big(I(y)\big).
\]
Therefore for every $n\in\w$ and each $x_0,\dots,x_n\in U$, we obtain
\[
\Big( \sum_{i=0}^n q_V\big(T(x_i)\big)^q\Big)^{1/q} =\Big( \sum_{i=0}^n q_W\big(I\circ T(x_i)\big)^q\Big)^{1/q} \leq C_W \cdot \sup\Big\{ \Big( \sum_{i=0}^n  |\langle\chi,x_i\rangle|^p\Big)^{1/p}: \chi\in U^\circ\Big\}.
\]
Thus $T$ is $U$-$(q,p)$-summing.\qed
%
\end{proof}


Let $H$ be a normed space, and let $E$ be a locally convex space. For every absolutely convex closed subset $U$ of $E$, we denote by $\LL_U(H,E)$ the family of all operators $S:H\to E$ such that $S(B_H)\subseteq U$.
The next theorem generalizes Proposition 2.7 of \cite{DJT}.

\begin{theorem} \label{t:p-summing-operator}
Let $1\leq p\leq q<\infty$, $E$ and $L$ be  locally convex spaces, $U$ be an absolutely convex closed subset of $E$, and let $T:E\to L$ be an operator. Denote $B_{\ell_{p^\ast}^0}$ {\rm(}or $B_{c_0^0}$ if $p=1${\rm)} by $B$.
If $T$ is $U$-$(q,p)$-summing, then for every $V\in\Nn_0^c(L)$, there is $C_V>0$ such that $T\circ S$ is $B$-$(q,p)$-summing  with $\pi_{q,p}^{B,V}(T\circ S)\leq C_V$ for each $S\in\LL_U(\ell_{p^\ast}^0, E)$ {\rm(}or for each $S\in\LL_U(c_0^0, E)$ if $p=1${\rm)}. The converse is true if $T(U)$ is a bounded subset of $L$.
\end{theorem}

\begin{proof}
Assume that $T$ is a $U$-$(q,p)$-summing operator. Fix an arbitrary $S\in\LL_U(\ell_{p^\ast}^0, E)$ (or $S\in\LL_U(c_0^0, E)$ if $p=1$). We show that for every $V\in\Nn_0^c(E)$, $T\circ S$ is $B$-$(q,p)$-summing  with $\pi_{q,p}^{B,V}(T)\leq C_V$, where the constant $C_V$ is defined by $T$. To this end, fix  $V\in\Nn_0^c(E)$ and let $x_0,\dots,x_n\in B$. Recall that $S(B)\subseteq U$ and observe also that for every $\chi\in U^\circ$ and each $x\in B$, we have
\[
|\langle S^\ast(\chi),x\rangle|=|\langle \chi, S(x)\rangle|\leq 1
\]
which means that $S^\ast(U^\circ)\subseteq B^\circ=B_{\ell_p}$. Then the $U$-$(q,p)$-summability of $T$, applied to $S(x_0),\dots,S(x_n)\in U$, implies
\[
\begin{aligned}
\Big( \sum_{i=0}^n q_V\big(TS(x_i)\big)^q\Big)^{1/q} & =\Big( \sum_{i=0}^n q_V\big(T\big(S(x_i)\big)\big)^q\Big)^{1/q} \leq C_V \cdot \sup_{\chi\in U^\circ} \Big( \sum_{i=0}^n  |\langle\chi,S(x_i)\rangle|^p\Big)^{1/p}\\
& = C_V \cdot \sup\Big\{ \Big( \sum_{i=0}^n  |\langle S^\ast(\chi),x_i\rangle|^p\Big)^{1/p}: \chi\in U^\circ\Big\}\\
& \leq  C_V \cdot \sup\Big\{ \Big( \sum_{i=0}^n  |\langle \eta,x_i\rangle|^p\Big)^{1/p}: \eta\in B^\circ \Big\}.
\end{aligned}
\]
Thus $T\circ S$ is $B$-$(q,p)$-summing  with $\pi_{q,p}^{B,V}(T)\leq C_V$.

Conversely, assume that $T(U)$ is a bounded subset of $L$ and for every $V\in\Nn_0^c(L)$, there is $C_V>0$ such that $T\circ S$ is $B$-$(q,p)$-summing  with $\pi_{q,p}^{B,V}(T\circ S)\leq C_V$ for each $S\in\LL_U(\ell_{p^\ast}^0, E)$ (or for each $S\in\LL_U(c_0^0, E)$ if $p=1$). We show that $T$ is $U$-$(q,p)$-summing. To this end, fix $V\in\Nn_0^c(L)$ and an arbitrary finite family $F=\{x_0,\dots,x_n\}\subseteq  U$. Set
\[
K:= \sup\Big\{ \Big( \sum_{i=0}^n  |\langle \chi,x_i\rangle|^p\Big)^{1/p}: \chi\in U^\circ \Big\} \;\; \big(\leq (n+1)^{1/p}\big)
\]
and consider the following two possible cases.
\smallskip

{\em Case 1. Assume that $K=0$.} Then $\spn(F)\subseteq U^{\circ\circ}=U$. Since $T(U)$ is a bounded subset of $L$ it follows that $F\subseteq \ker(T)$. Therefore
\[
\Big( \sum_{i=0}^n q_V\big(T(x_i)\big)^q\Big)^{1/q}=0= C_V\cdot \sup\Big\{ \Big( \sum_{i=0}^n  |\langle \chi,x_i\rangle|^p\Big)^{1/p}: \chi\in U^\circ \Big\}
\]
for every $C_V>0$.
\smallskip

{\em Case 2. Assume that $K>0$.} Define a finite-rank operator $S\in\LL(\ell_{p^\ast}^0, E)$ (or $S\in\LL(c_0^0, E)$ if $p=1$) by
\[
S\big(a_0 e_0^\ast +\cdots +a_n e_n^\ast +a_{n+1} e_{n+1}^\ast +\cdots\big):= \tfrac{1}{K} \big(a_0 x_0 +\cdots +a_n x_n\big),
\]
where $(a_n)\in \ell_{p^\ast}^0$ (or $\in c_0^0$ if $p=1$). Since $S$ has finite rank it is continuous.

We claim that $S\in\LL_U(\ell_{p^\ast}^0, E)$ (or $S\in\LL_U(c_0^0, E)$ if $p=1$). Indeed, for every $\chi\in U^\circ$ and each $(a_i)_{i\in\w}\in B_{\ell_{p^\ast}^0}$ (or $\in B_{c_0^0}$ if $p=1$), the H\"{o}lder inequality for  $p>1$ implies
\[
\begin{aligned}
\Big|\big\langle \chi, S\big( a_i\big)_{i\in\w}\big\rangle\Big| &= \tfrac{1}{K} \Big| \sum_{i=0}^n a_i \cdot \langle\chi,x_i\rangle\Big|\leq \tfrac{1}{K}\Big( \sum_{i=0}^n |a_i|^{p^\ast}\Big)^{1/p^\ast} \cdot \Big( \sum_{i=0}^n  |\langle \chi,x_i\rangle|^p\Big)^{1/p}\\
& \leq \tfrac{1}{K} \cdot \Big( \sum_{i=0}^n  |\langle \chi,x_i\rangle|^p\Big)^{1/p} \leq 1
\end{aligned}
\]
 and, for $p=1$, we have
\[
\Big|\big\langle \chi, S\big( a_i\big)_{i\in\w}\big\rangle\Big| =\tfrac{1}{K} \Big| \sum_{i=0}^n a_i \cdot \langle\chi,x_i\rangle\Big|\leq \tfrac{1}{K} \cdot \sup_{i\in\w} |a_i| \cdot \sum_{i=0}^n  |\langle \chi,x_i\rangle|\leq \tfrac{1}{K} \cdot  \sum_{i=0}^n  |\langle \chi,x_i\rangle|\leq 1.
\]
Therefore $S\big(B\big) \subseteq U^{\circ\circ}=U$, and hence $S\in\LL_U(\ell_{p^\ast}^0, E)$ (or $S\in\LL_U(c_0^0, E)$ if $p=1$). This proves the claim.

By assumption the operator $T\circ S$ is $B$-$(q,p)$-summing with $\pi_{q,p}^{B,V}(T\circ S)\leq C_V$. This means that (recall that $B^\circ=B_{\ell_p}$)
\[
\begin{aligned}
\Big( \sum_{i=0}^n q_V\big(T(x_i)\big)^q\Big)^{1/q} & = \Big( \sum_{i=0}^n K^q\cdot q_V\big(TS(e_i^\ast)\big)^q\Big)^{1/q} \leq K  \cdot C_V\cdot \sup_{\eta\in B_{\ell_{p}}} \Big( \sum_{i=0}^n  |\langle \eta,e_i^\ast\rangle|^p\Big)^{1/p}\\
& \leq  KC_V \cdot \sup\big\{ \|\eta\|_{\ell_p} : \eta\in B_{\ell_{p}} \big\}= C_V \cdot K \\
& = C_V \cdot \sup\Big\{ \Big( \sum_{i=0}^n  |\langle \chi,x_i\rangle|^p\Big)^{1/p}: \chi\in U^\circ \Big\}.
\end{aligned}
\]
Thus $T$ is a $U$-$(q,p)$-summing operator, as desired.\qed
\end{proof}

It turns out that two Pietsch characterizations of $p$-summing operators hold true for every locally convex space $E$. To obtain desired extensions we should have the condition that $U^\circ$ is weak$^\ast$ compact. Taking into account that $U^\circ$ is also absolutely convex, the Mackey--Arens theorem implies that $U$ should be a neighborhood of zero in the Mackey topology. This explains the condition $U\in\Nn_0^c(E_\mu)$ in the following generalization of the Pietsch Domination Theorem 2.12 of \cite{DJT}.

\begin{theorem}[\protect{Pietsch Domination Theorem}] \label{t:p-sum-Pietsch}
Let $1\leq p<\infty$, $E$ and $L$ be  locally convex spaces, $T:E\to L$ be an operator, $U\in\Nn_0^c(E_\mu)$ and let $K:=\big(U^\circ,\sigma(E',E){\restriction}_{U^\circ}\big)$. Then the following assertions are equivalent:
\begin{enumerate}
\item[{\rm (i)}]  $T$ is $U$-$p$-summing;
\item[{\rm (ii)}] for every $V\in\Nn_0^c(L)$ there exist  a regular probability measure $\mu_V$ on the compact space $K$ and a constant $C_V>0$ such that
    \[
    q_V\big(T(x)\big)^p\leq C^p_V \cdot \int_{K} |\langle\chi,x\rangle|^p d\mu_V(\chi)\;\; \mbox{ for every $x\in E$}.
    \]
\end{enumerate}
\end{theorem}

\begin{proof}
(i)$\Ra$(ii) In the proof we consider $C(K)$ over $\IR$. By the definition of $U$-$p$-summing operators, for every $V\in\Nn_0^c(L)$ there is a constant $C_V>0$ such that for every $n\in\w$ and each $x_0,\dots,x_n\in U$, we have
\begin{equation} \label{equ:Pietsch-000}
\Big( \sum_{i=0}^n q_V\big(T(x_i)\big)^p\Big)^{1/p} \leq C_V \cdot \sup\Big\{ \Big( \sum_{i=0}^n  |\langle\chi,x_i\rangle|^p\Big)^{1/p}: \chi\in U^\circ\Big\}.
\end{equation}
Using the constant $C_V$ from this inequality, for any finite subset $M$ of $E$, we define a continuous function $f_M: K\to \IR$ by
\[
f_M(\chi):= \sum_{x\in M} \Big(q_V\big(T(x)\big)^p - C^p_V \cdot |\langle\chi,x\rangle|^p\Big) \quad (\chi\in K).
\]

{\em Claim 1. The family $Q:=\{f_M: M\subseteq E \mbox{ is finite}\}\subseteq C(K)$ is convex.}  Indeed, if $M$ and $N$ are finite subsets of $E$ and $0<\lambda<1$, then $\lambda\cdot f_M+(1-\lambda)\cdot f_N = f_S$, where
$
S=\{ \lambda^{1/p} x: x\in M\SM N\} \cup \{ (1-\lambda)^{1/p} x: x\in N\SM M\}\cup (M\cap N).
$
This proves Claim 1.
\smallskip

Set $P:= \{ f\in C(K): f(\chi)>0 \mbox{ for every } \chi\in K\}$. Then:
\begin{enumerate}
\item[(a)] $P$ is an open, convex and positive cone in $C(K)$;
\item[(b)] $Q\cap P=\emptyset$.
\end{enumerate}
Indeed, (a) is clear. To check (b), let $f_M\in Q$. Since $U\in\Nn_0^c(E_\mu)$, the set $U$ is absorbing. Therefore, by homogeneity, (\ref{equ:Pietsch-000}) is satisfied for all $x_0,\dots,x_n\in E$.  It follows from (\ref{equ:Pietsch-000}) that there is $\chi\in K$ such that $f_M(\chi)\leq 0$. Thus $f_M\not\in P$, as desired.
\smallskip

Taking into account Claim 1 and (a)-(b), Proposition 2.13(ii) of \cite{fabian-10} implies that  there is a regular Borel measure $\mu_V\in C(K)'$ on $K$ such that
\begin{equation} \label{equ:Pietsch-1}
\langle\mu_V,g\rangle\leq c:=\sup\{ \langle\mu_V,h\rangle: h\in Q\} <\langle\mu_V,f\rangle
\end{equation}
for all $g\in Q$ and $f\in P$.
\smallskip

{\em Claim 2. $c=0$.} Indeed, since $g_{\{0\}} =0\in Q$ we obtain $c\geq 0$. On the other hand, since every positive constant function belongs to $P$,  (\ref{equ:Pietsch-1}) implies $c\leq 0$. Thus $c=0$.
\smallskip

Since $\mu_V$ is a continuous functional and, by Claim 2, $\langle\mu_V,f\rangle>0$ for every $f\in P$, it follows that $\langle\mu_V,f\rangle\geq 0$ for every $f\geq 0$ in $C(K)$. Therefore $\mu_V$ is a positive regular Borel measure on $K$. Dividing (\ref{equ:Pietsch-1}) by $\|\mu_V\|$ we can assume that $\mu_V$ is a probability measure. Since $c=0$ we can apply (\ref{equ:Pietsch-1}) to the function $g_{\{x\}}\in Q$, $x\in E$, to obtain
\[
\int_{K} \Big(q_V\big(T(x)\big)^p - C^p_V \cdot |\langle\chi,x\rangle|^p\Big) d\mu_V(\chi) \leq 0,
\]
or, since $\mu_V$ is a probability measure,
\[
q_V\big(T(x)\big)^p \leq C^p_V \cdot \int_K |\langle\chi,x\rangle|^p d\mu_V(\chi).
\]

(ii)$\Ra$(i) For each $V\in\Nn_0^c(L)$ and every $n\in\w$ and each $x_0,\dots,x_n\in U$, we apply (ii) to obtain
\[
\sum_{i=0}^n q_V\big(T(x_i)\big)^p\leq C^p_V \cdot \sum_{i=0}^n \int_{K} |\langle\chi,x_i\rangle|^p d\mu_V(\chi)\leq C^p_V \cdot\sup_{\chi\in U^\circ} \sum_{i=0}^n |\langle\chi,x_i\rangle|^p
\]
and hence $T$ is $U$-$p$-summing.\qed
\end{proof}

Let $K$ be a compact space, and let $\mu$ be a regular probability measure on $K$. Denote by $J_p$ and $J_p^\infty$ the identity inclusions of $C(K)$ into $L_p(\mu)$ and  $L_\infty(\mu)$ into $L_p(\mu)$, respectively. It is clear that $J_p$ and $J_p^\infty$ are continuous. If $L$ is a Banach space, we denote by $i_L:L\to \ell_\infty(B_{L'})$ the canonical isometric embedding.  Now we prove an analogue of the Pietsch Factorization Theorem 2.13 of \cite{DJT}.

\begin{theorem}[\protect{Pietsch Factorization Theorem}] \label{t:p-sum-Pietsch-f}
Let $1\leq p<\infty$, and let $T:E\to L$ be an operator from a locally convex space $E$ to a Banach space $L$. Then the following assertions are equivalent:
\begin{enumerate}
\item[{\rm (i)}]  $T$ is $U$-$p$-summing for some $U\in\Nn_0^c(E)$;
\item[{\rm (ii)}] there exist $U\in\Nn_0^c(E)$, a regular probability measure $\mu$ on the compact space $K:=\big(U^\circ,\sigma(E',E){\restriction}_{U^\circ}\big)$, an operator $I_U:E\to C(K)$, a closed subspace $H_p$ of $L_p(\mu)$ and an operator ${\hat T}:H_p\to L$ such that the following diagram is commutative
    \[
    \xymatrix{
    E \ar[r]^T \ar[d]_{I_E} & L \\
    I_U(E) \ar[r]^{J_p^E} \ar@{^{(}->}[d]_{\Id} &  H_p \ar[u]_{\hat T} \ar@{^{(}->}[d]^{\Id_{H_p}}\\
    C(K) \ar[r]^{J_p} & L_p(\mu)
    }
    \]
    in which $I_E$ is the co-restriction of $I_U$ to $I_U(E)$ and $\Id$ and $\Id_{H_p}$ are the corresponding identity inclusions of $I_U(E)$ into $C(K)$ and $H_p$ into $L_p(\mu)$.
\item[{\rm (iii)}] there exist  $U\in\Nn_0^c(E)$, a regular probability measure $\mu$ on the compact space $K:=\big(U^\circ,\sigma(E',E){\restriction}_{U^\circ}\big)$, an operator $I_U:E\to C(K)$,  and an operator ${\tilde T}:L_p(\mu)\to \ell_\infty(B_{L'})$ such that the following diagram is commutative
    \[
    \xymatrix{
    E \ar[r]^T \ar[dd]_{I_U} & L \ar[rd]^{i_L} &\\
    & & \ell_\infty(B_{L'})\\
    C(K) \ar[r]^{J_p} & L_p(\mu) \ar[ru]_{{\tilde T}} &
    }
    \]
\item[{\rm (iv)}] there exist  $U\in\Nn_0^c(E)$, a regular probability measure $\mu$ on the compact space $K:=\big(U^\circ,\sigma(E',E){\restriction}_{U^\circ}\big)$, an operator $I_U^\infty:E\to L_\infty(\mu)$,  and an operator ${\tilde T}:L_p(\mu)\to \ell_\infty(B_{L'})$ such that the following diagram is commutative
    \[
    \xymatrix{
    E \ar[r]^T \ar[dd]_{I_U^\infty} & L \ar[rd]^{i_L} &\\
    & & \ell_\infty(B_{L'}).\\
    L_\infty(\mu) \ar[r]^{J_p^\infty} & L_p(\mu) \ar[ru]_{{\tilde T}} &
    }
    \]
\end{enumerate}
\end{theorem}

\begin{proof}
(i)$\Ra$(ii) Assume that $T$ is a $U$-$p$-summing operator for some $U\in\Nn_0^c(E)$. Then, by the Pietsch Domination Theorem \ref{t:p-sum-Pietsch}, there exist a regular probability measure $\mu$ on the compact space $K:=\big(U^\circ,\sigma(E',E){\restriction}_{U^\circ}\big)$ and a constant $C>0$ such that
\begin{equation} \label{equ:Pietsch-12}
\|T(x)\|^p\leq C^p \cdot \int_{K} |\langle\chi,x\rangle|^p d\mu(\chi)\;\; \mbox{ for every $x\in E$}.
\end{equation}
Since every $x\in E$ is a continuous function on $K$, the canonical inclusion $I_U: E\to C(K)$, $I_U(x)(\chi):=\langle\chi,x\rangle$, is well-defined. As $|I_U(x)(\chi)|\leq 1 $  for every $x\in U$ and each $\chi\in K$, it follows that $I_U(U)\subseteq B_{C(K)}$ and hence $I_U$ is continuous. Denote by $I_E$ the co-restriction of $I_U$ onto the image $I_U(E)$, and let $\Id:I_U(E)\to C(K)$ be the identity inclusion.

Set $Y:=J_p\big(I_U(E)\big)\subseteq L_p(\mu)$  and define a map $R:Y\to L$  by
\[
R(y):=T(x), \;\; \; \mbox{ where $y=J_p(I_U(x))$.}
\]

The map $R$ is well-defined because if $y=J_p(I_U(x))=J_p(I_U(x'))$ for some $x'\in E$, then,  by the injectivity of $J_p$, $I_U(x)=I_U(x')$ and hence  $\langle\chi,x-x'\rangle=0$ for every $\chi\in U^\circ$. In particular, $\spn(x-x')\subseteq U^{\circ\circ}=U$. Since, by (\ref{equ:Pietsch-12}), $T(U)\subseteq C\cdot B_L$ and $B_L$ does not contain linear subspaces, we obtain $\spn(x-x')\subseteq \ker(T)$ and hence $T(x)=T(x')$. Thus $R$ is well-defined.

To show that $R$ is continuous, let $y=J_p(I_U(x))\in Y$ be such that
\[
\|y\|_{L_p(\mu)}^p := \int_{K} |\langle\chi,x\rangle|^p d\mu(\chi)\leq 1.
\]
Then, by (\ref{equ:Pietsch-12}), we obtain $\|R(y)\|=\|T(x)\|\leq C$. Thus $R$ is continuous.

Denote by $H_p$ the completion of  $Y$  in the Banach space $L_p(\mu)$, and let ${\hat T}$ be the extension of $R$ from $Y$ to $H_p$. Then, by construction, we have $T={\hat T} \circ J_p^E\circ I_E$, as desired.
\smallskip

(ii)$\Ra$(iii) Since the Banach space $\ell_\infty(B_{L'})$ is injective, the operator $i_L\circ {\hat T}: H_p\to \ell_\infty(B_{L'})$ has an extension ${\tilde T}: L_p(\mu)\to \ell_\infty(B_{L'})$. Clearly, ${\tilde T}$ is the operator we are looking for.
\smallskip

(iii)$\Ra$(iv) Denote by $J^\infty$ the identity mapping from $C(K)$ to $L_\infty(\mu)$. Set $I_U^\infty:= J^\infty\circ I_U$. Now it is clear from the diagram of (iii) that the diagram in (iv) is commutative.
\smallskip

(iv)$\Ra$(i) By Example 2.9(d) of \cite{DJT}, the operator $J^\infty_p$ is $B_{L_\infty(\mu)}$-$p$-summing. Therefore, by Theorem \ref{t:ideal-p-summing} and the commutativity of the diagram in (iv), the operator $i_L\circ T$ is $V$-$p$-summing for $V:=\big((I_U^\infty)^\ast\big(B^\circ_{L_\infty(\mu)}\big)\big)^\circ$.  Since $i_L$ is an isometry, Theorem \ref{t:injectivity-p-summing} implies that also $T$ is $V$-$p$-summing. To show that $V$ is a neighborhood of zero, observe that $B^\circ_{L_\infty(\mu)}$ is equicontinuous and hence so is $(I_U^\infty)^\ast\big(B^\circ_{L_\infty(\mu)}\big)\subseteq E'$. Thus $V\in\Nn_0^c(E)$, as desired. \qed
\end{proof}

It is well known that every $p$-summing operator between Banach spaces is weakly compact and completely continuous, see Theorem 2.17 of \cite{DJT}. The next theorem extends this result.
\begin{theorem} \label{t:p-sum-weakly-compact}
Let $1\leq p<\infty$, and let $T:E\to L$ be an operator from a locally convex space $E$ to a Banach space $L$. If $T$ is $U$-$p$-summing for some $U\in\Nn_0^c(E)$, then it is weakly {\rm(}sequentially{\rm)} compact and completely continuous.
\end{theorem}

\begin{proof}
By Theorem \ref{t:p-summing-inclusion} applied to $p_1=q_1=p$ and $p_2=q_2=q$, we can assume that $p>1$. By (iii) of the Pietsch Factorization Theorem \ref{t:p-sum-Pietsch-f}, we have $i_L\circ T={\tilde T}\circ J_p\circ I_U$. Since $L_p(\mu)$ is reflexive, all bounded subsets of $L_p(\mu)$ are relatively weakly (sequentially) compact. Therefore the operator $J_p$  and hence also $i_L\circ T={\tilde T}\circ J_p\circ I_U$ are weakly (sequentially) compact. It is well known that the canonical map $J_p:C(K)\to L_p(\mu)$ is completely continuous, and hence so is $i_L\circ T={\tilde T}\circ J_p\circ I_U$. Taking into account that $i_L$ is an isometric embedding, we obtain that $T$  is  weakly (sequentially) compact and completely continuous.\qed
\end{proof}

Since, by Proposition \ref{p:L^w-E-summand}, $E$ is a direct summand of $\ell_p[E]$ it follows that the barrelledness of $\ell_p[E]$ implies that $E$ is a barrelled space as well. As the space $E_w$ is not barrelled if $E\not= E_w$, the next proposition shows that the barrelledness of $\ell_p[E]$ in (ii) of Proposition \ref{p:qp-summing-norm} is not a necessary condition.

\begin{proposition} \label{p:finite-rank-qp-summing}
Let $1\leq p\leq q<\infty$, $E$ be a locally convex space, and let $L$ be a Banach space. Then every operator $T:E_w\to L$ is $U$-$(q,p)$-summing for some $U\in\Nn_0^c(E_w)$.
\end{proposition}

To prove Proposition \ref{p:finite-rank-qp-summing} we need the next simple lemma.
\begin{lemma} \label{l:weak-to-norm}
Let $E$ be a locally convex space such that $E=E_w$, and let $L$ be a normed space. Then every $T\in\LL(E,L)$ is finite-dimensional.
\end{lemma}

\begin{proof}
Observe that $T$ can be extended to an operator ${\bar T}$ from a completion ${\bar E}$ of $E$ to a completion ${\bar L}$ of $L$. As $E$ carries its weak topology, we obtain ${\bar E}=\IF^\kappa$ for some cardinal $\kappa$. Since ${\bar T}$ is continuous, there is a finite subset $\lambda$ of $\kappa$ such that ${\bar T}\big(\{0\}^\lambda \times \IF^{\kappa\SM \lambda} \big)$ is contained in the unit ball $B_{{\bar L}}$ of ${\bar L}$. Taking into account that $B_{{\bar L}}$ contains no non-trivial linear subspaces we obtain that $\{0\}^\lambda \times \IF^{\kappa\SM \lambda}$ is contained in the kernel $\ker({\bar T})$ of ${\bar T}$. Therefore ${\bar T}[\IF^\kappa]={\bar T}[\IF^\lambda]$ is finite-dimensional. Thus also $T$ is finite-dimensional.\qed
\end{proof}

\begin{proof}[Proof of Proposition \ref{p:finite-rank-qp-summing}]
By Lemma \ref{l:weak-to-norm}, $T$ is finite-dimensional. Therefore there are $\chi_0,\dots,\chi_m\in E'$ and $y_0,\dots,y_m\in L$ such that $T(x)=\sum_{j=0}^m \langle\chi_j,x\rangle y_j$. Set $U:=\{\chi_0,\dots,\chi_m\}^\circ$, so $U\in\Nn_0^c(E_w)$. We show that $T$ is $U$-$(q,p)$-summing. To this end, by Proposition \ref{p:p-summing-vector}(i), it suffices to prove that each operator $S_j:E\to L$ defined by $S_j(x)= \langle\chi_j,x\rangle y_j$ is $U$-$(q,p)$-summing. Set $C:=\max\{\|y_0\|,\dots,\|y_m\|\}$. Then for every $n\in\w$ and each each $x_0,\dots,x_n\in U$, we obtain
\[
\Big( \sum_{i=0}^n \|S_j(x_i)\|^q\Big)^{1/q} \leq \Big( \sum_{i=0}^n \|S_j(x_i)\|^p\Big)^{1/p} \leq C \cdot \Big( \sum_{i=0}^n  |\langle\chi_j,x_i\rangle|^p\Big)^{1/p}  \leq C \cdot   \sup_{\chi\in U^\circ} \Big( \sum_{i=0}^n  |\langle\chi,x_i\rangle|^p\Big)^{1/p}
\]
which means that $S_j$ is $U$-$(q,p)$-summing.\qed
\end{proof}

As an application of the obtained results we shall show that each $(p,p)$-$(V^\ast)$ subset of a Banach space is relatively weakly compact.

\begin{theorem} \label{t:Banach-Vpp}
If $1\leq p<\infty$, then every Banach space $E$ has the property $V^\ast_{(p,p)}$.
\end{theorem}

\begin{proof}
Let $A$ be a $(p,p)$-$(V^\ast)$ subset of $E$. Then, by Theorem \ref{t:V*-p-summing}, the adjoint operator $T_A^\ast: E'_\beta\to \ell_\infty(A)$ is $(p,p)$-convergent. Therefore, by Corollary \ref{c:qp-summable-Banach}, $T_A^\ast$ is $p$-summing. Now we apply Theorem \ref{t:p-sum-weakly-compact} to get that $T_A^\ast$ is  weakly compact. Then, by the Gantmacher Theorem \ref{t:Gantmacher}, the operator $T_A:\ell_1(A)\to E$ is also weakly compact. Therefore $A=T_A(A)$ is a relatively weakly compact subset of $E$. Thus  $E$ has the property $V^\ast_{(p,p)}$.\qed
\end{proof}

\begin{remark} {\em
The condition $p<\infty$ in Theorem \ref{t:Banach-Vpp} is essential. Indeed, let $E=c_0$. We show that $B_{c_0}$ is an $(\infty,\infty)$-$(V^\ast)$ set. Fix a weakly $\infty$-summable sequence $\{\chi_n\}_{n\in\w}\in (c_0)'=\ell_1$. By the Schur property, we have $\|\chi_n\|\to 0$. Therefore $\sup\{ |\langle \chi_n,x\rangle|:x\in B_{c_0}\}=\|\chi_n\|\to 0$, and hence $B_{c_0}$ is an $(\infty,\infty)$-$(V^\ast)$ set. On the other hand, since the sequence $\{e_0+\cdots+e_n\}_{n\in\w}\subseteq B_{c_0}$ has no cluster points in the weak topology, $B_{c_0}$ is not relatively weakly compact. Thus $c_0$ does not have the property $V^\ast_{(\infty,\infty)}$.\qed}
\end{remark}

Below we give another application of the obtained results, cf. Proposition 2.1 of \cite{Del-Pin}.

\begin{theorem} \label{t:V*-p=q}
Let $1\leq p\leq q<\infty$, and let $E$ be a locally convex space such that $\ell_p[E'_\beta]$ is barrelled. Then every $(p,p)$-$(V^\ast)$ subset of $E$ is also a $(q,q)$-$(V^\ast)$ set.
\end{theorem}

\begin{proof}
Let $A$ be a  $(p,p)$-$(V^\ast)$ subset of $E$. Recall that Theorem \ref{t:V*-p-summing} states that a bounded subset $A$ of $E$ is  a $(p,p)$-$(V^\ast)$ set if and only if the adjoint operator $T_A^\ast: E'_\beta\to \ell_\infty(A)$ is $(p,p)$-convergent. Since $\ell_p[E'_\beta]$ is barrelled, Corollary \ref{c:qp-summable-Banach} implies that $T_A^\ast$ is $U$-$(p,p)$-summing for some $U\in\Nn_0^c(E)$. By the inclusion property Theorem \ref{t:p-summing-inclusion}, the operator $T_A^\ast$ is also $U$-$(q,q)$-summing, and hence, by Proposition \ref{p:qp-summing-norm}, $T_A^\ast$ is $(q,q)$-convergent. Once more applying  Theorem \ref{t:V*-p-summing} we obtain that $A$ is a $(q,q)$-$(V^\ast)$ set.\qed
\end{proof}


\section{References}  \label{literature}
\bibliographystyle{amsplain}

\end{document}